\documentclass{amsart}

\usepackage{amsfonts,mathrsfs}
\usepackage{amsmath,amssymb,amsthm, amscd, bm}
\usepackage{tikz}
\usetikzlibrary{arrows,calc,matrix,topaths,positioning,scopes,shapes,decorations,decorations.markings} 
\usepackage{dsfont}
\usepackage{epsfig}
\usepackage{tikz-cd}
\usepackage{hyperref}
\usepackage{fullpage}
\usepackage[foot]{amsaddr}
\usepackage{adjustbox}
\usepackage{enumitem}
\newcommand{\subscript}[2]{$#1 _ #2$}
\usepackage[utf8x]{inputenc}

\makeatletter

\renewcommand{\email}[2][]{%
  \ifx\emails\@empty\relax\else{\g@addto@macro\emails{,\space}}\fi%
  \@ifnotempty{#1}{\g@addto@macro\emails{\textrm{(#1)}\space}}%
  \g@addto@macro\emails{#2}%
}
\makeatother

\author{Julien Korinman}
\address{Department of Mathematics, Faculty of Science and Engineering, Waseda University,3-4-1 Ohkubo, Shinjuku-ku, Tokyo, 169-8555, Japan}
\email{julien.korinman@gmail.com}

\subjclass{$57$R$56$, $57$N$10$, $57$M$25$.}

\keywords{Stated skein algebras, quantum Teichm\"uller spaces, character varieties}

\newcommand{\quotient}[2]{{\raisebox{.2em}{$#1$}\left/\raisebox{-.2em}{$#2$}\right.}}

\newcommand{\sslash}{\mathbin{/\mkern-6mu/}}
\newcommand{\Hom}{\operatorname{Hom}}
\newcommand{\tr}{\operatorname{tr}}
\newcommand{\Tr}{\operatorname{Tr}}
\newcommand{\SL}{\operatorname{SL}}
\newcommand{\id}{id}
\newcommand{\Span}{\operatorname{Span}}
\newcommand{\End}{\operatorname{End}}
\newcommand{\GL}{\operatorname{GL}}
\newcommand{\Vect}{\operatorname{Vect}}
\newcommand{\Mod}{\operatorname{Mod}}
\newcommand{\qdim}{\operatorname{qdim}}

\newcommand{\heightexch}[3]{
	\begin{tikzpicture}[baseline=-0.4ex,scale=0.5, >=stealth]
	\draw [fill=gray!60,gray!45] (-.7,-.75)  rectangle (.4,.75)   ;
	\draw[#1] (0.4,-0.75) to (.4,.75);
	\draw[line width=1.2] (0.4,-0.3) to (-.7,-.3);
	\draw[line width=1.2] (0.4,0.3) to (-.7,.3);
	\draw (0.65,0.3) node {\scriptsize{$#2$}}; 
	\draw (0.65,-0.3) node {\scriptsize{$#3$}}; 
	\end{tikzpicture}
}
\newcommand{\heightcurve}{
\begin{tikzpicture}[baseline=-0.4ex,scale=0.5]
\draw [fill=gray!20,gray!45] (-.7,-.75)  rectangle (.4,.75)   ;
\draw[-] (0.4,-0.75) to (.4,.75);
\draw[line width=1.2] (-.7,-0.3) to (-.4,-.3);
\draw[line width=1.2] (-.7,0.3) to (-.4,.3);
\draw[line width=1.15] (-.4,0) ++(-90:.3) arc (-90:90:.3);
\end{tikzpicture}
}

\newcommand{\heightcurveright}{
\begin{tikzpicture}[baseline=-0.4ex,scale=0.5]
\draw [fill=gray!20,gray!45] (-.7,-.75)  rectangle (.4,.75)   ;
\draw[-] (-0.7,-0.75) to (-.7,.75);
\draw[line width=1.2] (0.1,-0.3) to (.4,-.3);
\draw[line width=1.2] (0.1,0.3) to (.4,.3);
\draw[line width=1.15] (.1,0) ++(90:.3) arc (90:270:.3);
\end{tikzpicture}
}

\newcommand{\heightexchright}[3]{
	\begin{tikzpicture}[baseline=-0.4ex,scale=0.5, >=stealth]
	\draw [fill=gray!60,gray!45] (-.7,-.75)  rectangle (.4,.75)   ;
	\draw[#1] (-0.7,-0.75) to (-0.7,.75);
	\draw[line width=1.2] (0.4,-0.3) to (-.7,-.3);
	\draw[line width=1.2] (0.4,0.3) to (-.7,.3);
	\draw (-1,0.3) node {\scriptsize{$#2$}}; 
	\draw (-1,-0.3) node {\scriptsize{$#3$}}; 
	\end{tikzpicture}
}

\begin{document}

\theoremstyle{plain}
\newtheorem{theorem}{Theorem}[section]
\newtheorem{proposition}[theorem]{Proposition}
\newtheorem{corollary}[theorem]{Corollary}
\newtheorem{lemma}[theorem]{Lemma}
\theoremstyle{definition}
\newtheorem{notations}[theorem]{Notations}
\newtheorem{convention}[theorem]{Convention}
\newtheorem{problem}[theorem]{Problem}
\newtheorem{definition}[theorem]{Definition}
\theoremstyle{remark}
\newtheorem{remark}[theorem]{Remark}
\newtheorem{conjecture}[theorem]{Conjecture}
\newtheorem{example}[theorem]{Example}
\newtheorem{strategy}[theorem]{Strategy}
\newtheorem{question}[theorem]{Question}

\title[Stated skein algebras and their representations]{Stated skein algebras and their representations}
%
%
%

\date{}
\maketitle


\begin{abstract} 
This is a survey on stated skein algebras and their representations. 
\end{abstract}

\tableofcontents

\section{Introduction}

Let $\Sigma$ be a compact oriented surface, $k$ be a (unital, associative) commutative ring and $A\in k^{\times}$ an invertible element. The \textit{Kauffman-bracket skein algebra} $\mathcal{S}_A(\Sigma)$ is the quotient of the $k$ module freely generated by embedded framed links $L\subset \Sigma\times (0,1)$ by the ideal generated by elements $L-L'$, for $L,L'$ two isotopic links, and by the following skein relations: 
  	\begin{equation*} 
\begin{tikzpicture}[baseline=-0.4ex,scale=0.5,>=stealth]	
\draw [fill=gray!45,gray!45] (-.6,-.6)  rectangle (.6,.6)   ;
\draw[line width=1.2,-] (-0.4,-0.52) -- (.4,.53);
\draw[line width=1.2,-] (0.4,-0.52) -- (0.1,-0.12);
\draw[line width=1.2,-] (-0.1,0.12) -- (-.4,.53);
\end{tikzpicture}
=A
\begin{tikzpicture}[baseline=-0.4ex,scale=0.5,>=stealth] 
\draw [fill=gray!45,gray!45] (-.6,-.6)  rectangle (.6,.6)   ;
\draw[line width=1.2] (-0.4,-0.52) ..controls +(.3,.5).. (-.4,.53);
\draw[line width=1.2] (0.4,-0.52) ..controls +(-.3,.5).. (.4,.53);
\end{tikzpicture}
+A^{-1}
\begin{tikzpicture}[baseline=-0.4ex,scale=0.5,rotate=90]	
\draw [fill=gray!45,gray!45] (-.6,-.6)  rectangle (.6,.6)   ;
\draw[line width=1.2] (-0.4,-0.52) ..controls +(.3,.5).. (-.4,.53);
\draw[line width=1.2] (0.4,-0.52) ..controls +(-.3,.5).. (.4,.53);
\end{tikzpicture}
\hspace{.5cm}
\text{ and }\hspace{.5cm}
\begin{tikzpicture}[baseline=-0.4ex,scale=0.5,rotate=90] 
\draw [fill=gray!45,gray!45] (-.6,-.6)  rectangle (.6,.6)   ;
\draw[line width=1.2,black] (0,0)  circle (.4)   ;
\end{tikzpicture}
= -(A^2+A^{-2}) 
\begin{tikzpicture}[baseline=-0.4ex,scale=0.5,rotate=90] 
\draw [fill=gray!45,gray!45] (-.6,-.6)  rectangle (.6,.6)   ;
\end{tikzpicture}
.
\end{equation*}

The product of two classes links  $[L_1]$ and $[L_2]$ is defined by  isotoping $L_1$ and $L_2$  in $\Sigma\times (1/2, 1) $ and $\Sigma\times (0, 1/2)$ respectively and then setting $[L_1]\cdot [L_2]:=[L_1\cup L_2]$.

Skein algebras have been introduced by Przytycki \cite{Przytycki_skein} and Turaev \cite{Tu88} at the end of the 80's as a tool to study the $\mathrm{SU}(2)$ Witten-Reshetikhin-Turaev topological quantum field theories (\cite{Wi2, RT}). Skein algebras appear in TQFTs through their finite dimensional representations. Such representations exist if and only if the parameter $A$ is a root of unity. 
 Despite the apparent simplicity of their definition, these algebras and their representations are quite hard to study. A major breakthrough was made by Bonahon and Wong (\cite{BonahonWongqTrace, BonahonWong1}) who introduced three key original concepts to understand skein algebras which are: 
\begin{enumerate}
\item the stated skein algebras, 
\item the quantum traces and
\item the Chebyshev-Frobenius morphisms.
\end{enumerate}
Since the introduction of these new ideas, the amount of papers on the subject has grown very fast and many problems which were out of reach recently are now affordable. The goal of the present survey is to provide a self-content introduction on these three concepts and their use including proofs when the latters are short and enlightening.  In addition to  review some recent results on the subject, we dress a list of open problems/questions. 
Our leading problem will be
\begin{problem}\label{problem_classification}
 Classify all finite dimensional weight representations of (stated) skein algebras when $A$ is a root of unity of odd order.
 \end{problem}
  As we shall see, this problem is deeply connected to the study of the Poisson geometry of relative character varieties, more precisely to the computation of their symplectic leaves. Here we choose the order of $A$ to be odd and restrict to weight representations for simplicity.  Actually, even for the bigon, one of the simplest marked surface, Problem \ref{problem_classification} is undecidable  so we will reformulate later a more reasonable version in Problem \ref{problem_classification2}.

 The paper is organised as follows. We first define stated skein algebras and review their fundamental properties such as their behaviour for the gluing and fusion operations, their triangular decomposition, their Chebyshev-Frobenius morphisms and their finite presentations. We then review Bonahon-Wong's quantum trace which permits to embed skein algebras into some quantum tori. The third section reviews the Poisson geometry of relative character varieties, in particular the classification of their symplectic leaves. In the last section we review three families of representations of skein algebras coming from modular TQFTs, non semi-simple TQFTs and quantum Teichm\"uller theory. We then present two important theorems concerning the representation theory of skein algebras which are
 the Frohman-L\^e-Kania-Bartoszynska Unicity representation theorem and Brown-Gordon's theory of Poisson orders. Putting everything together, we'll review some recent results of Ganev-Jordan-Safranov and solve Problem \ref{problem_classification} in some simple cases.

\vspace{2mm}
\paragraph{\textbf{Acknowledgments.}} 
This manuscript is an expanded version of a proceeding of the conference "Intelligence in low dimensional topology" held at the RIMS. The author warmly thanks T.Ohtsuki for inviting him at the conference and 
 S.Baseilhac, F.Bonahon,  F.Costantino,  L.Funar , T.Q.T. L\^e, J.March\'e, J.Murakami,  A.Quesney, P.Roche and R.Santharoubane for useful discussions on the subject. He acknowledges  support from the Japanese Society for Promotion of Science (JSPS) and the Centre National de la Recherche Scientifique (CNRS).

 \section{Stated skein algebras}
 
 \subsection{Marked surfaces vs punctured surfaces}
 The key concept behind stated skein algebras is the introduction of marked surfaces.

\begin{definition} A \textit{marked surface} $\mathbf{\Sigma}=(\Sigma, \mathcal{A})$ is a compact oriented surface $\Sigma$ (possibly with boundary) with a finite set $\mathcal{A}=\{a_i\}_i$ of orientation-preserving immersions $a_i : [0,1] \hookrightarrow \partial \Sigma$, \textit{named boundary arcs}, whose restrictions to $(0,1)$ are embeddings and whose interiors are pairwise disjoint. 

 An \textit{embedding} $f:(\Sigma, \mathcal{A}) \to (\Sigma', \mathcal{A}')$ of marked surfaces is a orientation-preserving proper embedding $f:\Sigma \to \Sigma'$ so that for each boundary arc $a \in \mathcal{A}$ there exists $a' \in \mathcal{A}$ such that $f\circ a$ is the restriction of $a'$ to some subinterval of $[0,1]$. When several boundary arcs $a_1, \ldots, a_n$ in $\mathbf{\Sigma}$ are mapped to the same boundary arc $b$ of $\mathbf{\Sigma}'$ we include in the definition of $f$ the datum of a total ordering of $\{a_1,\ldots, a_n\}$. 
 Marked surfaces with embeddings form a category $\mathrm{MS}$ with monoidal structure given by disjoint union.  
\end{definition}

  By abuse of notations, we will often denote by the same letter the embedding $a_i$ and its image $a_i((0,1)) \subset \partial \Sigma$ and both call them boundary arcs. We will also abusively identify $\mathcal{A}$ with the disjoint union $\bigsqcup_i a_i((0,1)) \subset \partial \Sigma$ of open intervals. The main interest in considering marked surfaces is that they have a natural gluing operation. Let $\mathbf{\Sigma}=(\Sigma, \mathcal{A})$ be a marked surface and $a,b\in \mathcal{A}$ two boundary arcs. Set $\Sigma_{a\#b} := \quotient{\Sigma}{a(t) \sim b(1-t)}$ and $\mathcal{A}_{a\#b}:= \mathcal{A}\setminus a\cup b$. The marked surface $\mathbf{\Sigma}_{a\#b}=(\Sigma_{a\#b}, \mathcal{A}_{a\#b})$ is said obtained from $\mathbf{\Sigma}$ by gluing $a$ and $b$. We say that $\mathbf{\Sigma}=(\mathbf{\Sigma}, \mathcal{A})$ is \textit{unmarked} if $\mathcal{A}=\emptyset$.

  A related concept is the notion of \textit{punctured surfaces}: a punctured surface is a pair $(S, \mathcal{P})$ where $S$ is a compact oriented surface and $\mathcal{P}\subset \Sigma$ a finite subset of punctures which non trivially intersects each connected component of $\partial \Sigma$. To a punctured surface, one associates a marked surface $(\Sigma, \mathcal{A})$ where $\Sigma$ is obtained from $S$ by blowing up each inner puncture and the boundary arcs in $\mathcal{A}$ are the connected components of $\partial S \setminus \mathcal{P}$.  Both notions marked surfaces and punctured surfaces are used in the literature and are essentially equivalent. On the one hand, marked surfaces have the advantage of having a natural notion of morphisms, which is hard to translate in the language of punctured surfaces, and generalize naturally to $3$-manifolds. On the other hand, punctured surfaces make much more natural the concept of triangulation. In this survey, we will use marked surfaces $\mathbf{\Sigma}=(\Sigma, \mathcal{A})$ but will call punctures the connected components of $\partial \Sigma \setminus \mathcal{A}$ and \textit{inner punctures} the unmarked connected components of $\partial \Sigma$ (so they are circles really).
  
  \begin{notations} Let us name some marked surfaces. Let $D(n)$ be a disc $\mathbb{D}^2$ with $n$ pairwise disjoint open subdiscs removed and let $\Sigma_{g,n}$ be an oriented connected surface of genus $g$ with $n$ boundary components. We will sometimes write $\Sigma_g=\Sigma_{g,0}$.
  \begin{enumerate}
  \item The $n^{th}$\textit{-punctured monogon} $\mathbf{m}_n=(D(n), \{a\})$ is $D(n)$ with one boundary arc in $\partial \mathbb{D}^2$.
  \item The $n^{th}$\textit{-punctured bigon} $\mathbb{D}_n= (D(n), \{a,b\})$ is $D(n)$ with two boundary arcs in $\partial \mathbb{D}^2$. We call $\mathbb{B}:=\mathbb{D}_0$ simply the \textit{bigon}.
  \item A \textit{triangle} $\mathbb{T}= (\mathbb{D}^2, \{a_1, a_2, a_3\})$ is a disc with three boundary arcs on its boundary. Its boundary arcs are called \textit{edges}.
  \item We denote by $\mathbf{\Sigma}_{g,n}^0=(\Sigma_{g,n+1}, \{a\} )$ the  surface $\Sigma_{g,n}$ with a single boundary arc in one of its only boundary component and by $\underline{\mathbf{\Sigma}}_{g,n}= (\Sigma_{g,n}, \mathcal{A})$ the surface $\Sigma_{g,n}$ with exactly one boundary arc in each of its boundary component (so $\mathbf{\Sigma}_{g,0}^0=\underline{\mathbf{\Sigma}}_{g,1}$). 
  \end{enumerate}
  
  \end{notations}
  
  \begin{definition}
  A marked surface $\mathbf{\Sigma}$ is \textit{triangulable} if it can be obtained from a finite disjoint union $\mathbf{\Sigma}_{\Delta}= \bigsqcup_i \mathbb{T}_i$ of triangles by gluing some pairs of edges. A \textit{triangulation} $\Delta$ is then the data of the disjoint union $\mathbf{\Sigma}_{\Delta}$ together with the set of glued pair of edges.  
  \end{definition}
  The connected components $\mathbb{T}_i$ of $\mathbf{\Sigma}_{\Delta}$ are called \textit{faces} and their set is denoted $F(\Delta)$. The image in $\Sigma$ of the edges of the faces $\mathbb{T}_i$ are called \textit{edges} of $\Delta$ and their set is denoted $\mathcal{E}(\Delta)$. Note that each boundary arc is an edge in $\mathcal{E}(\Delta)$; the elements of the complementary $\mathring{\mathcal{E}}(\Delta):= \mathcal{E}(\Delta) \setminus \mathcal{A}$ are called \textit{inner edges}. The punctures of $\mathbf{\Sigma}$ are the vertices of the triangulation; in particular this set is non empty so $\Sigma$ cannot be closed.
  
  Note that the only connected marked surfaces which are not triangulable are: the bigon $\mathbb{B}$, the monogon $\mathbf{m}_0$ and the unmarked surfaces $\Sigma_{g,0}$, $\Sigma_{0,1}$ and $ \Sigma_{0,2}$.

  \subsection{Definition and first properties of stated skein algebras}

   A \textit{tangle}  is a  compact framed, properly embedded $1$-dimensional manifold $T\subset \Sigma \times (0,1)$ such that for every point of $\partial T \subset \mathcal{A}\times (0,1)$ the framing is parallel to the $(0,1)$ factor and points to the direction of $1$.   The \textit{height} of $(v,h)\in \Sigma \times (0,1)$ is $h$.  If $a$ is a boundary arc and $T$ a tangle, we impose that no two points in $\partial_aT:= \partial T \cap a\times(0,1)$  have the same heights, hence the set $\partial_aT$ is totally ordered by the heights. Two tangles are isotopic if they are isotopic through the class of tangles that preserve the boundary height orders. By convention, the empty set is a tangle only isotopic to itself.
 
\vspace{2mm}

\par Let $\pi : \Sigma\times (0,1)\rightarrow \Sigma$ be the projection with $\pi(v,h)=v$. A tangle $T$ is in \textit{generic position} if for each of its points, the framing is parallel to the $(0,1)$ factor and points in the direction of $1$ and is such that $\pi_{| T} : T\rightarrow \Sigma$ is an immersion with at most transversal double points in the interior of $\Sigma$. Every tangle is isotopic to a tangle in generic position. We call \textit{diagram}  the image $D=\pi(T)$ of a tangle in generic position, together with the over/undercrossing information at each double point. An isotopy class of diagram $D$ together with a total order of $\partial_a D:=\partial D\cap a$ for each boundary arc $a$, define uniquely an isotopy class of tangle. When choosing an orientation $\mathfrak{o}(a)$ of a boundary arc $a$ and a diagram $D$, the set $\partial_aD$ receives a natural order by setting that the points are increasing when going in the direction of $\mathfrak{o}(a)$. We will represent tangles by drawing a diagram and an orientation (an arrow) for each boundary arc, as in Figure \ref{fig_statedtangle}. When a boundary arc $a$ is oriented we assume that $\partial_a D$ is ordered according to the orientation. A \textit{state} of a tangle is a map $s:\partial T \rightarrow \{-, +\}$. A pair $(T,s)$ is called a \textit{stated tangle}. We define a \textit{stated diagram} $(D,s)$ in a similar manner.

\begin{figure}[!h] 
\centerline{\includegraphics[width=6cm]{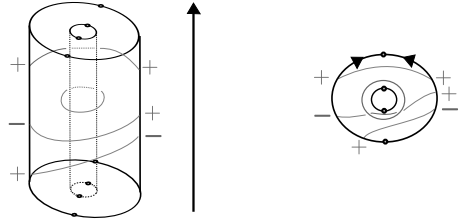} }
\caption{On the left: a stated tangle. On the right: its associated diagram. The arrows represent the height orders. } 
\label{fig_statedtangle} 
\end{figure} 

 \vspace{2mm}
\par  Let  $A^{1/2} \in k^{\times}$ an invertible element and denote by $A$ its square. Stated skein algebras were first introduced by Bonahon and Wong in \cite{BonahonWongqTrace}. The version we will present here is a refinement due to L\^e \cite{LeStatedSkein}.

\begin{definition}\label{def_stated_skein}\cite{LeStatedSkein} 
  The \textit{stated skein algebra}  $\mathcal{S}_{A}(\mathbf{\Sigma})$ is the  free $k$-module generated by isotopy classes of stated tangles in $\Sigma\times (0, 1)$ modulo the following relations \eqref{eq: skein 1} and \eqref{eq: skein 2}, 
  	\begin{equation}\label{eq: skein 1} 
\begin{tikzpicture}[baseline=-0.4ex,scale=0.5,>=stealth]	
\draw [fill=gray!45,gray!45] (-.6,-.6)  rectangle (.6,.6)   ;
\draw[line width=1.2,-] (-0.4,-0.52) -- (.4,.53);
\draw[line width=1.2,-] (0.4,-0.52) -- (0.1,-0.12);
\draw[line width=1.2,-] (-0.1,0.12) -- (-.4,.53);
\end{tikzpicture}
=A
\begin{tikzpicture}[baseline=-0.4ex,scale=0.5,>=stealth] 
\draw [fill=gray!45,gray!45] (-.6,-.6)  rectangle (.6,.6)   ;
\draw[line width=1.2] (-0.4,-0.52) ..controls +(.3,.5).. (-.4,.53);
\draw[line width=1.2] (0.4,-0.52) ..controls +(-.3,.5).. (.4,.53);
\end{tikzpicture}
+A^{-1}
\begin{tikzpicture}[baseline=-0.4ex,scale=0.5,rotate=90]	
\draw [fill=gray!45,gray!45] (-.6,-.6)  rectangle (.6,.6)   ;
\draw[line width=1.2] (-0.4,-0.52) ..controls +(.3,.5).. (-.4,.53);
\draw[line width=1.2] (0.4,-0.52) ..controls +(-.3,.5).. (.4,.53);
\end{tikzpicture}
\hspace{.5cm}
\text{ and }\hspace{.5cm}
\begin{tikzpicture}[baseline=-0.4ex,scale=0.5,rotate=90] 
\draw [fill=gray!45,gray!45] (-.6,-.6)  rectangle (.6,.6)   ;
\draw[line width=1.2,black] (0,0)  circle (.4)   ;
\end{tikzpicture}
= -(A^2+A^{-2}) 
\begin{tikzpicture}[baseline=-0.4ex,scale=0.5,rotate=90] 
\draw [fill=gray!45,gray!45] (-.6,-.6)  rectangle (.6,.6)   ;
\end{tikzpicture}
;
\end{equation}

\begin{equation}\label{eq: skein 2} 
\begin{tikzpicture}[baseline=-0.4ex,scale=0.5,>=stealth]
\draw [fill=gray!45,gray!45] (-.7,-.75)  rectangle (.4,.75)   ;
\draw[->] (0.4,-0.75) to (.4,.75);
\draw[line width=1.2] (0.4,-0.3) to (0,-.3);
\draw[line width=1.2] (0.4,0.3) to (0,.3);
\draw[line width=1.1] (0,0) ++(90:.3) arc (90:270:.3);
\draw (0.65,0.3) node {\scriptsize{$+$}}; 
\draw (0.65,-0.3) node {\scriptsize{$+$}}; 
\end{tikzpicture}
=
\begin{tikzpicture}[baseline=-0.4ex,scale=0.5,>=stealth]
\draw [fill=gray!45,gray!45] (-.7,-.75)  rectangle (.4,.75)   ;
\draw[->] (0.4,-0.75) to (.4,.75);
\draw[line width=1.2] (0.4,-0.3) to (0,-.3);
\draw[line width=1.2] (0.4,0.3) to (0,.3);
\draw[line width=1.1] (0,0) ++(90:.3) arc (90:270:.3);
\draw (0.65,0.3) node {\scriptsize{$-$}}; 
\draw (0.65,-0.3) node {\scriptsize{$-$}}; 
\end{tikzpicture}
=0,
\hspace{.2cm}
\begin{tikzpicture}[baseline=-0.4ex,scale=0.5,>=stealth]
\draw [fill=gray!45,gray!45] (-.7,-.75)  rectangle (.4,.75)   ;
\draw[->] (0.4,-0.75) to (.4,.75);
\draw[line width=1.2] (0.4,-0.3) to (0,-.3);
\draw[line width=1.2] (0.4,0.3) to (0,.3);
\draw[line width=1.1] (0,0) ++(90:.3) arc (90:270:.3);
\draw (0.65,0.3) node {\scriptsize{$+$}}; 
\draw (0.65,-0.3) node {\scriptsize{$-$}}; 
\end{tikzpicture}
=A^{-1/2}
\begin{tikzpicture}[baseline=-0.4ex,scale=0.5,>=stealth]
\draw [fill=gray!45,gray!45] (-.7,-.75)  rectangle (.4,.75)   ;
\draw[-] (0.4,-0.75) to (.4,.75);
\end{tikzpicture}
\hspace{.1cm} \text{ and }
\hspace{.1cm}
A^{1/2}
\heightexch{->}{-}{+}
- A^{5/2}
\heightexch{->}{+}{-}
=
\heightcurve.
\end{equation}
The product of two classes of stated tangles $[T_1,s_1]$ and $[T_2,s_2]$ is defined by  isotoping $T_1$ and $T_2$  in $\Sigma\times (1/2, 1) $ and $\Sigma\times (0, 1/2)$ respectively and then setting $[T_1,s_1]\cdot [T_2,s_2]=[T_1\cup T_2, s_1\cup s_2]$. Figure \ref{fig_product} illustrates this product.
\end{definition}
\par For an unmarked surface, $\mathcal{S}_{A}(\Sigma, \emptyset)$ coincides with the usual Kauffman-bracket skein algebra.

\begin{figure}[!h] 
\centerline{\includegraphics[width=8cm]{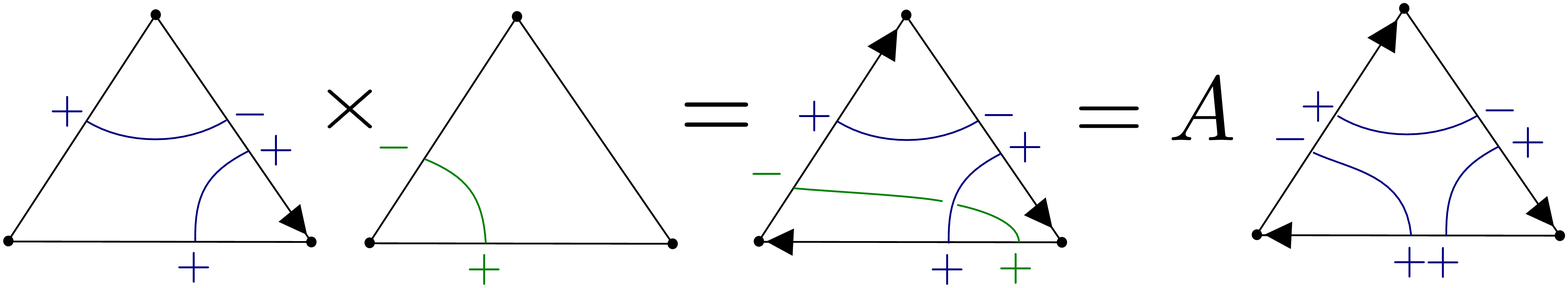} }
\caption{An illustration of the product in stated skein algebras.} 
\label{fig_product} 
\end{figure} 

\textbf{Functoriality}

Now consider an embedding $f:\mathbf{\Sigma}_1\to \mathbf{\Sigma}_2$ of marked surfaces and define a proper embedding $\widetilde{f} : \Sigma_1\times (0,1) \to \Sigma_2\times (0,1)$ such that: $(1)$ $\widetilde{f}(x,t)=(f(x), \varphi(x,t))$ for $\varphi$ a smooth map and $(2)$ if $a_1, a_2$ are two boundary arcs of $\mathbf{\Sigma}_1$ mapped to the same boundary arc $b$ of $\mathbf{\Sigma}_2$ and the ordering of $f$ is $a_1< a_2$, then for all $x_1\in a_1, x_2\in a_2, t_1, t_2 \in (0,1)$ one has $\varphi(x_1,t_1)<\varphi(x_2, t_2)$. 
It is an easy consequence of the definition that the formula $f_*([T,s]):= (\widetilde{f}(T), s\circ \widetilde{f}^{-1})$ defines a morphism of algebras $f_* : \mathcal{S}_A(\mathbf{\Sigma}_1) \to \mathcal{S}_A(\mathbf{\Sigma}_2)$ independent on the choice of $\widetilde{f}$ and that the assignment $\mathbf{\Sigma}\to \mathcal{S}_A(\mathbf{\Sigma})$ defines a symmetric monoidal functor 
$$ \mathcal{S}_A : \mathrm{MS} \to \mathrm{Alg}_k.$$
Here is an alternative easier $2$-dimensional definition of $f_*$. Suppose we have fixed some orientations $\mathfrak{o_1}, \mathfrak{o_2}$ of the boundary arcs of $\mathbf{\Sigma}_1$ and $\mathbf{\Sigma}_2$ and that $f: \Sigma_1 \to \Sigma_2$ is a proper oriented embedding sending $\mathcal{A}_1$ to $\mathcal{A}_2$ in such a way that it preserves the orientations of the boundary arcs. If $a_1, a_2 \in \mathcal{A}_1$ are sent to the same boundary arc $b\in \mathcal{A}_2$, they are naturally ordered by  $<_b$, so $f$ together with $\mathfrak{o_1}, \mathfrak{o}_2$ defines a morphism in $\mathrm{MS}$ and we can choose $\widetilde{f}:= f\times \id$, so $f_*$ is defined on stated diagrams by $f_*([D,s]):= [f(D), s\circ f^{-1}]$. 
\vspace{2mm}

\par \textbf{Bases}

The first theorem we state is the existence of bases for stated skein algebras. Whereas proving the existence of bases for usual Kauffman-bracket skein algebras (given by multicurves) is an easy exercise, the construction of bases for marked surfaces is a highly non trivial result based on the Diamond lemma.

A closed component of a diagram $D$ is trivial if it bounds an embedded disc in $\Sigma$. An open component of $D$ is trivial if it can be isotoped, relatively to its boundary, inside some boundary arc. A diagram is \textit{simple} if it has neither double point nor trivial component. By convention, the empty set is a simple diagram. Let $\mathfrak{o}$ denote an arbitrary orientation of the boundary arcs of $\mathbf{\Sigma}$. For each boundary arc $a$ we write $<_{\mathfrak{o}}$ the induced total order on $\partial_a D$. A state $s: \partial D \rightarrow \{ - , + \}$ is $\mathfrak{o}-$\textit{increasing} if for any boundary arc $a$ and any two points $x,y \in \partial_a D$, then $x<_{\mathfrak{o}} y$ implies $s(x)< s(y)$, with the convention $- < +$. 

\begin{definition}\label{def_basis}
 We denote by $\mathcal{B}^{\mathfrak{o}}\subset \mathcal{S}_{A}(\mathbf{\Sigma})$ the set of classes of stated diagrams $(D,s)$ such that $D$ is simple and $s$ is $\mathfrak{o}$-increasing. 
\end{definition}

\begin{theorem}\label{theorem_basis}(L\^e \cite[Theorem $2.11$]{LeStatedSkein})  The set $\mathcal{B}^{\mathfrak{o}}$ is a  basis of $\mathcal{S}_{A}(\mathbf{\Sigma})$. \end{theorem}

The basis $\mathcal{B}^{\mathfrak{o}}$ is independent on the choice of the ring $k$ and of $A^{1/2}\in k^{\times}$, in particular when $k=\mathbb{Z}[A^{\pm 1/2}]$, the corresponding stated skein algebra is flat. This fact will have deep consequences when defining deformation quantizations of relative character varieties.
\vspace{2mm}
\par \textbf{Off puncture ideal} An easy but useful consequence of Theorem \ref{theorem_basis} is the following. Suppose that $\mathbf{\Sigma}'$ is obtained from $\mathbf{\Sigma}$ by removing a puncture $p$, that is either by filling a closed unmarked component of $\partial \Sigma$ (inner puncture) or fusioning two adjacent boundary arcs (boundary puncture) to a single one (in which case we need to order the two boundary arcs). Then the inclusion $\Sigma \subset \Sigma'$ defines  a morphism $f: \mathbf{\Sigma} \to \mathbf{\Sigma}'$. Since any diagram in $\Sigma'$ is isotopic to a diagram in $\Sigma'$, the morphism $f_* : \mathcal{S}_A(\mathbf{\Sigma}) \to \mathcal{S}_A(\mathbf{\Sigma}')$ is surjective.
The \textit{off-puncture ideal} is $\mathcal{I}_p:= \mathrm{ker} (f_*)$, so we have an exact sequence:
\begin{equation}\label{eq_off_puncture}
0 \to \mathcal{I}_p \to \mathcal{S}_A(\mathbf{\Sigma}) \xrightarrow{f_*} \mathcal{S}_A(\mathbf{\Sigma}') \to 0.
\end{equation}

\begin{proposition}\label{prop_offpuncture}(K.-Quesney \cite[Proposition $2.18$]{KojuQuesneyClassicalShadows} )
The off puncture ideal $\mathcal{I}_p$ is generated by the elements $[D_1,s_1]- [D_2,s_2]$, where $(D_1,s_1)$ and $(D_2,s_2)$ are two connected stated diagrams (so closed curves or stated arcs) in $\Sigma$ which are isotopic in $\Sigma'$. 
\end{proposition}

\vspace{2mm}
\par 
\textbf{Some useful skein relations} 
Let us state some useful skein relations, which are direct consequences of the definition. We first define some matrices:
$$ C= \begin{pmatrix} C_+^+ & C_-^+ \\ C_+^- & C_-^- \end{pmatrix} := \begin{pmatrix} 0 & A^{-1/2} \\ -A^{-5/2} & 0 \end{pmatrix}. \mbox{ Therefore }C^{-1}= -A^3 C =  \begin{pmatrix} 0 & -A^{5/2} \\ A^{1/2} & 0 \end{pmatrix}.$$

$$ \mathscr{R} =  
\begin{pmatrix}
 \mathscr{R}_{++}^{++} & \mathscr{R}_{+-}^{++} &\mathscr{R}_{-+}^{++} &\mathscr{R}_{--}^{++} \\
\mathscr{R}_{++}^{+-} &\mathscr{R}_{+-}^{+-} &\mathscr{R}_{-+}^{+-} &\mathscr{R}_{--}^{+-}  \\
\mathscr{R}_{++}^{-+} &\mathscr{R}_{+-}^{-+} &\mathscr{R}_{-+}^{-+} &\mathscr{R}_{--}^{-+} \\
\mathscr{R}_{++}^{--} &\mathscr{R}_{+-}^{--} &\mathscr{R}_{-+}^{--} &\mathscr{R}_{--}^{--} 
  \end{pmatrix}
:= \begin{pmatrix} A & 0 & 0 & 0 \\ 0 & 0 &A^{-1} & 0 \\ 0 & A^{-1} & A-A^{-3} & 0 \\ 0 & 0 & 0 & A \end{pmatrix}, \mbox{so } 
\mathscr{R}^{-1}  =
\begin{pmatrix} A^{-1} & 0 & 0 & 0 \\ 0 & A^{-1} - A^3 & A & 0 \\ 0 & A & 0 & 0 \\ 0 & 0 & 0 & A^{-1}\end{pmatrix}.
$$

\par $\bullet$  \textit{The trivial arc relations:} 

\begin{equation}\label{trivial_arc_rel}
\begin{tikzpicture}[baseline=-0.4ex,scale=0.5,>=stealth]
\draw [fill=gray!45,gray!45] (-.7,-.75)  rectangle (.4,.75)   ;
\draw[->] (0.4,-0.75) to (.4,.75);
\draw[line width=1.2] (0.4,-0.3) to (0,-.3);
\draw[line width=1.2] (0.4,0.3) to (0,.3);
\draw[line width=1.1] (0,0) ++(90:.3) arc (90:270:.3);
\draw (0.65,0.3) node {\scriptsize{$i$}}; 
\draw (0.65,-0.3) node {\scriptsize{$j$}}; 
\end{tikzpicture}
= C^i_j 
\hspace{.2cm}
\begin{tikzpicture}[baseline=-0.4ex,scale=0.5,>=stealth]
\draw [fill=gray!45,gray!45] (-.7,-.75)  rectangle (.4,.75)   ;
\draw[-] (0.4,-0.75) to (.4,.75);
\end{tikzpicture}
, \hspace{.4cm}
\begin{tikzpicture}[baseline=-0.4ex,scale=0.5,>=stealth]
\draw [fill=gray!45,gray!45] (-.7,-.75)  rectangle (.4,.75)   ;
\draw[->] (-0.7,-0.75) to (-.7,.75);
\draw[line width=1.2] (-0.7,-0.3) to (-0.3,-.3);
\draw[line width=1.2] (-0.7,0.3) to (-0.3,.3);
\draw[line width=1.15] (-.4,0) ++(-90:.3) arc (-90:90:.3);
\draw (-0.9,0.3) node {\scriptsize{$i$}}; 
\draw (-0.9,-0.3) node {\scriptsize{$j$}}; 
\end{tikzpicture}
=(C^{-1})^i_j 
\hspace{.2cm}
\begin{tikzpicture}[baseline=-0.4ex,scale=0.5,>=stealth]
\draw [fill=gray!45,gray!45] (-.7,-.75)  rectangle (.4,.75)   ;
\draw[-] (-0.7,-0.75) to (-0.7,.75);
\end{tikzpicture}.
\end{equation}

\par $\bullet$  \textit{The cutting arc relations:}

\begin{equation}\label{cutting_arc_rel}
\heightcurveright
= \sum_{i,j = \pm} C^i_j
 \hspace{.2cm} 
\heightexchright{->}{i}{j}
, \hspace{.4cm} 
\heightcurve =
\sum_{i,j = \pm} (C^{-1})_j^i
\hspace{.2cm}
\heightexch{->}{i}{j}
\end{equation}.

\par $\bullet$ \textit{The height exchange relations:}

\begin{equation}\label{height_exchange_rel}
\heightexch{->}{i}{j}= \adjustbox{valign=c}{\includegraphics[width=0.9cm]{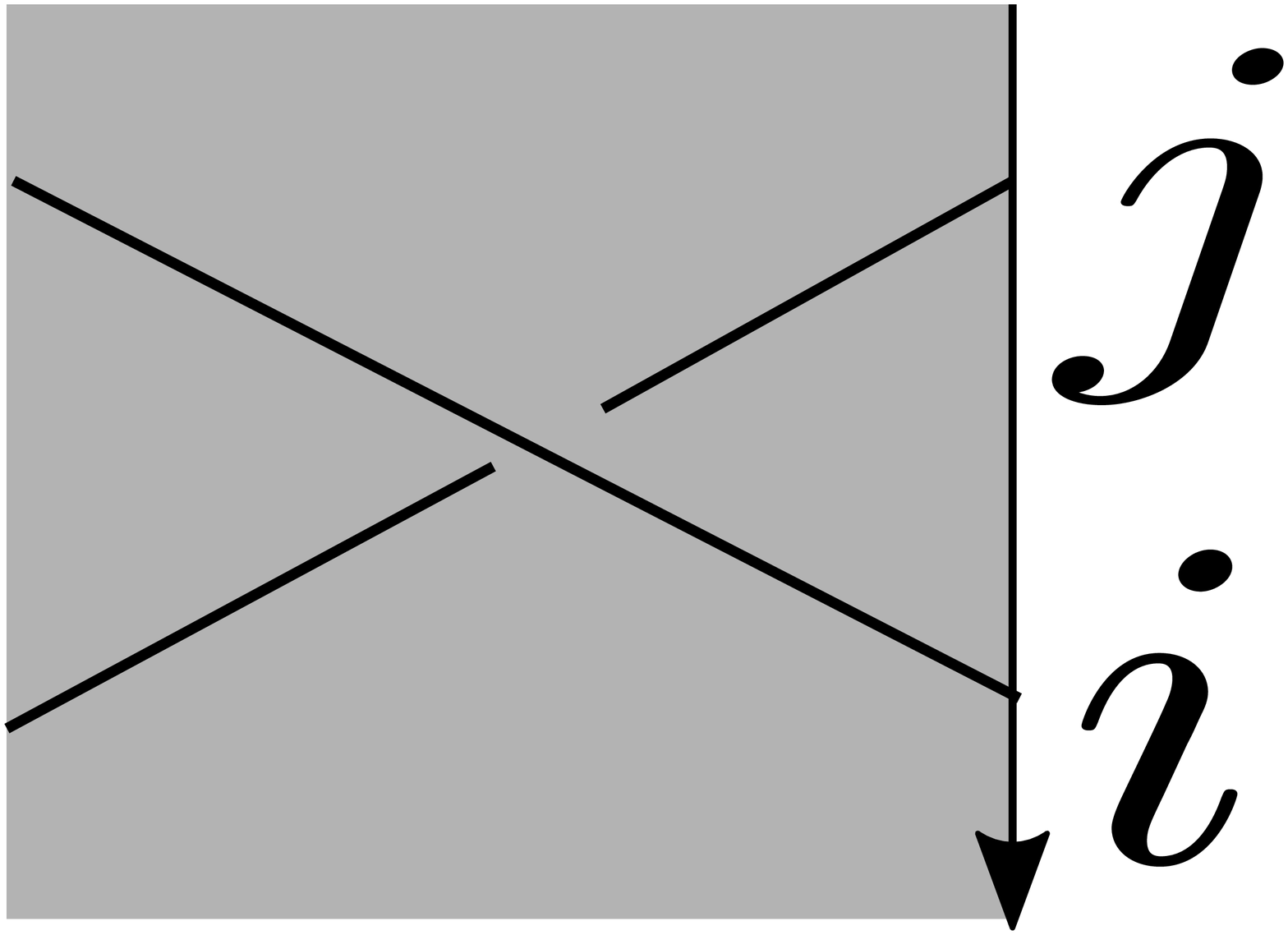}} =  \sum_{k,l = \pm} \mathscr{R}_{i j}^{k l} 
 \hspace{.2cm} 
 \heightexch{<-}{l}{k}
 , \hspace{.4cm}
 \heightexch{<-}{j}{i} = \adjustbox{valign=c}{\includegraphics[width=0.9cm]{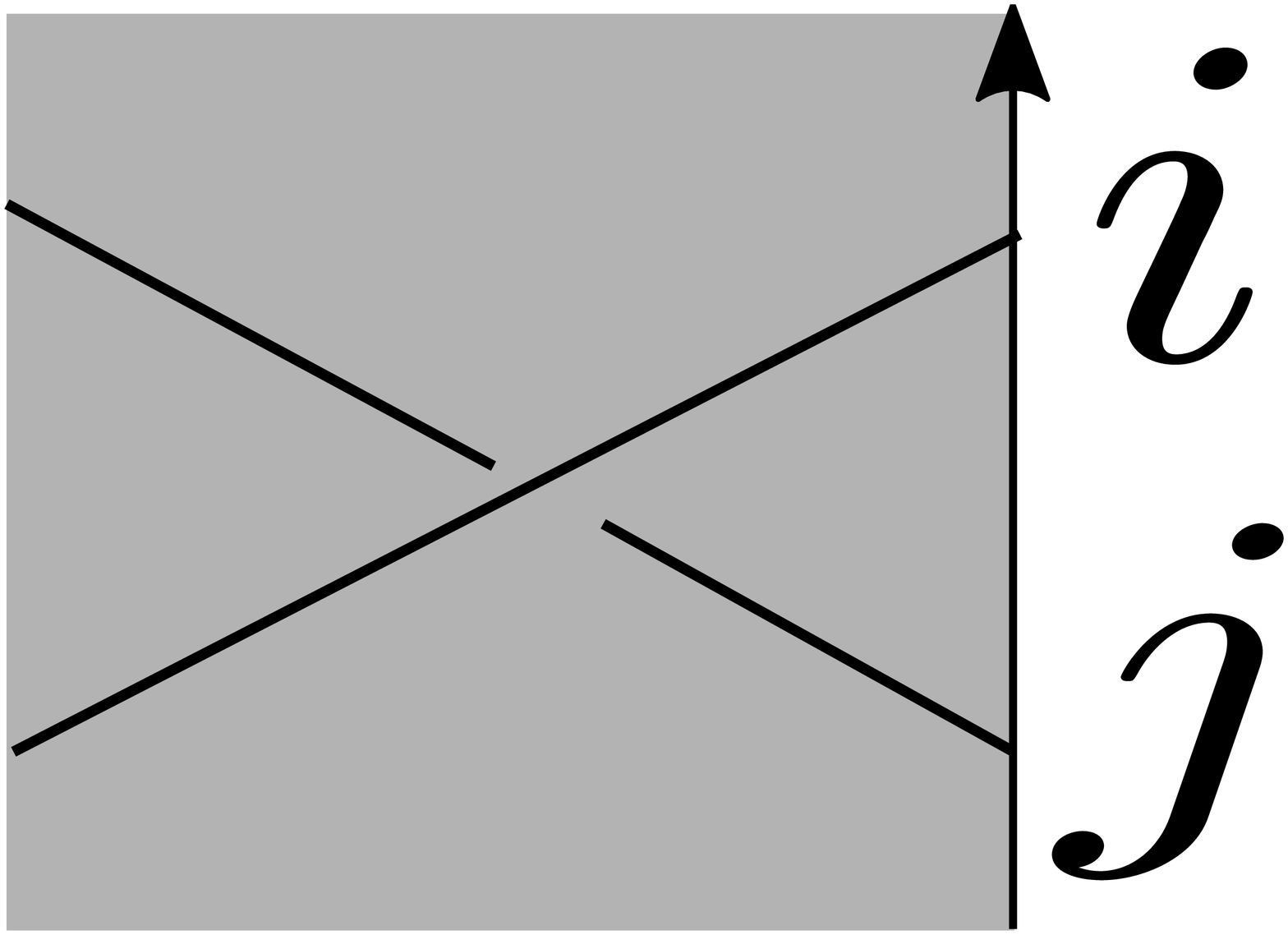}} = \sum_{k,l = \pm} (\mathscr{R}^{-1})_{i j}^{k l} 
 \hspace{.2cm}
 \heightexch{->}{k}{l}.
\end{equation}

\vspace{2mm}
\par 
\textbf{Graduations} 
Fix an orientation $\mathfrak{o}$ of the boundary arcs of $\mathbf{\Sigma}$. A simple diagram $D$ induces a relative homology class $[D]\in \mathrm{H}_1(\Sigma, \mathcal{A}; \mathbb{Z}/2\mathbb{Z})$. 

\begin{definition}
For $\chi \in \mathrm{H}_1(\Sigma, \mathcal{A}; \mathbb{Z}/2\mathbb{Z})$, we define the subspace:
$$ \mathcal{S}_A^{(\chi)}(\mathbf{\Sigma}):= \mathrm{Span} \left( [D,s]\in \mathcal{B}^{\mathfrak{o}} | [D]=\chi \right) \subset \mathcal{S}_A(\mathbf{\Sigma}).$$
\end{definition}
It follows from the defining skein relations \eqref{eq: skein 1} and \eqref{eq: skein 2} that $ \mathcal{S}_A^{(\chi)}(\mathbf{\Sigma})$ does not depend on the choice of $\mathfrak{o}$ and that $\mathcal{S}_A(\mathbf{\Sigma})= \oplus_{\chi \in \mathrm{H}_1(\Sigma, \mathcal{A}; \mathbb{Z}/2\mathbb{Z})}  \mathcal{S}_A^{(\chi)}(\mathbf{\Sigma})$ is a graded algebra, \textit{i.e.} that $ \mathcal{S}_A^{(\chi_1)}(\mathbf{\Sigma})\cdot  \mathcal{S}_A^{(\chi_2)}(\mathbf{\Sigma})\subset  \mathcal{S}_A^{(\chi_1+\chi_2)}(\mathbf{\Sigma})$.

\vspace{2mm}
\par 
\textbf{Reduced stated skein algebras} 

Let $\mathbf{\Sigma}$ a marked surface and $p$ a boundary puncture between two consecutive boundary arcs $a$ and $b$ on the same boundary component $\partial$ of $\partial \Sigma$.
The orientation of $\Sigma$ induces an orientation of $\partial$ so a cyclic ordering of the elements of $\partial \cap \mathcal{A}$ we suppose that $b$ is followed by $a$ in this ordering.  We denote by $\alpha(p)$ an arc with one endpoint $v_a\in a$ and one endpoint $v_b \in b$ such that $\alpha(p)$ can be isotoped inside $\partial$. Let $\alpha(p)_{ij}\in \mathcal{S}_A(\mathbf{\Sigma})$ be the class of the stated arc $(\alpha(p), s)$ where $s(v_a)=i$ and $s(v_b)=j$. The following definition will be justified in the next section when studying the quantum trace.

\begin{definition}
We call \textit{bad arc} associated to $p$ the element $\alpha(p)_{-+}\in \mathcal{S}_A(\mathbf{\Sigma})$ (see Figure \ref{fig_bad_arc}). The \textit{reduced stated skein algebra} $\mathcal{S}^{red}_A(\mathbf{\Sigma})$ is the quotient of $\mathcal{S}_A(\mathbf{\Sigma})$ by the ideal generated by all bad arcs.
\end{definition}

\begin{figure}[!h] 
\centerline{\includegraphics[width=4cm]{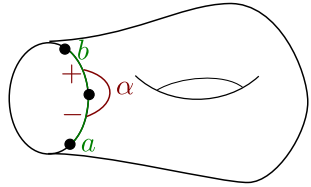} }
\caption{A bad arc.} 
\label{fig_bad_arc} 
\end{figure} 

The orientation of $\Sigma$ induces an orientation $\mathfrak{o}^+$ of the boundary arcs of $\mathcal{A}$. Let $\overline{\mathcal{B}}^{\mathfrak{o}^+} \subset \mathcal{S}_A^{red}(\mathbf{\Sigma})$ be the set of classes $[D,s]^{\mathfrak{o}^+}$ where $D$ is simple, $s$ is $\mathfrak{o}^+$-increasing and $(D,s)$ does not contain any bad arc.

\begin{theorem}(Costantino-Le \cite[Theorem $7.1$]{CostantinoLe19})
The set $\overline{\mathcal{B}}^{\mathfrak{o}^+} $ is a basis of $\mathcal{S}^{red}_A(\mathbf{\Sigma})$.
\end{theorem}

  \subsection{The gluing/cutting/excision formula and triangular decomposition}
  The main interest in extending skein algebras to marked surfaces is their behaviour for the gluing operation that we now describe.

  Let $a$, $b$ be two distinct boundary arcs of $\mathbf{\Sigma}$, denote by $\pi : \Sigma\rightarrow \Sigma_{a\#b}$ the projection and $c:=\pi(a)=\pi(b)$. Let $(T_0, s_0)$ be a stated framed tangle of $\Sigma_{a\#b}\times (0,1)$ transversed to $c\times (0,1)$ and such that the heights of the points of $T_0 \cap c\times (0,1)$ are pairwise distinct and the framing of the points of $T_0 \cap c\times (0,1)$ is vertical towards $1$. Let $T\subset \Sigma \times (0,1)$ be the framed tangle obtained by cutting $T_0$ along $c$. 
Any two states $s_a : \partial_a T \rightarrow \{-,+\}$ and $s_b : \partial_b T \rightarrow \{-,+\}$ give rise to a state $(s_a, s, s_b)$ on $T$. 
Both the sets $\partial_a T$ and $\partial_b T$ are in canonical bijection with the set $T_0\cap c$ by the map $\pi$. Hence the two sets of states $s_a$ and $s_b$ are both in canonical bijection with the set $\mathrm{St}(c):=\{ s: c\cap T_0 \rightarrow \{-,+\} \}$. 

\begin{definition}\label{def_gluing_map}
Let $\theta_{a\#b}: \mathcal{S}_{A}(\mathbf{\Sigma}_{a\#b}) \rightarrow \mathcal{S}_{A}(\mathbf{\Sigma})$ be the linear map given, for any $(T_0, s_0)$ as above, by: 
$$ \theta_{a\#b} \left( [T_0,s_0] \right) := \sum_{s \in \mathrm{St}(c)} [T, (s, s_0 , s) ].$$
\end{definition}

\begin{figure}[!h] 
\centerline{\includegraphics[width=8cm]{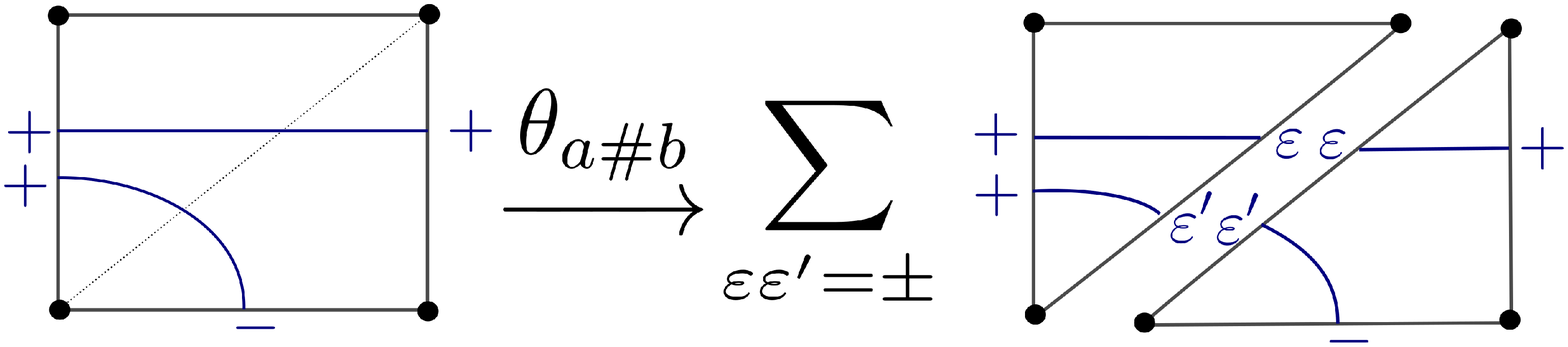} }
\caption{An illustration of the gluing map $\theta_{a\#b}$.} 
\label{fig_gluingmap} 
\end{figure} 

\begin{theorem}\label{theorem_gluing}( L\^e \cite[Theorem $3.1$]{LeStatedSkein})
 The linear map $\theta_{a\#b}: \mathcal{S}_A(\mathbf{\Sigma}_{a\#b}) \rightarrow \mathcal{S}_A(\mathbf{\Sigma})$ is an injective morphism of algebras. Moreover the gluing operation is coassociative in the sense that if $a,b,c,d$ are four distinct boundary arcs, then we have $\theta_{a\#b} \circ \theta_{c\#d} = \theta_{c\#d} \circ \theta_{a\#b}$.
\end{theorem}

Consider the bigon $\mathbb{B}$, that is a disc with two boundary arcs, say $a_L$ and $a_R$. Take two copies $\mathbb{B}\bigsqcup \mathbb{B}'$: when gluing $a_R$ to $a_L'$ we get another bigon, so we have a map:
$$ \Delta:= \theta_{a_R\#a_L'} : \mathcal{S}_A(\mathbb{B})^{\otimes 2} \to \mathcal{S}_A(\mathbb{B}).$$
The coassociativity of Theorem \ref{theorem_gluing} implies that $\Delta$ is coassociative. Denote by $\alpha_{ij} \in \mathcal{S}_A(\mathbb{B})$ the class of a single arc $\alpha$ connecting $a_L$ to $a_R$ with state $i$ on $\alpha\cap a_L$ and $j$ on $\alpha \cap a_R$. L\^e proved in \cite{LeStatedSkein} that $\mathcal{S}_A(\mathbb{B})$ is generated by the $\alpha_{ij}$ together with the following relations, where we state $q:=A^2$:
\begin{align}\label{relbigone}
\alpha_{++}\alpha_{+-} &= q^{-1}\alpha_{+-}\alpha_{++} & \alpha_{++}\alpha_{-+}&=q^{-1}\alpha_{-+}\alpha_{++}
\\ \alpha_{--}\alpha_{+-} &= q\alpha_{+-}\alpha_{--} & \alpha_{--}\alpha_{-+}&=q\alpha_{-+}\alpha_{--}
\\ \alpha_{++}\alpha_{--}&=1+q^{-1}\alpha_{+-}\alpha_{-+} &  \alpha_{--}\alpha_{++}&=1 + q\alpha_{+-}\alpha_{-+} 
\\ \alpha_{-+}\alpha_{+-}&=\alpha_{+-}\alpha_{-+} & &
\end{align}
We turn $\mathcal{S}_A(\mathbb{B})$ into a Hopf algebra using the counit and antipode given by

 \begin{equation*}
 \begin{pmatrix} \epsilon(\alpha_{++}) & \epsilon(\alpha_{+-}) \\ \epsilon(\alpha_{-+}) & \epsilon(\alpha_{--}) \end{pmatrix} =
\begin{pmatrix} 1 &0 \\ 0& 1 \end{pmatrix}  
\text{ and }
\begin{pmatrix} S(\alpha_{++}) & S(\alpha_{+-}) \\ 	S(\alpha_{-+}) & S(\alpha_{--}) \end{pmatrix} 
	= 
	\begin{pmatrix} \alpha_{--} & -q\alpha_{+-} \\ -q^{-1}\alpha_{-+} & \alpha_{++} \end{pmatrix} .
 \end{equation*}

\begin{notations} We denote by $\mathcal{O}_q[\SL_2]$ the Hopf algebra $(\mathcal{S}_A(\mathbb{B}), \Delta, \epsilon, S)$.\end{notations}

Now consider a marked surface $\mathbf{\Sigma}$ and a boundary arc $c$. By gluing a bigon $\mathbb{B}$ along $\mathbf{\Sigma}$ while gluing $b$ with $c$, one obtains a punctured surface isomorphic to $\mathbf{\Sigma}$, hence a map 
$$\Delta_c^L:= \theta_{b\#c} :   \mathcal{S}_A(\mathbf{\Sigma}) \rightarrow  \mathcal{O}_q[\SL_2] \otimes \mathcal{S}_A(\mathbf{\Sigma})$$
 which endows $ \mathcal{S}_A(\mathbf{\Sigma})$ with a structure of left $\mathcal{O}_q[\SL_2]$ comodule (see Figure \ref{fig_comodule}). Similarly, gluing $c$ with $a$ induces a right comodule morphism $\Delta_c^R := \theta_{c\#a} :  \mathcal{S}_A(\mathbf{\Sigma}) \rightarrow  \mathcal{S}_A(\mathbf{\Sigma})\otimes  \mathcal{O}_q[\SL_2] $.

\begin{figure}[!h] 
\centerline{\includegraphics[width=6cm]{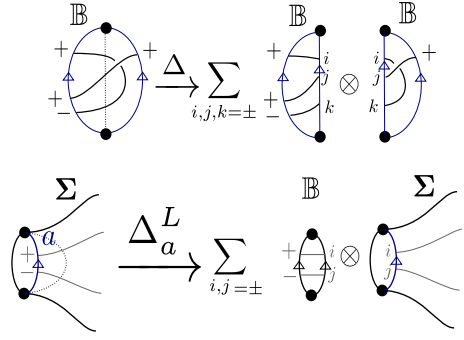} }
\caption{An illustration of the coproduct $\Delta$ in $\mathcal{O}_q[\SL_2]$ and of the left comodule map $\Delta_a^L$.} 
\label{fig_comodule} 
\end{figure}

 Probably the most fundamental property which justifies the importance and usefulness of stated skein algebra is the following: 

\begin{theorem}\label{theorem_gluingformula}(K.-Quesney \cite[Theorem $1.1$]{KojuQuesneyClassicalShadows}, Costantino-L\^e \cite[Theorem $4.7$]{CostantinoLe19}, see also Cooke \cite{Cooke_FactorisationHomSkein} in a different language)
Let $\mathbf{\Sigma}$ be a marked surface and $a,b$ two boundary arcs. The following sequence is exact: 
$$ 0 \to \mathcal{S}_A(\mathbf{\Sigma}_{a\#b}) \xrightarrow{\theta_{a\#b}} \mathcal{S}_A(\mathbf{\Sigma}) \xrightarrow{\Delta^L_a - \sigma \circ \Delta^R_b}\mathcal{O}_q[\SL_2] \otimes \mathcal{S}_A(\mathbf{\Sigma}), $$
where $\sigma(x\otimes y) := y\otimes x$. 
\end{theorem}

In particular, when $\mathbf{\Sigma}$ is equipped with a triangulation $\mathbb{T}$, i.e. when $\mathbf{\Sigma}$ is obtained from a disjoint union of triangles by gluing some pairs $(e_i, e_i')$ of edges, by composing the maps $\theta_{e_i\#e_i'}$, we get an injective map: 
$$ \theta^{\Delta} : \mathcal{S}_A(\mathbf{\Sigma}) \hookrightarrow \otimes_{\mathbb{T}\in F(\Delta)} \mathcal{S}_A(\mathbb{T}).$$

Also, by composing the comodule maps $\Delta^L_{e_i\# e_i'}$, one  gets a Hopf comodule map

 \begin{equation*}
 \Delta^L : \otimes_{\mathbb{T}\in F(\Delta)} \mathcal{S}_A(\mathbb{T}) \rightarrow \left( \otimes_{e\in \mathring{\mathcal{E}}(\Delta)} \mathcal{O}_q[\SL_2] \right)
 \otimes \left(  \otimes_{\mathbb{T}\in F(\Delta)} \mathcal{S}_A(\mathbb{T}) \right).
 \end{equation*}
 
 and a right comodule map
 \begin{equation*}
 \Delta^R :  \otimes_{\mathbb{T}\in F(\Delta)} \mathcal{S}_A(\mathbb{T}) \rightarrow \left(  \otimes_{\mathbb{T}\in F(\Delta)} \mathcal{S}_A(\mathbb{T}) \right)\otimes \left( \otimes_{e\in \mathring{\mathcal{E}}(\Delta)} \mathcal{O}_q[\SL_2] \right). 
 \end{equation*} 
 
 \begin{corollary}
 	The following sequence is exact:
 \begin{equation}\label{eq_triangular}
 0 \rightarrow \mathcal{S}_A(\mathbf{\Sigma}) \xrightarrow{\theta^{\Delta}} 
 \otimes_{\mathbb{T}\in F(\Delta)} \mathcal{S}_A(\mathbb{T}) \xrightarrow{\Delta^L -\sigma\circ\Delta^R}  \left( \otimes_{e\in \mathring{\mathcal{E}}(\Delta)}\mathcal{O}_q[\SL_2] \right) \otimes \left(  \otimes_{\mathbb{T}\in F(\Delta)} \mathcal{S}_A(\mathbb{T}) \right).
 \end{equation} 
  \end{corollary}     

  For further use, let us emphasize a second consequence of the gluing property. Let $\mathbf{\Sigma}=(\Sigma, \mathcal{A})$ and let $\mathcal{B}\subset \mathcal{A}$ a subset of boundary arcs and consider $\underline{\mathbf{\Sigma}}:= (\Sigma, \mathcal{A}\setminus \mathcal{B})$ obtained by removing the boundary arcs of $\mathcal{B}$. The identity map of $\Sigma$ defines an embedding $ \underline{\mathbf{\Sigma}} \to \mathbf{\Sigma}$, so a morphism $\iota : \mathcal{S}_A(\underline{\mathbf{\Sigma}} ) \to \mathcal{S}_A(\mathbf{\Sigma})$. Note that  $\underline{\mathbf{\Sigma}}$ is obtained from  $\mathbf{\Sigma}$ by gluing a monogon $\mathbf{m}_0$ to each boundary arc of $\mathcal{B}$. After identifying  $\mathcal{S}_A(\mathbf{m}_0)\cong k$ through the isomorphism sending the empty link to $1\in k$, $\iota$ can be alternatively defined as the gluing map:
  $$ \iota: \mathcal{S}_A(\underline{\mathbf{\Sigma}} ) \xrightarrow{\theta} \mathcal{S}_A( \mathbf{\Sigma}\cup \mathbf{m}_0^{\cup \mathcal{B}}) \cong \left( \mathcal{S}_A(\mathbf{m}_0) \right)^{\otimes \mathbb{B}} \otimes \mathcal{S}_A(\mathbf{\Sigma})\cong \mathcal{S}_A(\mathbf{\Sigma}).$$
  From Theorem \ref{theorem_gluingformula}, we get an exact sequence:
  $$ 0 \to \mathcal{S}_A(\underline{\mathbf{\Sigma}} ) \xrightarrow{\theta}  \mathcal{S}_A( \mathbf{\Sigma}\cup \mathbf{m}_0^{\cup \mathcal{B}}) \xrightarrow{\Delta^L -\sigma \circ \Delta^R} \left(\mathcal{O}_q[\SL_2] \right)^{\otimes \mathcal{B}} \otimes  \mathcal{S}_A( \mathbf{\Sigma}\cup \mathbf{m}_0^{\cup \mathcal{B}}).$$
  Using the identification $\mathcal{S}_A( \mathbf{\Sigma}\cup \mathbf{m}_0^{\cup \mathcal{B}})\cong \mathcal{S}_A(\mathbf{\Sigma})$, the comodule maps $\Delta^L, \Delta^R$ induces comodules maps $\Delta^L_{\mathcal{B}} :  \mathcal{S}_A(\mathbf{\Sigma}) \to  \left(\mathcal{O}_q[\SL_2] \right)^{\otimes \mathcal{B}} \otimes  \mathcal{S}_A( \mathbf{\Sigma})$ and $\Delta^R_{\mathcal{B}} :  \mathcal{S}_A(\mathbf{\Sigma}) \to  \mathcal{S}_A( \mathbf{\Sigma}) \otimes \left(\mathcal{O}_q[\SL_2] \right)^{\otimes \mathcal{B}} $. By definition, $\Delta^R$ is defined by gluing some bigons $\mathbb{B}$ to the monogons $\mathbf{m}_0$. While gluing a bigon to a monogon, we get another monogon and using the identification $\mathcal{S}_A(\mathbf{m}_0)\cong k$, the gluing map $\mathcal{S}_A(\mathbf{m}_0)\to \mathcal{S}_A(\mathbf{m_0}\cup \mathbb{B})$ corresponds to the unit map $\eta: k\to \mathcal{O}_q[\SL_2]$, therefore $\Delta^R_{\mathcal{B}}=\id \otimes (\eta\circ \epsilon)^{\mathcal{B}}$ and we get the 
  
  \begin{corollary}\label{coro_remove_BA}
  The following sequence is exact:
  $$ 0 \to  \mathcal{S}_A(\underline{\mathbf{\Sigma}} ) \xrightarrow{\iota}  \mathcal{S}_A(\mathbf{\Sigma} ) \xrightarrow{ \Delta^L_{\mathcal{B}} - (\eta \circ \epsilon)^{\mathcal{B}}\otimes \id} 
   \left(\mathcal{O}_q[\SL_2] \right)^{\otimes \mathcal{B}} \otimes  \mathcal{S}_A(\mathbf{\Sigma} ).$$
   Said differently, $ \mathcal{S}_A(\underline{\mathbf{\Sigma}} ) \subset   \mathcal{S}_A(\mathbf{\Sigma} )$ is the subalgebra of coinvariant vectors for the left co-action $\Delta^L_{\mathcal{B}}$.
  \end{corollary}
  
    \subsection{The quantum fusion operation}
    
   The fusion operation was introduced by Alekseev and Malkin in the study of relative character varieties. Its interpretation in terms of marked surfaces and skein algebras was made by Costantino and L\^e. Recall that the triangle $\mathbb{T}$ is a disc with three boundary arcs say $e_1, e_2, e_3$.
   
   \begin{definition}
   Let $\mathbf{\Sigma}$ a marked surface with some boundary arcs $a$ and $b$. 
   
   The \textit{fusion of } $a$ and $b$ is the marked surface  $\mathbf{\Sigma}_{a \circledast b}$ obtained from $\mathbf{\Sigma}\bigsqcup \mathbb{T}$ by gluing $(a,e_1)$ and $(b,e_2)$, i.e.
   $$ \mathbf{\Sigma}_{a \circledast b}:= \left( \mathbf{\Sigma}\bigsqcup \mathbb{T} \right) _{a\#e_1, b\#e_2}.$$
   \end{definition}
   
   \begin{example}
   \begin{enumerate}
   \item the triangle $\mathbb{T}$ is obtained from $\mathbb{B}\bigsqcup \mathbb{B}'$ by fusioning $a_R$ and $a_L'$. We write $\mathbb{T}=\mathbb{B}\circledast \mathbb{B}$.
   \item $\mathbf{m}_1$ is obtained from $\mathbb{B}$ by fusioning its two boundary arcs.
   \item While fusioning the two boundary arcs of  $\mathbf{\Sigma}_{g_1,n_1}^0 \bigsqcup \mathbf{\Sigma}_{g_2,n_2}^0$ we get $\mathbf{\Sigma}^0_{g_1+g_2, n_1+n_2}$. We thus write: 
   $$ \mathbf{\Sigma}_{g,n}^{0} = {\mathbf{\Sigma}^0_{1,0}}^{\circledast g} \circledast \mathbf{m}_1^{\circledast n}.$$
   \end{enumerate}
   \end{example}
   In order to understand the fusion operation at the level of stated skein algebras, we need to endow $\mathcal{O}_q[\SL_2]$ with a structure of cobraided Hopf algebra.
   Define a co-R matrix  $\mathbf{r}: \mathcal{O}_q[\SL_2]^{\otimes 2} \to k$ by the formula
  $$ \mathbf{r} \left( (\adjustbox{valign=c}{\includegraphics[width=1.3cm]{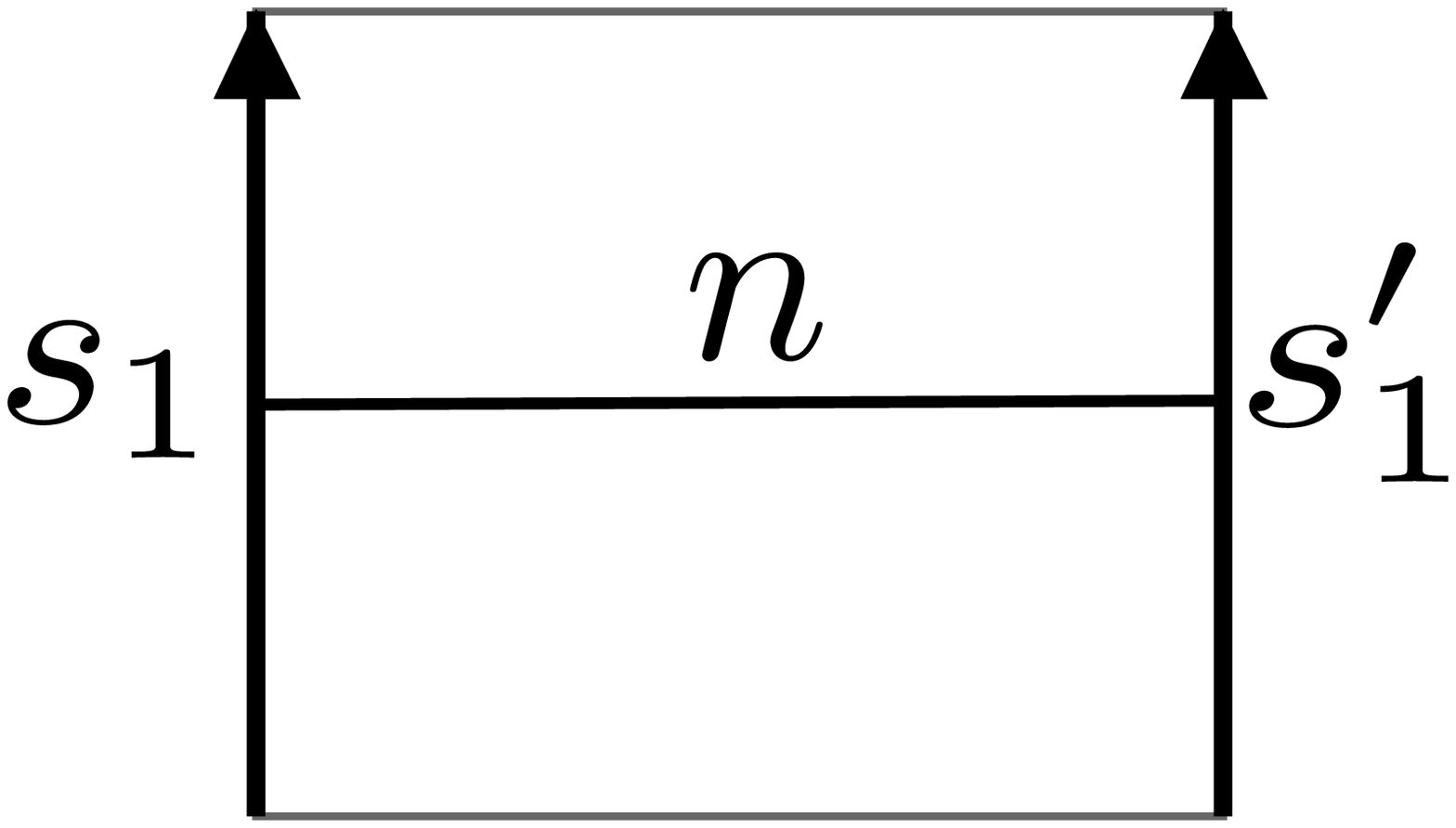}} \otimes (\adjustbox{valign=c}{\includegraphics[width=1.3cm]{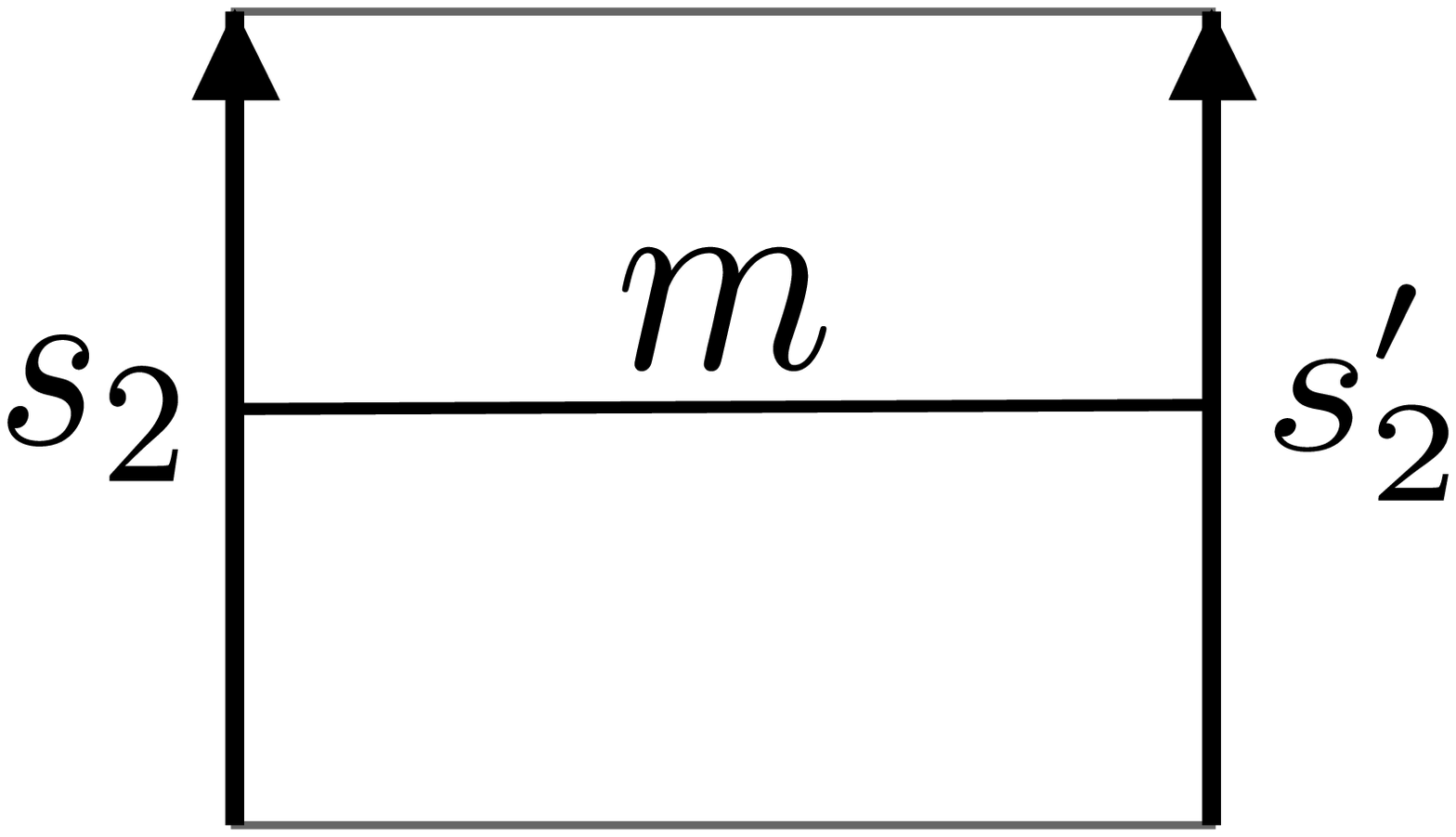}}  \right)
  := \epsilon\left(  \adjustbox{valign=c}{\includegraphics[width=1.3cm]{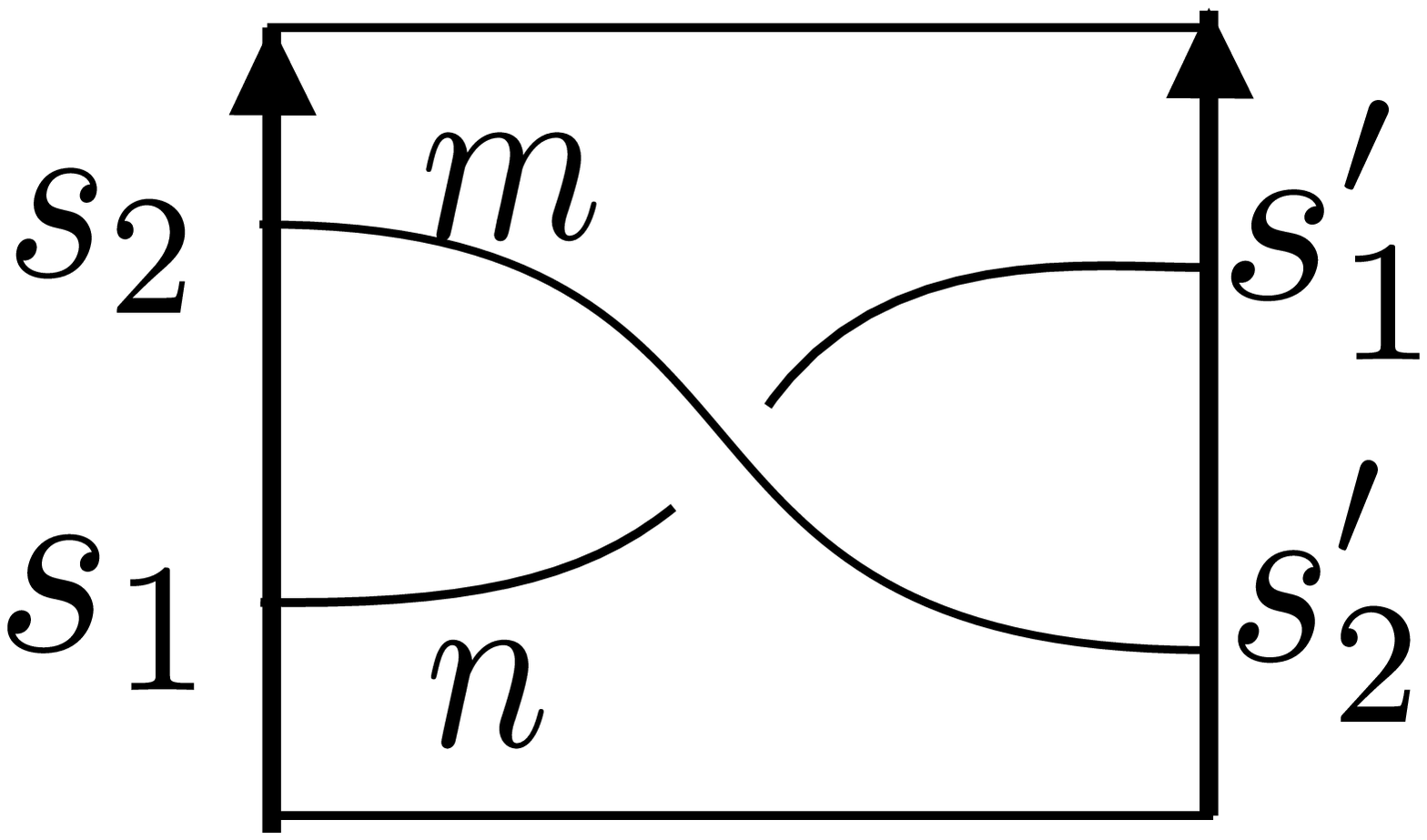}} \right),$$
  where a letter $n$ on top of a strand means that we take $n$ parallel copies. Note that the height exchange formula \eqref{height_exchange_rel} implies that $\mathbf{r}$ is characterized by 
  $$ \mathbf{r} (\alpha_{ij}\otimes \alpha_{kl}) = \mathscr{R}^{jl}_{ki}.$$
  
  Recall that, by definition, a bilinear form $\mathbf{r}: \mathcal{O}_q[\SL_2]^{\otimes 2} \to k$  is a co-R matrix means that there exists another linear form $\overline{\mathbf{r}}: \mathcal{O}_q[\SL_2]^{\otimes 2} \to k$  it satisfies the axioms
  \begin{enumerate}[label=(\subscript{E}{{\arabic*}})]
\item $\mathbf{r}\star \overline{\mathbf{r}} = \overline{\mathbf{r}}\star r = \epsilon\otimes \epsilon$; 
\item $\mu^{op}= \mathbf{r} \star \mu \star \overline{\mathbf{r}}$; 
\item $\mathbf{r}(\mu \otimes \id) = \mathbf{r}_{13}\star \mathbf{r}_{23}$ and $\mathbf{r}(\id \otimes \mu)= \mathbf{r}_{13}\star \mathbf{r}_{12}$.
\end{enumerate}
where $\star$ is the convolution operation on linear forms and $\mathbf{r}_{12}:= \mathbf{r}\otimes \epsilon, \mathbf{r}_{23}:= \epsilon \otimes \mathbf{r}, \mathbf{r}_{13}:= (\epsilon \otimes \mathbf{r})(\tau \otimes \id)$. Define $\overline{\mathbf{r}}$ by: 
 $$ \overline{\mathbf{r}} \left( \adjustbox{valign=c}{\includegraphics[width=1.3cm]{Tangle1.eps}} \otimes (\adjustbox{valign=c}{\includegraphics[width=1.3cm]{Tangle2.eps}}  \right)
  := \epsilon\left(  \adjustbox{valign=c}{\includegraphics[width=1.3cm]{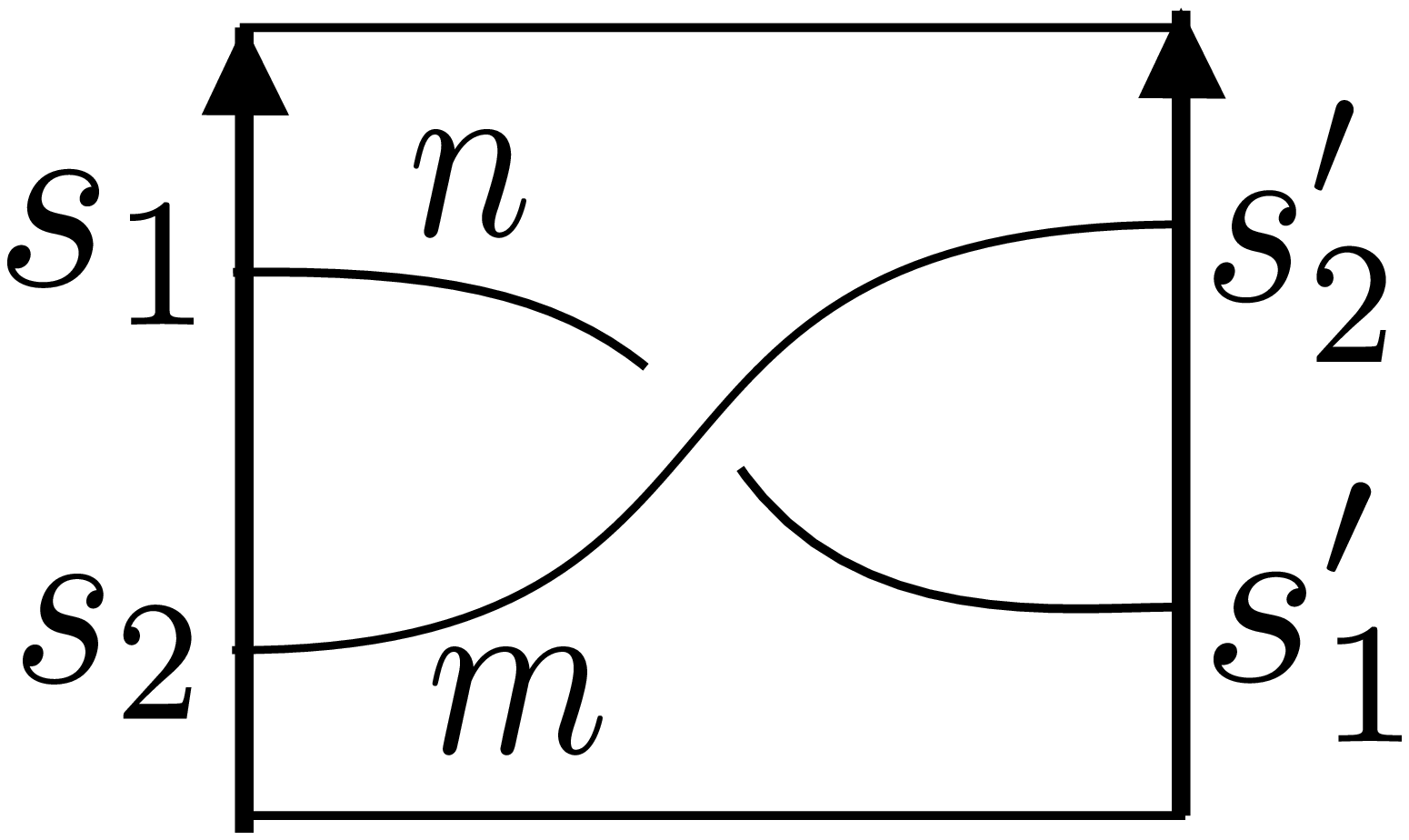}} \right).$$

Said differently, $\overline{\mathbf{r}}(\alpha_{ij}\otimes \alpha{kl})= (\mathscr{R}^{-1})^{jl}_{ki}$.
The axioms are proved graphically as follows:

\par \textit{Axiom $(E_1)$:} 
\begin{multline*}
 (\mathbf{r}\star \overline{\mathbf{r}}) \left(\adjustbox{valign=c}{\includegraphics[width=1.3cm]{Tangle1.eps}} \otimes \adjustbox{valign=c}{\includegraphics[width=1.3cm]{Tangle2.eps}} \right) = \sum_{i,j} \epsilon \left( \adjustbox{valign=c}{\includegraphics[width=1.3cm]{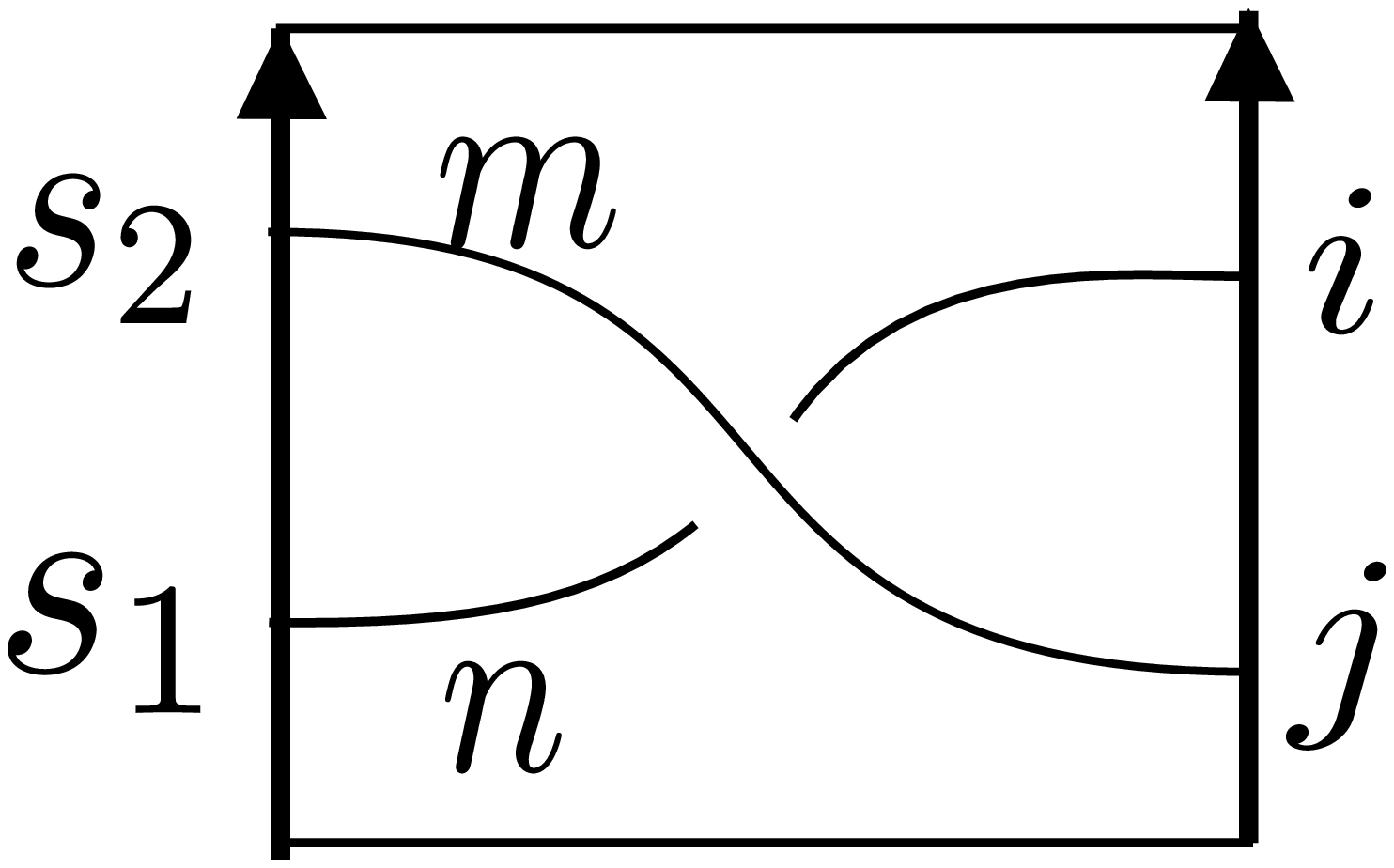}}\right) \epsilon \left( \adjustbox{valign=c}{\includegraphics[width=1.3cm]{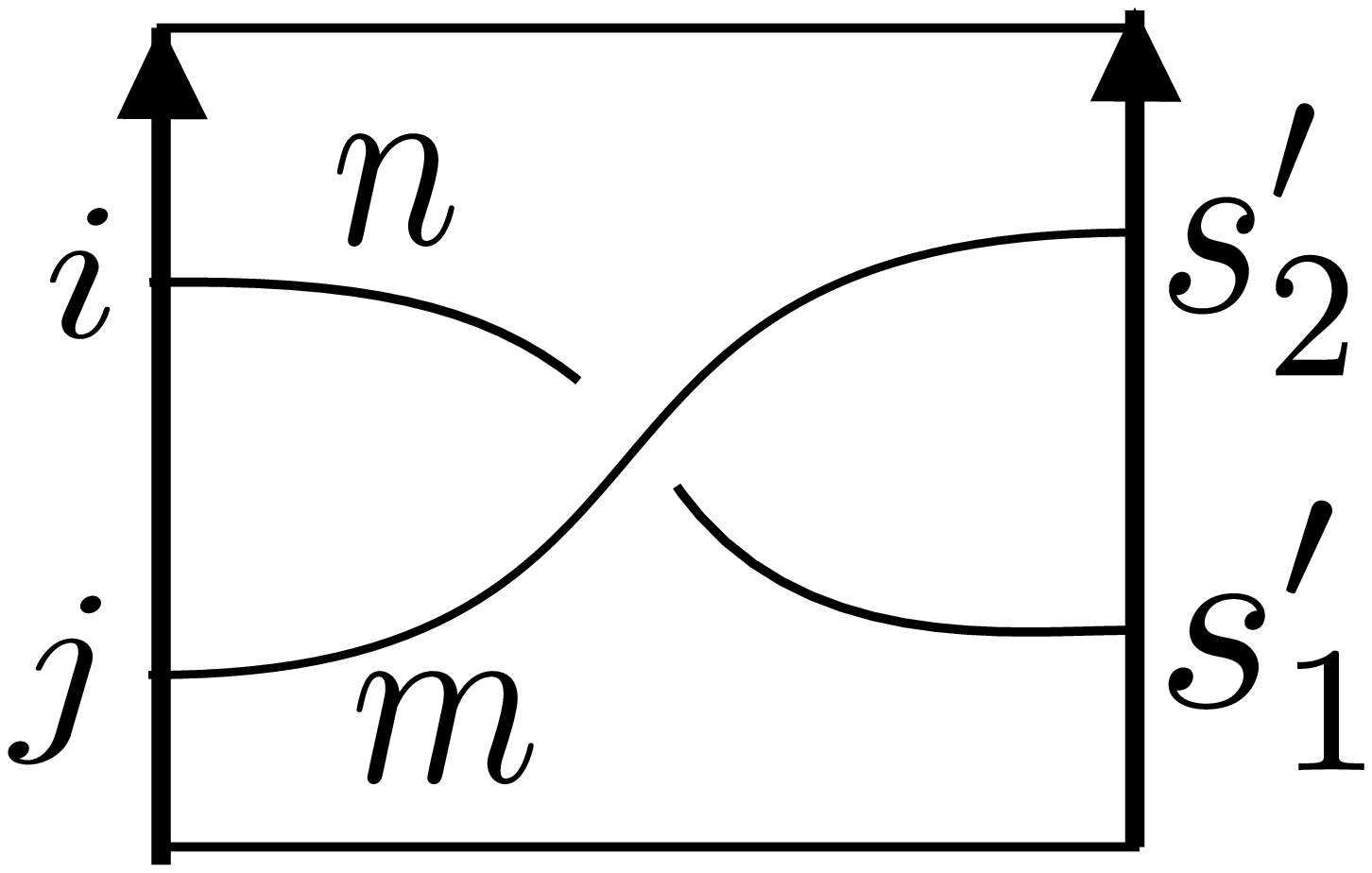}}\right) = \epsilon \left( \adjustbox{valign=c}{\includegraphics[width=1.5cm]{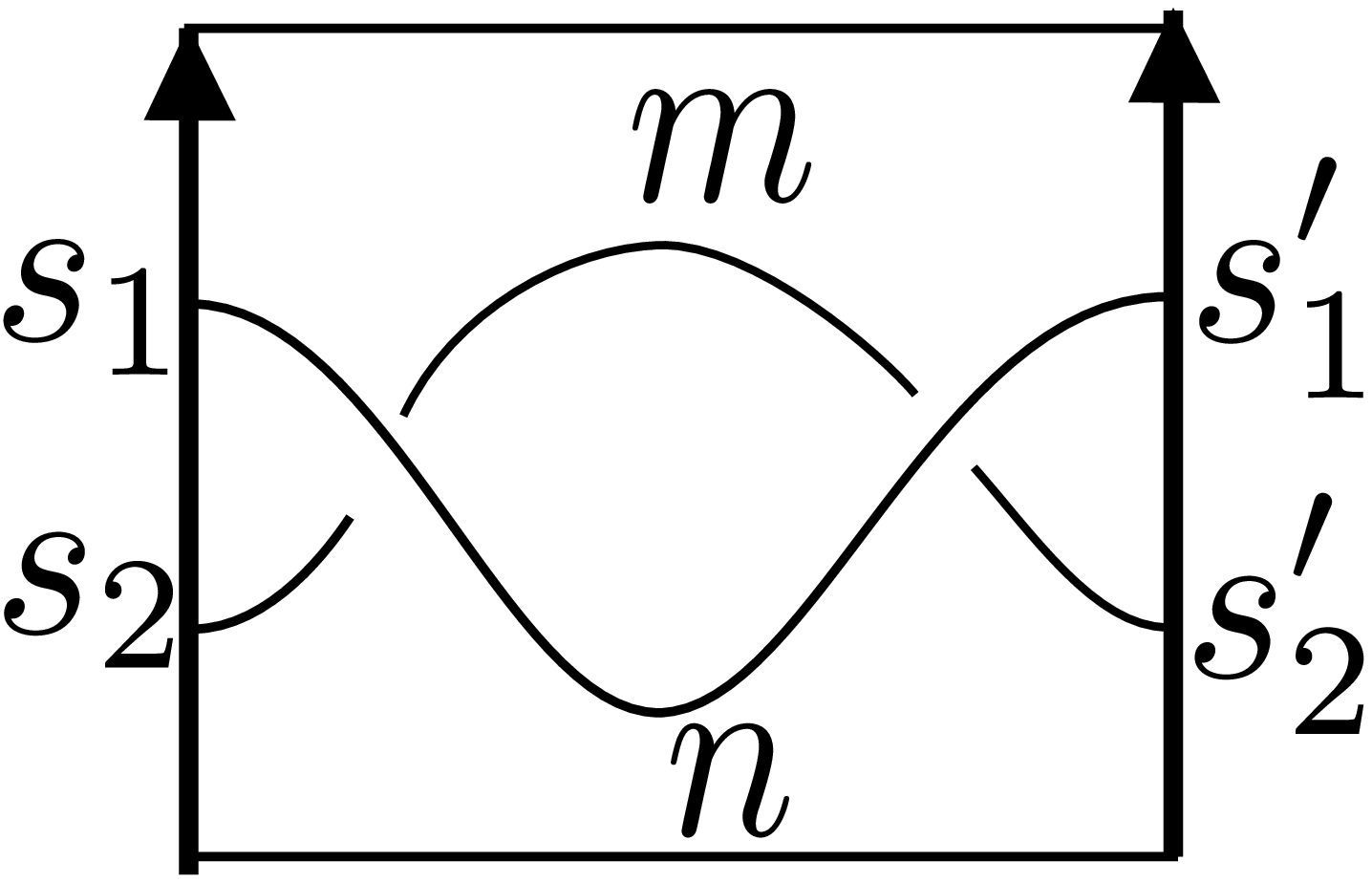}}\right) 
 \\ = \epsilon \left( \adjustbox{valign=c}{\includegraphics[width=1.3cm]{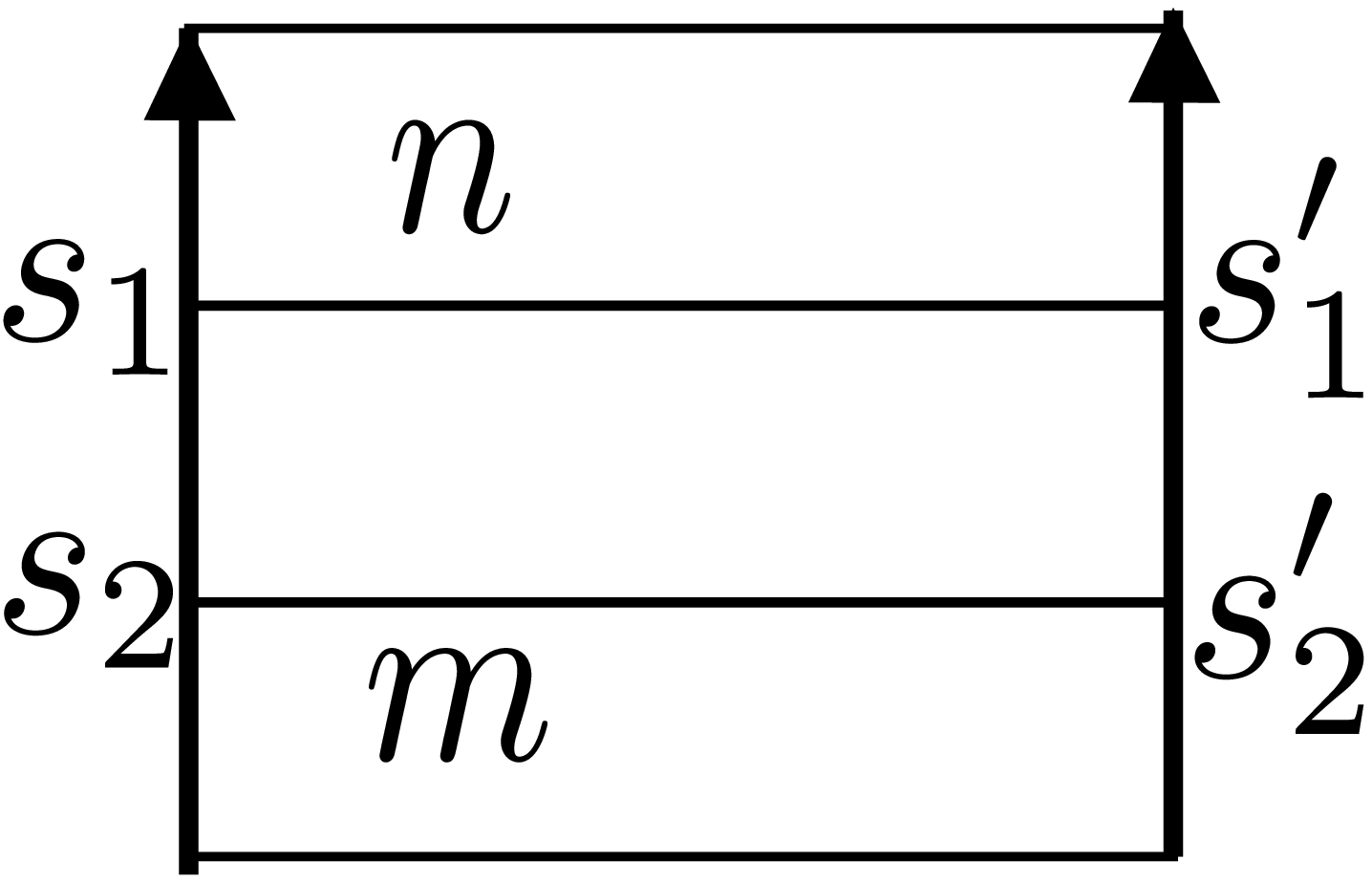}}\right) =(\epsilon \otimes \epsilon) \left(\adjustbox{valign=c}{\includegraphics[width=1.3cm]{Tangle1.eps}} \otimes \adjustbox{valign=c}{\includegraphics[width=1.3cm]{Tangle2.eps}} \right).
 \end{multline*}
The equality $\overline{\mathbf{r}}\star \mathbf{r} = \epsilon \otimes \epsilon$ is proved similarly.
\\ \textit{Axiom $(E_2)$:}
\begin{multline*}
(\mathbf{r} \star \mu \star \overline{\mathbf{r}}) \left(\adjustbox{valign=c}{\includegraphics[width=1.3cm]{Tangle1.eps}} \otimes \adjustbox{valign=c}{\includegraphics[width=1.3cm]{Tangle2.eps}} \right) =
\sum_{i,j,k,l} \epsilon \left( \adjustbox{valign=c}{\includegraphics[width=1.3cm]{R11.eps}} \right) \adjustbox{valign=c}{\includegraphics[width=1.3cm]{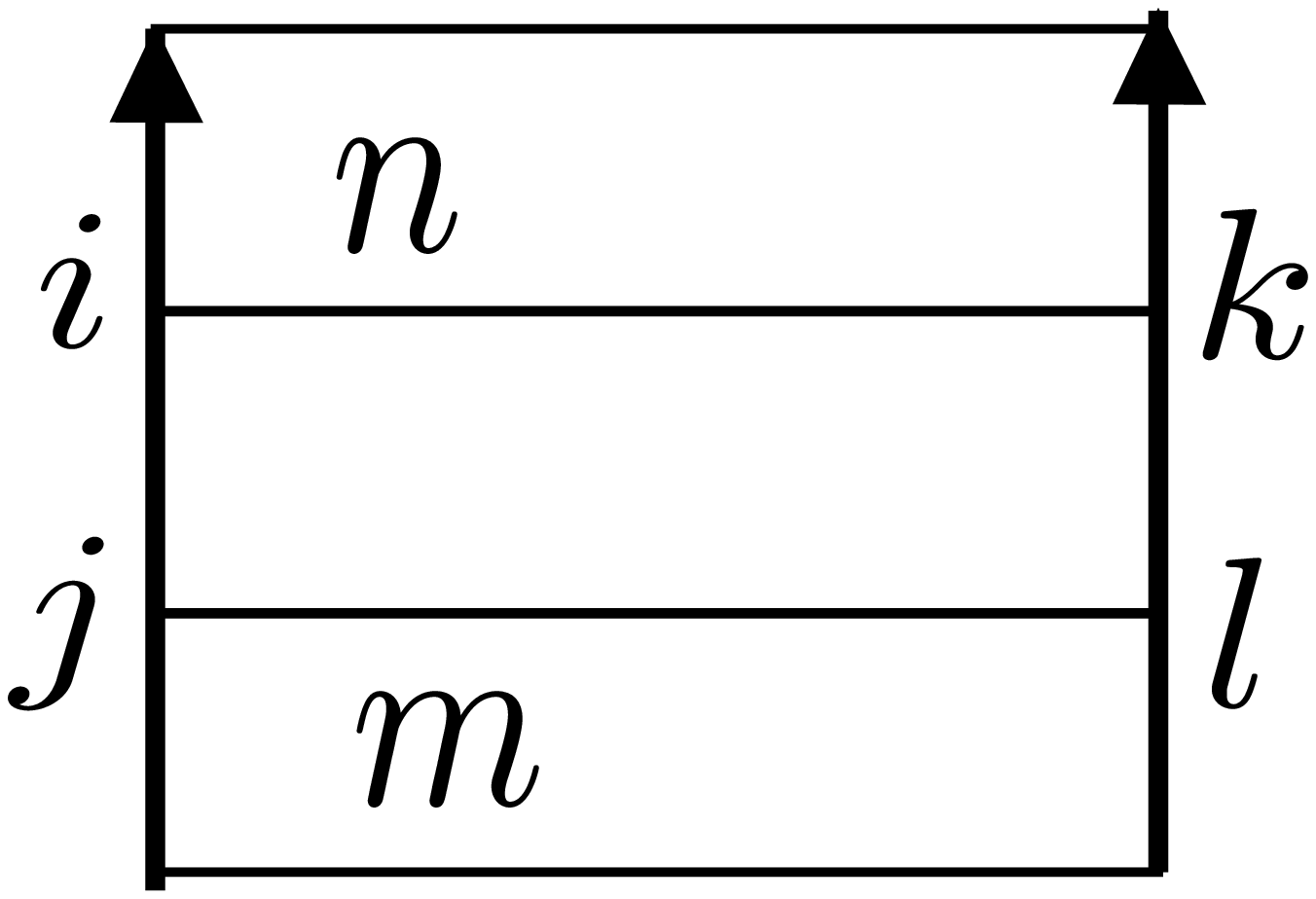}} 
\epsilon \left( \adjustbox{valign=c}{\includegraphics[width=1.3cm]{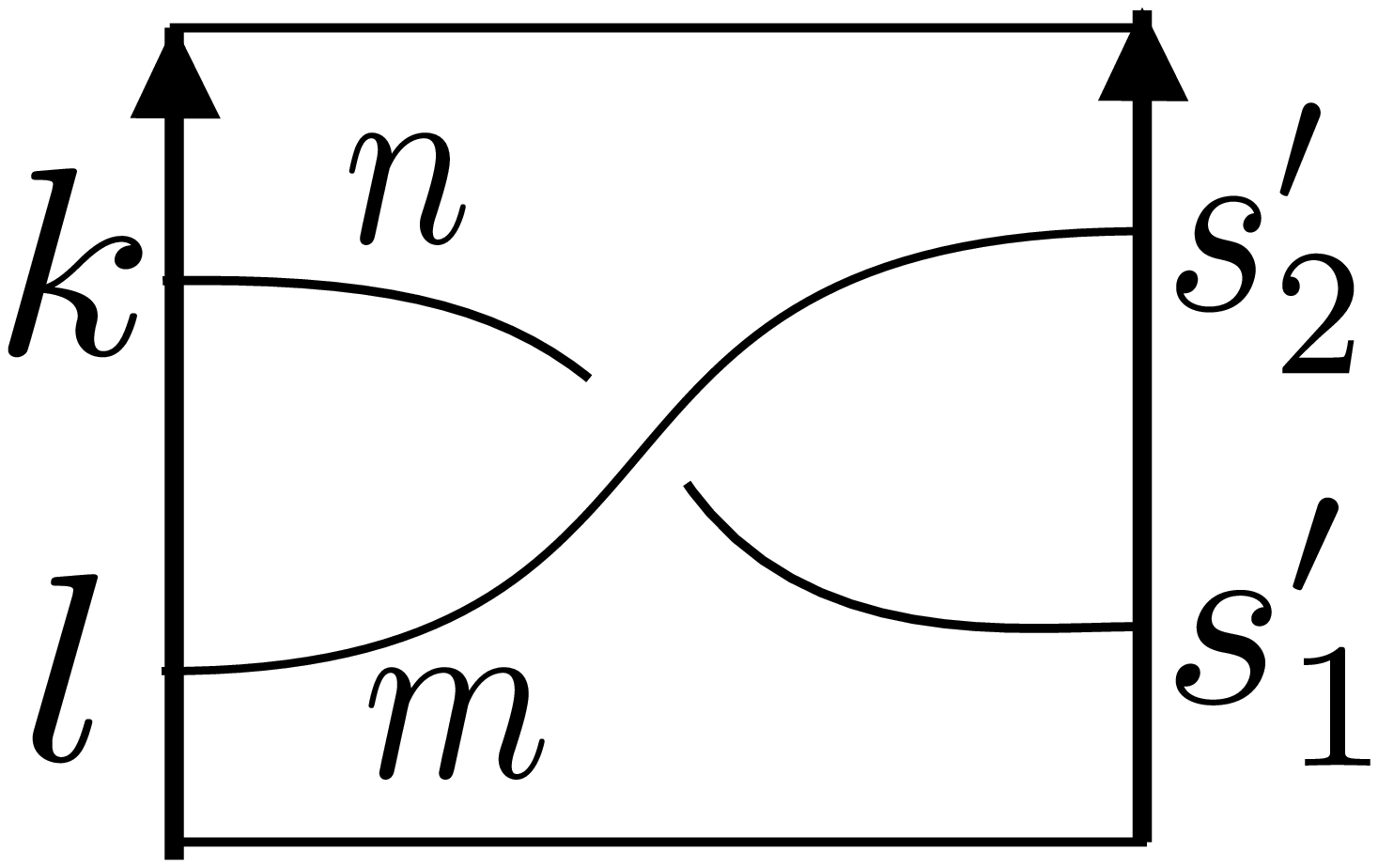}}  \right) 
\\ = \adjustbox{valign=c}{\includegraphics[width=1.3cm]{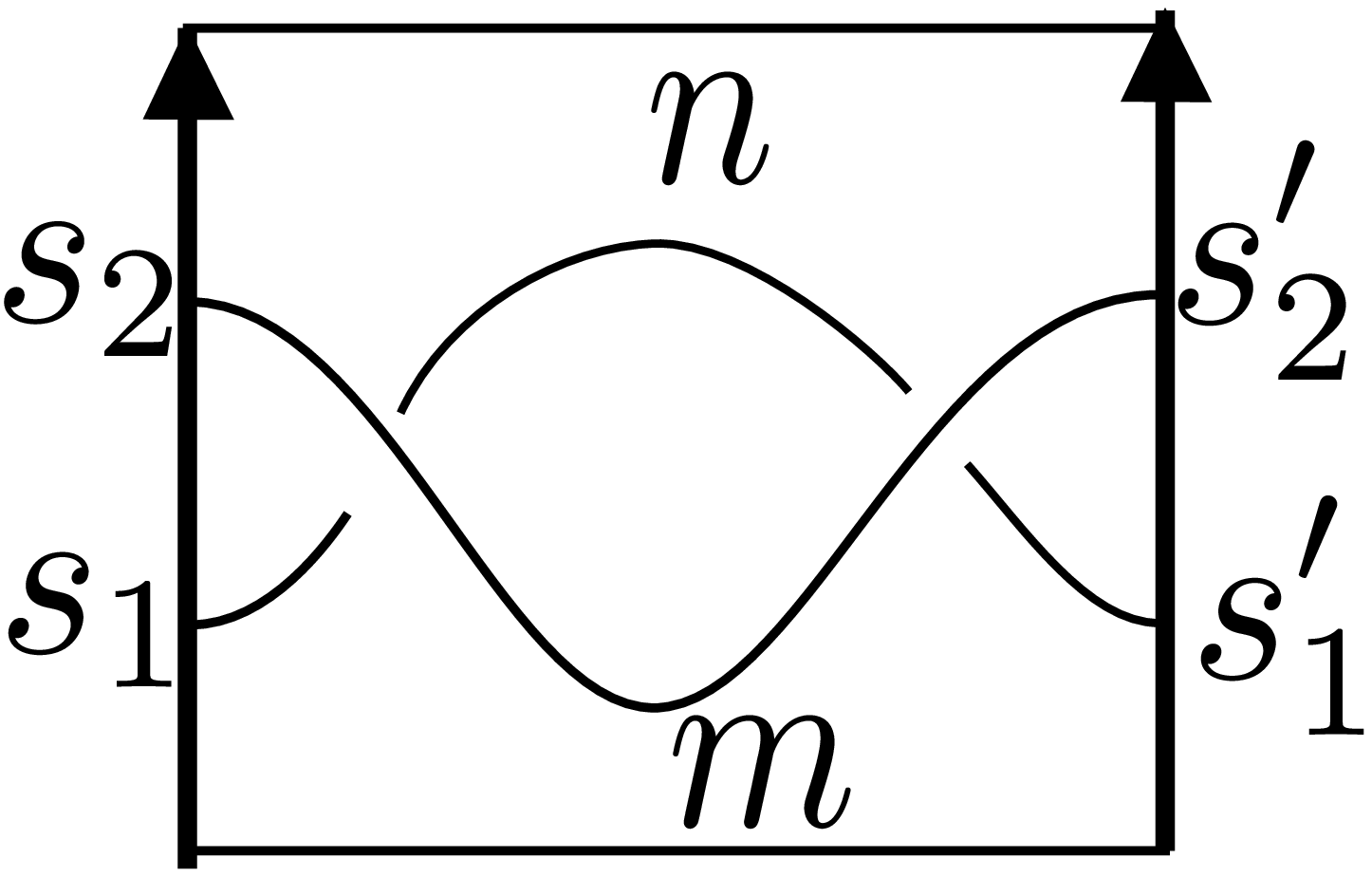}} = \adjustbox{valign=c}{\includegraphics[width=1.3cm]{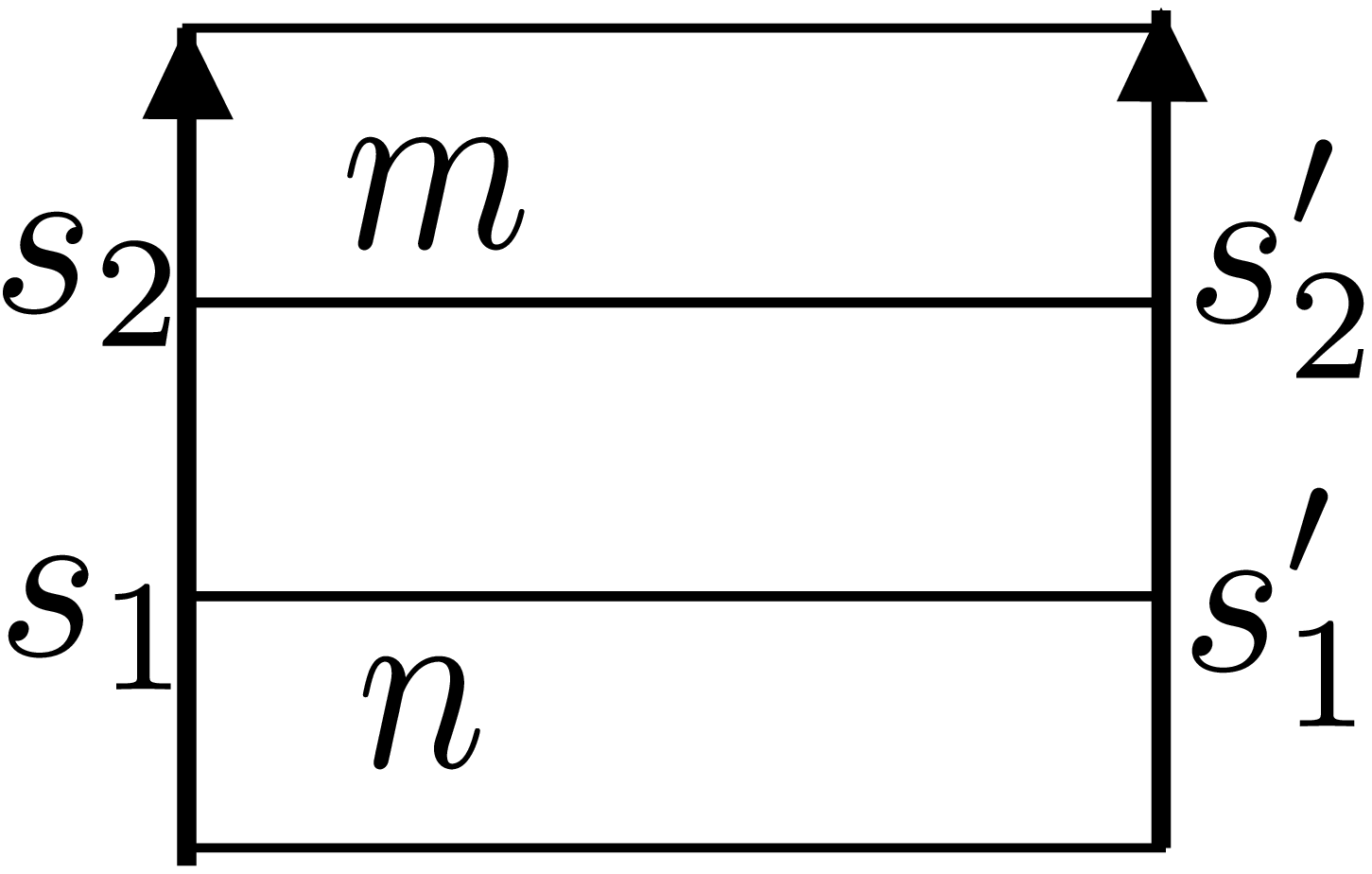}} = \mu^{op} \left(\adjustbox{valign=c}{\includegraphics[width=1.3cm]{Tangle1.eps}} \otimes \adjustbox{valign=c}{\includegraphics[width=1.3cm]{Tangle2.eps}} \right) .
\end{multline*}
\textit{Axiom $(E_3)$:}
\begin{multline*}
\mathbf{r}(\mu \otimes \id) \left( \adjustbox{valign=c}{\includegraphics[width=1.3cm]{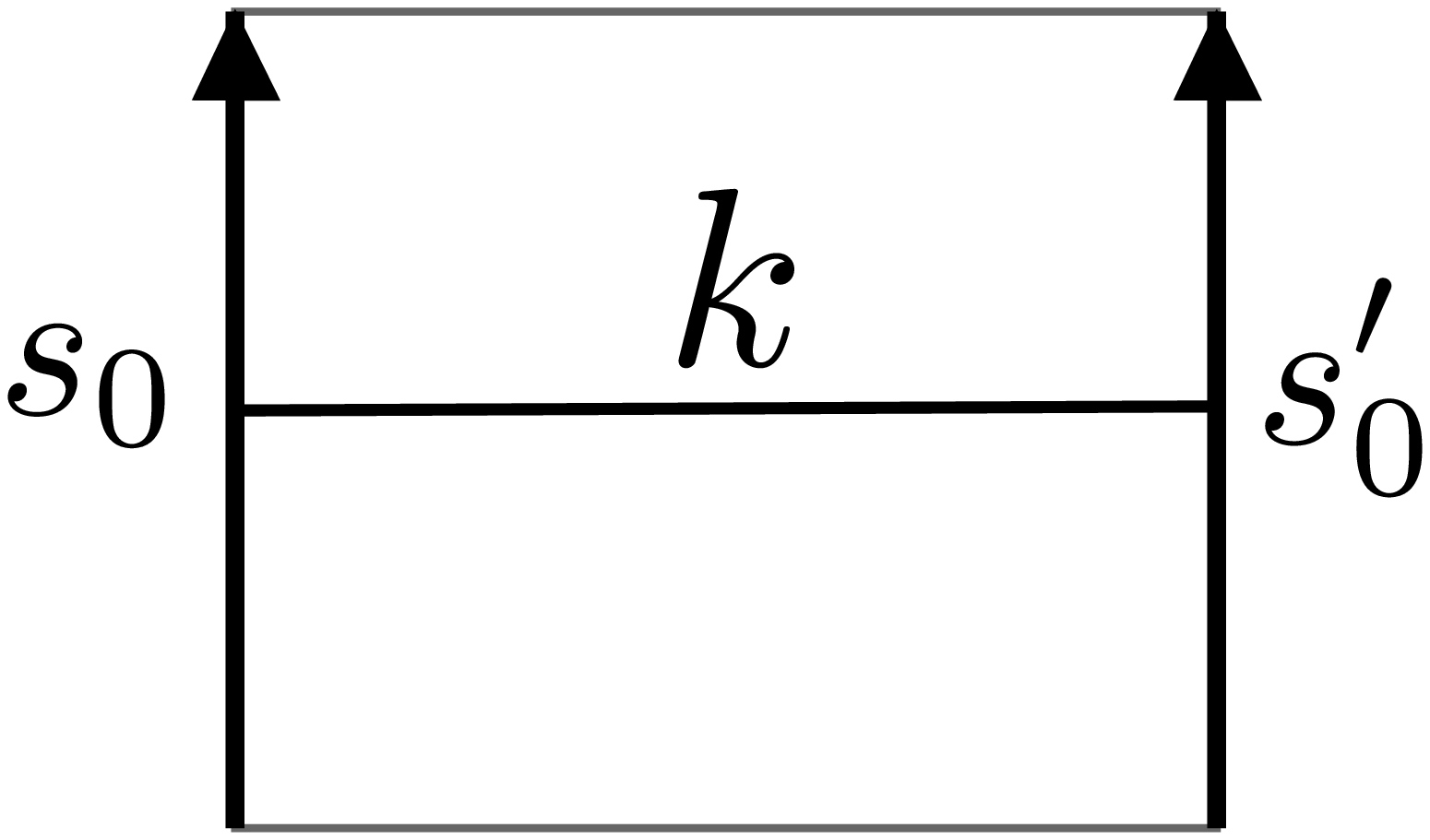}} \otimes \adjustbox{valign=c}{\includegraphics[width=1.3cm]{Tangle1.eps}} \otimes \adjustbox{valign=c}{\includegraphics[width=1.3cm]{Tangle2.eps}} \right) = \epsilon \left( \adjustbox{valign=c}{\includegraphics[width=1.3cm]{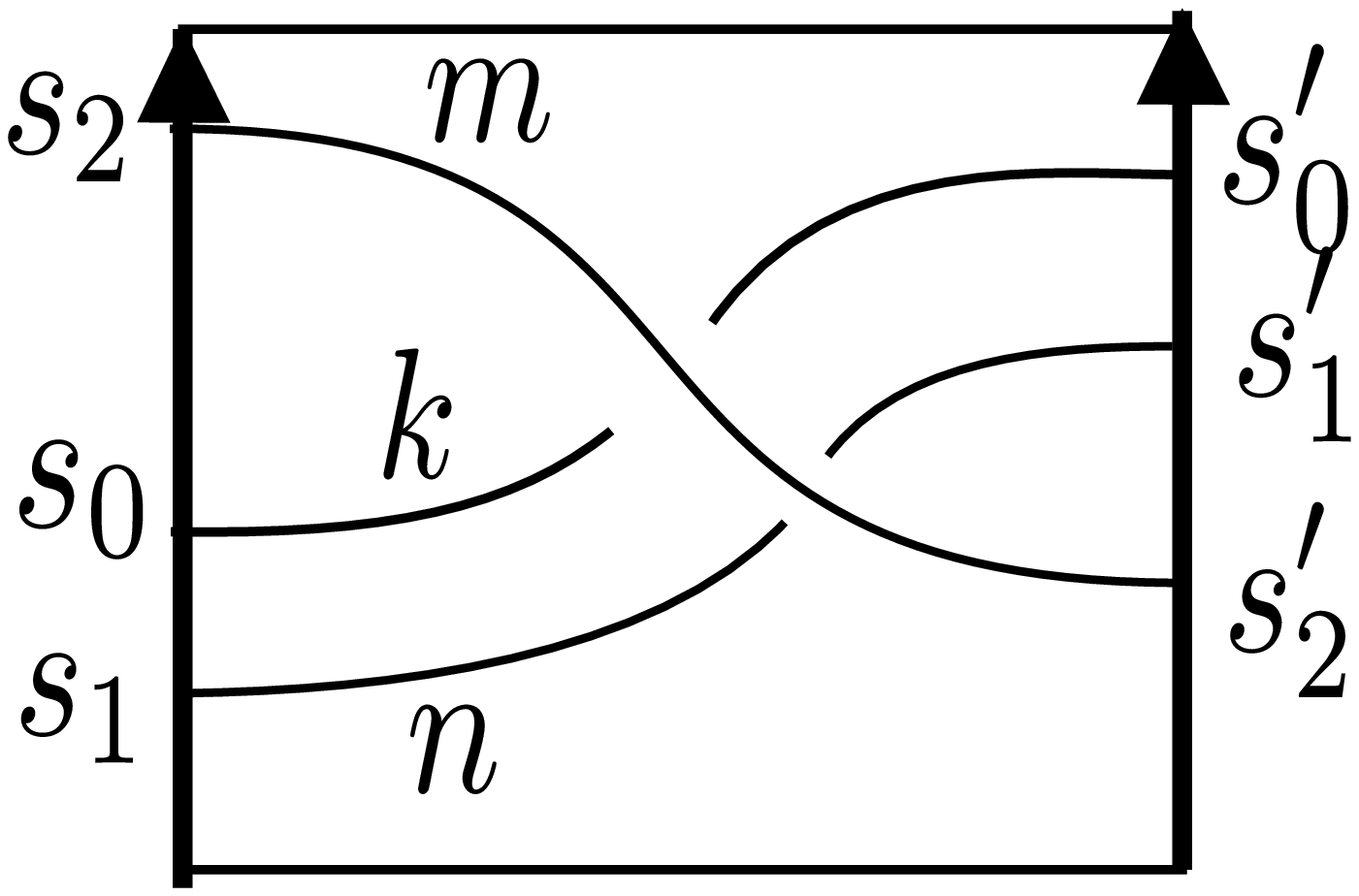}} \right)
\\ \sum_{i,j,k} \epsilon \left( \adjustbox{valign=c}{\includegraphics[width=1.3cm]{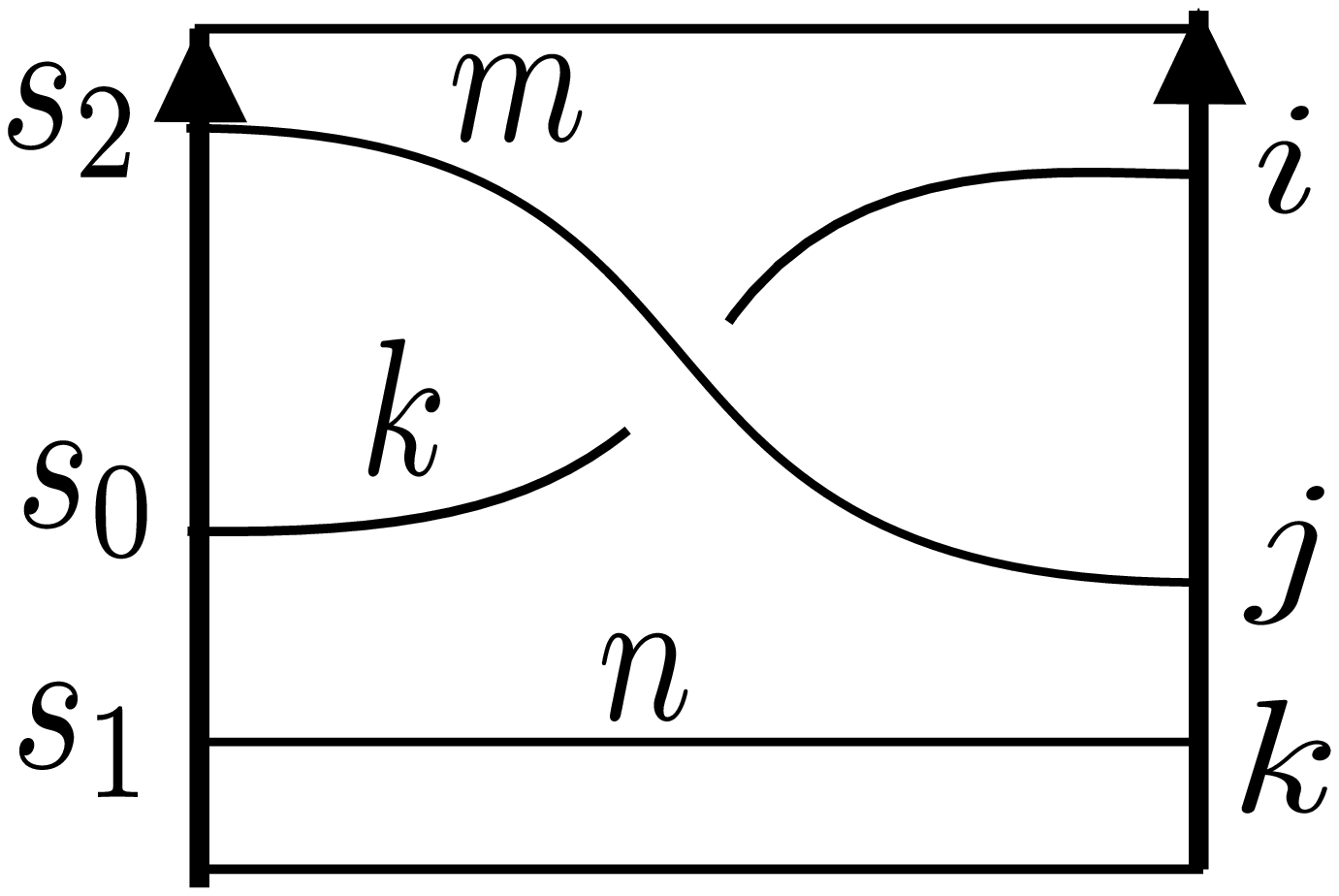}} \right)\epsilon \left( \adjustbox{valign=c}{\includegraphics[width=1.3cm]{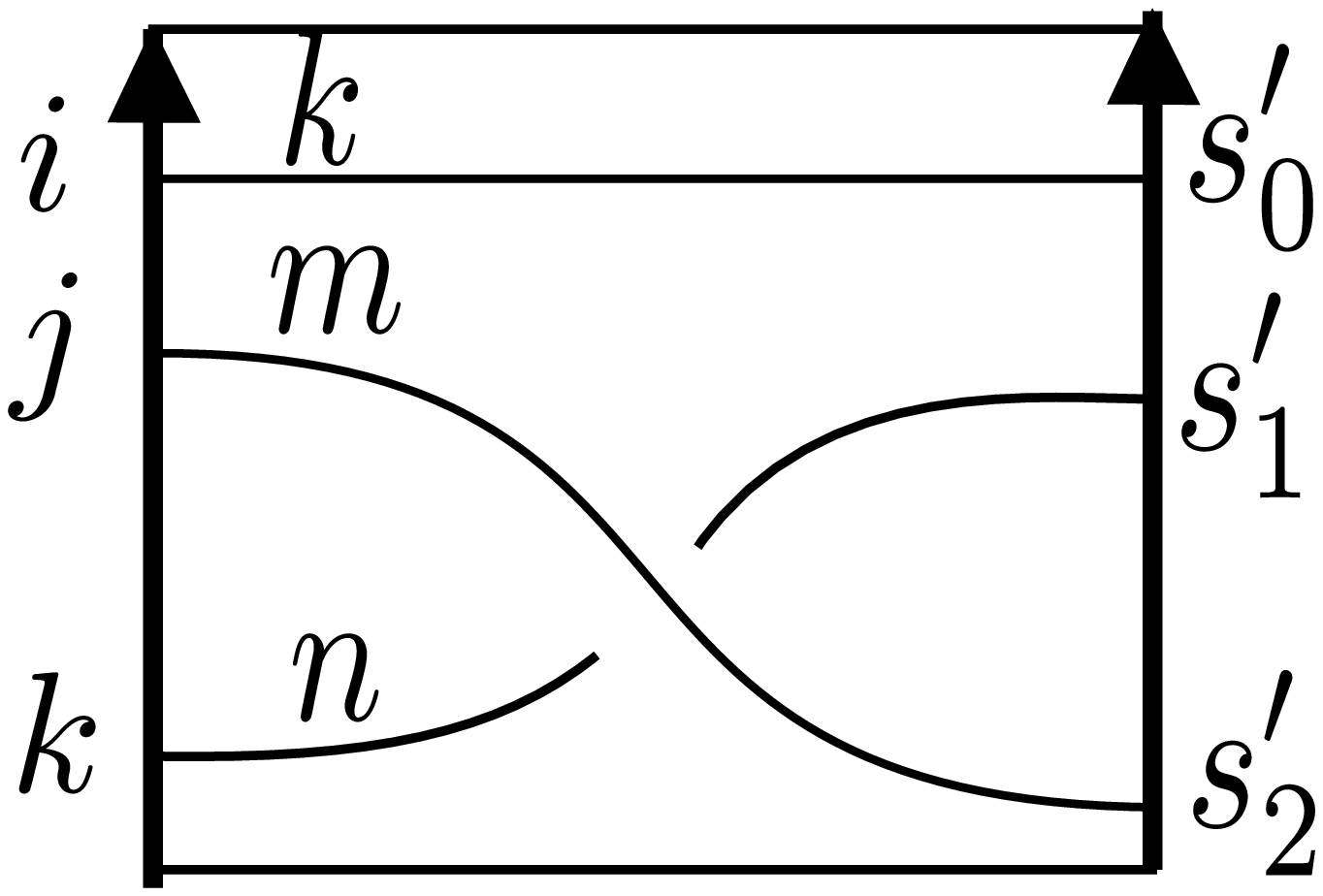}} \right)
= r_{13}\star r_{23}  \left( \adjustbox{valign=c}{\includegraphics[width=1.3cm]{Tangle0.eps}} \otimes \adjustbox{valign=c}{\includegraphics[width=1.3cm]{Tangle1.eps}} \otimes \adjustbox{valign=c}{\includegraphics[width=1.3cm]{Tangle2.eps}} \right).
\end{multline*}
The equality $\mathbf{r}\star (\id \otimes \mu) = \mathbf{r}_{13} \star \mathbf{r}_{12}$ is proved similarly.
   
   \vspace{2mm}
   \par When $(H,\mu, \eta, \Delta, \epsilon, S, \mathbf{r})$ is a cobraided Hopf algebra, its category $H-\mathrm{Comod}$ of finite dimensional left comodules has a natural structure of braided category where for $(V, \Delta_V)$ and $(W, \Delta_W)$ two $H$-comodules, the braiding is given by: 
$$
 c_{V,W} : V\otimes W \xrightarrow{\tau_{V,W}} W\otimes V \xrightarrow{\Delta_W\otimes \Delta_V} H\otimes W \otimes H \otimes V \xrightarrow{\id_H \otimes \tau_{W,H} \otimes \id_V} H\otimes H \otimes W \otimes V \xrightarrow{r\otimes \id_W \otimes \id_V} W\otimes V.
$$

\begin{definition}\label{def_fusion_quantique}
Let $(H,\mu_H, \eta, \Delta, \epsilon_H, S, \mathbf{r})$  be a cobraided Hopf algebra and $\mathbf{A}=(A, \mu, \epsilon)$ be an algebra object in $H\otimes H-\mathrm{Comod}$ and denote by $\Delta_{H\otimes H} : A \to H \otimes H \otimes A$ its comodule map. Write $\Delta^1:= (\id \otimes \epsilon \otimes \id) \circ \Delta_{H\otimes H}: A\to H\otimes A$ and $\Delta^2: (\epsilon \otimes \id \otimes \id) \circ \Delta_{H\otimes H}: A \to H \otimes A$.
 The \textit{fusion} $\mathbf{A}_{1\circledast 2}$ is the algebra object $(A_{1\circledast 2}, \mu_{1\circledast 2}, \epsilon_{1\circledast 2})$ in $H-\mathrm{Comod}$ where 
 \begin{enumerate}
 \item $A_{1\circledast 2}= A$ as a $k$-module and $\epsilon_{1 \circledast 2}=\epsilon$;
 \item The product is the composition
 $$ \mu_{1 \circledast 2} : A\otimes A \xrightarrow{\Delta_1\otimes \Delta_2} H\otimes A \otimes H \otimes A \xrightarrow{\id \otimes \tau_{A,H} \otimes \id} H\otimes H \otimes A \otimes A \xrightarrow{\mathbf{r}\circ \tau_{H,H} \otimes \mu} A.$$
 \item The comodule map is $\Delta_H:= (\mu_H \otimes \id) \circ \Delta_{H\otimes H} $.
 \end{enumerate}
 \end{definition}
 For instance, if $V$ and $W$ are two algebra objects in $H-\mathrm{Comod}$, then $V\otimes_k W$ is an algebra object in $H^{\otimes 2}-\mathrm{Comod}$ and its fusion $V\circledast W$ is called the \textit{cobraided tensor product}. Identify $V$ with $V\otimes 1$ and $W$ with $1\otimes W$ in $V\otimes W$. Its product  is characterized by the formula
 $$ \mu (x\otimes y) = \left\{ \begin{array}{ll} 
 \mu_V(x\otimes y) &\mbox{, if }x,y \in V; \\
 \mu_W(x\otimes y)&\mbox{, if }x,y \in W; \\ 
 x\otimes y & \mbox{, if }x\in V, y \in W; \\
 c_{W,V}(x\otimes y) & \mbox{, if }x\in W, y \in V.
 \end{array} \right.
 $$
 
 Now consider a marked surface $\mathbf{\Sigma}_{a\circledast b}$ obtained by fusioning two boundary arcs $a$ and $b$.  Fix an orientation $\mathfrak{o}
 $ of the boundary arcs of $\mathbf{\Sigma}\bigsqcup \mathbb{T}$, as in Figure \ref{fig_fusion} and let $\mathfrak{o}'$ the induced orientation of the boundary arcs of $\mathbf{\Sigma}'$. 
 Define a linear map $\Psi_{a\circledast b} : \mathcal{S}_A(\mathbf{\Sigma}) \to \mathcal{S}_A(\mathbf{\Sigma}_{a\#b})$ by $\Psi_{a\circledast b} ([D,s]^{\mathfrak{o}}):= [D',s']$ where $(D',s')$ is obtained from $(D,s)$ by gluing to each point of $D\cap a$ a straight line in $\mathbb{T}$ between $e_1$ and $e_3$ and by gluing to each point of $D\cap b$ a straight line in $\mathbb{T}$ between $e_2$ and $e_3$. Figure \ref{fig_fusion} illustrates $\Psi_{a\circledast b}$.
  \begin{figure}[!h] 
\centerline{\includegraphics[width=6cm]{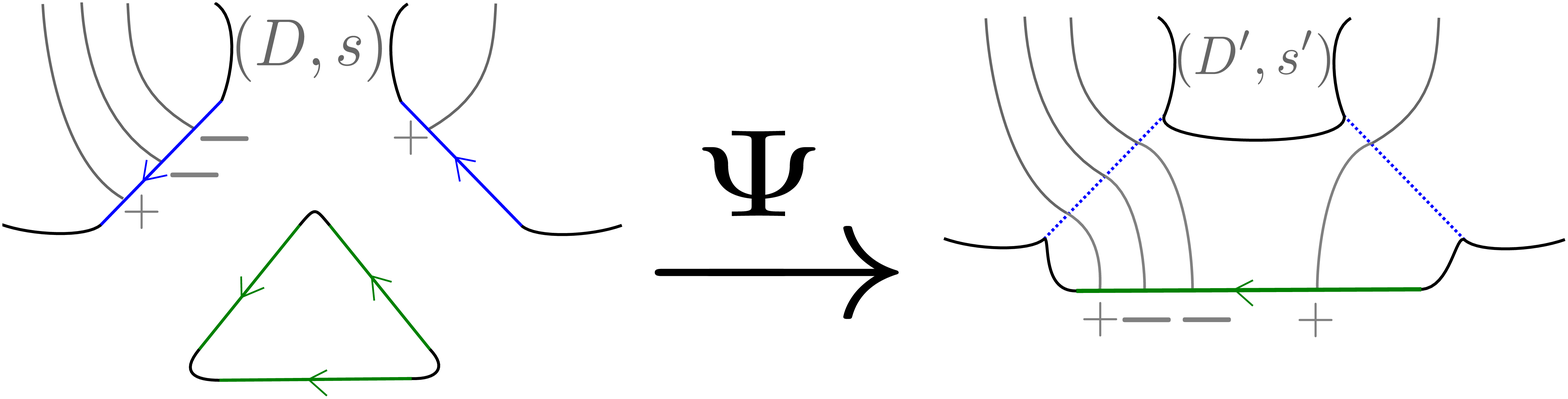} }
\caption{An illustration of $\Psi_{a\circledast b}$.} 
\label{fig_fusion} 
\end{figure} 
 \begin{theorem}\label{theorem_fusion}(Costantino-L\^e \cite[Theorem $4.13$]{CostantinoLe19}) The linear map $\Psi_{a\circledast b}$ is an isomorphism of $k$-modules which identifies $\mathcal{S}_A(\mathbf{\Sigma}_{a\circledast b})$  with the fusion $\mathcal{S}_A(\mathbf{\Sigma})_{a\circledast b}$.
 \end{theorem}
 
 \begin{proof}[Proof (Higgins)]
 The fact that  $\Psi_{a\circledast b}$ is surjective is an easy consequence of the cutting arc relation \eqref{cutting_arc_rel}. The injectivity is proved using the following elegant argument of Higgins in \cite{Higgins_SSkeinSL3}. Since $\mathbf{\Sigma}_{a\circledast b}$ is obtained from $\mathbf{\Sigma}\bigsqcup \mathbb{T}$ by gluing some boundary arcs, we have a gluing map $\theta : \mathcal{S}_A(\mathbf{\Sigma}_{a\circledast b}) \to \mathcal{S}_A (\mathbb{T}) \otimes \mathcal{S}_A(\mathbf{\Sigma})$. Let $i: \mathbb{T} \to \mathbb{B}$ be the embedding of marked surfaces sending $e_3$ to $a_R$ and $e_1, e_2$ to $a_L$ with $e_1>e_2$ and denote by $i_* : \mathcal{S}_A \to \mathcal{O}_q[\SL_2]$ the induced morphism. Consider the composition
 $$ \Phi: \mathcal{S}_A(\mathbf{\Sigma}_{a\circledast b}) \xrightarrow{\theta} \mathcal{S}_A (\mathbb{T}) \otimes \mathcal{S}_A(\mathbf{\Sigma})\xrightarrow{i_*\otimes \id} \mathcal{O}_q[\SL_2] \otimes \mathcal{S}_A(\mathbf{\Sigma})\xrightarrow{\epsilon \otimes \id} \mathcal{S}_A(\mathbf{\Sigma}).$$
 As illustrated in Figure \ref{fig_defusion}, it is easy to see that $\Phi$ is a left inverse to $\Psi_{a\circledast b}$, thus $\Psi_{a\circledast b}$ is an isomorphism. It remains to prove that 
 the pullback by $\Psi_{a\circledast b}$ of the product in $ \mathcal{S}_A(\mathbf{\Sigma}_{a\circledast b}) $ is the fusion product 
$  \mu_{a\circledast b}$. This fact is illustrated in Figure \ref{fig_fusion_product}.

 \begin{figure}[!h] 
\centerline{\includegraphics[width=16cm]{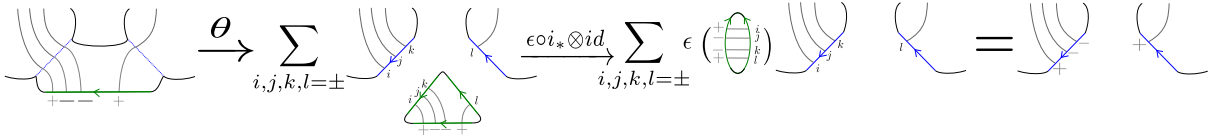} }
\caption{An illustration of the equality $\Phi \circ \Psi_{a\circledast b} = \id$.} 
\label{fig_defusion} 
\end{figure} 

 \begin{figure}[!h] 
\centerline{\includegraphics[width=10cm]{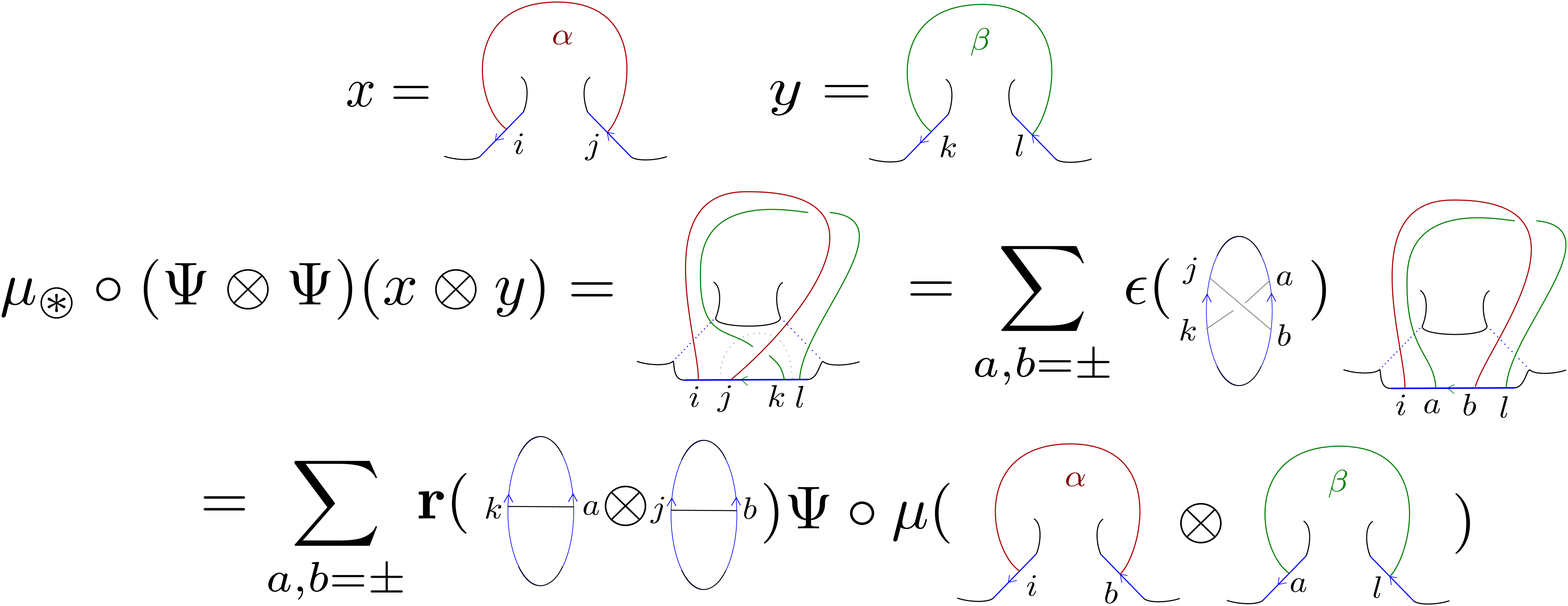} }
\caption{An illustration of the fusion product $\mu_{\circledast}$.} 
\label{fig_fusion_product} 
\end{figure}

 \end{proof}

   \subsection{The triangular strategy}
   
   We have now all the ingredients to show the powerfulness of stated skein algebras. 
   
   \begin{strategy}\label{strategy}(Triangular strategy)
   Suppose you want to prove that (stated) skein algebras satisfies a certain property (P).
   
   \begin{enumerate}
   \item First prove (P) in the case of the bigon. Since $\mathcal{S}_A(\mathcal{B})=\mathcal{O}_q[\SL_2]$ has been overly studied, property (P) has probably already been proved in the literature in this case (you might look at  \cite{Kassel, BrownGoodearl} for instance).
  \item Using the fact that $\mathcal{S}_A(\mathbb{T})\cong \mathcal{O}_q[\SL_2]\circledast \mathcal{O}_q[\SL_2]$, deduce that property (P) holds for the triangle.
  \item Using the embedding $\theta^{\Delta} : \mathcal{S}_A(\mathbf{\Sigma}) \hookrightarrow \otimes_{\mathbb{T}\in F(\Delta)} \mathcal{S}_A(\mathbb{T})$ and the exact sequence \eqref{eq_triangular}, deduce that property (P) is true for any triangulable surfaces.
  \item It remains to deal with non triangulable connected surfaces (other than the bigon). Since $\mathcal{S}_A(\mathbf{m}_0)\cong \mathcal{S}_A(\mathbb{D}^2, \emptyset) \cong  k$ and $\mathcal{S}_A(\Sigma_{0,1}, \emptyset) \cong k[X]$, property (P) is probably trivial in these cases, so it remains to deal with closed connected (unmarked) surfaces $(\Sigma_{g,0}, \emptyset)$. Since $\mathbf{\Sigma}=(\Sigma_{g,0}, \emptyset)$ is obtained from $\mathbf{\Sigma}'=(\Sigma_{g,1}, \emptyset)$ by removing a puncture (the unique boundary component), using the exact sequence 
$$  0 \to \mathcal{I}_p \to \mathcal{S}_A(\mathbf{\Sigma}) \xrightarrow{f_*} \mathcal{S}_A(\mathbf{\Sigma}') \to 0, $$
and the fact (Proposition \ref{prop_offpuncture}) that $\mathcal{I}_p$ is generated by elements $[\gamma]-[\gamma']$ where $\gamma, \gamma'$ are two closed curves in $\Sigma_{g,1}$ which are isotopic in $\Sigma_{g,0}$, you can probably deduce that property (P) holds for $\mathbf{\Sigma}$ from the fact that it holds for the triangulable surface $\mathbf{\Sigma}'$.
\end{enumerate}
\end{strategy}

Let us illustrate this strategy on the following:
\begin{theorem}\cite{BonahonWongqTrace, PrzytyckiSikora_SkeinDomain, LeStatedSkein} \label{theorem_domain}
If $k$ is a domain, so is $\mathcal{S}_A(\mathbf{\Sigma})$.
\end{theorem}

Proving Theorem \ref{theorem_domain} without the use of stated skein algebras is a difficult problem. In the particular case where $k=\mathbb{C}$, $A=-1$ and $\mathbf{\Sigma}$ is closed, it was proved independently by Przytycki-Sikora in \cite{PS00} and Charles-Marche in \cite{ChaMa} that the (commutative) skein algebra is reduced. 

The more general case of triangulable surfaces for any domain $k$ was proved by Bonahon and Wong in \cite{BonahonWongqTrace} using the quantum trace (see also M\"uller \cite{Muller} for a similar strategy using another quantum torus). The case of closed surfaces was  proved by Przytycki-Sikora in \cite{PrzytyckiSikora_SkeinDomain} using suitable filtrations from pants decomposition. 
 Using the triangular strategy, L\^e arrived at the following trivial proof:

\begin{proof} (L\^e \cite{LeStatedSkein})
The fact that $\mathcal{O}_q[\SL_2]$ is a domain is classical and a proof (using a PBW basis) can be found in any textbook on the subject like \cite{Kassel}. It easily results that the cobraided tensor product $\mathcal{O}_q[\SL_2] \circledast \mathcal{O}_q[\SL_2]$ is a domain. Since a subring of a domain is a domain, using the embedding $\theta^{\Delta} : \mathcal{S}_A(\mathbf{\Sigma}) \hookrightarrow \otimes_{\mathbb{T}\in F(\Delta)} \mathcal{S}_A(\mathbb{T})$, we see that $\mathcal{S}_A(\mathbf{\Sigma})$ is a domain whenever $\mathbf{\Sigma}$ is triangulable. 

\end{proof}

In this survey we will give three more illustrations of how the triangular strategy can be used to get (almost) trivial proofs of deep results on skein algebras. They are: $(1)$ the construction of the Chebyshev-Frobenius morphism, $(2)$ the construction of the quantum trace (which was the original motivation for the introduction of stated skein algebras) and $(3)$ the proof that (stated) skein algebras are deformation quantizations of (relative) character variety.

  \subsection{Chebyshev-Frobenius morphisms}

\begin{definition}  The $N$-th Chebyshev polynomial of first kind is the polynomial  $T_N(X) \in \mathbb{Z}[X]$ defined by the recursive formulas $T_0(X)=2$, $T_1(X)=X$ and $T_{n+2}(X)=XT_{n+1}(X) -T_n(X)$ for $n\geq 0$.
\end{definition}
The Chebyshev polynomial is alternatively defined by the equality
$$ \tr (A^N) = T_N(\tr(A)), \quad \forall A \in \SL_2(\mathbb{C}).$$
It appears at the quantum level as follows. Suppose that $q^2=A^4$ is a root of unity of order $N$. Then in $\mathcal{O}_q[\SL_2]$, we have the following equality
\begin{equation}\label{eq_miraculous}
T_N (\alpha_{++} + \alpha_{--}) = \alpha_{++}^N + \alpha_{--}^N.
\end{equation}
Consider a stated arc $\alpha_{ij}$ in some marked surface $\mathbf{\Sigma}$ and denote by $\alpha_{ij}^{(N)}$ the stated tangle obtained by taking $N$ parallel copies of $\alpha_{ij}$ pushed along the framing direction. In the case where both endpoints of $\alpha$ lye in two distinct boundary arcs, one has the equality
$ \alpha_{ij}^{(N)} = (\alpha_{ij})^N$
in $\mathcal{S}_A(\mathbf{\Sigma})$, but when both endpoints lye in the same boundary arc, they are distinct. More precisely, suppose the two endpoints, say $v$ and $w$, of $\alpha$ lye in the same boundary arc with $h(v)>h(w)$. Then  $\alpha_{ij}^{(N)}$ is defined by a stated tangle $(\alpha^{(N)}, s)$ where $\alpha^{(N)}=\alpha_1 \cup \ldots \cup \alpha_N$ represents $N$ copies of $\alpha$ and the endpoints $v_i, w_i$ of the copy $\alpha{(i)}$ are chosen such that $h(v_N) > \ldots > h(v_1) > h(w_N)>\ldots >h(w_1)$.

\begin{theorem}\label{theorem_chebyshev}(Bonahon-Wong \cite{BonahonWongqTrace} for unmarked surfaces, K.-Quesney \cite{KojuQuesneyClassicalShadows} for marked surfaces, see also \cite{BloomquistLe})
Suppose $A\in \mathbb{C}^{\times}$ is a root of unity of odd order $N$, then there is an embedding 
$$ Ch_A^{\mathbf{\Sigma}} : \mathcal{S}_{+1}(\mathbf{\Sigma}) \hookrightarrow \mathcal{Z}\left( \mathcal{S}_A(\mathbf{\Sigma}) \right)$$
sending the (commutative) algebra at $+1$ into the center of the skein algebra at $A^{1/2}$. Moreover, $Ch_A^{\mathbf{\Sigma}}$ is characterized by the facts that if $\gamma$ is a closed curve, then $Ch_A^{\mathbf{\Sigma}}(\gamma) = T_N(\gamma)$ and if $\alpha_{ij}$ is a stated arc, then $Ch_A^{\mathbf{\Sigma}}(\alpha_{ij})= \alpha_{ij}^{(N)}$.
\end{theorem}

We call $Ch_A^{\mathbf{\Sigma}}$ the \textit{Chebyshev-Frobenius morphism}. In this survey, we made the choice to focus on skein algebras at roots of unity of odd orders. Theorem \ref{theorem_chebyshev} has an analogue for roots of unity of even orders (see \cite{Le_QTraceMuller, BloomquistLe}), though in these cases we only get at best a central embedding of the $0$-th graded part of the skein algebra. For this reason, we restrict on the simpler case where $N$ is odd.

\begin{proof}(Sketch)
We apply the triangular strategy. The existence of a Chebyshev-Frobenius morphism $Ch_A^{\mathbb{B}}: \mathcal{O}[\SL_2]\hookrightarrow \mathcal{O}_q[\SL_2]$ for the bigon, sending $\alpha_{ij}$ to $\alpha_{ij}^N$  is well known since the original work of Lusztig \cite{Lusztig_QGroupsRoots1} and elementary proofs can be found in any textbook such as \cite{BrownGoodearl}. It is easy to deduce a central embedding $Ch_A^{\mathbb{T}}: \mathcal{O}[\SL_2] \circledast \mathcal{O}[\SL_2] \hookrightarrow \mathcal{O}_q[\SL_2] \circledast \mathcal{O}_q[\SL_2]$.  Now for a triangulated marked surface $(\mathbf{\Sigma}, \Delta)$, we define $Ch_A^{\mathbf{\Sigma}}$ as the unique injective morphism making commuting the following diagram:

$$ 
\begin{tikzcd}
0 \arrow[r] & \mathcal{S}_{+1}(\mathbf{\Sigma}) \arrow[r, "\theta^{\Delta}"] 
\arrow[d, dotted, "\exists!", "Ch_A^{\mathbf{\Sigma}}"'] 
& 
\otimes_{\mathbb{T}\in F(\Delta)} \mathcal{S}_{+1}(\mathbb{T})
 \arrow[d, "\otimes_{\mathbb{T}} Ch_A^{\mathbb{T}}"] 
\arrow[r, "\Delta^L - \sigma \circ \Delta^R"] 
 &
\left( \otimes_{e\in \mathring{\mathcal{E}}(\Delta)} \mathcal{O}[\SL_2] \right) \otimes \left( \otimes_{\mathbb{T}\in F(\Delta)} \mathcal{S}_{+1}(\mathbb{T})\right)
\arrow[d, "(\otimes_e Ch_A^{\mathbb{B}}) \otimes (\otimes_{\mathbb{T}} Ch_A^{\mathbb{T}})"]
\\
0 \arrow[r] & \mathcal{S}_A(\mathbf{\Sigma}) \arrow[r, "\theta^{\Delta}"] & 
\otimes_{\mathbb{T}\in F(\Delta)} \mathcal{S}_A(\mathbb{T}) \arrow[r, "\Delta^L - \sigma \circ \Delta^R"] &
\left( \otimes_{e\in \mathring{\mathcal{E}}(\Delta)} \mathcal{O}_q[\SL_2] \right) \otimes \left( \otimes_{\mathbb{T}\in F(\Delta)} \mathcal{S}_A(\mathbb{T})\right)
\end{tikzcd}
$$
Clearly, the image of $Ch_A^{\mathbf{\Sigma}}$ is central since the image of $Ch_A^{\mathbf{\Sigma}}$ is central. 
Working a little bit, and using Equation \eqref{eq_miraculous}, one can show that $Ch_A^{\mathbf{\Sigma}}(\gamma)=T_N(\gamma)$ and  $Ch_A^{\mathbf{\Sigma}}(\alpha_{ij})=\alpha_{ij}^{(N)}$; we refer to \cite{KojuQuesneyClassicalShadows} for details on this delicate step. Eventually, for a closed unmarked surface $\mathbf{\Sigma}=(\Sigma_{g,0}, \emptyset)$, we consider the triangulable surface $\mathbf{\Sigma}'=(\Sigma_{g,1}, \emptyset)$. By Proposition \ref{prop_offpuncture}, the off puncture ideal $\mathcal{I}_p \subset \mathcal{S}_A(\mathbf{\Sigma}')$ is generated by elements $\gamma - \gamma'$, where $\gamma, \gamma'$ are two closed curves isotopic in $\Sigma_{g,0}$. Since $Ch_A^{\mathbf{\Sigma}'}(\gamma - \gamma') = T_N(\gamma)-T_N(\gamma') \in \mathcal{I}_p$, the Chebyshev-Frobenius morphism preserves the off-puncture ideal, so we define $Ch_A^{\mathbf{\Sigma}}$ as the unique morphism making commuting the following diagram
$$\begin{tikzcd}
0 \arrow[r] & \mathcal{I}_p \arrow[r] \arrow[d, "Ch_A^{\mathbf{\Sigma}'}"] & \mathcal{S}_{+1}(\mathbf{\Sigma}') \arrow[d, hook, "Ch_A^{\mathbf{\Sigma}'}"] \arrow[r, "i_*"] &
\mathcal{S}_{+1}(\mathbf{\Sigma}) \arrow[d, dotted, "\exists!"', "Ch_A^{\mathbf{\Sigma}}"] \\
0 \arrow[r] & \mathcal{I}_p \arrow[r]& \mathcal{S}_A(\mathbf{\Sigma}') \arrow[r, "i_*"] &
\mathcal{S}_A(\mathbf{\Sigma})
\end{tikzcd}
$$
Clearly the image of $Ch_A^{\mathbf{\Sigma}}$ is central and we easily deduce its injectivity from the fact that $Ch_A^{\mathbf{\Sigma}}(\gamma)=T_N(\gamma)$ and using that multicurves form a basis.

\end{proof}

  \subsection{Center and PI dimension of stated skein algebras at roots of unity}
  
  \begin{definition} Let $\mathbf{\Sigma}=(\Sigma, \mathcal{A})$ a marked surface.
  \begin{enumerate}
  \item For $p$ an inner puncture (an unmarked connected component of $\partial \Sigma$), we denote by $\gamma_p \in \mathcal{S}_A(\mathbf{\Sigma})$ the class of a peripheral curve encircling $p$ once.
  \item For $\partial \in \pi_0(\partial \Sigma)$ a boundary component which intersects $\mathcal{A}$ non trivially, denote by $p_1, \ldots, p_n$ the boundary punctures in $\partial$ cyclically ordered by $\mathfrak{o}^+$  and define the elements in $\mathcal{S}^{red}_A(\mathbf{\Sigma})$:
  $$ \alpha_{\partial} := \alpha(p_1)_{++} \ldots \alpha(p_n)_{++}, \quad \mbox{ and } \quad \alpha_{\partial}^{-1}:= \alpha(p_1)_{--} \ldots \alpha_(p_n)_{--}.$$
  We easily see that in $\mathcal{S}^{red}_A(\mathbf{\Sigma})$, we have $\alpha_{\partial} \alpha_{\partial}^{-1} =1$.
  \item For $\partial \in \pi_0(\partial \Sigma)$ a boundary component whose intersection with $\mathcal{A}$ is $2n$, for $n\geq 1$, denote by $p_1, \ldots, p_{2n}$ the boundary punctures in $\partial$ cyclically ordered by $\mathfrak{o}^+$. For $k\in \{1, \ldots, N-1\}$, write the product of bad arcs:
 $$\beta_{\partial}^{(N,k)} := \alpha(p_1)_{-+}^{k} \alpha(p_2)_{-+}^{N-k} \ldots \alpha(p_{2n-1})_{-+}^k \alpha(p_{2n})_{-+}^{N-k} \in \mathcal{S}_A(\mathbf{\Sigma}).$$
 We will call \textit{even} such a boundary component $\partial$.
\end{enumerate}
\end{definition}

When trying to classify the representations of stated skein algebras, the knowledge of its center is an important information. 

\begin{theorem}\label{theorem_Posner}(Posner-Formanek  \cite[Theorem $I.13.3$]{BrownGoodearl}) Let $R$ a prime ring which has finite rank $r$ over its center $Z$. Let 
 $S:= Z\setminus \{0\}$ and consider the localization $Q:= S^{-1}R$. Then $Q$ is a central simple algebra with center $S^{-1}Z$.
 \end{theorem}
 
 So there is an algebraic extension $\overline{Z}$ of  $S^{-1}Z$, such that $R\otimes_Z \overline{Z} \cong Mat_D(\overline{Z})$ is a matrix algebra.
  In particular, the rank $r=D^2$ is a perfect square and we call \textit{PI-dimension} of $R$ its square root $D$. Computing the PI-dimensions of stated skein algebras is an important step towards the classification of its representations.
By Theorem \ref{theorem_domain}, when the ground ring $k$ is a domain, the stated skein algebras are domains, so are prime.  

\begin{theorem}\label{theorem_center} Suppose $A\in \mathbb{C}^{\times}$ is a root of unity of odd order $N$ and $\mathbf{\Sigma}$ a marked surface.
\begin{enumerate}
\item (Frohman-L\^e-Kania-Bartoszynska \cite{FrohmanKaniaLe_UnicityRep,FrohmanKaniaLe_DimSkein}) If $\mathbf{\Sigma}$ is unmarked, then $(i)$ the center of $\mathcal{S}_A(\mathbf{\Sigma})$ is generated by the image of the Chebyshev-Frobenius morphism together with the eventual peripheral curves $\gamma_p$ for $p$ an inner puncture. $(ii)$ $\mathcal{S}_A(\mathbf{\Sigma})$ is finitely generated over the image of the Chebyshev-Frobenius morphism (so over its center) and $(iii)$ for $\mathbf{\Sigma}=(\Sigma_{g,n}, \emptyset)$ the PI-dimension of $\mathcal{S}_A(\mathbf{\Sigma})$ is $N^{3g-3+n}$.
\item (K. \cite{KojuAzumayaSkein}) For any marked surface then $(i)$ the center of $\mathcal{S}_A^{red}(\mathbf{\Sigma})$ is generated by the image of the Chebyshev-Frobenius morphism together with the peripheral curves $\gamma_p$ associated to inner punctures and the elements $\alpha_{\partial}^{\pm 1}$ associated to boundary components $\partial \in \pi_0(\partial \Sigma)$. $(ii)$ both $\mathcal{S}_A(\mathbf{\Sigma})$ and $\mathcal{S}^{red}_A(\mathbf{\Sigma})$ are finitely generated over the image of the Chebyshev-Frobenius morphisms (so over their center). $(iii)$ For $\mathbf{\Sigma}=(\Sigma_{g,n}, \mathcal{A})$,  the PI-dimension of $\mathcal{S}^{red}_A(\mathbf{\Sigma})$ is $N^{3g-3+n+|\mathcal{A}|}$.
\item (L\^e-Yu \cite{LeYu_ToAppear}: To appear) For any marked surfaces then $(i)$ the center of $\mathcal{S}_A(\mathbf{\Sigma})$ is generated by the image of the Chebyshev-Frobenius morphism together with the peripheral curves $\gamma_p$ and the elements $\beta^{(N,k)}_{\partial}$ associated to even boundary components and integers $1\leq k \leq N-1$. $(ii)$ For $\mathbf{\Sigma}=(\Sigma_{g,n}, \mathcal{A})$, the PI-dimension of $\mathcal{S}_A(\mathbf{\Sigma})$ is $N^{3g-3+n_{even}+\frac{3}{2}(|\mathcal{A}|+n_{odd})}$, where $n_{odd}, n_{even}$ are the number of boundary components with an odd and even number of boundary arcs respectively (clearly $|\mathcal{A}|$ and $n_{odd}$ have the same parity).
\end{enumerate}
\end{theorem}

The third item of Theorem \ref{theorem_center} is, at the time of writing the present survey, still not prepublished yet, though it was announced in \cite{LeYu_Survey}. In the particular case where $\mathbf{\Sigma}=\mathbb{B}$ (so for $\mathcal{O}_q[\SL_2]$), this is a classical theorem of Enriquez. For $\mathbf{\Sigma}=\mathbf{\Sigma}_{g,0}^0$, this was proved by Ganev-Jordan-Safranov in \cite{GanevJordanSafranov_FrobeniusMorphism}. Note that Theorem \ref{theorem_center} implies that, when $\mathcal{A}\neq \emptyset$, then 
$$ PI-Dim (\mathcal{S}^{red}_A(\mathbf{\Sigma})) < PI-Dim (\mathcal{S}_A(\mathbf{\Sigma})).$$

  \subsection{Finite presentations}\label{sec_presentations}

  For unmarked surfaces, except in genus $0$ and $1$ (\cite{BullockPrzytycki_00}), no finite presentation for the Kauffman-bracket skein algebras is known, though a conjecture in that direction was formulated in \cite[Conjecture $1.2$]{SantharoubaneSkeinGenerators}. However, it is well-known that they are finitely generated (\cite{BullockGeneratorsSkein, AbdielFrohman_SkeinFrobenius, FrohmanKania_SkeinRootUnity, SantharoubaneSkeinGenerators}). The stated skein algebras of marked surfaces is way much easier to study than in the unmarked case: many finite presentations are known. To describe them, we first need some definitions. Let $\mathbf{\Sigma}$ be a connected marked surface with $\mathcal{A}\neq \emptyset$. For each boundary arc $a\in \mathcal{A}$, fix a point $v_a \in a$ and denote by $\mathbb{V}=\{v_a\}_{a\in \mathcal{A}}$ the set of such points. Let $\Pi_1(\Sigma, \mathbb{V})$ be the full subcategory of the fundamental groupoid $\Pi_1(\Sigma)$ generated by $\mathbb{V}$. In other terms, the set of objects of $\Pi_1(\Sigma, \mathbb{V})$ is $\mathbb{V}$ and a morphism $\alpha: v_a\to v_b$ is an homotopy class of path $c_{\alpha}: [0,1] \to \Sigma$ such that $c_{\alpha}(0)=v_a$ and $c_{\alpha}(1)=v_b$; composition is the concatenation of paths. For a path $\alpha: v_a \to v_b$ we write the source  $s(\alpha):= v_a$ and the target $t(\alpha):=v_b$.
  
  \begin{definition}\label{def_presentation}
  \begin{enumerate}
  \item A \textit{set of generators} for $\Pi_1(\Sigma, \mathbb{V})$ is a finite set of paths $\mathbb{G}= \{\beta_1, \ldots, \beta_n\}$ such that any path $\beta$ can be decomposed as
  $$ \beta= \beta_{i_1}^{\varepsilon_1} \ldots \beta_{i_k}^{\varepsilon_k}$$
for some $\varepsilon_l \in \pm$. 
  \item Let $\mathcal{F}(\mathbb{G})$ denote the free semi-group generated by the elements of $\mathbb{G}$ and let $\mathrm{Rel}_{\mathbb{G}}$ (the set of relations) denote the subset of $\mathcal{F}(\mathbb{G})$ of elements of the form $R=\beta_{1}\star \ldots \star \beta_{n}$ such that $s(\beta_{i})= t(\beta_{i+1})$ and such that the path $\beta_1\ldots \beta_n$ is trivial. We write $R^{-1}:= \beta_n^{-1} \star \ldots \star \beta_1^{-1}$. A finite subset $\mathbb{RL}\subset \mathrm{Rel}_{\mathbb{G}}$ is called a \textit{finite set of relations} if   every word  $R\in \mathrm{Rel}_{\mathbb{G}}$ can be decomposed as $ R = \beta \star R_1^{\varepsilon_1} \star \ldots \star R_m^{\varepsilon_m}\star \beta^{-1}$, where $R_i \in \mathbb{RL}$, $\varepsilon_i \in \{ \pm 1 \}$ and $\beta=\beta_{1}\star \ldots \star \beta_{n}\in \mathcal{F}(\mathbb{G})$ is such that $s(\beta_{i})= t(\beta_{i+1})$.
A pair $\mathbb{P}:=(\mathbb{G}, \mathbb{RL})$ is called a \textit{finite presentation} of $\Pi_1(\Sigma_{\mathcal{P}}, \mathbb{V})$. 
  \end{enumerate}
  \end{definition}
  
  To each finite presentation of $\Pi_1(\Sigma, \mathbb{V})$, we will associate a finite presentation of $\mathcal{S}_A(\mathbf{\Sigma})$. Let us first give some exemples of such presentations. 
  
  \begin{example}\label{example_presentations}
  \begin{enumerate}
  \item The triangle $\mathbb{T}$ has a finite presentation with $3$ generators $\mathbb{G}=\{ \beta_1, \beta_2, \beta_3\}$ and unique relation $\mathbb{RL}=\{\beta_3\star \beta_2 \star \beta_1\}$. It has also another presentation with generators $\{\beta_1, \beta_2\}$ and no relation.
  \item Let $(\Gamma, c)$ be a ciliated oriented graph, that is a finite oriented graph together with the data, for each vertex, of a total ordering of its adjacent half-edges. Place a  a disc $D_v$ on top of  each vertex $v$ and a band $B_e$ on top of each edge $e$,  then glue the discs to the band using the underlying cyclic ordering of the adjacent half-edges: we thus get a surface $\Sigma(\Gamma)$. 
  For each vertex $v$ adjacent to half-edges ordered as $e_1<e_2<\ldots <e_n$ place one boundary arc $a_v$ on the boundary of $D_v$ containing the $e_i$ and such that while orienting $a_v$ using the orientation of $\Sigma(\Gamma)$, we have $e_1<_{a_v}e_2<_{a_v} \ldots <_{a_v} e_n$.
  Set $\mathcal{A}(c)=\{a_v\}_{v\in V(\Gamma)}$. By isotoping each vertex $v$ of $\Gamma \subset \Sigma(\Gamma)$ to $a_v$, we get the generating graph of a set of generators $\mathbb{G}= \mathcal{E}(\Gamma)$ (the oriented edges) such that $\mathbb{P}(\Gamma, c):=(\mathcal{E}(\Gamma), \emptyset)$ is a finite presentation of $\Pi_1(\Sigma(\Gamma), \mathbb{V})$  with no relation. Clearly, any connected marked surface with non trivial marking admits such a ciliated graph and therefore admits a presentation without relation. Conversely, any marked surface equipped with a presentation with no relation defines a ciliated oriented graph in an obvious way.
  \item The surface $\mathbf{\Sigma}_{g,n}^0$ admits a special presentation using the so-called Daisy graph illustrated in Figure \ref{daisy_graph} used by Alekseev and its collaborators as we shall review.
    
  \begin{figure}[!h] 
\centerline{\includegraphics[width=3cm]{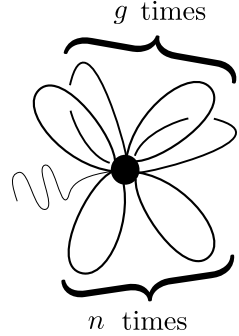} }
\caption{The Daisy graph $\Gamma_{g,n}$ defining a finite presentation of the fundamental groupoid of $\mathbf{\Sigma}_{g,n}^0$.} 
\label{daisy_graph} 
\end{figure} 
  \end{enumerate}
  \end{example}
  
  For $\beta$ a path in $\Pi_1(\Sigma(\Gamma), \mathbb{V})$ and $i,j\in \{\pm \}$, we want to associate $\beta_{ij} \in \mathcal{S}_A(\mathbf{\Sigma})$ as the arc $\beta$ with state $i$ on $s(\beta)$ and $j$ on $t(\beta)$. There is an ambiguity in this definition when $s(\beta)=t(\beta)=v$ since $(1)$ to associate an arc (=connected diagram) to the path $\beta$ we need to separate its endpoints inside the boundary arc and $(2)$ to associate a tangle in $\Sigma\times (0,1)$ to an arc, we need to specify the height order of its endpoints. So there are four possibilities to lift a loop path to a tangle that are illustrated in Figure \ref{fig_types} that we call type $b,c,d,e$. We now suppose that in a given presentation $(\mathbb{G}, \mathbb{RL})$ each path $\beta\in \mathbb{G}$ is equipped with a tangle representative in such a way that the associated arcs have no self-intersection point and are pairwise non-intersecting, so $\beta_{ij} \in \mathcal{S}_A(\mathbf{\Sigma})$ is unambiguously defined. 
  
  \begin{figure}[!h] 
\centerline{\includegraphics[width=8cm]{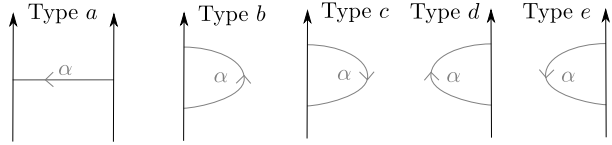} }
\caption{Five types of oriented arcs.} 
\label{fig_types} 
\end{figure} 
  
  \begin{notations}\label{notations_presentations}
 \begin{enumerate}
\item For $\alpha\in \mathbb{G}$, write $M(\alpha) := \begin{pmatrix} \alpha_{++} & \alpha_{+-} \\ \alpha_{-+} & \alpha_{--} \end{pmatrix}$ the $2\times 2$ matrix with coefficients in $\mathcal{S}_A(\mathbf{\Sigma})$. The relations among the generators of $\mathcal{S}_A(\mathbf{\Sigma})$ that we will soon define are simpler when written using of the following matrix
$$
N(\alpha) := \left\{ 
\begin{array}{ll} 
M(\alpha) & \mbox{, if }\alpha \mbox{ is of type }a; \\
M(\alpha)C & \mbox{, if }\alpha \mbox{ is of type }b; \\
M(\alpha) ^tC & \mbox{, if }\alpha \mbox{ is of type }c; \\
C^{-1}M(\alpha) & \mbox{, if }\alpha \mbox{ is of type }d; \\
{}^t C^{-1} M(\alpha) & \mbox{, if }\alpha \mbox{ is of type }e;
\end{array} \right.
$$
 where ${}^tM$ denotes the transpose of $M$.
\item Let $M_{a,b}(R)$ the ring of $a\times b$ matrices with coefficients in some ring $R$ (here $R$ will be $\mathcal{S}_A(\mathbf{\Sigma})$). The \textit{Kronecker product} $\odot : M_{a,b}(R) \otimes M_{c,d}(R) \rightarrow M_{ac, bd}(R)$ is defined by $(A\odot B)_{j,l}^{i,k} = A^i_j B^k_l$. For instance 
$$ M(\alpha) \odot M(\beta) = 
\begin{pmatrix}
 \alpha_{++} \beta_{++} & \alpha_{++} \beta_{+-} & \alpha_{+-} \beta_{++} & \alpha_{+-} \beta_{+-} \\
 \alpha_{++} \beta_{-+} &\alpha_{++} \beta_{--} &\alpha_{+-} \beta_{-+} &\alpha_{+-} \beta_{--}  \\
 \alpha_{-+} \beta_{++} &\alpha_{-+} \beta_{+-} &\alpha_{--} \beta_{++} &\alpha_{--} \beta_{+-} \\
 \alpha_{-+} \beta_{-+} &\alpha_{-+} \beta_{--} &\alpha_{--} \beta_{-+} &\alpha_{--} \beta_{--} 
 \end{pmatrix}.
 $$

 \item Define the flip matrix:
 $$ \tau = \begin{pmatrix}
 1 & 0 & 0 & 0 \\
 0 & 0 & 1 & 0 \\
 0 & 1 & 0 & 0 \\
 0 & 0 & 0 & 1
 \end{pmatrix}.
 $$
\item For $M= \begin{pmatrix} a & b \\ c & d \end{pmatrix}$, we set $\mathrm{det}_q (M) := ad-q^{-1}bc$ and $\mathrm{det}_{q^2}(M) := ad - q^{-2}bc$.
\end{enumerate}
\end{notations} 

Consider a relation $R=\beta_k^{\varepsilon_k}\star \ldots \beta_1^{\varepsilon_1} \in \mathbb{RL}$ and suppose that the concatenation of the representatives arcs of the $\beta_i^{\varepsilon_i}$ forms a contractible closed curve $\gamma$ bounding a disc in $\Sigma$ and such that the orientation of $\gamma$ coincides with the orientation induced from the disc. We suppose that all relations in $\mathbb{RL}$ have this form.
In this case, a simple application of the trivial curve and cutting arc relations, illustrated in the case of the triangle in Figure \ref{fig_polygonrel} and proved in \cite[Lemma $2.23$]{KojuPresentationSSkein}, shows the equality
  \begin{equation}\label{eq_trivial_loops_rel}
\mathds{1}_2 = C M(\beta_k) C^{-1} M(\beta_{k-1}) C^{-1} \ldots C^{-1} M(\beta_1).
\end{equation}
The relations between the $(\beta_k)_{ij}$ obtained from Equation \ref{eq_trivial_loops_rel} (using eventually $M(\beta)^{-1}={}^tM(\beta)$) will be referred to as the \textit{trivial loops relations}.

\begin{figure}[!h] 
\centerline{\includegraphics[width=12cm]{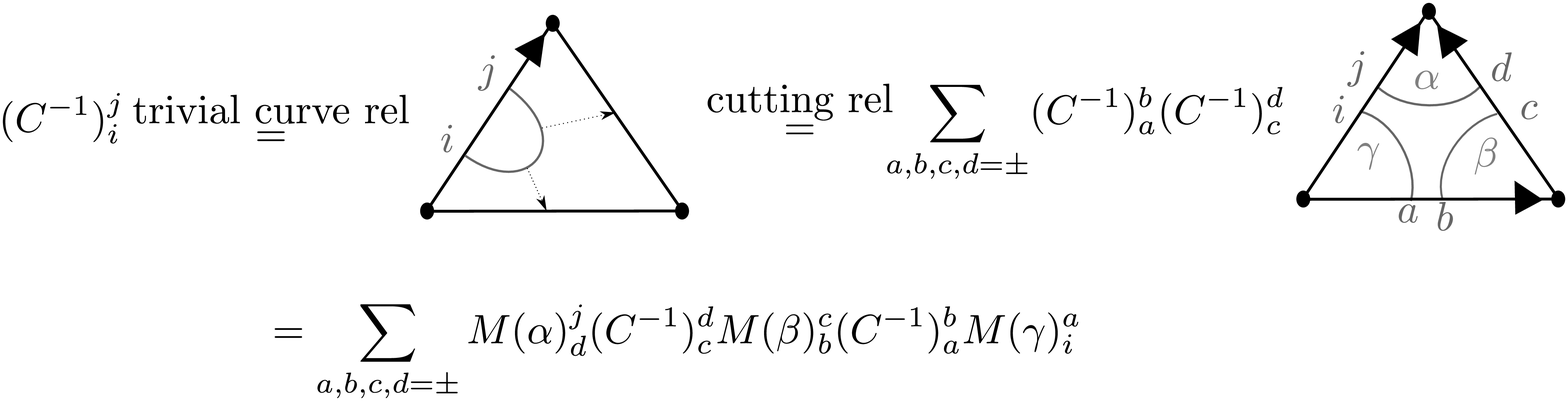} }
\caption{An illustration of  Equation \eqref{eq_trivial_loops_rel} in the case of the triangle.} 
\label{fig_polygonrel} 
\end{figure} 
  
  An easy computation shows (\cite[Lemma $2.25$]{KojuPresentationSSkein}) the following \textit{q-determinant relation}:
  \begin{equation}\label{qdet_rel}
\mathrm{det}_q (N(\alpha))=1 \mbox{, if }\alpha \mbox{ is of type a, and } \mathrm{det}_{q^2}(N(\alpha))= 1 \mbox{, else}.
\end{equation}

Eventually, given $\alpha, \beta \in \mathbb{G}$, we can express any product $\alpha_{ij} \beta_{kl}$ as a linear combination of products of the form $\beta_{ab}\alpha_{cd}$. We call such equations the \textit{arc exchange relations}. Up to orientation reversing, there are $10$ possibilities for the configuration of $\alpha \cup \beta$ depending wether their endpoints lye in distinct boundary arcs or not and depending on the height orders. These $10$ possibilities are depicted in Figure \ref{fig_arcrelations}. 

\begin{figure}[!h] 
\centerline{\includegraphics[width=12cm]{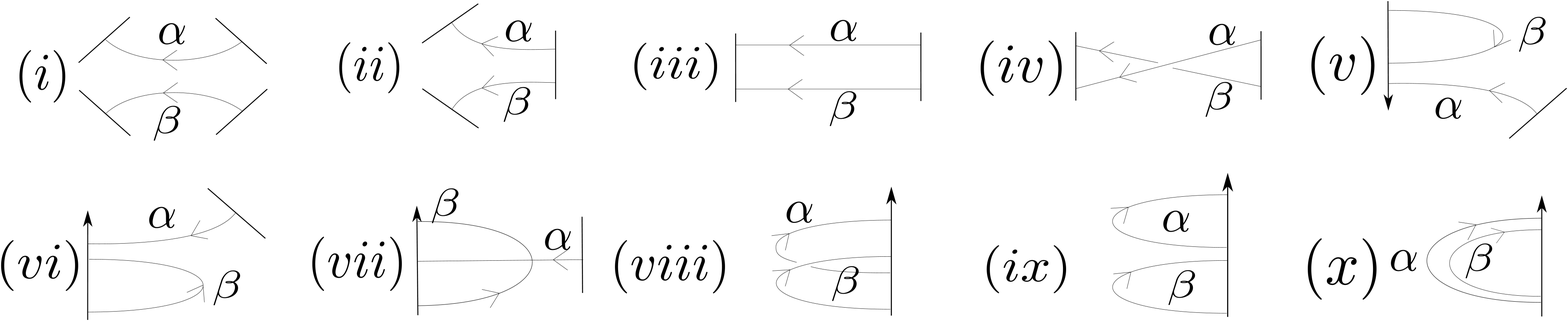} }
\caption{Ten configurations for two non-intersecting oriented arcs.} 
\label{fig_arcrelations} 
\end{figure} 

In each of these $10$ cases, the arc exchange relations write: 

\begin{itemize}
\item[Case (i)] $
N(\alpha) \odot N(\beta) = \tau \left( N(\beta) \odot N(\alpha)\right) \tau.
$
\item[Case (ii)] 
$N(\alpha)\odot N(\beta) = \tau (N(\beta) \odot N(\alpha)) \mathscr{R}.$
\item[Case (iii)] 
$N(\alpha)\odot N(\beta) = \mathscr{R}^{-1}(N(\beta) \odot N(\alpha)) \mathscr{R}.$
\item[Case (iv)] 
$N(\alpha)\odot N(\beta) = \mathscr{R}(N(\beta) \odot N(\alpha)) \mathscr{R}.$
\item[Case(v)]
$N(\alpha) \odot N(\beta) = \mathscr{R}^{-1} \left( N(\beta) \odot \mathds{1}_2 \right) \mathscr{R} \left( N(\alpha) \odot \mathds{1}_2 \right).$
\item[Case (vi)] $N(\alpha) \odot N(\beta) = \mathscr{R}^{-1}\left( N(\beta) \odot \mathds{1}_2 \right) \mathscr{R} \left( N(\alpha) \odot \mathds{1}_2 \right).
$
\item[Case (vii)] $
N(\alpha) \odot N(\beta) = \mathscr{R} \left( N(\beta) \odot \mathds{1}_2 \right) \mathscr{R} \left( N(\alpha) \odot \mathds{1}_2 \right).
$
\item[Case (viii)] $
\left( \mathds{1}_2 \odot N(\alpha) \right) \mathscr{R} \left( \mathds{1}_2 \odot N(\beta) \right) \mathscr{R}^{-1} = 
\mathscr{R} \left( \mathds{1}_2 \odot N(\beta) \right) \mathscr{R}^{-1} \left( \mathds{1}_2 \odot N(\alpha) \right).
$
\item[Case (ix)]$
\mathscr{R}^{-1} \left( \mathds{1}_2 \odot N(\alpha) \right) \mathscr{R} \left( \mathds{1}_2 \odot N(\beta) \right) =
\left( \mathds{1}_2 \odot N(\beta) \right) \mathscr{R}^{-1} \left( \mathds{1}_2 \odot N(\alpha) \right) \mathscr{R}.
$
\item[Case (x)]$
\left( \mathds{1}_2 \odot N(\alpha) \right) \mathscr{R}^{-1} \left( \mathds{1}_2 \odot N(\beta) \right)\mathscr{R} =
\mathscr{R} \left( \mathds{1}_2 \odot N(\beta) \right) \mathscr{R}^{-1} \left( \mathds{1}_2 \odot N(\alpha) \right).
$
\end{itemize}
  
  \begin{theorem}\label{theorem_presentation}(L\^e \cite{LeStatedSkein} for the bigon and the triangle, Faitg \cite{Faitg_LGFT_SSkein} for $\mathbf{\Sigma}_{g,n}^{0}$, K. \cite[Theorem $1.1$]{KojuPresentationSSkein} in general)
  For a connected marked surface $\mathbf{\Sigma}$ with non trivial marking and for a finite presentation $(\mathbb{G}, \mathbb{RL})$ of $\Pi_1(\Sigma, \mathbb{V})$, the stated skein algebra $\mathcal{S}_A(\mathbf{\Sigma})$ has finite presentation with generators the elements $\beta_{ij}$ with $\beta \in \mathbb{G}$ and $i,j=\pm$, and relations the trivial loops relations \eqref{eq_trivial_loops_rel}, the q-determinant relations \eqref{qdet_rel} and the arc exchange relations.
  \end{theorem}
  While choosing a finite presentation with no relation (such as the ones given by ciliated graphs), the associated presentation of $\mathcal{S}_A(\mathbf{\Sigma})$ is quadratic inhomogeneous. A suitable application of the Diamond Lemma for PBW bases then show that
  
  \begin{theorem}(K. \cite[Theorem $1.2$]{KojuPresentationSSkein})   For a connected marked surface $\mathbf{\Sigma}$ with non trivial marking, the stated skein algebra $\mathcal{S}_A(\mathbf{\Sigma})$ is Koszul.
  \end{theorem}
  
  For a ciliated graph $(\Gamma,c)$, the algebra presented using $\mathbb{P}(\Gamma,c)$ was first defined by Alekseev-Grosse-Schomerus \cite{AlekseevGrosseSchomerus_LatticeCS1,AlekseevGrosseSchomerus_LatticeCS2, AlekseevSchomerus_RepCS} and Buffenoir-Roche \cite{BuffenoirRoche, BuffenoirRoche2} and is called the \textit{quantum moduli algebra}. So Theorem \ref{theorem_presentation} says that stated skein algebras are essentially a presentation-free reformulation of quantum moduli algebras. It particular, this relation makes  obvious the fact that quantum moduli algebras only depend on the underlying marked surface and not on the choice of the ciliated graph. Relations between these algebras were previously enlightened in \cite{BullockFrohmanKania_LGFT} for unmarked surfaces and in \cite{Faitg_LGFT_SSkein, BenzviBrochierJordan_FactAlg1} for $\mathbf{\Sigma}_{g,n}^0$ with its daisy graph.

  \section{Quantum Teichm\"uller theory}
  
  \subsection{Quantum tori}
  
  \begin{definition}
  \begin{enumerate}
  \item A \textit{quadratic pair} is a pair  $\mathbb{E}=(E,\left(\cdot, \cdot\right)_E)$ where $E$ is a free $\mathbb{Z}$ module of finite rank and $\left(\cdot, \cdot\right)_E : E\times E \rightarrow \mathbb{Z}$ is a skew-symmetric bilinear form.  A morphism of quadratic pairs is a linear maps that preserves the pairings. Quadratic pairs with direct sum $\oplus$ form a symmetric monoidal category $(\mathrm{QP}, \oplus)$.
  \item For $k$ a commutative (unital, associative) ring and $\omega \in k^{\times}$ invertible, the \textit{quantum torus} $\mathbb{T}_{\omega}(\mathbb{E})$ is the complex algebra with underlying vector space $\mathbb{C}[E]$ and product given by $[x]\cdot [y] := \omega^{(x,y)_E}[x+y]$. Given $e=(e_1,\ldots, e_n)$ a basis of $E$, the quantum torus $\mathbb{T}_{\omega}(\mathbb{E})$ is isomorphic to the complex  algebra generated by invertible elements $Z_{e_i}^{\pm 1}$ with relations $Z_{e_i}Z_{e_j}=\omega^{2(e_i,e_j)_E}Z_{e_j}Z_{e_i}$. 
  Clearly, this defines a monoidal functor $\mathbb{T}_{\omega} : (\mathrm{QP}, \oplus) \to (\mathrm{Alg}_k, \otimes_k)$.
  \item The \textit{classical torus} $\mathcal{X}(\mathbb{E})$ is the Poisson affine variety with algebra of regular functions $\mathbb{C}[E]$ and Poisson bracket induced by
  $ \{[x], [y]\} = 2(x,y)[x+y]$. 
  \end{enumerate}
  \end{definition}
  
  The representation theory of quantum tori at roots of unity is particularly simple:
  
  \begin{theorem}(DeConcini-Procesi \cite[Proposition $7.2$]{DeConciniProcesiBook}) If $\omega \in \mathbb{C}^*$ is a root of unity, then $\mathbb{T}_{\omega}(\mathbb{E})$ is Azumaya of constant rank.
  \end{theorem}
  
  In particular,  $\mathbb{T}_{\omega}(\mathbb{E})$ is semi-simple and the character map $\chi: \mathrm{Irrep} \to \mathrm{Specm}(\mathcal{Z})$ gives a bijection between the set of isomorphism classes of its irreducible representations and the set of character $\chi: \mathcal{Z}\to \mathbb{C}$ over its center $\mathcal{Z}$. Moreover, its simple modules all have the same dimension which is the PI-dimension of  $\mathbb{T}_{\omega}(\mathbb{E})$. So in order to classify the representations of a quantum torus, we need to find its center and to compute its PI-dimension. Suppose that $\omega^2$ is a root of unity of order $N$. 
  Let $E_N \subset E$ be the left kernel of the composition 
  $$ \left(\cdot, \cdot \right)_N : E\times E \xrightarrow{(\cdot, \cdot )} \mathbb{Z} \rightarrow \mathbb{Z}/N\mathbb{Z}.$$
  It is easy to see that the center of $\mathbb{T}_{\omega}(\mathbb{E})$ is generated by the elements $Z_{e_0}$ for $e_0 \in E_N$. Define the \textit{Frobenius morphism}
  $$Fr_{\omega}: \mathbb{T}_{+1}(\mathbb{E}) \hookrightarrow \mathbb{T}_{\omega}(\mathbb{E})$$
  by $Fr_{\omega}(Z_e):= Z_e^N$. Since $N\cdot E \subset K_N$, the image of the Frobenius morphism is central. Clearly, if $e_0$ is in the kernel of $\left(\cdot, \cdot \right)$, then its is in the kernel of $\left(\cdot, \cdot \right)_N$ so $Z_{e_0}$ is central regarding whether $\omega$ is a root of unity or not. We call \textit{Casimir} such elements.

  \subsection{The balanced Chekhov-Fock algebra}\label{sec_CF}
  
  The goal of this section is to present a deep and powerful result of Bonahon and Wong \cite{BonahonWongqTrace} which states that skein algebras can be embedded into some quantum tori. The quantum tori in question are called the balanced Chekhov-Fock algebras. Fix $(\mathbf{\Sigma}, \Delta)$ a triangulated marked surface.
  
  \begin{definition}
  A map $k : \mathcal{E}(\Delta)\to \mathbb{Z}$ is \textit{balanced} if for any face $\mathbb{T}$ of the triangulation with edges $a,b,c$ then $k(a)+k(b)+k(c)$ is even. We denote by $K_{\Delta}$ the $\mathbb{Z}$-module of balanced maps.  For  $e$ and $e'$ two edges, denote by $a_{e,e'}$ the number of faces $\mathbb{T}\in F(\Delta)$ such that $e$ and $e'$ are edges of $\mathbb{T}$ and such that we pass from $e$ to $e'$ in the counter-clockwise direction in $\mathbb{T}$. The \textit{Weil-Petersson}  form $\left(\cdot, \cdot \right)^{WP}: K_{\Delta} \times K_{\Delta} \rightarrow \mathbb{Z}$ is the skew-symmetric form defined  by $\left( \mathbf{k}_1, \mathbf{k}_2\right)^{WP}:= \sum_{e,e'} \mathbf{k}_1(e)\mathbf{k}_2(e')(a_{e,e'}-a_{e',e})$.

  The \textit{balanced Chekhov-Fock algebra} is the quantum torus $\mathcal{Z}_{\omega}(\mathbf{\Sigma}, \Delta):= \mathbb{T}_{\omega}(K_{\Delta}, (\cdot, \cdot)^{WP})$.
  \end{definition}
  
  If $\mathcal{E}(\Delta)=\{e_1, \ldots, e_n\}$, we will write $Z^{\mathbf{k}}=[Z_{e_1}^{k(e_1)}\ldots Z_{e_n}^{k(e_n)}]$. Given $e\in \mathcal{E}(\Delta)$, one has a balanced map $\mathbf{k}_e$ sending $e$ to $2$ and all other edges to $0$. We write $X_e:= Z^{\mathbf{k}_e}$. The balanced Chekhov-Fock algebra has a natural graduation defined as follows.
   An edge $e\in \mathcal{E}(\Delta)$ defines a Borel-Moore homology class $[e]\in \mathrm{H}_1^c(\Sigma; \mathbb{Z}/2\mathbb{Z})$. Given a balanced monomial $Z^{\mathbf{k}}\in \mathcal{Z}_{\omega}(\mathbf{\Sigma}, \Delta)$, there exists a unique class $[\mathbf{k}] \in \mathrm{H}_1(\Sigma, \mathcal{A}; \mathbb{Z}/2\mathbb{Z})$ whose algebraic intersection with $[e]$ is $\mathbf{k}(e)$ modulo $2$. 

\begin{definition}
For $\chi \in \mathrm{H}_1(\Sigma, \mathcal{A}; \mathbb{Z}/2\mathbb{Z})$, we define the subspace:
$$ \mathcal{Z}_{\omega}^{(\chi)}(\mathbf{\Sigma},\Delta) := \mathrm{Span} \left( Z^{\mathbf{k}} | [\mathbf{k}]=\chi \right)\subset \mathcal{Z}_{\omega}(\mathbf{\Sigma},\Delta). $$
\end{definition}

It follows from definitions that $\mathcal{Z}_{\omega}(\mathbf{\Sigma}, \Delta)=  \oplus_{\chi \in \mathrm{H}_1(\Sigma, \mathcal{A}; \mathbb{Z}/2\mathbb{Z})}  \mathcal{Z}_{\omega}^{(\chi)}(\mathbf{\Sigma}, \Delta)$ is a graded algebra. Note that the $0$-graded part $ \mathcal{Z}_{\omega}^{(0)}(\mathbf{\Sigma}, \Delta)$ is generated, as an algebra, by the elements $X_e$ for $e\in \mathcal{E}(\Delta)$. This is an exponential version of Chekhov and Fock's original quantum Teichm\"uller space in \cite{ChekhovFock}.

Balanced Chekhov-Fock algebras have the same behaviour for the gluing operation than stated skein algebras. Let $(\mathbf{\Sigma}_{a\#b}, \Delta_{a\#b})$ be the triangulated punctured surface obtained from $(\mathbf{\Sigma}, \Delta)$ by gluing two boundary arcs together and call $c$ the edge of $ \Delta_{a\#b}$ corresponding to $a$ and $b$.

Define an injective algebra morphism $\theta_{a\#b} : \mathcal{Z}_{\omega}(\mathbf{\Sigma}_{|a\#b}, \Delta_{|a\#b}) \hookrightarrow \mathcal{Z}_{\omega}(\mathbf{\Sigma},\Delta)$  by $\theta_{a\#b}(Z^{\mathbf{k}}) := Z^{\mathbf{k'}}$, where $\mathbf{k'}(a)=\mathbf{k'}(b):=\mathbf{k}(c)$ and $\mathbf{k'}(e)=\mathbf{k}(e)$, if $e\neq a,b$.

\par The images of the gluing maps for balanced Chekhov-Fock algebras are characterized in a similar manner than for the stated skein algebras. For $a$ a boundary arc of $\mathbf{\Sigma}$, define some left and right comodule maps $\Delta_a^L : \mathcal{Z}_{\omega}(\mathbf{\Sigma}, \Delta) \rightarrow k[Z_a^{\pm 1}] \otimes \mathcal{Z}_{\omega}(\mathbf{\Sigma}, \Delta)$ and $\Delta_a^R : \mathcal{Z}_{\omega}(\mathbf{\Sigma}, \Delta) \rightarrow  \mathcal{Z}_{\omega}(\mathbf{\Sigma}, \Delta) \otimes k[Z_b^{\pm 1}] $ by the formulas $\Delta_a^L (Z^{\mathbf{k}}):= Z_a^{\mathbf{k}(a)} \otimes Z^{\mathbf{k}}$ and $\Delta_a^R(Z^{\mathbf{k}}):= Z^{\mathbf{k}}\otimes Z_a^{\mathbf{k}(a)}$. 
By definition, the following sequence is exact:
$$ 0 \to \mathcal{Z}_{\omega}(\mathbf{\Sigma}_{a\#b}, \Delta_{a\#b}) \xrightarrow{\theta_{a\#b}} \mathcal{Z}_{\omega}(\mathbf{\Sigma},\Delta) \xrightarrow{\Delta_a^L - \sigma \circ \Delta_b^R} k[Z^{\pm 1}] \otimes \mathcal{Z}_{\omega}(\mathbf{\Sigma},\Delta).$$

In particular, like for stated skein algebra, we get an exact sequence
\begin{equation}\label{eq_triangulationCF}
0 \to \mathcal{Z}_{\omega}(\mathbf{\Sigma}, \Delta) \xrightarrow{\theta^{\Delta}} \otimes_{\mathbb{T}\in F(\Delta)} \mathcal{Z}_{\omega}(\mathbb{T}) \xrightarrow{\Delta^L - \sigma \circ \Delta^R} \left( \otimes_{e \in \mathring{\mathcal{E}}(\Delta)} k[Z_e^{\pm 1}]\right) \otimes \left(  \otimes_{\mathbb{T}\in F(\Delta)} \mathcal{Z}_{\omega}(\mathbb{T}) \right).
\end{equation}

We now turn to the problem of classifying the representations of $\mathcal{Z}_{\omega}(\mathbf{\Sigma}, \Delta)$.

\begin{definition}\label{def_central_elements}
\begin{itemize}
\item  Let $p$ be an inner puncture. For each edge $e\in \mathcal{E}(\Delta)$, denote by $\mathbf{k}_p(e)\in \{0,1,2\}$ the number of endpoints of $e$ equal to $p$. The \textit{central inner puncture element} is  $H_p:=Z^{\mathbf{k}_p}\in \mathcal{Z}_{\omega}(\mathbf{\Sigma}, \Delta)$.
\item  Let  $\partial$ a connected component of $\partial \Sigma$. For each edge $e$, denote by $\mathbf{k}_{\partial}(e)\in \{0, 1, 2\}$ the number of endpoints of $e$ lying in $\partial$. The \textit{central boundary element} is $H_{\partial}:=Z^{\mathbf{k}_{\partial}}\in \mathcal{Z}_{\omega}(\mathbf{\Sigma}, \Delta)$. 
\end{itemize}
\end{definition}

It follows from the definition of the Weil-Petersson form that the elements $H_p$ and $H_{\partial}$ are Casimir (so central). 

\begin{theorem}\label{theorem_centerCF}(Bonahon-Wong \cite{BonahonWong1} for unmarked surfaces, K.-Quesney \cite{KojuQuesneyQNonAb} in general)
Suppose $\omega \in \mathbb{C}^{\times}$ is a root of unity of odd order $N$, then 
\begin{enumerate}
\item The center of $\mathcal{Z}_{\omega}(\mathbf{\Sigma}, \Delta)$ is generated by the Casimir elements $H_{\partial}^{\pm 1}$ and $H_p^{\pm 1}$ together with the image of the Frobenius morphism.
\item The PI-dimension of $\mathcal{Z}_{\omega}(\mathbf{\Sigma},\Delta)$ is equal to the PI-dimension of $\mathcal{S}^{red}_{\omega}(\mathbf{\Sigma})$ so, when $\mathbf{\Sigma}=(\Sigma_{g,n}, \mathcal{A})$, it is equal to $N^{3g-3+n +|\mathcal{A}| }$.
\end{enumerate}
\end{theorem}

Let us give the main ingredients of the proof of Theorem \ref{theorem_centerCF}. The idea is to find a geometric interpretation of the quadratic pair $(K_{\Delta}, (\cdot, \cdot)^{WP})$ and for this, we will associate to the triangulated marked surface $(\mathbf{\Sigma}, \Delta)$ a double branched covering.  Let  $\Gamma^{\dagger}\subset \Sigma$ be the dual of the one-squeleton of the triangulation. The graph $\Gamma^{\dagger}$ has one vertex inside each face and one edge $e^*$ intersecting once transversally each edge $e$ of the triangulation. Denote by $B$ the set of its vertices. Let $[\Gamma^{\dagger}]$ denotes its Borel-Moore relative homology class in $\mathrm{H}_1^c(\Sigma \setminus B, \partial \Sigma; \mathbb{Z})$ and $\phi\in \mathrm{H}^1(\Sigma \setminus B; \mathbb{Z}/2\mathbb{Z})$ the Poincar\'e-Lefschetz dual of $[\Gamma^{\dagger}]$ sending a class $[\alpha]$ to its algebraic intersection with $[\Gamma^{\dagger}]$ modulo $2$.  

\begin{definition} The covering $ \pi : \hat{\Sigma} \rightarrow \Sigma$ is the double covering of $\Sigma$ branched along $B$ defined by $\phi$. \end{definition}

   \par   Figure \ref{figtrianglecov} illustrates the covering. Let $\widehat{B}=\pi^{-1}(B)$ and $\widehat{\mathcal{A}}=\pi^{-1}(\mathcal{A})$. Write $\widehat{\mathbf{\Sigma}}=(\widehat{\Sigma}\setminus \widehat{B}, \widehat{\mathcal{A}})$.
   
     \begin{figure}[!h] 
\centerline{\includegraphics[width=12cm]{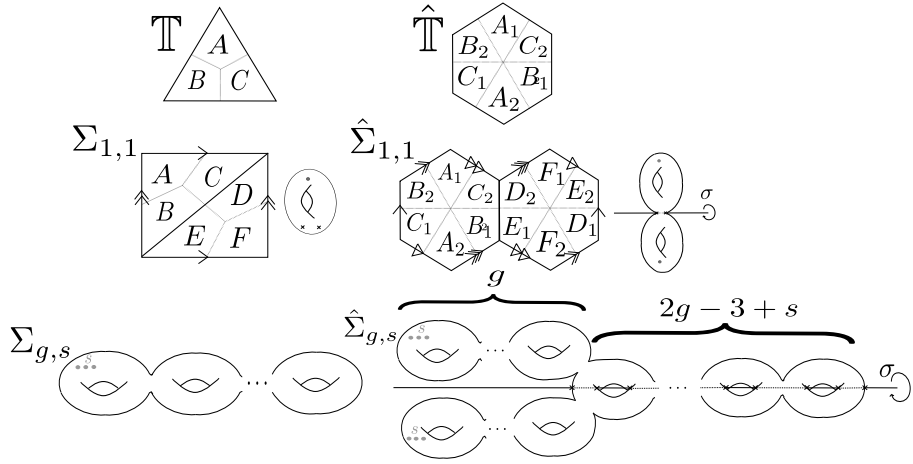} }
\caption{On the top, the triangle and its double branched covering. The dashed lines represent the dual graph $\Gamma^{\dagger}$, its edges are the branched cuts and its vertices the branched point of the covering. The covering involution is the central symmetry along the branched point.  In the middle, a triangulated once-punctured and the corresponding covering, a twice-punctured genus two surface. At the bottom, the double covering of a genus $g$ surface with $s$ punctures. } 
\label{figtrianglecov} 
\end{figure} 

Given $(\Sigma, \mathcal{A})$ a marked surface, we now define a skew-symmetric form $(\cdot, \cdot) : \mathrm{H}_1(\Sigma, \mathcal{A}; \mathbb{Z}) \to \mathbb{Z}$
 Let $c_1, c_2 : [0,1] \rightarrow \Sigma$ be two cycles in $\mathrm{Z}_1(\Sigma, \mathcal{A} ;  \mathbb{Z})$, that is such that for $i=1,2$, either $c_i(0)=c_i(1)$ or $c_i(0), c_i(1) \in \mathcal{A}$. Further suppose that the images of $c_1$ and $c_2$ intersect transversally in the interior of $\Sigma$ along simple crossings. Let $v\in c_1 \cap c_2\subset \mathring{\Sigma}$ be an intersection point and denote by $e_1, e_2 \in T_v \Sigma$ the tangent vectors of $c_1$ and $c_2$ respectively at the point $v$. We define the sign intersection $\varepsilon(v)= +1$ if $(e_1, e_2)$ is an oriented basis of $T_v \Sigma$ and $\varepsilon(v)=-1$ else. Let $a$ be a boundary arc and denote by $S(a)$ the set of pairs $(v_1, v_2)$ of points such that $v_i \in c_i \cap a$. Note that $c_1$ and $c_2$ do not have intersection point in $a$ by definition. Given $(v_1, v_2) \in S(a)$, we define a sign $\varepsilon(v_1, v_2) \in \{ \pm 1 \}$ as follows. Isotope $c_1$ around $a$ to bring $v_1$ in the same position than $v_2$ and denote by $c_1'$ the new geometric curve. The isotopy should preserve the transversality condition and should not make appear any new inner intersection points. Then define $e_1, e_2\in T_{v_2} \Sigma$ the tangent vectors at $v_2$ of $c_1'$ and $c_2$ respectively. Define $\varepsilon(v_1, v_2) = +1$ if $(e_1, e_2)$ is an oriented basis and $\varepsilon(v_1, v_2)=-1$ else. 

\begin{definition}\label{def_relative_intersection} The relative intersection form  $(\cdot, \cdot) : \mathrm{H}_1(\Sigma, \mathcal{A} ;  \mathbb{Z})^{\otimes 2} \rightarrow \frac{1}{2} \mathbb{Z}$ is defined by:
$$ (c_1, c_2) := \sum_a \sum_{(v_1, v_2)\in S(a)} \frac{1}{2} \varepsilon(v_1, v_2)  + \sum_{v\in c_1\cap c_2} \varepsilon(v)$$
\end{definition}
An easy computation shows that the value $(c_1, c_2)$ only depends on the relative homology classes of $c_1, c_2$. Note that when $\mathcal{A}=\emptyset$ then $(\cdot, \cdot)$ is the usual intersection form.

Now let us consider again a triangulated marked surface $(\Sigma, \mathcal{A})$ with double covering $(\widehat{\Sigma}\setminus \widehat{B}, \widehat{\mathcal{A}})$. Denote by $\sigma: \widehat{\Sigma} \to \widehat{\Sigma}$ the covering involution and let $\mathrm{H}_1^{\sigma}(\widehat{\Sigma}\setminus \widehat{B}, \widehat{\mathcal{A}}; \mathbb{Z}) \subset \mathrm{H}_1(\widehat{\Sigma}\setminus \widehat{B}, \widehat{\mathcal{A}}; \mathbb{Z})$ the submodule of classes $[\alpha]$ such that $\sigma_* ([\alpha]) = -[\alpha]$. The relative intersection form on $\mathrm{H}_1(\widehat{\Sigma}\setminus \widehat{B}, \widehat{\mathcal{A}}; \mathbb{Z})$ restricts to a skew-symmetric map (still denoted by the same letter) on $\mathrm{H}_1^{\sigma}(\widehat{\Sigma}\setminus \widehat{B}, \widehat{\mathcal{A}}; \mathbb{Z}) $ with integral values and 

\begin{theorem}\label{theorem_doublecovering} (Bonahon-Wong \cite{BonahonWong1} for unmarked surfaces, K.-Quesney \cite{KojuQuesneyQNonAb} in general)
The quadratic pair $(K_{\Delta}, (\cdot, \cdot)^{WP})$ is isomorphic to $(\mathrm{H}_1^{\sigma}(\widehat{\Sigma}\setminus \widehat{B}, \widehat{\mathcal{A}}; \mathbb{Z}) , (\cdot, \cdot))$.
\end{theorem}

Theorem \ref{theorem_doublecovering} is proved by first noticing that it is trivial for a triangle and then studying the behaviour of the gluing operation for the quadratic pairs. Now Theorem \ref{theorem_centerCF} appears more naturally. For instance consider an unmarked surface $\mathbf{\Sigma}=(\Sigma_{g,s}, \emptyset)$. The double branched covering $\widehat{\mathbf{\Sigma}}$ is depicted in Figure \ref{figtrianglecov} and is obtained by gluing together $g$ pairs of tori which are exchanged by $\sigma$ together with $2g-3+s$ tori on which the covering the elliptic involution together with two balls having $s$ inner punctured exchanged by the involution. From this geometric description, we get a decomposition 
$$(\mathrm{H}_1^{\sigma}(\widehat{\Sigma}\setminus \widehat{B}, \widehat{\mathcal{A}}; \mathbb{Z}) , (\cdot, \cdot)) \cong \begin{pmatrix} 0 & 1 \\ -1 & 0 \end{pmatrix}^{\oplus g} \oplus \begin{pmatrix} 0 & 2 \\ -2 & 0 \end{pmatrix}^{\oplus 2g-3+s} \oplus (0)^{\oplus s}, $$
 from which we deduce Theorem \ref{theorem_centerCF}. The case of marked surfaces is similar though slightly more technical.

Note that the set of closed points of the Poisson torus $\mathcal{X}(\mathrm{H}_1^{\sigma}(\widehat{\Sigma}\setminus \widehat{B}, \widehat{\mathcal{A}}; \mathbb{Z}) , (\cdot, \cdot))$ can be described geometrically as the set of gauge equivalence of flat connections on a trivial $\mathbb{C}^*$ bundle over $\widehat{\Sigma}\setminus \widehat{B}$ which is equivariant in the sense that its holonomy along an arc $\alpha$ is the inverse of the holonomy along $\sigma(\alpha)$. This moduli space was considered by Gaitto-Moore-Neitzke in \cite{GMN12, GMN13}. 
  
  \subsection{Bonahon-Wong's quantum trace}
  
  Stated skein algebras were designed by Bonahon and Wong in \cite{BonahonWongqTrace} to permit the definition of the quantum trace. Thanks to the deep work of L\^e in \cite{LeStatedSkein}, the definition of the quantum trace is now very simple. Fix a commutative ring $k$ with $A^{1/2}\in k^{\times}$ and write $\omega:= A^{-1/2}$ its inverse.
   First consider the triangle $\mathbb{T}$ with edges $e_1, e_2, e_3$ and arcs $\alpha_i$ as in Figure \ref{fig_triangle}. 
   
   \begin{figure}[!h] 
\centerline{\includegraphics[width=2cm]{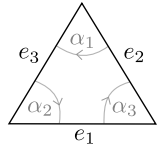} }
\caption{The triangle and some arcs.} 
\label{fig_triangle} 
\end{figure}

   For $i\in \mathbb{Z}/3\mathbb{Z}$, let $\mathbf{k}_i $ be the balanced map sending $e_i$ to $0$ and $e_{i+1}, e_{i+2}$ to $1$. Using the explicit presentation of the triangle in Theorem \ref{theorem_presentation}, it is easy to see 
   that the linear map $\Tr^{\mathbb{T}} : \mathcal{S}_A(\mathbb{T}) \rightarrow \mathcal{Z}_{\omega}(\mathbb{T})$, sending $(\alpha_i)_{-+}$ to $0$, sending $(\alpha_i)_{++}$ to $Z^{\mathbf{k}_i}$ and sending $(\alpha_i)_{--}$ to $Z^{-\mathbf{k}_i}$, extends to a morphism of algebras. Note that $\mathcal{S}_A(\mathbb{T})$ is a $\mathcal{O}_q[\SL_2]^{\otimes \mathcal{E}(\mathbb{T})}$-comodule and that  $\mathcal{Z}_{\omega}(\mathbb{T})$ is a $k[Z^{\pm 1}]^{\otimes \mathcal{E}(\mathbb{T})}$-comodule. Using the morphism $\varphi :  \mathcal{O}_q[\SL_2] \to
   k[Z^{\pm 1}]$ defined by $\varphi(\alpha_{+-})=\varphi(\alpha_{-+})=0$, $\varphi(\alpha_{\varepsilon \varepsilon})= Z^{\varepsilon}$, we see that $\tr^{\mathbb{T}}$ is equivariant for the induced $k[Z^{\pm 1}]^{\otimes \mathcal{E}(\mathbb{T})}$-coaction. In particular, for $(\mathbf{\Sigma}, \Delta)$ a triangulated marked surface, we get the commutative diagram

\begin{equation}\label{diagram_qtrace}
\begin{tikzcd}
0 \arrow[r] & \mathcal{S}_A(\mathbf{\Sigma}) \arrow[r, "\theta^{\Delta}"] \arrow[d, dotted, "\tr^{\Delta}", "\exists!"'] 
& \otimes_{\mathbb{T}\in F(\Delta)} \mathcal{S}_A(\mathbb{T}) \arrow[d, "\otimes_{\mathbb{T}}\tr^{\mathbb{T}}"] 
\arrow[r, "\Delta^L - \sigma \circ \Delta^R"]
& \left( \otimes_{e\in \mathring{\mathcal{E}}(\Delta)} \mathcal{O}_q[\SL_2] \right) \otimes \left(  \otimes_{\mathbb{T}\in F(\Delta)} \mathcal{S}_A(\mathbb{T}) \right)
\arrow[d, "(\otimes_e \varphi) \otimes (\otimes_{\mathbb{T}} \Tr^{\mathbb{T}})"]
\\
0 \arrow[r] & \mathcal{Z}_{\omega}(\mathbf{\Sigma}, \Delta) \arrow[r,  "\theta^{\Delta}"] 
& \otimes_{\mathbb{T}\in F(\Delta)}\mathcal{Z}_{\omega}(\mathbb{T})
\arrow[r, "\Delta^L - \sigma \circ \Delta^R"]
& \left( \otimes_{e\in \mathring{\mathcal{E}}(\Delta)}k[Z_e^{\pm 1}] \right) \otimes \left(  \otimes_{\mathbb{T}\in F(\Delta)}\mathcal{Z}_{\omega}(\mathbb{T}) \right)
\end{tikzcd}
\end{equation}
  
  \begin{definition} The \textit{quantum trace} is the unique algebra morphism $\Tr^{\Delta} : \mathcal{S}_A(\mathbf{\Sigma}) \to \mathcal{Z}_{\omega}(\mathbf{\Sigma}, \Delta)$ making commuting Diagram \ref{diagram_qtrace}.
  \end{definition}
  
  The following theorem justifies the definition of the reduced stated skein algebra.
  \begin{theorem}(Bonahon-Wong \cite{BonahonWongqTrace} for unmarked surfaces, Costantino-L\^e \cite[Theorem $7.12$]{CostantinoLe19} in general)
  The quantum trace induces an injective morphism of algebras $\Tr^{\Delta} : \mathcal{S}^{red}_A(\mathbf{\Sigma}) \hookrightarrow \mathcal{Z}_{\omega}(\mathbf{\Sigma}, \Delta)$.
  \end{theorem}
  
  \begin{proof}
  For the triangle, it is a straightforward consequence of the presentation of the triangle (Theorem \ref{theorem_presentation}) that the quantum trace induces an isomorphism $\Tr^{\mathbb{T}} : \mathcal{S}^{red}_A(\mathbb{T}) \xrightarrow{\cong} \mathcal{Z}_{\omega}(\mathbb{T})$. In general, the injectivity follows from the commutativity of the following diagram
  $$\begin{tikzcd}
\mathcal{S}^{red}_A(\mathbf{\Sigma}) \arrow[r, hook, "\theta^{\Delta}"] \arrow[d, "\tr^{\Delta}"] 
& \otimes_{\mathbb{T}\in F(\Delta)} \mathcal{S}^{red}_A(\mathbb{T}) \arrow[d, "\otimes_{\mathbb{T}}\tr^{\mathbb{T}}", "\cong"'] \\
\mathcal{Z}_{\omega}(\mathbf{\Sigma}, \Delta) \arrow[r, hook, "\theta^{\Delta}"] 
& \otimes_{\mathbb{T}\in F(\Delta)}\mathcal{Z}_{\omega}(\mathbb{T})
\end{tikzcd}$$
  
  \end{proof}
  
  The quantum trace preserves the $\mathrm{H}_1(\Sigma, \mathcal{A}; \mathbb{Z}/2\mathbb{Z})$ grading (\cite[Lemma $2.27$]{KojuQGroupsBraidings}). 
  Also the quantum trace intertwines the Chebyshev-Frobenius morphism with the Frobenius morphism in the sense that the following diagram commutes:
  $$\begin{tikzcd}
  \mathcal{S}^{red}_{+1}(\mathbf{\Sigma}) \arrow[r, hook, "\tr^{\Delta}"] \arrow[d, hook, "Ch_A"] & 
  \mathcal{Z}_{+1}(\mathbf{\Sigma}, \Delta) \arrow[d, hook, "Fr_{\omega}"] \\
    \mathcal{S}^{red}_{A}(\mathbf{\Sigma}) \arrow[r, hook, "\tr^{\Delta}"]& 
  \mathcal{Z}_{\omega}(\mathbf{\Sigma}, \Delta) 
  \end{tikzcd}
  $$
  The commutativity of this diagram follows from the obvious fact that it commutes when $\mathbf{\Sigma}=\mathbb{T}$. So the Chebyshev-Frobenius morphism (for triangulable surfaces) is the restriction of the Frobenius morphism: this is how Bonahon and Wong made it first appear in \cite{BonahonWong1}.

  \subsection{L\^e-Yu's enhancement of the quantum trace}
  
  Bonahon and Wong's quantum trace embeds the reduced stated skein algebras into quantum tori. In order to embed the whole stated skein algebra into some quantum torus, L\^e and Yu defined in \cite{LeYu_SSkeinQTraces} an enhancement of the quantum trace that we now describe. Let $(\mathbf{\Sigma}, \Delta)$ be a triangulated marked surface and denote by $(\mathbf{\Sigma}^*, \Delta^*)$ the triangulated marked surface obtained from $(\mathbf{\Sigma}, \Delta)$ by gluing a triangle along each boundary arc of $\mathbf{\Sigma}$. So each boundary arc of $\mathbf{\Sigma}$ corresponds to two boundary arcs, say $a'$ and $a''$ in $\mathbf{\Sigma}^*$. Let $i: \mathbf{\Sigma} \to \mathbf{\Sigma}^*$ be the embedding which is the identity outside a small neighborhood of $\mathcal{A}$ and embedding $a$ into $a'$. Note that the morphism $i_* : \mathcal{S}_A(\mathbf{\Sigma}) \to \mathcal{S}_A(\mathbf{\Sigma}^*)$ sends injectivly the basis $\mathcal{B}^{\mathfrak{o}_+}(\mathbf{\Sigma})$ of $\mathcal{S}_A(\mathbf{\Sigma})$  to a subset of the basis $\overline{\mathcal{B}}^{\mathfrak{o}_+}(\mathbf{\Sigma}^*)$ of $\mathcal{S}^{red}_A(\mathbf{\Sigma}^*)$ (which does not contain any bad arc) so $i_*$ induces an injective morphism 
  $$ j : \mathcal{S}_A(\mathbf{\Sigma}) \hookrightarrow \mathcal{S}_A^{red}(\mathbf{\Sigma}^*).$$
  We can thus embeds $\mathcal{S}_A(\mathbf{\Sigma})$ into some quantum torus using the composition: 
  $$ \phi: \mathcal{S}_A(\mathbf{\Sigma}) \xrightarrow{j} \mathcal{S}_A^{red}(\mathbf{\Sigma}^*) \xrightarrow{\tr^{\Delta^*}} \mathcal{Z}_{\omega}(\mathbf{\Sigma}^*, \Delta^*).$$
  The quantum torus $ \mathcal{Z}_{\omega}(\mathbf{\Sigma}^*, \Delta^*)$ is very large and the construction can be refined as follows. 
    Let $\overline{\mathcal{E}}:= \mathcal{E} \bigsqcup \mathcal{A}$. Since a boundary arc $a$ is an edge of the triangulation, we denote it by $a$ when seen as an element of $\mathcal{E}(\Delta) $ and $\widehat{a}$ when seen as an element of $\mathcal{A}$. Let $\overline{K}_{\Delta}$ denote the set of maps $\mathbf{k}: \overline{\mathcal{E}}(\Delta) \to \mathbb{Z}$ such that $(1)$ for any face of $\Delta$ with edges $a,b,c$, then $\mathbf{k}(a)+\mathbf{k}(b)+\mathbf{k}(c)$ is even and $(2)$  each $\mathbf{k}(\widehat{a})$ is even. Define an injective linear map $i: \overline{K}_{\Delta} \hookrightarrow K_{\Delta^*}$ sending $\mathbf{k}$ to $\mathbf{k}'$ where $\mathbf{k}'(e)=\mathbf{k}(e)$ if $e\in \mathring{\mathcal{E}}(\Delta)$ and for $a\in \mathcal{A}$,   $\mathbf{k}'(a)=\mathbf{k}(a)+\mathbf{k}(\widehat{a})$, $\mathbf{k}'(a')=-\mathbf{k}(a)$ and $\mathbf{k}'(a'')=0$.
    Let $\overline{\mathcal{Z}}_{\omega}(\mathbf{\Sigma}, \Delta) \subset \mathcal{Z}_{\omega}(\mathbf{\Sigma}^*, \Delta^*)$ be the submodule spanned by elements $Z^{i(\mathbf{k})}$ for $\mathbf{k}\in \overline{K}_{\Delta}$. An easy computation shows that the map $\phi$ takes values in the smaller quantum torus $\overline{\mathcal{Z}}_{\omega}(\mathbf{\Sigma}, \Delta) $.
    
    \begin{definition} The injective morphism $\phi: \mathcal{S}_A(\mathbf{\Sigma}) \hookrightarrow \overline{\mathcal{Z}}_{\omega}(\mathbf{\Sigma}, \Delta) $ will be referred to as the \textit{refined quantum trace}.
    \end{definition}
  
  \subsection{Closed surfaces}\label{sec_ClosedCase}
  
  It remains to consider closed surfaces. Let $\mathcal{W}_q$ be the quantum torus generated by $X^{\pm 1}, Y^{\pm 1}$ with relation $XY=qYX$. Let $\Theta$ be the involutive automorphism  defined by $\Theta(X)=X^{-1}$ and $\Theta(Y)=Y^{-1}$. It decomposes the quantum torus into its eigenspaces $\mathcal{W}_q= \mathcal{W}_q^+ \oplus \mathcal{W}_q^-$ where $\mathcal{W}_q^{\pm}$ is the submodule of elements $z$ such that $\Theta(z)=\pm z$.
   In genus $1$, one has the 
  
  \begin{theorem}(Frohman-Gelca \cite{FG00}) For $\mathbf{\Sigma}=(\Sigma_{1,0}, \emptyset)$ the closed torus, there is an isomorphism $\phi^{FG}: \mathcal{S}_A(\Sigma_{1,0}) \xrightarrow{\cong} \mathcal{W}_q^+$ sending the meridian $\lambda \subset \Sigma_{1,0}$ to $X+X^{-1}$ and the longitude $\mu \subset \Sigma_{1,0}$ to $Y+Y^{-1}$.
  \end{theorem}

An important remaining problem is the
  
  \begin{problem} For $g\geq 2$, find an embedding of  $\mathcal{S}_A(\Sigma_{g,0})$ into some quantum torus. If possible, choose one having the same PI-dimension than  $\mathcal{S}_A(\Sigma_{g,0})$ when $A$ is a root of unity.
  \end{problem}
  
  A candidate for the quantum torus was proposed in \cite{KojuQuesneyQNonAb} based on pants decomposition. An important step has been announced by Karuo  \cite{Karuo_ToAppear} who will construct a filtration of $\mathcal{S}_A(\Sigma_{g,0})$, based on pants decomposition, and prove that the associated graded algebra can be embedded into some quantum torus.

  \section{Relative character varieties}
  
  \subsection{Deformation quantization}\label{sec_DefoQuantization}
  
  Let $A_q$ be an associative unital $\mathbb{C}[q^{\pm 1}]$-algebra which is free and flat as a $\mathbb{C}[q^{\pm 1}]$-module. Consider the $\mathbb{C}$ algebra $A_{+1}=A_q\otimes_{q=1} \mathbb{C}$ and the $\mathbb{C}[[\hbar]]$ algebra $A_{\hbar}=A_q\otimes_{q=\exp(\hbar)} \mathbb{C}[[\hbar]]$. We suppose that $A_{+1}$ is commutative. Fix $\mathcal{B}$ a basis of $A_q$, so by flatness, $\mathcal{B}$ can be also considered as a basis of $A_{\hbar}$ and $A_{+1}$.
The basis $\mathcal{B}$ induces an isomorphism of $\mathbb{C}[[\hbar]]$-modules $\Psi^{\mathcal{B}}: A_{+1}\otimes_{\mathbb{C}} \mathbb{C}[[\hbar]] \xrightarrow{\cong} A_{\hbar}$ sending $b\in \mathcal{B}$ to itself.  Denote by $\star$ the pull-back in $ A_{+1}\otimes_{\mathbb{C}} \mathbb{C}[[\hbar]]$ of the product in $A_{\hbar}$. Define a Poisson bracket $\{\cdot, \cdot\}$ on $A_{+1}$ by the formula:
$$ x \star y - y \star x = \hbar \{x, y\} \pmod{\hbar^2}.$$
Note that the associativity of $\star$ implies the Jacobi identity for $\{\cdot, \cdot\}$. 
  We say that the algebra $(A_{+1}\otimes_{\mathbb{C}} \mathbb{C}[[\hbar]], \star)$ is a \textit{deformation quantization} of the Poisson algebra $(A_{+1}, \{\cdot, \cdot\})$ and refer to (\cite{KontsevichQuantizationPoisson}, \cite{GRS_QuantizationDeformation}  $II.2$)  for details on the subject. If $\mathcal{B}'$ is another basis, then $\Psi^{\mathcal{B}'} \circ \Psi^{\mathcal{B}}$ is an algebra isomorphism $(A_{+1}\otimes_{\mathbb{C}} \mathbb{C}[[\hbar]], \star_{\mathcal{B}}) \cong (A_{+1}\otimes_{\mathbb{C}} \mathbb{C}[[\hbar]], \star_{\mathcal{B}'})$ whose reduction modulo $\hbar$ is the identity:  such an isomorphism is called a \textit{gauge equivalence} and it is clear that two gauge equivalent star products induces the same Poisson bracket, in particular $\{\cdot, \cdot\}$ is independent on the choice of $\mathcal{B}$. Note that when $A_{+1}$ is reduced and finitely generated, $\mathrm{Specm}(A_{+1})$ is a Poisson variety. For instance, for $\mathbb{E}$ a quadratic pair, the quantum torus $\mathbb{T}_{\hbar}(\mathbb{E})$ where $q= \exp(\hbar)$, is a deformation quantization of the 
  Poisson torus $\mathcal{\chi}(\mathbb{E})$. Indeed, from the equality $Z_{a+b}= q^{-(a,b)}Z_aZ_b= q^{+(a,b)}Z_bZ_a$ and setting $q=\exp{\hbar}$, we get that 
  $$ Z_a \star Z_b - Z_b \star Z_a = (q^{(a,b)}-q^{-(a,b)})Z_{a+b}= \hbar 2(a,b) Z_{a+b} \pmod{\hbar^2}, $$
  giving the Poisson bracket $\{ Z_a, Z_b\}= 2(a,b) Z_{a+b}$ of $\mathcal{X}(\mathbb{E})$.
  
\begin{remark}\label{remarkPoisson} If $\Psi_A : A^1_q \to A^2_q$ is an algebra morphism, it induces some morphisms $\Psi_{\hbar}: A^1_{\hbar} \to A^2_{\hbar}$ and $\Psi_{+1}: A^1_{+1} \to A^2_{+1}$. Since $\Psi_{+1}$ is the reduction modulo $\hbar$ of $\Psi_{\hbar}$, it follows from the definition of the Poisson bracket that $\Psi_{+1}$ is a Poisson morphism.
\end{remark}

  \begin{definition}
  For $\mathbf{\Sigma}$ a marked surface, we denote by $\mathcal{X}(\mathbf{\Sigma}):= \mathrm{Specm}(\mathcal{S}_{+1}(\mathbf{\Sigma}))$ and $\mathcal{X}^{red}(\mathbf{\Sigma}):= \mathrm{Specm}(\mathcal{S}^{red}_{+1}(\mathbf{\Sigma}))$ the Poisson varieties associated to $A^{1/2}:= \exp(\hbar/4)$ (so that $q=\exp(\hbar)$).
  \end{definition}
Note that by Remark \ref{remarkPoisson}, the gluing maps $\theta_{a\#b}$ and the comodule maps $\Delta^{L/R}_a$ are Poisson.  

\begin{notations} When considering an affine complex variety $X$, we will denote by $\mathcal{O}[X]$ its algebra of regular functions (so $X=\mathrm{Specm}(\mathcal{O}[X])$). A closed point $x\in X$ is a maximal ideal $\mathfrak{m}_x \subset \mathbb{C}[X]$ which is the kernel of a (unique) character $\chi_x : \mathcal{O}[X]\to \mathbb{C}$. Even though $x$ and $\mathfrak{m}_x$ are the same object, we use three letters $x, \mathrm{m}_x, \chi_x$ to denote it. 
\end{notations}

  In view of classifying the weight representations of stated algebras, we will be interested in computing the symplectic leaves of these Poisson varieties. Let $X$ be an affine complex Poisson variety. Define a first partition $X=X^0\bigsqcup \ldots \bigsqcup X^n$ where $X^0$ is the smooth locus of $X$ and for $i=0, \ldots, n-1$, $X^{i+1}$ is the smooth locus of $X\setminus X^{i}$. Then each $X^i$ is a smooth complex affine variety that can be seen as an analytic Poisson variety. Define an equivalence relation $\sim$ on $X^i$ by 
  writing $x\sim y$ if there exists a finite sequence $x=p_0, p_1, \ldots, p_k=y$ and functions $h_0, \ldots, h_{k-1} \in \mathcal{O}[X^i]$ such that $p_{i+1}$ is obtained from $p_i$ by deriving along the Hamiltonian flow of $h_i$. Write $X^i= \bigsqcup_j X^{i,j}$ the orbits of this relation. Note that the $X^{i,j}$ are analytic subvarieties: they are the biggest connected smooth symplectic subvarieties of $X^i$.  
  
  \begin{definition} 
  The elements $X^{i,j}$ of the (analytic) partition $X= \bigsqcup_{i,j} X^{i,j}$ are called the \textit{symplectic leaves} of $X$.
  \end{definition}

  Recall the coaction map $\Delta^L : \mathcal{S}_A(\mathbf{\Sigma}) \to (\mathcal{O}_q[\SL_2])^{\otimes \mathcal{A}} \otimes \mathcal{S}_A(\mathbf{\Sigma})$ coming from gluing bigons to the boundary arcs. It induces an algebraic group action $\nabla^L: (\SL_2)^{\mathcal{A}} \times \mathcal{X}(\mathbf{\Sigma}) \to  \mathcal{X}(\mathbf{\Sigma})$ which is Poisson when $\SL_2$ has the Poisson structure of $\mathcal{X}(\mathbb{B})$ (see bellow). Using the diagonal inclusion $\mathbb{C}^* \hookrightarrow \SL_2(\mathbb{C})$ sending $z$ to $\begin{pmatrix} z & 0 \\ 0 & z^{-1} \end{pmatrix}$, we get an algebraic (left) group action $\nabla: (\mathbb{C}^*)^{\mathcal{A}} \times \mathcal{X}(\mathbf{\Sigma})\to  \mathcal{X}(\mathbf{\Sigma})$. This action restricts to a group action $\nabla^{red}: (\mathbb{C}^*)^{\mathcal{A}} \times \mathcal{X}^{red}(\mathbf{\Sigma}) \to  \mathcal{X}^{red}(\mathbf{\Sigma})$. Define an equivalence relation $\sim$ on these varieties by setting $x\sim y$ if there exists $g\in  (\mathbb{C}^*)^{\mathcal{A}} $ such that $g\cdot x$ and $y$ belong to the same symplectic leaf.
  
  \begin{definition} The \textit{equivariant symplectic leaves} of $\mathcal{X}(\mathbf{\Sigma})$ and $\mathcal{X}^{red}(\mathbf{\Sigma})$ are the equivalence classes for this relation.
  \end{definition}
  
  As we shall see in Section \ref{sec_PO}, the problem of classifying the weight representations of stated skein algebras is deeply related to the following:
 \begin{problem}\label{problem_classification_leafs} 
 Compute the equivariant symplectic leaves of $\mathcal{X}(\mathbf{\Sigma})$ and $\mathcal{X}^{red}(\mathbf{\Sigma})$.
 \end{problem}
  
  As we shall review, the problem was solved:
  \begin{enumerate}
  \item for closed surfaces by Goldman \cite{Goldman_symplectic}; 
  \item for unmarked non-closed surfaces, independently  by Fock-Rosly \cite{FockRosly} and Guruprasad-Huebschmann-Jeffrey-Weitsman \cite{ GHJW_ModSpacesParBd};
  \item for the once-punctured monogon $\mathbf{m}_1$ and  $\mathbb{D}_1^+:=\underline{\mathbf{\Sigma}}_{0,2}$ by Alekseev-Malkin \cite{AlekseevMalkin_PoissonLie};
  \item for the bigon, independently by Alekseev-Malkin \cite{AlekseevMalkin_PoissonLie} and Hodges-Levasseur \cite{HodgesLevasseur_OqG};
  \item Ganev-Jordan-Safranov \cite{GanevJordanSafranov_FrobeniusMorphism} found an explicit open dense symplectic leaf in $\mathcal{X}(\mathbf{\Sigma}_{g,0}^0)$, for $g\geq 1$. 
  \end{enumerate}

  Let us first state a trivial, but useful result towards the resolution of this problem. Consider an algebra $A_q$ as before and let $x\in \mathcal{A}_q$ be such that $xA_q=A_qx$, i.e. the left and right (and bilateral) ideals generated by $x$ coincide. Let $I_q=(x)\subset A_q$ this ideal and $I=I_q\otimes_{q=1} \mathbb{C} \subset A_{+1}$. Since we have $[I_q, A_q] \subset I_q$, it follows from the definition of the Poisson bracket that $\{ I, A_{+1}\} \subset I$, i.e. that $I$ is a Poisson ideal of $A_{+1}$. Partition the set $X=\mathrm{Specm}(A_{+1})$ into $X=X^0 \bigsqcup X^1$ where $X^0$ is the open subset of $x\in X$ such that $\chi_x(I)=0$ and $X^1$ its closed complement. Clearly each set $X^i$ is a disjoint union of symplectic leaves, i.e. the partition into symplectic leaves is a refinement of the partition $X=X^0 \bigsqcup X^1$.

  An easy skein manipulation shows the following:
  
  \begin{lemma}\label{lemma_leyu}(L\^e-Yu \cite[Lemma $4.4. (a)$]{LeYu_SSkeinQTraces}) Let $p\in \mathcal{P}^{\partial}$ be a boundary puncture and $\alpha(p)_{-+}\in \mathcal{S}_A(\mathbf{\Sigma})$ its associated bad arc. For any $[D,s] \in \mathcal{B}^{\mathfrak{o}^+}$, there exists $n\in \mathbb{Z}$ such that $\alpha(p)_{-+} [D,s]= A^{n/2} [D,s] \alpha(p)_{-+}$. In particular $\alpha(p)_{-+} \mathcal{S}_A(\mathbf{\Sigma})= \mathcal{S}_A(\mathbf{\Sigma}) \alpha(p)_{-+}$. 
  \end{lemma}
  
  For $\mathbf{\varepsilon} : \mathcal{P}^{\partial} \to \{0, 1\}$, denote by $\mathcal{X}^{(\mathbf{\varepsilon})}(\mathbf{\Sigma}) \subset \mathcal{X}(\mathbf{\Sigma})$ the subset of these $x\in  \mathcal{X}(\mathbf{\Sigma})$ such that $\chi_x(\alpha(p)_{-+})=0$ if $\mathbf{\varepsilon}(p)=0$ and $\chi_x(\alpha(p)_{-+})\neq 0$ else.
  
  \begin{definition} We call the \textit{bad arcs partition} the partition $\mathcal{X}(\mathbf{\Sigma})= \bigsqcup_{\mathbf{\varepsilon}} \mathcal{X}^{(\mathbf{\varepsilon})}(\mathbf{\Sigma})$. 
  \end{definition}
  
  Note that for $\mathbf{\varepsilon}=0$ (the map sending every $p$ to $0$), one has $\mathcal{X}^{(0)}(\mathbf{\Sigma})= \mathcal{X}^{red}(\mathbf{\Sigma})$ by definition. 
  By Lemma \ref{lemma_leyu} and the preceding discussion, the partition into symplectic leaves is a refinement of the bad arc partition. Moreover, it is clear from the definition of the $(\mathbb{C}^*)^{\mathcal{A}}$ action that each set $\mathcal{X}^{(\mathbf{\varepsilon})}(\mathbf{\Sigma})$ is preserved  by $(\mathbb{C}^*)^{\mathcal{A}}$, so we have our first tool towards the resolution of Problem \ref{problem_classification_leafs}:
  
  \begin{lemma}
  The partition into equivariant symplectic leaves is a refinement of the bad arcs partition.
  \end{lemma}
  
  Let us state a second obvious remark towards the resolution of Problem \ref{problem_classification_leafs}. Recall from Definition \ref{def_central_elements}, that for each inner puncture $p$ we defined a central element $\gamma_p \in \mathcal{S}_A(\Sigma)$ and for each boundary component $\partial$, we defined an invertible central element $\alpha_{\partial} \in \mathcal{S}_A^{red}(\Sigma)$. Let $\mathrm{Cas} \subset \mathcal{S}_{+1}(\mathbf{\Sigma})$ (resp. $\mathrm{Cas}^{red}\subset \mathcal{S}_{+1}^{red}(\mathbf{\Sigma})$) denote the subgroup generated by the elements $\gamma_p$ (resp. by the elements $\gamma_p, \alpha_{\partial}^{\pm 1}$). Since these elements are central in the skein algebras with parameter $A=\exp(\hbar/2)$, the elements in $\mathrm{Cas}$ and $\mathrm{Cas}^{red}$ are Casimir elements, i.e. they are in the kernel of the Poisson bracket. Therefore, if we consider the following \textit{Casimir partition}
  $$ \mathcal{X}(\mathbf{\Sigma}) = \bigsqcup_{\pi : \mathrm{Cas} \to \mathbb{C}} \mathcal{X}_{(\pi)}(\mathbf{\Sigma}), \quad \mathcal{X}^{red}(\mathbf{\Sigma})=\bigsqcup_{\pi: \mathrm{Cas}^{red}\to \mathbb{C}} \mathcal{X}^{red}_{(\pi)}(\mathbf{\Sigma}), $$
  where the $\pi$ are characters over the Casimir groups and $ \mathcal{X}_{(\pi)}(\mathbf{\Sigma})$ is the (algebraic) subset of elements $x$ such that $\chi_x(c)=\pi(c)$, for all $c\in \mathrm{Cas}$ and similarly for the reduced version, then
  
  \begin{lemma}\label{lemma_Casimir_partition}
  The partition into symplectic leaves is a refinement of the Casimir partition.
  \end{lemma}
  
  Note that the group $(\mathbb{C}^*)^{\mathcal{A}}$ preserves the Casimir leaves $\mathcal{X}_{(\pi)}(\mathbf{\Sigma})$ but not the leaves $ \mathcal{X}^{red}_{(\pi)}(\mathbf{\Sigma})$.   
  
  Let us make a final remark:
  
  \begin{lemma}\label{lemma_singleton} If $\mathcal{S}_A^{red}(\mathbf{\Sigma})$ is commutative for $A$ generic, then for every $x\in \mathcal{X}^{red}(\mathbf{\Sigma})=\mathcal{X}^{(0)}(\mathbf{\Sigma})$, then the singleton $\{x\}$ is a symplectic leaf of $\mathcal{X}(\mathbf{\Sigma})$.
  \end{lemma}
  
  \begin{proof} Let $\mathcal{I}^{bad}\subset \mathcal{S}_{A}^{red}(\mathbf{\Sigma})$ be the ideal generated by the bad arcs. If $\mathcal{S}_A^{red}(\mathbf{\Sigma})$ is commutative for $A=\exp(\hbar)$, then we have $[\mathcal{S}_{\hbar}(\mathbf{\Sigma}), \mathcal{S}_{\hbar}(\mathbf{\Sigma})] \subset \mathcal{I}^{bad}$, so by definition of the Poisson bracket we have $\{\mathcal{S}_{+1}(\mathbf{\Sigma}), \mathcal{S}_{+1}(\mathbf{\Sigma})\} \subset \mathcal{I}^{bad}$. Therefore the restriction of the Poisson bracket to $\mathcal{X}^{red}(\mathbf{\Sigma})$ vanishes.
  \end{proof}

  \subsection{First examples and the work of Semenov-Tian-Shansky}\label{sec_STS}
  
  Let us study the cases of the bigon $\mathbb{B}$, the once-punctured bigon $\mathbb{D}_1$ and the once-punctured monogon $\mathbf{m}_1$ for which Problem \ref{problem_classification_leafs} has been solved by Alekseev-Malkin \cite{AlekseevMalkin_PoissonLie} and Hodges-Levasseur \cite{HodgesLevasseur_OqG}.
  At this stage, it is useful to understand the structure of the quantum $R$-matrix $\mathscr{R}$. Consider the matrices in $\mathfrak{sl}_2$:
  $$ E = \begin{pmatrix} 0 & 0 \\ 1 & 0 \end{pmatrix}, \quad F = \begin{pmatrix} 0 & 1 \\ 0 & 0 \end{pmatrix},\quad H= \begin{pmatrix} 1 &0\\0 & -1 \end{pmatrix}.$$
  Consider these matrices as operators acting on a $2$-dimensional vector space $V$ with ordered basis $(v_+, v_-)$ (the standard representation of $\mathfrak{sl}_2$). Define an endomorphism $q^{\frac{H\otimes H}{2}}\in \End(V\otimes V)$ by $q^{\frac{H\otimes H}{2}}\cdot  v_{\varepsilon_1} \otimes v_{\varepsilon_2} := A^{\varepsilon_1 \varepsilon_2} v_{\varepsilon_1} \otimes v_{\varepsilon_2}$. Then $\mathscr{R}$ is the matrix in the ordered basis $(v_+\otimes v_+, v_+\otimes v_-, v_-\otimes v_+, v_-\otimes v_-)$ of the operator 
 $$ \mathscr{R}= \tau \circ q^{\frac{H\otimes H}{2}}  \exp_q\left( (q-q^{-1})E \otimes F \right) =  \tau \circ q^{\frac{H\otimes H}{2}} \circ \left( \mathds{1}_2 + (q-q^{-1}) E \otimes F \right) .$$
Define the classical r-matrices: 
$$ r^+:= \frac{1}{2} H\otimes H +2E\otimes F, \quad r^{-} := \frac{1}{2} H\otimes H +2F\otimes E. $$
Then, writing $A^{1/2}=\exp(\hbar/2)$,  one has 
\begin{align}\label{eq_classicalr}
 & \mathscr{R}\equiv \tau ( \mathds{1}\otimes \mathds{1} + \hbar r^+) \pmod{\hbar^2} \equiv ( \mathds{1}\otimes \mathds{1} + \hbar r^-)\tau \pmod{\hbar^2}, \\
 \label{eq_classicalr2}
&\mathscr{R}^{-1}\equiv ( \mathds{1}\otimes \mathds{1} - \hbar r^+)\tau \pmod{\hbar^2} \equiv \tau ( \mathds{1}\otimes \mathds{1} - \hbar r^-) \pmod{\hbar^2}, 
\end{align}

  \vspace{2mm}
  \par \textbf{The bigon:} Let $\alpha\subset \mathbb{B}$ be an arc linking $a_L$ and $a_R$ and $\alpha_{ij}$ the stated arc with state $i$ on $a_L$ and $j$ on $a_R$. By Theorem \ref{theorem_presentation}, $\mathcal{S}_A(\mathbb{B})=\mathcal{O}_q[\SL_2]$ is generated by the elements $\alpha_{ij}, i,j = \pm$ together with the relations, written in matrix form using $N:= \begin{pmatrix} \alpha_{++} & \alpha_{+-} \\ \alpha_{-+} & \alpha_{--} \end{pmatrix}$, given by $\det_q(N)=1$ and the so-called Faddeev-Reshetikhin-Takhtadjian relation
  $$ N \odot N = \mathscr{R}^{-1} (N\odot N) \mathscr{R}.$$
  So we get Equation \eqref{relbigone} by identifying the matrix coefficients. Replacing $A$ by $\exp(\hbar/2)$ and developing using Equations \eqref{eq_classicalr} \eqref{eq_classicalr2}, we find
  \begin{align*}
  & N \odot N \equiv (1- \hbar r^+) \tau (N\odot N) \tau (1+\hbar r^+) \pmod{\hbar^2} \\
  \Leftrightarrow & \tau (N\odot N) \tau -N\odot N \equiv \hbar \left( r^{+} (N\odot N) - (N\odot N) r^+ \right) \pmod{\hbar^2}
  \end{align*}
  
  Therefore $\mathcal{X}(\mathbb{B})$ can be identified with the variety $\SL_2$ together with the Poisson bracket defined by 
  \begin{equation}\label{eq_DrinfeldBracket}
  \{ N \otimes N \}^D = r^+ (N \odot N) - (N\odot N) r^+.
  \end{equation}
  
  For further use, we define $\mathcal{O}[\SL_2]:= \quotient{\mathbb{C}[a,b,c,d]}{(ad-bd-1)}$ and the isomorphism $\mathcal{S}_{+1}(\mathbb{B}) \cong \mathcal{O}[\SL_2]$ identifies $\alpha_{++}, \alpha_{+-}, \alpha_{-+}, \alpha_{--}$ with $a,b,c,d$ respectively.
  Since the coproduct is defined as a gluing map, it is Poisson. We will denote by $\SL_2^D$ the obtained Poisson-Lie group (the D stands for Drinfel'd who first defined it in \cite{DrinfeldrMatrix}). We can rewrite Equation \eqref{eq_DrinfeldBracket} as 
\begin{align*}
\{ a, b\}^D &= -ab & \{a, c\}^D &= -ac &  \{ b, c \}^D &= 0
\\ \{ d, b \}^D &= db & \{d, c \}^D &= dc & \{ a, d\}^D &= -2bc
\end{align*}

  In this case, the bad arcs partition writes 
  $$\SL_2^D = X_{00} \bigsqcup X_{01} \bigsqcup X_{10} \bigsqcup X_{11}, $$
  where $\begin{pmatrix} a & b \\ c & d \end{pmatrix}$ is in $X_{11}$ if $bd\neq 0$, is in $X_{10}$ if $b=0, c\neq 0$, is in $X_{01}$ if $c=0, b\neq 0$ and is in $X_{00}$ if it is diagonal. The Weil group of $\SL_2$ is  $W= \{w_0, w_1\}$ where $w_0$ is the class of the identity $\dot{w}_0:= \mathds{1}_2$ and $w_1$ is the class of $\dot{w}_1:= \begin{pmatrix} 0 & 1 \\ -1 & 0 \end{pmatrix}$. Denote by $B^+$ (resp. $B^-$) the subgroup of $\SL_2$ of upper (resp. lower) triangular matrices. A simple computation shows that 
  $$X_{ij} = B^+ \dot{w}_i B^+ \cap B^- \dot{w}_j B^-.$$
  In other terms, the bad arc partition of $\SL_2^D$ coincides with the double Bruhat cells decomposition of $\SL_2$. As we shall review, Hodges and Levasseur have proved that in this case, the bad arc partition and the equivariant symplectic leaves partition coincide, i.e. the double Bruhat cells are the equivariant symplectic leaves.
  
    \vspace{2mm}
  \par \textbf{The once-punctured monogon:}
  Write $\mathcal{O}_q[\SL_2^{STS}] := \mathcal{S}_A(\mathbf{m}_1)$. This algebra is also a deformation of $\SL_2$ but with a different Poisson structure, that we will denote by $\SL_2^{STS}$ by reference to Semenov-Tian-Shansky who first defined it in \cite{STS_DressingAction}. This Poisson structure is no longer compatible with the coproduct, i.e. $\SL_2^{STS}$ is not a Poisson Lie group. Let $\alpha \subset \mathbf{m}_1$ be the unique corner arc with endpoint $v$ and $w$ such that $v<_{\mathfrak{o}_+} w$ and $\alpha_{ij}$ the stated arc with state $i$ on $v$ and $j$ on $w$ and $h(v)<h(w)$ (in the terminology of Section \ref{sec_presentations}, $\alpha$ is of type $d$).
  
   This time, in order to identify $\mathcal{S}_{+1}(\mathbf{m}_1)$ with $\mathcal{O}[\SL_2]$, we set 
  $$N=\begin{pmatrix} a & b \\ c & d \end{pmatrix} := C^{-1} \begin{pmatrix} \alpha_{++} & \alpha_{+-} \\ \alpha_{-+} & \alpha_{--} \end{pmatrix} =\begin{pmatrix} -\alpha_{-+} & -\alpha_{--} \\ \alpha_{++} & \alpha_{+-} \end{pmatrix}.$$
   Theorem \ref{theorem_presentation} shows that $\mathcal{O}_q[\SL_2^{STS}]$ is generated by $a,b,c,d$ with relations $\det_{q^2} (N)=1$ and 
   $$ (\mathds{1}_1 \odot N) \mathscr{R}^{-1}(\mathds{1}_1 \odot N) \mathscr{R} = \mathscr{R} (\mathds{1}_1 \odot N) \mathscr{R}^{-1}(\mathds{1}_1 \odot N)$$. 
  Replacing $A$ by $\exp(\hbar/2)$ and developing using Equations \eqref{eq_classicalr} \eqref{eq_classicalr2} as before, we find   
  $$ \{N\otimes N\}^{STS} = - (\mathds{1}_1 \odot N)r^+ (N\odot \mathds{1}_2) +\tau (N\odot N) \tau r^+ - r^- (N\odot N) + (N\odot \mathds{1}_2)r^- (\mathds{1}_2 \odot N).$$
  We can develop these equations to find
  \begin{align*}
 & [\alpha_{++}, \alpha_{+-}]= (q^{-1}-q) \alpha_{-+} \alpha_{++} & [\alpha_{+-}, \alpha_{--}] = (q^{-1}-q) \alpha_{--} \alpha_{-+} \\
  &\alpha_{--}\alpha_{-+} = q^2 \alpha_{-+} \alpha_{--} & \alpha_{-+}\alpha_{++} = q^2 \alpha_{++}\alpha_{-+}\\
  &[\alpha_{+-}, \alpha_{-+}]=0 & [\alpha_{++}, \alpha_{--}]= (q^{-1}-q)\alpha_{-+}^2 +(1-q^2)\alpha_{+-}\alpha_{-+} \\
  &\alpha_{--}\alpha_{++}-q^2 \alpha_{+-} \alpha_{-+}=A &{} 
  \end{align*}
From which we get 
\begin{align*}
\{c,d\}^{STS} &= 2ac & \{d,b\}^{STS}&= 2ab & \{d,a\}^{STS} &=0 \\
\{b,a\}^{STS} &= 2ab & \{a,c\}^{STS} &= 2ac & \{c,b\}^{STS}&= 2a(a-d).
\end{align*}

The bad arcs partition writes 
$$\SL_2^{STS} = \SL_2^0 \bigsqcup \SL_2^1, $$
 where $\SL_2^0=B^- B^+$ is the subset of matrices $\begin{pmatrix} a & b\\ c & d\end{pmatrix}$ with $a\neq 0$ and $\SL_2^1= B^- C B^+$ is the set of those matrices for which $a=0$. So the bad arcs partition coincides with the simple Bruhat cells decomposition associated to the Cartan decomposition $(B^-, B^+)$. 
 
Note that Lemmas \ref{lemma_singleton} and \ref{lemma_Casimir_partition} imply that
 \begin{enumerate}
 \item For $g\in \SL_2^1$, the singleton $\{g \}$ is a symplectic leaf of $\SL_2^{STS}$ and
 \item the symplectic leaves of $\SL_2^0$ are included in the algebraic sets $\SL_2^0(c):=\{g\in \SL_2^0 | tr(g)=c\}$ for $c\in \mathbb{C}$ (since $a+d$ is Casimir).
 \end{enumerate}
As we shall review, it follows from the work of Alekseev-Malkin in \cite{AlekseevMalkin_PoissonLie} that the symplectic leaves in $\SL_2^0$ are the analytic sets $\SL_2^0\cap C$, for $C$ a conjugacy class. The algebraic sets $\SL_2^0(c)$ will be called  symplectic cores of $\SL_2^{STS}$ in Section \ref{sec_PO}. Note that $\mathbf{m}_1$ has exactly one boundary arc and the induced Poisson action of $\SL_2^D$ on $\SL_2^{STS}$ is the action by conjugaison. So, writing $\varphi(z):= \begin{pmatrix} z & 0 \\ 0 & z^{-1} \end{pmatrix}$, the sets $\{ \varphi(z) g \varphi(z^{-1}), z \in \mathbb{C}^*\}$ for $g\in \SL_2^1$ are equivariant symplectic leaves.

Note also that the identification  $\mathcal{O}[\SL_2] \cong\mathcal{S}_{+1}(\mathbf{m}_1)$ that we made contains an arbitrary choice: we could have precompose this identification by the Cartan involution, i.e. have stated $N:= M(\alpha)C^{-1}$ instead of $N:=C^{-1}M(\alpha)$. With this new identification, the  bad arcs decomposition would have written
$$ {\SL_2'}^{STS} = {}^0\SL_2 \bigsqcup {}^1\SL_2, $$
where ${}^0\SL_2= B^+B^-$ is the set of these $\begin{pmatrix} a & b\\ c & d\end{pmatrix}\in \SL_2$ with $d\neq 0$ and ${}^1 \SL_2$ is its complementary.

\begin{remark}
The algebra $\mathcal{O}_q[\SL_2^{STS}]$ appeared in the work of Lyubashenko in \cite{Lyubashenko_ModularTransoTensorCat} who revisited Shan's Tannakian theorem in \cite{Majid_QGroups}. In short, Shan's theorem says that for a Tannakian ribbon category, i.e. a ribbon category $\mathcal{C}$ with a tensor faithful functor $F: \mathcal{C} \to \Vect$, where $\Vect$ is the category of finite dimensional vector spaces, then there exists a coribbon Hopf algebra $\mathcal{H}$ and a ribbon equivalence of categories $\mathcal{C} \to \mathcal{H}-\mathrm{Comod}$. 
Following the classical work of Grothendieck-Deligne, Shan defined $\mathcal{H}=\mathrm{Nat}(F)$ has a space of natural transformations of $F$. Lyubashenko found a much more explicit definition for $\mathcal{H}$. He considers $\mathcal{F}:= \int^{V\in \mathcal{C}} V^*\otimes V$ and turns $\mathcal{F}$ into an algebra object using the product defined  as follows: for $V, W \in \mathcal{C}$ then $\mu : [V\otimes V^*] \otimes [W \otimes W^*] \to [V\otimes W] \otimes [W^* \otimes V^*]$ is the morphism in $\mathcal{C}$ given by the image by the Reshetikhin-Turaev functor of 
 $$ \mu := F^{RT} \left( \adjustbox{valign=c}{\includegraphics[width=2cm]{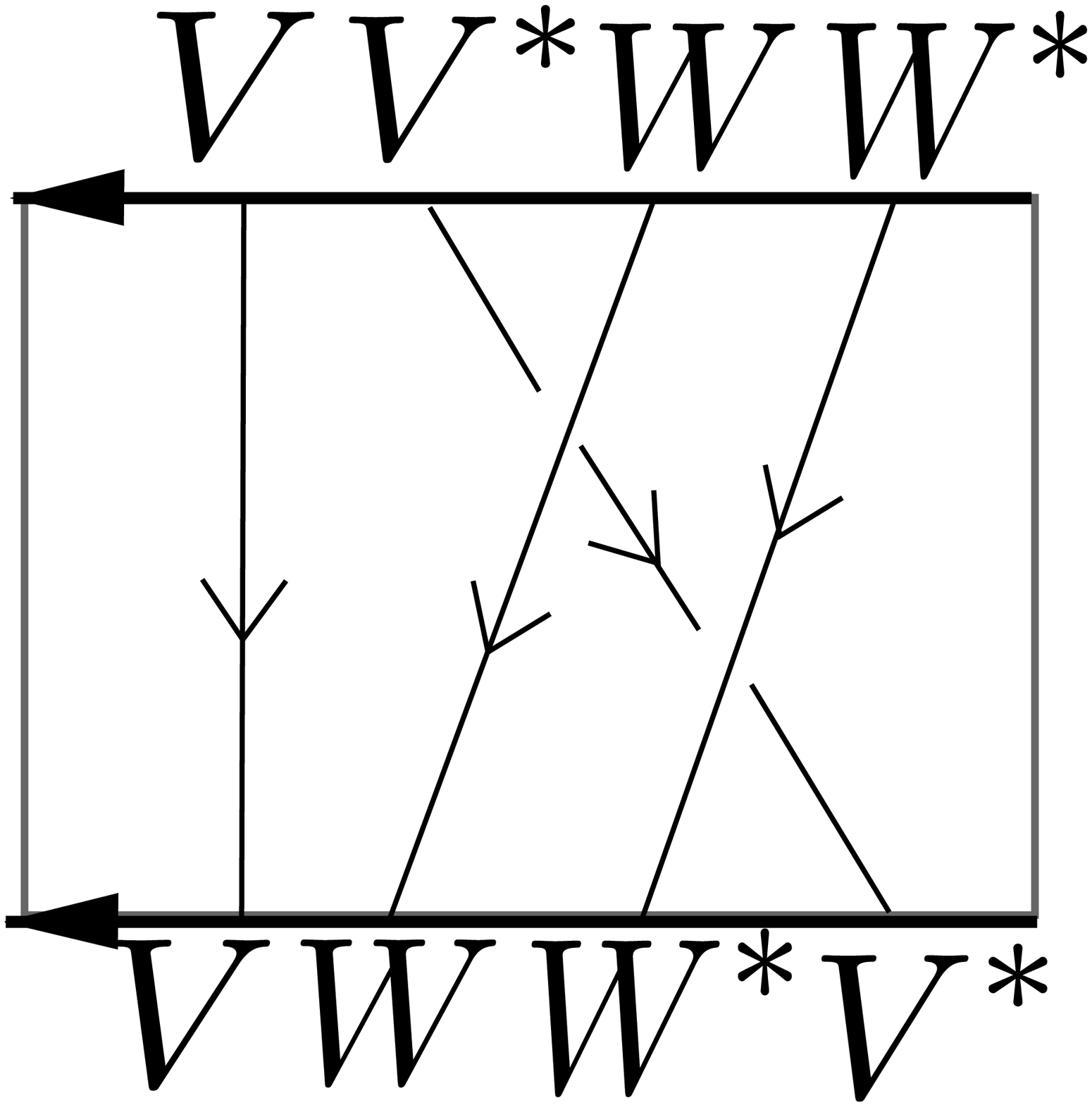}} \right). $$
Then $\mathcal{H}=F(\mathcal{F})\in \Vect$ is the desired cobraided Hopf algebra with  coproduct $\Delta([v^* \otimes v]):= \sum_i [v^* \otimes v_i] \otimes [v_i^* \otimes v]$, for $\{v_i\}_i$ a basis of $V$, co-unit $\epsilon([v^*\otimes v])=ev_V(v^* \otimes v)$  and co-R-matrix $\mathbf{r}([v^*\otimes v] \otimes [w^*\otimes w]):= \epsilon( [(v^*\otimes w^*) \otimes c_{V,W}(v,w)])$.
 As noticed by Gunningham-Jordan-Safranov in \cite{GunninghamJordanSafranov_FinitenessConjecture}, when $\mathcal{C}=U_q\mathfrak{sl}_2-\mathrm{Mod}$ with $q$ generic, and $F$ is the forgetful functor, the Lyubashenko algebra $\mathcal{H}$ is the stated skein algebra of the once-punctured monogon.  So Lyubashenko's Tannakian theorem states that the category of $U_q\mathfrak{sl}_2$ modules ($q$ is generic here) is equivalent to the category of $\mathcal{O}_q[\SL_2^{STS}]$-comodules. This is where the literature becomes confusing: since $\mathcal{O}_q[\SL_2]$ is the restricted dual of $U_q\mathfrak{sl}_2$, the category of $U_q\mathfrak{sl}_2$ modules is equivalent to the category of $\mathcal{O}_q[\SL_2]$-comodules as well ! Worst: Lyubashenko named $\mathcal{O}_q[\SL_2^{STS}]$ the \textit{quantum function algebra} (hence the letter $\mathcal{F}$), so with the same name that $\mathcal{O}_q[\SL_2]$. Because of this confusion, many authors denote by the same symbol the algebras $\mathcal{O}_q[\SL_2]$ and $\mathcal{O}_q[\SL_2^{STS}]$ though, as we just saw, they are not isomorphic as algebras and deform two very different Poisson structures. Since $\mathbf{m}_1$ is obtained from $\mathbb{B}$ by fusioning the two boundary arcs, the relation between the two algebras is clear: $\mathcal{O}_q[\SL_2^{STS}]$ is the fusion (in the sense of Definition \ref{def_fusion_quantique}) of $\mathcal{O}_q[\SL_2]$.
 
 \end{remark}
 
 \par \textbf{The Heisenberg double and its twist}
 We now consider the once-punctured bigon $\mathbb{D}_1$ and the marked surface $\mathbb{D}_1^+:=\underline{\mathbf{\Sigma}}_{0,2}$, so both are annuli with two boundary arcs, and both boundary arcs are in the same boundary component in $\mathbb{D}_1$ whereas the two boundary arcs in $\mathbb{D}_1^+$ are in two distinct boundary components. Write $D_q(\SL_2):= \mathcal{S}_A(\mathbb{D}_1)$ and $D_q^+(\SL_2):=\mathcal{S}_A(\mathbb{D}_1^+)$. Both their fundamental groupoids admit  presentations with two generators $\alpha, \beta$, drawn in Figure \ref{fig_D1}, and no relation. 
 
    \begin{figure}[!h] 
\centerline{\includegraphics[width=5cm]{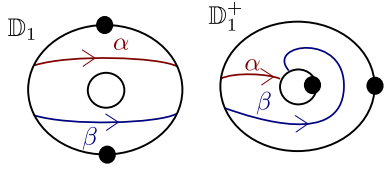} }
\caption{The marked surfaces $\mathbb{D}_1$ and $\mathbb{D}_1^+$ defining the Heisenberg doubles $D_q(\SL_2)$ and $D_q^+(\SL_2)$.} 
\label{fig_D1}
\end{figure}

 Consider the matrices $N(\alpha)= \begin{pmatrix} \alpha_{++} & \alpha_{+-} \\ \alpha_{-+} & \alpha_{--} \end{pmatrix}, N(\beta)=  \begin{pmatrix} \beta_{++} & \beta_{+-} \\ \beta_{-+} & \beta_{--} \end{pmatrix}$. Theorem \ref{theorem_presentation} shows that $D_q(\SL_2)$ is generated by the elements $\alpha_{ij}, \beta_{kl}$ with relations $\det_q(N(\alpha))=\det_q(N(\beta))=1$ and 
 $$ N(\alpha)\odot N(\alpha) = \mathscr{R}^{-1} (N(\alpha)\odot N(\alpha))\mathscr{R}, N(\beta)\odot N(\beta) = \mathscr{R}^{-1} (N(\beta)\odot N(\beta))\mathscr{R}, N(\alpha)\odot N(\beta) = \mathscr{R}^{-1} (N(\alpha)\odot N(\beta))\mathscr{R},$$
  therefore the Poisson structure is described by
 $$\{N(a)\otimes N(b)\}^D = r^+ (N(a) \odot N(b)) - (N(a)\odot N(b))r^+ \quad \mbox{, for }(a,b)= (\alpha, \alpha), (\beta, \beta), (\alpha, \beta).$$
The algebra $D_q(\SL_2)$ has a structure of Hopf algebra that we now describe. Let $i: \mathbb{D}_1\to \mathbb{D}_2$ be the embedding which is the identity on the boundary and embeds the two inner punctures of $\mathbb{D}_2$ inside the inner puncture of $\mathbb{D}_1$. Note that $\mathbb{D}_2$ is obtained from $\mathbb{D}_1\cup \mathbb{D}_1$ by gluing two boundary arcs, so we get a gluing map $\theta: \mathcal{S}_A(\mathbb{D}_2) \hookrightarrow D_q(\SL_2)^{\otimes 2}$. The \textit{Bigelow coproduct} is then defined as the composition
$$ \Delta : D_q(\SL_2) \xrightarrow{i_*} \mathcal{S}_A(\mathbb{D}_2) \xrightarrow{\theta} D_q(\SL_2)^{\otimes 2}.$$

    \begin{figure}[!h] 
\centerline{\includegraphics[width=9cm]{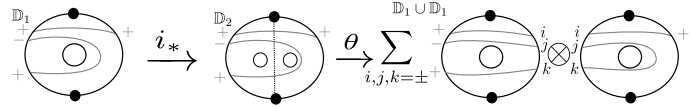} }
\caption{An illustration of Bigelow's coproduct.} 
\label{fig_coproduct}
\end{figure} 

It appeared in the work of Bigelow in a different language in \cite{BigelowQGroups}. In particular, $D(\SL_2)$ received a structure of Poisson-Lie group from this coproduct and was called the \textit{Heisenberg double} by Drinfel'd. Note that the reduced skein algebra $\mathcal{S}_A^{red}(\mathbb{D}_1)$ is also a Hopf algebra with this structure and it is proved in \cite[Theorem1]{KojuQGroupsBraidings} that it is isomorphic to the simply laced quantum group $\widetilde{U}_q\mathfrak{gl}_2$.

 In $D_q^+(\SL_2)$, the relations read 
 $$ N(\alpha)\odot N(\alpha) = \mathscr{R}^{-1} (N(\alpha)\odot N(\alpha))\mathscr{R}, N(\beta)\odot N(\beta) = \mathscr{R}^{-1} (N(\beta)\odot N(\beta))\mathscr{R}, N(\alpha)\odot N(\beta) = \mathscr{R} (N(\alpha)\odot N(\beta))\mathscr{R}$$
 so $D_+(\SL_2):= \mathcal{X}(\mathbb{D}_1^+)$ has a Poisson bracket described by 
 $$ \{N(\alpha) \otimes N(\beta) \}^+ = r_+ (N(\alpha)\odot N(\beta)) + (N(\alpha)\odot N(\beta))r_+, \quad \{N(a)\otimes N(a)\}^+ = r^+ (N(a) \odot N(a)) - (N(a)\odot N(a))r^+,$$
for $a=\alpha, \beta$.
The Poisson variety $D_+(\SL_2)$ was studied by Alekseev-Malkin in  \cite{AlekseevMalkin_PoissonLie} inspired by the work of Semenov-Tian-Shansky. More precisely, they consider the bracket $\{\cdot, \cdot\}^{AM}:= - \{\cdot, \cdot\}^+$ for which, using the notations $r:= r^+$ and $r^*:= -r^-$, one has (compare with \cite[Equation $(80$)]{AlekseevMalkin_PoissonLie}):
$$ \{N(\alpha)\odot N(\beta) \}^{AM}= - \left( r ( N(\alpha) \odot N(\beta)) + (N(\alpha)\odot N(\beta))r^* \right).$$

For $D_+(\SL_2)$, the bad arcs partition writes
 $$ D_+(\SL_2)= D_{00}\bigsqcup D_{01} \bigsqcup D_{10} \bigsqcup D_{11}, $$
 where 
  $$ D_{ij} = \{(g_1, g_2) | g_2^{-1}g_1 \in \SL_2^i , g_2g_1^{-1} \in \SL_2^j \}.$$
  Note that, by comparing the previous formulas for the Poisson brackets, it is clear that the diagonal embedding $\delta: \SL_2^D \to D(\SL_2)$ sending $g$ to $(g,g)$ is Poisson. Denote by $\SL_2^{\delta}\subset D(\SL_2)$ its image. Introduce also the subalgebra
 $$ \SL_2^* := \left\{ \left( \begin{pmatrix} \lambda & x \\ 0 & \lambda^{-1} \end{pmatrix}, \begin{pmatrix} \lambda^{-1} & 0 \\ y & \lambda \end{pmatrix} \right), \lambda \in \mathbb{C}^*, x,y \in \mathbb{C} \right\} \subset D(\SL_2).$$
 The bad arcs leaves admit the following characterisation, which appear in \cite{AlekseevMalkin_PoissonLie}:
 $$ D_{ij} = \SL_2^{\delta} (\dot{w}_i, \mathds{1}_2) \SL_2^* \cap \SL_2^* (\mathds{1}_2, \dot{w}_j) \SL_2^{\delta}.$$

 \begin{theorem}\label{theorem_leaves_D1}(Alekseev-Malkin \cite[Theorem $2$]{AlekseevMalkin_PoissonLie}) The bad arcs leaves $D_{ij}$ are symplectic in $D_+(\SL_2)$.
 \end{theorem}

 Let $\mathfrak{sl}_2, \mathfrak{sl}_2^*, D(\mathfrak{sl}_2)$ denote the Lie algebras of $\SL_2^D\cong \SL_2^{\delta}, \SL_2^*$ and $D(\SL_2)$ respectively and note that $\mathfrak{sl}_2\oplus \mathfrak{sl}_2^* = D(\mathfrak{sl}_2)$. The triple $(\mathfrak{sl}_2, \mathfrak{sl}_2^*, D(\mathfrak{sl}_2))$ is called a \textit{Manin triple} 
 and it encodes (and is determined by) at the infinitesimal level the Poisson Lie group structure of $\SL_2^D$
 (see \cite{ChariPressley} for a detailed account on the subject). Note that the $(2:1)$ map $\varphi: \SL_2^* \to (\SL_2^0)^{STS}$, sending $(g_1, g_2)$ to $g_2^{-1}g_1$ is Poisson: it corresponds to the embedding of marked surfaces drawn in Figure \ref{fig_embeddingD1}. 
 
    \begin{figure}[!h] 
\centerline{\includegraphics[width=4cm]{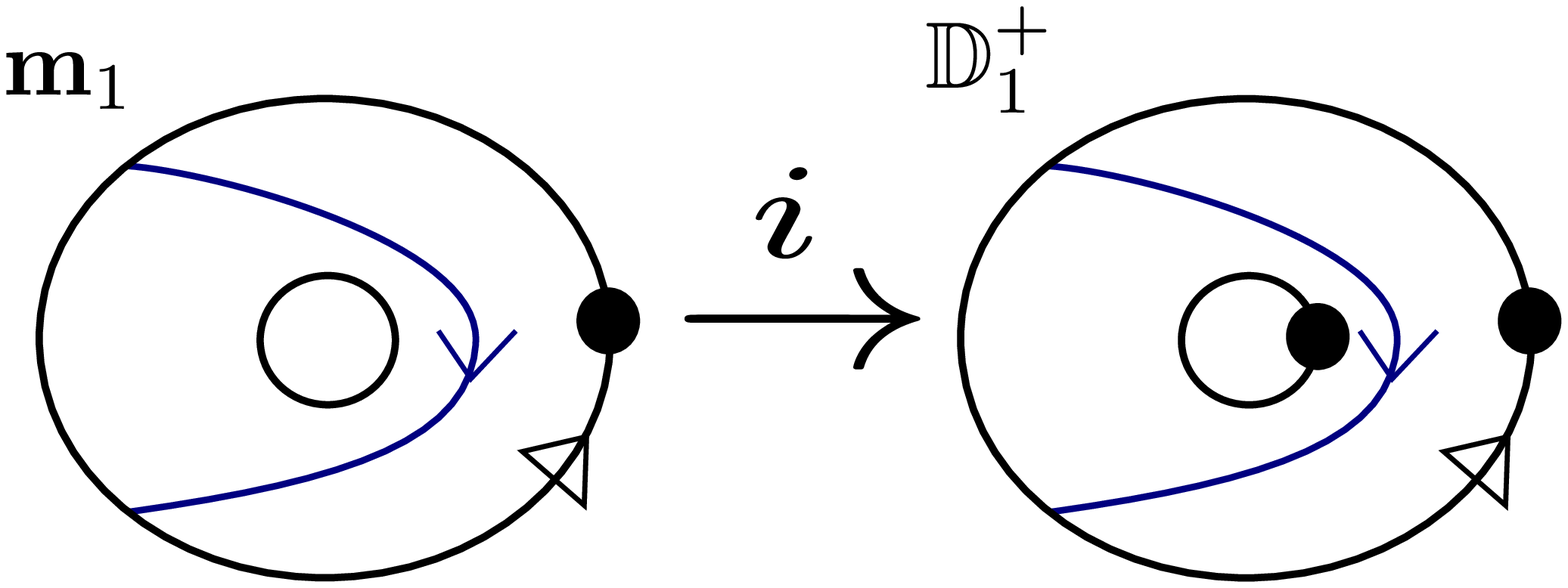} }
\caption{The embedding $\mathbf{m}_1 \to \mathbb{D}^+_1$ defining the map $\varphi: D_+(\SL_2)\to (\SL_2^{0})^{STS}$.} 
\label{fig_embeddingD1}
\end{figure} 
 
 \vspace{2mm}
\par \textbf{The work of Semenov-Tian-Shansky and its generalizations:}
 A useful tool to compute the symplectic leaves of a Poisson variety is the notion of dual pairs introduced by Weinstein in \cite{Weinstein_DualPairs}. For $X$ a smooth Poisson variety and $A\subset \mathcal{O}(X)$ a subspace, define
 $$A^{\perp} := \{ f \in \mathcal{O}(X) | \{f, a\} = 0, \forall a\in A\}.$$
 
 \begin{definition}
 Let $X, Y, Z$ three smooth Poisson manifolds such that $Z$ is symplectic. A \textit{dual pair} is a pair $(f,g)$ of Poisson morphisms $f: Z\to X, g: Z\to Y$ which are surjective submersions and such that 
 $$ f^* (\mathcal{O}(X))^{\perp} = g^{*}(\mathcal{O}(Y)) \quad \mbox{and} \quad g^{*}(\mathcal{O}(Y))^{\perp} =  f^* (\mathcal{O}(X))^{\perp}.$$
 \end{definition}
 
 \begin{theorem}\label{theorem_Weinstein}(Weinstein \cite[Section $8$]{Weinstein_DualPairs}) For $(f,g)$ a dual pair, then the symplectic leaves of $X$ are the connected components of the $f(g^{-1}(y))$ for $y\in Y$ and the symplectic leaves of $Y$ are the connected components of the $g(f^{-1}(x))$ for $x\in X$.
 \end{theorem}
 
 Let $(\mathfrak{g}, \mathfrak{g}^*, D(\mathfrak{g}))$ be a Manin triple of finite dimensional Lie bialgebras and let $G,G^*, D(G)$ denote the connected, simply connected Poisson Lie groups obtained by exponenting the three bialgebras. The equality maps $\mathfrak{g}\oplus \mathfrak{g}^* = D(\mathfrak{g})$ and $\mathfrak{g}^*\oplus \mathfrak{g} = D(\mathfrak{g})$ induce some Poisson maps $\Psi_1: G\times G^* \to D(G)$ and $\Psi_2: G^*\times G \to D(G)$ which are local diffeomorphisms in the neighbourhoods of the identities. The Poisson Lie group is called \textit{factorizable} if $\Psi_1$ and $\Psi_2$ are global diffeomorphisms. In this case, consider the diffeomorphism $\Psi_2^{-1} \circ \Psi_1: G \times G^* \cong G^* \times G$ sending $(g, g^*)$ to $\Psi_2^{-1} \circ \Psi_1(g,g^*)= (g^{g^*}, g\cdot g^*)$. The formulas $(g,g^*) \mapsto g^{g^*}$ and $(g,g^*) \mapsto g\cdot g^*$ define a right action of $G^*$ on $G$ and a left action of $G$ on $G^*$ respectively, which are called \textit{dressing actions}. In the factorizable case, Semonov-Tian-Shansky proved in \cite[Proposition $14$]{STS_DressingAction} that $D(G)$ is symplectic and that in the correspondance
 $$ \begin{tikzcd} 
 {} & D(G) \arrow[ld, "\pi_1"] \arrow[rd, "\pi_2"']& {} \\
 G \cong  {}_{G^*}\backslash^{D(G)} & {} & \quotient{D(G)}{G^*} \cong G
 \end{tikzcd}
 $$
 the pair $(\pi_1, \pi_2)$ is a dual pair. Using Theorem \ref{theorem_Weinstein}, Semenov-Tian-Shansky concludes in \cite[Proposition $15$]{STS_DressingAction} that the symplectic leaves of $G$ are the connected components of the $G^*$ dressing orbits $\pi_1(\pi_2^{-1}(g)), g\in G$ with a similar result by exchanging $G$ and $G^*$. 
 
 \vspace{2mm}
 \par Unfortunately, the Poisson Lie group $\SL_2^D$ is not factorizable: the product map $\Psi_1: \SL_2^{\delta}\times \SL_2^* \to D(\SL_2)$ is $2:1$ and is not surjective.
 Moreover, by Lemma \ref{lemma_Casimir_partition}, $D(\SL_2)$ is very far from been symplectic: it has an infinite number of symplectic leaves. The idea of Alekseev-Malkin is to replace $D(\SL_2)$ by the symplectic leaves $D_{ij}^+$ of $D_+(\SL_2)$. The inclusion maps $\SL_2^{\delta}, \SL_2^* \hookrightarrow D_+(\SL_2)$ are also Poisson. Let 
 $$\Gamma:= \SL_2^{\delta} \cap \SL_2^* = \{ (\pm \mathds{1}_2, \pm \mathds{1}_2)\} \cong \mathbb{Z}/2\mathbb{Z}$$
 and write $\overline{\SL_2}^D= \quotient{\SL_2^{\delta}}{\Gamma}$ and $\overline{\SL_2^*}=\quotient{\SL_2^*}{\Gamma}$. Recall the definition of the simple Bruhat cells $\SL_2^i$ and ${}^i\SL_2$ corresponding to the Cartan data $(B_+, B_-)$ and $(B_-, B_+)$ respectively. Alekseev and Malkin proved that in the correspondance

  $$ \begin{tikzcd} 
 {} & D_{ij}^+ \arrow[ld, "\pi_1"] \arrow[rd, "\pi_2"']& {} \\
 {}^i \overline{\SL_2}^{D} \cong {}_{\SL_2^* \cap D_{ij}} \backslash {D_{ij}^+}& {} & \quotient{D_{ij}^+}{\SL_2^*\cap D_{ij}} \cong {\overline{\SL_2}^j}^{D}
 \end{tikzcd}
 $$
 the maps $(\pi_1, \pi_2)$ form a dual pair, so Weinstein's Theorem \ref{theorem_Weinstein} implies that the symplectic leaves of $\overline{\SL_2}^D$ are the intersections of the $\SL_2^*$ dressing orbits with the Bruhat cells ${\overline{\SL_2}^j}^{D}$. Since the quotient map $\SL_2^D \to \overline{\SL_2}^D$ is a regular Poisson covering (it is \'etale), the symplectic leaves of $\SL_2^D$ are the pull-backs by this map of the symplectic leaves of $\overline{\SL_2}^D$ , and we get the
 
 \begin{theorem}\label{theorem_leaves_bigon}(Hodges-Levasseur \cite[Theorem $B.2.1$]{HodgesLevasseur_OqG}, Alekseev-Malkin \cite[Section $4$]{AlekseevMalkin_PoissonLie}) The symplectic leaves of $\SL_2^D$ are
 \begin{enumerate}
 \item The leaves $\mathfrak{F}_{11}(\lambda):= \left\{ \begin{pmatrix} a & \lambda b \\ b & d\end{pmatrix}| b\neq 0, ad-\lambda b^2=1\right\}, \lambda \in \mathbb{C}^*$;  
 \item the leaves $X_{01}$ and $X_{10}$; 
 \item the singletons $\{g\}$ for $g\in X_{11}$.
 \end{enumerate}
 In particular, the equivariant symplectic leaves of $\SL_2^D$ are the double Bruhat cells $X_{ij}$.
 \end{theorem}
 
 Similarly, by exchanging the roles of $\SL_2$ and $\SL_2^*$, we obtain that the symplectic leaves of $\overline{\SL_2^*}$ and, using the Poisson isomorphism $\varphi: \overline{\SL_2^*}\xrightarrow{\cong} {\SL_2^0}^{STS}$, sending $(g_+, g_-)$ to $g_-^{-1}g_+$, and noting that the dressing action of $\SL_2^D$ on $\overline{\SL_2^*}$ corresponds through $\varphi$ to the action of $\SL_2^D$ on ${\SL_2^0}^{STS}$ by conjugacy, 
 we see that the leaves $\SL_2^0\cap C$, for $C$ a conjugacy class, are symplectic in $\SL_2^{STS}$. Together with Lemma \ref{lemma_singleton}, we get the
 
 \begin{theorem}\label{theorem_leaves_monogon}(Alekseev-Malkin \cite[Section $4$]{AlekseevMalkin_PoissonLie})
 The symplectic leaves of $\SL_2^{STS}$ are
 \begin{enumerate}
 \item The leaves $\SL_2^0 \cap C$, for $C$ a conjugacy class; 
 \item the singletons $\{g\}$ for $g\in \SL_2^1$.
 \end{enumerate}
 \end{theorem}

 \subsection{The classical fusion operation}
 
 We now describe the classical equivalent of Theorem \ref{theorem_fusion}. Let $\mathbf{r}_{\hbar} : \mathcal{O}_{\hbar}[\SL_2]^{\otimes 2}\to \mathbb{C}[[\hbar]]$ be the co-R-matrix for parameter $q:=\exp(\hbar)$ and note that 
 $$ \mathbf{r}_{\hbar} \equiv \epsilon \otimes \epsilon + \hbar r^+ \pmod{\hbar^2}.$$
 In this formula, we see $r^+ \in \mathfrak{sl}_2\otimes \mathfrak{sl}_2$ as an element of the Zariski tangent space of $\SL_2\times \SL_2$ at the neutral element $(\mathds{1}_2, \mathds{1}_2)$, i.e. $r^+ \in \mathrm{Der}(\mathcal{O}[\SL_2]^{\otimes 2}, \mathbb{C}_{\varepsilon \otimes \varepsilon})$ is a derivation valued in $\mathbb{C}$ with $\mathcal{O}[\SL_2]^{\otimes 2}$-module structure induced by $\epsilon \otimes \epsilon$.

  Let $G$ be an affine algebraic Poisson Lie group with Poisson structure given by a classical r-matrix $r^+ \in \mathfrak{g}\otimes \mathfrak{g}$ (i.e. with Poisson structure given by the cocomutator $\delta: \mathfrak{g}\to \mathfrak{g}\otimes \mathfrak{g}$, $\delta(X)=[X, r^+]$). A $G$-Poisson affine variety is a complex affine variety $X$ with an algebraic Poisson action $G\times X \to X$. 
  
  \begin{definition}
  Let $X$ be a $G^2$-Poisson affine variety and denote by $\Delta_{G\times G} : \mathcal{O}[X]\to \mathcal{O}[G]\otimes \mathcal{O}[X]$ it comodule map. Wite $\Delta^1:= (\id\otimes \epsilon\times \id)\circ \Delta_{G\times G}$ and $\Delta^2:= (\epsilon \otimes \id \otimes \id)$.
  The \textit{fusion} of $X$ is the $G$-Poisson affine variety $X^{\circledast}$ define by
  \begin{enumerate}
  \item As a $\mathbb{C}$-algebra, $\mathcal{O}[X^{\circledast}]=\mathcal{O}[X]$.
  \item  For $x\in \mathcal{O}[X]$ and $i=1,2$, write $\Delta^i (x) = \sum x_{(i)}'\otimes x_{(i)}''$. The Poisson bracket is defined by
  $$\{ x, y \}^{\circledast} := \{x,y\} + \sum r^+ (y'_{(2)} \otimes x'_{(1)}) x_{(1)}''y_{(2)}'' - \sum r^+ (x'_{(2)}\otimes y'_{(1)})y''_{(1)}x''_{(2)}.$$
   \item The $G$ action is given by the comodule map $\Delta_G:= (\mu_G \otimes \id) \circ \Delta_{G\times G} $.
  \end{enumerate}
  \end{definition}
  
  In the particular case where $X$ is smooth, consider $X$ as a smooth manifold and denote by $\pi_X $ the Poisson bivector field defining the Poisson structure (i.e. $\{f,g\}(x)=\left< D_xf \otimes D_xg, \pi_{X,x}\right>$). Let $r^- := \sigma(r^+)$, where $\sigma(x\otimes y)=y\otimes x$. Let $a_{G\times G}: \mathfrak{g}\otimes \mathfrak{g} \to \Gamma(X, T_X)$ the infinitesimal action induced by the action of $G^2$ on $X$. Then the fusion $X^{\circledast}$ is the manifold $X$ with the Poisson bivector field
  $$ \pi_{X^{\circledast}} = \pi_X + a_{G\times G} (r^- - r^+).$$
 This is using this formula that the concept of fusion was introduced in the work of Alekseev-Malkin \cite{AlekseevMalkin_PoissonCharVar}. As we shall see, for a connected marked surface with non-trivial marking, then 
 $\mathcal{X}(\mathbf{\Sigma})$ is smooth. Recall the coaction $\Delta^L : \mathcal{S}_{+1}(\mathbf{\Sigma}) \to \mathcal{O}[\SL_2]^{\otimes \mathcal{A}} \otimes \mathcal{S}_{+1}(\mathbf{\Sigma})$ induced by gluing some bigons to the boundary arcs. This endow $\mathcal{X}(\mathbf{\Sigma})$ with a structure of $(\SL_2^D)^{\mathcal{A}}$ Poisson variety. In particular, by choosing two boundary arcs $a,b$, we get a structure of $(\SL_2^D)^2$-Poisson variety on $\mathcal{X}(\mathbf{\Sigma})$.
 As a consequence of Costantino and L\^e's Theorem \ref{theorem_fusion}, we get the 
 
 \begin{theorem}\label{coro_fusion}
 If $\mathbf{\Sigma}_{a\circledast b}$ is obtained from $\mathbf{\Sigma}$ by fusioning $a$ and $b$, then $\mathcal{X}(\mathbf{\Sigma}_{a\circledast b})\cong \mathcal{X}(\mathbf{\Sigma}^{a\circledast b})$.
 \end{theorem}
 
 \begin{proof}
 Recall from Theorem \ref{theorem_fusion} the isomorphism of vector spaces $\Psi_{a\circledast b} : \mathcal{S}_{\hbar}(\mathbf{\Sigma}) \cong \mathcal{S}_{\hbar}(\mathbf{\Sigma}_{a \circledast b})$ so that the  pull-back $\mu_{\circledast}$ of the product in $ \mathcal{S}_{\hbar}(\mathbf{\Sigma}_{a \circledast b})$ is the fusion, in the sense of Definition \ref{def_fusion_quantique}, of the product $\mu$ of $ \mathcal{S}_{\hbar}(\mathbf{\Sigma})$. Denote by $\star_{\circledast}$ and $\star$ the products in $\mathcal{S}_{+1}(\mathbf{\Sigma})\otimes \mathbb{C}[[\hbar]]$ corresponding to $\mu_{\circledast}$ and $\mu$ respectively.
 For $x,y\in \mathcal{S}_{+1}(\mathbf{\Sigma})$, we have
 $$ x \star_{\circledast} y = \sum \mathbf{r}_{\hbar}(y'_{(2)} \otimes x'_{(1)} ) x''_{(1)} \star y''_{(2)}, \quad \mbox{and} \quad y\star_{\circledast} x = \sum \mathbf{r}_{\hbar}(x'_{(2)}\otimes y'_{(1)})y''_{(1)} \star x''_{(2)}.$$
 Using the equalities $ \mathbf{r}_{\hbar}(x\otimes y) \equiv \epsilon(x)  \epsilon(y) + \hbar r^+(x\otimes y)  \pmod{\hbar^2}$ and $x\star y - y\star x = \hbar \{x,y\} \pmod{\hbar^2}$, we find:
 $$ x\star_{\circledast} y - y\star_{\circledast} x \equiv \hbar\left( \{x,y\} + \sum r^+ (y'_{(2)} \otimes x'_{(1)}) x_{(1)}''y_{(2)}'' - \sum r^+ (x'_{(2)}\otimes y'_{(1)})y''_{(1)}x''_{(2)}\right) \pmod{\hbar^2}.$$
 \end{proof}
 
The main motivation for the authors of \cite{AlekseevMalkin_PoissonCharVar} to introduce the notion of fusion was to get the following decomposition (which is \cite[Theorem 2]{AlekseevMalkin_PoissonCharVar}): 
 $$ \mathcal{X}(\mathbf{\Sigma}_{g,n}^{0}) \cong \mathcal{X}(\mathbf{\Sigma}_{1,0}^0) ^{\circledast g} \circledast (\SL_2^{STS})^{\circledast n},$$
 where $ \mathcal{X}(\mathbf{\Sigma}_{1,0}^0)$ is obtained from $D_+(\SL_2)$ by fusion. This reduces the study of $ \mathcal{X}(\mathbf{\Sigma}_{g,n}^{0}) $ to the study of $D_+(\SL_2)$ and $\SL_2^{STS}$.
 A useful concept, also introduced by Alekseev-Kosmann-Schwarzbach-Meinrenken in \cite{AlekseevKosmannMeinrenken} is the
 \begin{definition}
 A smooth $G$-Poisson variety $X$ is \textit{non-degenerate} if the map 
 $$ (a_G, \pi^{\#}_X) : \mathfrak{g}\oplus T_X^{*} \to T_X$$ 
 is surjective, where $a_G : \mathfrak{g} \to T_X$ is the infinitesimal $G$ action and $\pi_X^{\#}$ is the map induced from the Poisson bivector field $\pi_X$.
 \end{definition}
 
 \begin{lemma}\label{lemma_XXX}( Ganev-Jordan-Safranov \cite[Proposition $2.13$]{GanevJordanSafranov_FrobeniusMorphism} following \cite[Section $10$]{AlekseevKosmannMeinrenken})
 Let $X$  a smooth $G^2\times H$-Poisson variety and $X^{\circledast}$ is the $G\times H$-Poisson variety obtained by fusion. If $X$ is non degenerate then $X^{\circledast}$ is non-degenerate.
 \end{lemma}
 
 \begin{lemma}\label{lemma_ND}
 For $g\geq 1$, the $\SL_2^D$-Poisson variety $\mathcal{X}(\mathbf{\Sigma}_{g,0}^0)$ is non-degenerate.
 \end{lemma}
 
 \begin{proof}
 The fact that $D_+(\SL_2)=\mathcal{X}(\underline{\mathbf{\Sigma}}_{0,2})$ is non-degenerate follows from Theorem \ref{theorem_leaves_D1}. By fusioning the two boundary arcs of $D_+(\SL_2)$ we obtain $\mathbf{\Sigma}_{1,0}^0$ and by fusioning $g$ copies of $\mathbf{\Sigma}_{1,0}^0$ we obtain $\mathbf{\Sigma}_{g,0}^0$, so the lemma follows from Lemma \ref{lemma_XXX} and Theorem \ref{coro_fusion}.
 \end{proof}

 Let $i: \mathbf{m}_1 \to \mathbf{\Sigma}_{g,0}^0$ be the (unique up to isomorphism) embedding of marked surfaces illustrated in Figure \ref{fig_MomentMap}. Denote by $\mu_q : \mathcal{O}_q[\SL_2^{STS}] \to \mathcal{S}_A(\mathbf{\Sigma}_{g,n}^0)$ the induced algebra morphism and denote by $\mu: \mathcal{X}(\mathbf{\Sigma}_{g,0}^0) \to \SL_2^{STS}$ the Poisson morphism induced by $\mu_{+1}$. 
 \begin{figure}[!h] 
\centerline{\includegraphics[width=7cm]{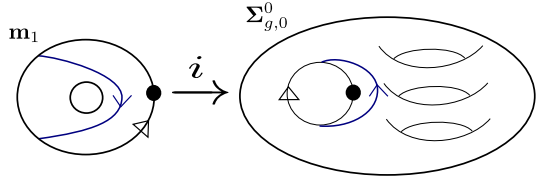} }
\caption{The embedding $\mathbf{m}_1 \to \mathbf{\Sigma}_{g,0}^0$ defining the moment map $\mu: \mathcal{X}(\mathbf{\Sigma}_{g,0}^0) \to \SL_2^{STS}$.} 
\label{fig_MomentMap}
\end{figure} 
 
 The map $\mu$ was called \textit{Lie group valued moment map} in \cite{AlekseevMalkinMeinrenken_LieGroupMomentMap} and the map $\mu_q$ was named \textit{quantum moment map} in \cite{GanevJordanSafranov_FrobeniusMorphism}. Note that $\mu_q$ sends the unique bad arc of $\mathcal{O}_q[\SL_2^{STS}] $ to the unique bad arc of $\mathcal{S}_A(\mathbf{\Sigma}_{g,0}^0)$, so the bad arcs decomposition
 $$ \mathcal{X}(\mathbf{\Sigma}^0_{g,0})= \mathcal{X}^0(\mathbf{\Sigma}^0_{g,0}) \bigsqcup \mathcal{X}^1(\mathbf{\Sigma}^0_{g,0}), $$
 is such that $\mathcal{X}^i(\mathbf{\Sigma}^0_{g,0})= \mu^{-1}(\SL_2^i)$.

 \begin{theorem}\label{theorem_OpenDense}(Ganev-Jordan-Safranov \cite[Theorem $2.14$]{GanevJordanSafranov_FrobeniusMorphism}) Let $g\geq 1$. The open dense subset $\mathcal{X}^0(\mathbf{\Sigma}^0_{g,0})\subset \mathcal{X}(\mathbf{\Sigma}^0_{g,0})$ is symplectic.
 \end{theorem}
 
 \begin{proof} Let $a: \mathfrak{sl}_2 \to \Gamma( {\SL_2^0}^{STS}, T_{\SL_2^0})$ be the infinitesimal action of the conjugacy action $\SL_2^D \circlearrowright {\SL_2^0}^{STS}$. By Theorem \ref{theorem_leaves_monogon}, since the orbits of this action are symplectic, then the image of $a_x$ is included in the image of $\pi_{\SL_2^0, x}^{\#}$ for all $x\in \SL_2^0$, i.e. the infinitesimal action by conjugacy is Hamiltonian. Since $\mu$ is Poisson and $\SL_2^D$ equivariant, the infinitesimal action of $\SL_2^D$ on $ \mathcal{X}^0(\mathbf{\Sigma}^0_{g,0})$ is also Hamiltonian. Moreover, since $ \mathcal{X}^0(\underline{\mathbf{\Sigma}}_{g,n})$ is non-degenerate, for any $x\in  \mathcal{X}^0(\mathbf{\Sigma}^0_{g,0})$ the map $\pi_{\mathbf{\Sigma}^0_{g,0}, x}^{\#} $ is surjective, therefore $\mathcal{X}^0(\mathbf{\Sigma}^0_{g,0})$ is symplectic.
 \end{proof}
 
 \begin{remark}
 In \cite[Theorem $5.9$]{SjamaarLerman}, Sjamaar and Lerman proved that for any compact Lie group $G$, the relative character variety $\mathcal{X}_G(\underline{\mathbf{\Sigma}}_{g,n})$  contains a connected open dense symplectic leaf and state that "Although most of the results of this paper hold for proper actions of arbitrary Lie groups, the proofs are technically more difficult". Theorem \ref{theorem_OpenDense} is an explicit version of the case $G=\SL_2$ and $n=1$.
\end{remark}

  \subsection{Relative character varieties}

The Poisson variety $\mathcal{X}(\mathbf{\Sigma})$ admits a geometric interpretation that we now  briefly describe and refer to \cite{FockRosly, AlekseevMalkin_PoissonCharVar, KojuTriangularCharVar} for details. Let $\Pi_1(\Sigma)$ be the fundamental groupoid of $\Sigma$ whose objects are points in $\Sigma$ and morphisms are paths $\alpha: v_0 \to v_1$. We keep the same notations than in Section \ref{sec_presentations}. The inclusion $\mathcal{A}\subset \Sigma$ induces a fully-faithful functor $\Pi_1(\mathcal{A})\subset \Pi_1(\Sigma)$. The \textit{representations space} is the set of functors
$$\mathcal{R}_{\SL_2}(\mathbf{\Sigma}):= \{ \rho : \Pi_1(\Sigma) \to \SL_2, | \rho(\alpha^0)=\mathds{1}_2, \forall \alpha^0 \in \Pi_1(\mathcal{A}) \}.$$
The \textit{gauge group} $\mathcal{G}$ is the group of maps $g:\Sigma \to \SL_2$ with finite support such that $g(v)=\mathds{1}_2$ for all $v\in \mathcal{A}$. It acts on $\mathcal{R}_{\SL_2}(\mathbf{\Sigma})$ by the formula
$$ g\cdot \rho (\alpha) = g(t(\alpha)) \rho(\alpha) g(s(\alpha))^{-1} , \quad \forall g\in \mathcal{G}, \rho \in \mathcal{R}_{\SL_2}(\mathbf{\Sigma}),  \alpha \in \Pi_1(\Sigma).$$
 The set $\mathcal{R}_{\SL_2}(\mathbf{\Sigma})$ is the set of closed points of an affine scheme over $\mathrm{Spec}(\mathbb{C})$ whose algebra of regular functions is 
 $$\mathcal{O}[\mathcal{R}_{\SL_2}]:= \quotient{ \mathbb{C}[X_{\alpha, i,j} | \alpha \in \Pi_1(\Sigma), i,j=\pm ]}{ \left( M(\alpha)M(\beta)= M(\alpha \beta), \det(M(\alpha))=1, M(\alpha^0)=\mathds{1}_2, \mbox{ if } \alpha^0 \in \Pi_1(\mathcal{A}) \right)}, $$
 where we wrote $M(\alpha):=\begin{pmatrix}X_{\alpha, +,+} & X_{\alpha, +,-} \\ X_{\alpha, -,+} & X_{\alpha, -,-} \end{pmatrix}$. Similarly, the elements of $\mathcal{G}$ form the closed points of an affine group scheme over $\mathrm{Spec}(\mathbb{C})$ whose Hopf algebra of regular functions is 
 $$\mathcal{O}[\mathcal{G}] = \quotient{ \mathbb{C}[ X_{v, i,j} | v\in \Sigma, i,j =\pm ] }{\left(\det(M(v))=1, M(v_0)=\mathds{1}_2 \mbox{ if }v_0 \in \mathcal{A}\right)}, $$
  where $M(v)= \begin{pmatrix} X_{v,+,+} & X_{v, +,-} \\ X_{v, -, +} & X_{v, -, -} \end{pmatrix}$. 
  Using the notation
  $$\begin{pmatrix} a & b \\ c & d \end{pmatrix} \boxtimes \begin{pmatrix} a' & b' \\ c' & d' \end{pmatrix} = \begin{pmatrix} a\otimes a' +b\otimes c' & a\otimes b' + b\otimes d' \\ c\otimes a' + d\otimes c' & c\otimes b' + d \otimes d' \end{pmatrix}, $$
  then the coproduct $\Delta$ on $\mathcal{O}[\mathcal{G}] $ is given by  $\Delta(M(v))= M(v)\boxtimes M(v)$ and the antipode by $S(M(v)):= \begin{pmatrix} X_{v, -, -}  & -X_{v, +, -} \\ X_{v, -, +} & X_{v, +, +}  \end{pmatrix}$.
   The action of $\mathcal{G}$ on $\mathcal{R}_{\SL_2}(\mathbf{\Sigma})$ is algebraic, given by the comodule map $\Delta^L(M(\alpha))= (\mu \otimes \id)(\id \otimes \tau) \left( M(t(\alpha))\boxtimes M(\alpha) \boxtimes S(M(s(\alpha))) \right)$.

  \begin{definition}\label{def_CharVar}
  The \textit{relative character variety} is the algebraic quotient 
  $$ \mathcal{X}_{\SL_2}(\mathbf{\Sigma}):= \mathcal{R}_{\SL_2}(\mathbf{\Sigma}) \sslash \mathcal{G}.$$
  \end{definition}
  Said differently, the algebra $\mathcal{O}[\mathcal{X}_{\SL_2}(\mathbf{\Sigma})]$ is defined as the set of coinvariant vectors for the coaction, i.e. as the kernel
  $$ 0 \to \mathcal{O}[\mathcal{X}_{\SL_2}(\mathbf{\Sigma})] \to \mathcal{O}[\mathcal{R}_{\SL_2}(\mathbf{\Sigma})] \xrightarrow{\Delta^L - \eta \circ \epsilon} \mathcal{O}[\mathcal{G}] \otimes \mathcal{O}[\mathcal{R}_{\SL_2}(\mathbf{\Sigma})].$$
  
\textbf{Curve functions}  For $\gamma \in \Sigma$ a simple closed curve, choose $v\in \gamma$ and let $\gamma_v : v \to v$ the path in $\Pi_1(\Sigma)$ representing $\gamma$. Fix $f\in \mathcal{O}[\SL_2]^{\SL_2}$ a regular function invariant by conjugacy, i.e. $f=P(a+d)$ is a polynomial in the trace function. Then the element $f_{\gamma} :=P(X_{\gamma_v, +,+} + X_{\gamma_v, -, -}) \in \mathcal{O}[\mathcal{X}_{\SL_2}(\mathbf{\Sigma})]$ is a coinvariant vector which only depends on the isotopy class of $\gamma$ and not on its orientation (because $\tr(A^{-1})=\tr(A)$ for $A\in \SL_2$). It corresponds to the function $f_{\gamma}([\rho])= P(\tr (\rho(\gamma)))$. Similarly, for $g\in \mathcal{O}[\SL_2]$ an arbitrary function, so $g=P(a,b,c,d)$ is a polynomial in the matrix coefficients $a,b,c,d \in \mathcal{O}[\SL_2]$, and for $\alpha\subset \Sigma$ a path with $s(\alpha), t(\alpha) \in \mathcal{A}$, the function $g_{\alpha} := P(X_{\alpha, ++}, X_{\alpha, -+}, X_{\alpha,-+}, X_{\alpha, --}) \in \mathcal{O}[\mathcal{X}_{\SL_2}(\mathbf{\Sigma})$ is a coinvariant vector which only depends on the isotopy class of $\alpha$ relatively to $\mathcal{A}$ (and depends on its orientation). So $g_{\alpha}([\rho])=g(\rho(\alpha))$. We call \textit{curve functions} such functions $f_{\gamma}, h_{\alpha}$ and a theorem of Procesi \cite{Procesi87} implies that curves functions generate $\mathcal{O}[\mathcal{X}_{\SL_2}(\mathbf{\Sigma})] $ (see \cite{KojuTriangularCharVar} for details). 
  
  \textbf{Poisson structure} We now define some Poisson structures on $ \mathcal{X}_{\SL_2}(\mathbf{\Sigma})$ which depend on a choice $\mathfrak{o}$ of orientations of the boundary arcs of $\mathbf{\Sigma}$. Let $f_{\alpha}, g_{\beta}$ be two curve functions and fix some geometric representatives of $\alpha$ and $\beta$ which only intersect transversally in the interior of $\Sigma$ and let us introduce some notations.  Fix $\rho \in \mathcal{R}_{\SL_2}(\mathbf{\Sigma})$. Recall the notations introduced in Section \ref{sec_CF} while defining the relative intersection form. In particular for $v\in \alpha\cap \beta$ we have defined a sign $\epsilon(v)$ depending whether the tangent vectors of $\alpha$ and $\beta$ at $v$ form an oriented bases of $T_v\Sigma$ or not. Also 
   for $a$ a boundary arc, we write $S(a):=a\cap \alpha \times a \cap \beta$ and for $(v_1,v_2)\in S(a)$, we have defined a sign $\epsilon(v_1,v_2)=\pm1$ in the same manner.
   Define $\mathfrak{o}(v_1,v_2)= +1$ if $v_1<_{\mathfrak{o}} v_2$ and $\mathfrak{o}(v_1,v_2)=-1$ else.
    Let $\left(\cdot , \cdot \right) : \mathfrak{sl}_2 \otimes \mathfrak{sl}_2 \to \mathbb{C}$ be the non degenerate Ad-invariant symmetric form normalized so that its dual is the symmetric part $\frac{1}{2}H\otimes H$ of $r^+$. 
    Let $X_{f, \alpha} \in \mathfrak{\sl_2}$ be the unique vector satisfying $(X_{f, \alpha}, Y)=D_{\rho(\alpha)}f (Y)$. For $v$ a point in $\alpha$, decompose the path in $\alpha=\alpha^+ \alpha^-$ where $s(\alpha^+)=t(\alpha^-)=v$ and write $X_{f,\alpha}(v):= \rho(\alpha^+)^{-1} X_{f, \alpha} \rho(\alpha^-)^{-1}\in \mathfrak{sl}_2$. Extend $(\cdot, \cdot)$ to a bilinear form on $\mathfrak{sl}_2^{\otimes 2}$ by $(x\otimes y, z \otimes t):= (x, z) (y,t)$.

  \begin{definition}\label{GoldmanFormula} The Poisson bracket $\{\cdot, \cdot \}^{\mathfrak{o}}$ on $\mathcal{O}[\mathcal{X}_{\SL_2}(\mathbf{\Sigma})]$ is defined by the following generalized Goldman formula
  \begin{eqnarray*} \{ f_{\alpha} , g_{\beta} \}^{\mathfrak{o}} ([\rho])  &=&   \sum_{a\in \mathcal{A}} \sum_{(v_1, v_2)\in S(a)} \varepsilon(v_1, v_2) \left( X_{f, \alpha}(v_1)\otimes X_{g, \beta}(v_2), r^{\mathfrak{o}(v_1,v_2)} \right) \\ &&+2\sum_{v\in c_1\cap c_2} \varepsilon(v) \left( X_{f, \alpha}(v), X_{g, \beta}(v) \right) 
 \end{eqnarray*}
  \end{definition}
  That the Poisson bracket of two regular functions is a regular function and the fact that $\{\cdot, \cdot\}^{\mathfrak{o}}$ satisfies the Jacobi identity are non trivial facts.

  \textbf{Discrete models} The relative character variety is an algebraic variety, i.e. $\mathcal{O}[\mathcal{X}_{\SL_2}(\mathbf{\Sigma})]$ is reduced and finitely generated. To prove this, fix $\mathbb{V}\subset \Sigma$ a finite subset such that $(1)$ $\mathbb{V}$ intersects each boundary arc $a$ once, in say $v_a$, and $(2)$ $\mathbb{V}$ intersects each connected component of $\Sigma$ at least once. Let $\Pi_1(\Sigma, \mathbb{V}) \subset \Pi_1(\Sigma)$ the full subcategory generated by $\mathbb{V}$ and define a finite presentation $\mathbb{P}=(\mathbb{G}, \mathbb{RL})$ of $\Pi_1(\Sigma, \mathbb{V})$ as in Definition \ref{def_presentation}. Write $\mathbb{RL}=(R_1, \ldots, R_m)$ and consider the $R_j$ as Laurent monomials in the generators $\beta_i \in \mathbb{G}$. Define the algebraic map $\mathcal{R}: (\SL_2)^{\mathbb{G}} \to (\SL_2)^{\mathbb{RL}}$ by $\mathcal{R}=(R_1, \ldots, R_m)$. The \textit{discrete representations space} is the affine subvariety $\mathcal{R}_{\SL_2}(\mathbf{\Sigma}, \mathbb{P}) \subset (\SL_2)^{\mathbb{G}} $ of elements $\rho$ such that $\mathcal{R} (\rho)= (\mathds{1}_2, \ldots, \mathds{1}_2)$. Write $\mathring{\mathbb{V}} \subset \mathbb{V}$ the subset of elements which are not in $\mathcal{A}$ and define the \textit{discrete gauge group} $\mathbb{G}_{\mathbb{V}}:= (\SL_2)^{\mathring{\mathbb{V}}}\cong \{ g: \mathbb{V} \to \SL_2 | g(v_a)=\mathds{1}_2,  \forall a \in \mathcal{A}\}$. It acts on $\mathcal{R}_{\SL_2}(\mathbf{\Sigma}, \mathbb{P})$ algebraically by the same formula
  $$ g\cdot \rho (\beta) = g(t(\beta)) \rho(\beta) g(s(\beta))^{-1} , \quad \forall g\in \mathcal{G}_{\mathbb{V}}, \rho \in \mathcal{R}_{\SL_2}(\mathbf{\Sigma}, \mathbb{P}),  \beta \in \mathbb{G}.$$
  The \textit{discrete character variety} is the algebraic quotient 
  $$\mathcal{X}_{\SL_2}(\mathbf{\Sigma}, \mathbb{P}) := \mathcal{R}_{\SL_2}(\mathbf{\Sigma}, \mathbb{P}) \sslash \mathcal{G}_{\mathbb{V}}.$$
  Now consider the algebraic morphisms $Res: \mathcal{R}_{\SL_2}(\mathbf{\Sigma}) \to \mathcal{R}_{\SL_2}(\mathbf{\Sigma}, \mathbb{P})$ and $\mathcal{G}\to \mathcal{G}_{\mathbb{V}}$ sending $\rho\in \mathcal{R}_{\SL_2}(\mathbf{\Sigma}) $ to its restriction to $\mathbb{G}$ and sending $g\in \mathcal{G}$ to its restriction to $\mathbb{V}$. The morphism $Res$ is clearly $\mathcal{G}_{\mathbb{V}}$ equivariant for the induced structure so it defines a morphism
  $$\Psi_{\mathbb{P}}: \mathcal{X}_{\SL_2}(\mathbf{\Sigma}) \to \mathcal{X}_{\SL_2}(\mathbf{\Sigma}, \mathbb{P}).$$
  
  \begin{theorem}\label{theorem_discrete_model}(K. \cite[Proposition $2.5$ ]{KojuTriangularCharVar}) $\Psi_{\mathbb{P}}$ is an isomorphism.
  \end{theorem}
  In particular, $\mathcal{X}_{\SL_2}(\mathbf{\Sigma})$ is an algebraic variety.
  
   \textbf{Character varieties for closed surfaces} 
   Let us now consider the case where $\mathbf{\Sigma}=(\Sigma, \emptyset)$ is closed and choose $\mathbb{V}=\{v\}$ a singleton. Then Theorem \ref{theorem_discrete_model} shows that
  $$ \mathcal{X}_{\SL_2}(\Sigma, \emptyset) \cong \Hom(\pi_1(\Sigma, v), \SL_2) \sslash \SL_2, $$
where $\SL_2$ acts by conjugacy.  Under this form, character varieties were introduced by Culler and Shalen in \cite{CullerShalenCharVar} as a powerful tool to analyse $3$-manifolds, in particular to study their incompressible surfaces. Suppose that $\Sigma$ is closed and connected.
The action of $\SL_2$ on $\Hom(\pi_1(\Sigma, v), \SL_2)$ by conjugacy is not free.
\begin{definition}
 We decompose the set of representations $\rho: \pi_1(\Sigma, v) \to \SL_2$  into three classes: 
  \begin{enumerate}
   \item The \textit{central representations} taking value in $\pm \mathds{1}_2$ and for which the stabilizer is $\SL_2$.
   \item The \textit{diagonal representations} which are conjugate to a non central representation valued in the subgroup $D\subset \SL_2$ of diagonal matrices and for which the stabilizer is the group of diagonal matrices.
   \item The \textit{irreducible representations} for which the stabilizer is $\pm \mathds{1}_2$.
   \end{enumerate}
  
  Accordingly, we decompose the (closed points of the) character variety in three classes
  $$ \mathcal{X}_{\SL_2}(\Sigma, \emptyset)= \mathcal{X}_{\SL_2}^0(\Sigma) \bigsqcup \mathcal{X}_{\SL_2}^1(\Sigma) \bigsqcup \mathcal{X}_{\SL_2}^2(\Sigma), $$
  where $\mathcal{X}_{\SL_2}^i(\Sigma)$ is the set of classes of  irreducible, diagonal and central representations when $i=0, 1, 2$ respectively.
  \end{definition}
  
   When $\Sigma$ has genus $1$, the set  $\mathcal{X}_{\SL_2}^0(\Sigma)$ is empty and the smooth locus of the character variety is $ \mathcal{X}_{\SL_2}^1(\Sigma)$. When $\Sigma$ has genus $g\geq 2$, the smooth locus is $\mathcal{X}_{\SL_2}^0(\Sigma)$. The Poisson variety $\mathcal{X}_{\SL_2}^{diag}(\Sigma)=\mathcal{X}_{\SL_2}^1(\Sigma) \bigsqcup \mathcal{X}_{\SL_2}^2(\Sigma)$ is easy to study. Consider the quadratic pair $\mathbb{E}=(\mathrm{H}_1(\Sigma, \mathbb{Z}), (\cdot, \cdot))$ formed by the first homology group with the intersection form and denote by $\mathcal{X}_{\mathbb{C}^*}(\Sigma):= \mathcal{X}(\mathbb{E})\cong \mathrm{H}^1(\Sigma, \mathbb{C}^*)$ the associated symplectic torus (which is the $\mathbb{C}^*$ character variety). The diagonal inclusion $\varphi: \mathbb{C}^* \hookrightarrow \SL_2$, sending $z$ to $\begin{pmatrix} z & 0 \\ 0 & z^{-1} \end{pmatrix}$, induces a morphism $\Phi : \mathcal{X}_{\mathbb{C}^*} (\Sigma) \hookrightarrow \mathcal{X}_{\SL_2}^{diag}(\Sigma)$ sending $\chi: \mathrm{H}_1(\Sigma; \mathbb{Z}) \to \mathbb{C}^*$ to the representation $\Phi(\chi)$ sending $\gamma \in \pi_1(\Sigma, v)$ to $\varphi(\chi[\gamma])$. Clearly $\Phi$ is a double covering which is branched along the set $ \mathcal{X}_{\SL_2}^2(\Sigma)$ of central representations.
  
   In algebraic language, $\mathcal{O}[\mathcal{X}_{\mathbb{C}^*}(\Sigma)]=\mathbb{C}[\mathrm{H}_1(\Sigma; \mathbb{Z})]$ and for $\rho: \pi_1(\Sigma)\to D$ a diagonal representation, the value $\tau_{\gamma}(\rho)=\tr(\rho(\gamma))$ only depends on the homology class $[\gamma]\in \mathrm{H}_1(\Sigma; \mathbb{Z})]$.
 The  morphism $\Phi^*: \mathcal{O}[\mathcal{X}^{diag}_{\SL_2}(\Sigma, \emptyset)]\to \mathcal{O}[\mathcal{X}_{\mathbb{C}^*}(\Sigma)]$ sends a trace function $\tau_{\gamma}$ to the homology class $[\gamma]$. It is an immediate consequence of Goldman formula in Definition \ref{GoldmanFormula} that for $\rho \in \mathrm{R}_{\SL_2}(\Sigma)$ a functor valued in the subset of diagonal matrices, then 
  $$ \{ \tau_{[\alpha]} , \tau_{[\beta]} \} (\rho)= \sum_{v \in \alpha \cap \beta} \epsilon(v) \tau_{[\alpha]+[\beta]}(\rho)= ([\alpha], [\beta]) \tau_{[\alpha]+[\beta]}(\rho).$$
  Therefore the double branched covering $\Phi : \mathcal{X}_{\mathbb{C}^*} (\Sigma) \hookrightarrow \mathcal{X}_{\SL_2}^{diag}(\Sigma)$ is Poisson and since $\mathcal{X}_{\mathbb{C}^*} (\Sigma) $ is symplectic we have proved the 
  
  \begin{lemma}\label{lemma_diagonal_symplectic}
  The leaf  $\mathcal{X}_{\SL_2}^1(\Sigma)$ is symplectic.
  \end{lemma}

   Atiyah and Bott revisited in \cite{AB} character varieties in the language of gauge theory by studying 2 dimensional Yang-Mills theory. Let $\Omega_F^1(\Sigma, \mathfrak{sl}_2)$ denote the set of $\mathfrak{sl}_2$ valued $1$-forms $A$ on $\Sigma$ with vanishing curvature, i.e. such that $FA:=dA +\frac{1}{2}[A\wedge A] =0$. Let $\mathcal{G}=C^{\infty}(\Sigma, \SL_2)$ act on  $\Omega_F^1(\Sigma, \mathfrak{sl}_2)$ by $A^g= g^{-1}Ag+g^{-1}dg$ and write
   $$\mathcal{M}_{\SL_2}(\Sigma) = \quotient{\Omega_F^1(\Sigma, \mathfrak{sl}_2)}{\mathcal{G}}.$$
   The set $\mathcal{M}_{\SL_2}(\Sigma)$ is isomorphic to the moduli space of $\SL_2$-flat structures on $\Sigma$, i.e. the set of isomorphism classes of pairs $(E,\nabla)$ where $E\to \Sigma$ is a principal $\SL_2$-bundle and $\nabla$ a flat connection. Since $\pi_1(\SL_2)=\pi_2(\SL_2)=0$, such a bundle is always isomorphic to the trivial bundle $\Sigma\times\SL_2$ in which $\nabla=d+A$ for $A\in \Omega_F^1(\Sigma, \mathfrak{sl}_2)$. The group $\mathcal{G}$ then identifies with the group of automorphisms of the trivial bundle so the correspondance is clear. By imposing some Sobolev regularity to the $1$-forms $A$ and the gauge group elements, one can turn $\Omega^1(\Sigma, \mathfrak{sl}_2)$ into a Banach space and  $\Omega_F^1(\Sigma, \mathfrak{sl}_2)$  becomes a Banach subvariety. Its tangent space at $A$ identifies with the space $\mathrm{Z}^1_A(\Sigma, \mathfrak{sl}_2)$ of these forms $\alpha\in \Omega^1(\Sigma, \mathfrak{sl}_2)$ such that $d\alpha+[\alpha \wedge A]=0$. 
    When $\Sigma$ is closed and equipped with some Riemannian metric, Atiyah and Bott defined a symplectic form on $\Omega^1_F(\Sigma, \mathfrak{sl}_2)$ by
   $$\omega^{AB}(\alpha, \beta) := \int_{\Sigma} \left( \alpha\wedge \beta \right).$$
   The action of the gauge group $\mathcal{G}$ on $\Omega_F^1(\Sigma, \mathfrak{sl}_2)$ is not free, so the quotient $\mathcal{M}_{\SL_2}(\Sigma)$ is not a manifold. However if we restrict to the subvariety of the principal orbits, i.e. these $A$ for which the stabilizer is minimal, we get a subset $\mathcal{M}^0_{\SL_2}(\Sigma)\subset \mathcal{M}_{\SL_2}(\Sigma)$ with the structure of a symplectic manifold (it can be seen as a symplectic reduction with the curvature map $F: \Omega^1(\Sigma, \mathfrak{sl}_2) \to \Omega^2(\Sigma, \mathfrak{sl}_2) \cong \Omega^0(\Sigma,\mathfrak{sl}_2)^*= Lie(\mathcal{G})^*$ playing the role of a moment map).
    The holonomy map gives a bijection $Hol: \Omega_F^1(\Sigma, \mathfrak{sl}_2) \cong \mathcal{R}_{\SL_2}(\Sigma, \emptyset)$ (more precisely, one needs to replace the fundamental group by the path groupoid) which induces a diffeomorphism $Hol : \mathcal{M}_{\SL_2}^0(\Sigma) \cong \mathcal{X}_{\SL_2}^0(\Sigma)$. This is while trying to describe algebraically the Atiyah-Bott symplectic structure induced by $Hol$ on $\mathcal{X}_{\SL_2}^0(\Sigma)$, and by generalizing previous computations of Wolpert on the Teichm\"uller space, that Goldman discovered in \cite{Goldman86} the formula in Definition \ref{GoldmanFormula} in the closed case. More precisely, fix a representation $\rho: \pi_1(\Sigma, v) \to \SL_2$ so that $Ad\circ \rho$ induces a structure of $\pi_1(\Sigma,v)$ module on $\mathfrak{sl}_2$. Let $\widetilde{\Sigma}$ be a universal covering of $\Sigma$ and consider the cochain complex $C^{\bullet}(\Sigma, \mathrm{Ad}_{\rho}):=\Hom_{\mathbb{Z}[\pi_1(\Sigma)]}(C^{\bullet}(\widetilde{\Sigma}; \mathbb{Z}), \mathfrak{sl}_2)$ with chain map $d\otimes \id$. 
    The Zariski tangent space of $ \mathcal{X}_{\SL_2}(\Sigma)$ is canonically isomorphic to the first cohomology group $\mathrm{H}^1_{\rho}(\Sigma, \mathfrak{sl}_2)$ of this complex and Goldman reinterpreted in \cite{Goldman_symplectic} the Atiyah-Bott symplectic form as the composition
    $$\omega^G : \mathrm{H}^1_{\rho}(\Sigma, \mathfrak{sl}_2)^{\otimes 2} \xrightarrow{\cup} \mathrm{H}^2_{\rho}(\Sigma, \mathfrak{sl}_2^{\otimes 2}) \xrightarrow{ (\cdot, \cdot)}  \mathrm{H}^2(\Sigma, \mathbb{C})\xrightarrow{\int_{\Sigma}} \mathbb{C}, $$
    where $(\cdot, \cdot)$ denote the Killing form normalized as in Definition \ref{GoldmanFormula}. Now $\omega^G$ endows the smooth locus of $\mathcal{X}_{\SL_2}(\Sigma)$ with a structure of symplectic manifold and 
    the main theorem of \cite{Goldman86}, is that for $\tau_{\alpha}, \tau_{\beta}$ two curve functions and for $[\rho]$ in the smooth locus, then $\{\tau_{\alpha}, \tau_{\beta}\}([\rho])$ is given by the formula of Definition \ref{GoldmanFormula}. Therefore one has the: 
    
    \begin{theorem}\label{theorem_smoothsymplectic}(Goldman \cite{Goldman_symplectic}) The leaf $\mathcal{X}_{\SL_2}^0(\Sigma)$ is symplectic. \end{theorem}
    
    \begin{corollary}\label{coro_closed_surface} For a closed surface $\Sigma$, the symplectic leaves of $\mathcal{X}_{\SL_2}(\Sigma)$ are $(1)$ the leaf $\mathcal{X}_{\SL_2}^0(\Sigma)$ (eventually empty in genus $1$), $(2)$ the leaf  $\mathcal{X}_{\SL_2}^1(\Sigma)$ and $(3)$ the singletons $\{ [\rho] \}$ for $\rho$ a central representation.
   \end{corollary}

  \textbf{Character varieties for marked surfaces}
  \par Recall from Example \ref{example_presentations} that to an oriented ciliated graph $(\Gamma,c)$, one can associate a marked surface $\mathbf{\Sigma}(\Gamma, c)$ whose set of boundary arcs is (non empty and) in correspondance with the set $\mathbb{V}$ of vertices of $\Gamma$, 
  together with a finite presentation $\mathbb{P}=(\mathcal{E}(\Gamma), \emptyset)$ without relation of $\Pi_1(\Sigma, \mathbb{V})$ for which the generators corresponds to the oriented edges of $\Gamma$. Moreover, any connected marked surface with non-empty set of boundary arcs can be obtained from such a graph. Since the presentation has no relation and since all points in $\mathbb{V}$ are in the boundary arcs, the discrete gauge group is trivial and 
  $$ \mathcal{X}_{\SL_2}(\mathbf{\Sigma}(\Gamma,c)) \cong \mathcal{X}_{\SL_2}(\mathbf{\Sigma}(\Gamma,c), \mathbb{P}) = \mathcal{R}_{\SL_2}(\mathbf{\Sigma}(\Gamma,c), \mathbb{P})= (\SL_2)^{\mathcal{E}(\Gamma)}.$$
  This is under this form that relative character varieties for marked surfaces first appeared in the work of Fock-Rosly in \cite{FockRosly}. Note that for
   $\mathbf{\Sigma}=(\Sigma_{g,n}, \mathcal{A})$ a connected marked surface with $\mathcal{A}\neq \emptyset$, we thus have an isomorphism of varieties
  $$\mathcal{X}_{\SL_2}(\Sigma_{g,n}, \mathcal{A}) \cong (\SL_2)^{2g-1+n+|\mathcal{A}|}.$$
  In particular, $\mathcal{X}_{\SL_2}(\mathbf{\Sigma})$ is smooth.  Fock and Rosly defined an explicit Poisson bivector field on $(\SL_2)^{\mathcal{E}(\Gamma)}$ for which the Poisson bracket of two elements $f_{\alpha}$ and $g_{\beta}$, for $\alpha, \beta \in \mathcal{E}(\Gamma)$, corresponds to the generalized Goldman formula in \ref{GoldmanFormula} (see \cite[Proposition $B.2$]{KojuTriangularCharVar} for details).

 The discrete models $\mathcal{X}_{\SL_2}(\mathbf{\Sigma}_{g,n}^0, \mathbb{P})$, for $\mathbb{P}$ the presentation induced by a Daisy graph as defined in Example \ref{example_presentations}, appeared apparently independently in the work of Alekseev-Malkin \cite{AlekseevMalkin_PoissonCharVar}. Note that the fusion property $\mathcal{X}_{\SL_2}(\mathbf{\Sigma}^{\circledast}) \cong \mathcal{X}_{\SL_2}(\mathbf{\Sigma})_{\circledast}$ (equivalent to Theorem \ref{coro_fusion}) can be read directly from the generalized Goldman formula. The decomposition $\mathcal{X}_{\SL_2}(\mathbf{\Sigma}_{g,n}^0) \cong\mathcal{X}_{\SL_2}(\mathbf{\Sigma}_{1,0}^0)^{\circledast g} \circledast \mathcal{X}_{\SL_2}(\mathbf{\Sigma}_{0,1}^0)^{\circledast n}$ was the initial motivation for the authors of \cite{AlekseevMalkin_PoissonCharVar} to introduce the notion of fusion. 
 Another discrete model appeared also in the work of Guruprasad, Huebschmann, Jeffrey and Weinstein in \cite{GHJW_ModSpacesParBd}, where the authors study marked surfaces having exactly one boundary arc per boundary component.
 The definition we gave here in Definitions \ref{def_CharVar} and \ref{GoldmanFormula} of relative character varieties is due to the author in \cite{KojuTriangularCharVar} and has the advantage of been independent on the choice of finite presentation $\mathbb{P}$. In particular by identifying canonically each discrete model with a continuous model in Theorem \ref{theorem_discrete_model}, it makes obvious the fact that all discrete models developed in \cite{FockRosly, AlekseevMalkin_PoissonCharVar, GHJW_ModSpacesParBd} are canonically isomorphic and only depend on the underlying marked surface. The (continuous) relative character variety we present here can be though as a presentation free reformulation of the discrete models of Fock-Rosly in the same manner than stated skein algebra are a presentation free reformulation of the quantum moduli spaces.
 
 \vspace{2mm}
 \par
 A new feature is the behaviour of relative character varieties for the gluing operation similar to that of stated skein algebras. Consider two boundary arcs $a,b$ in $\mathbf{\Sigma}$. The projection map $p: \Sigma \to \Sigma_{a\#b}$ induces a functor $p_* : \Pi_1(\Sigma) \to \Pi_1(\Sigma_{a\#b})$ which itself induces a regular (dominant) morphism $\theta_{a\#b}^*: \mathcal{R}_{\SL_2}(\mathbf{\Sigma}) \to \mathcal{R}_{\SL_2}(\mathbf{\Sigma}_{a\#b})$ sending a functor $\rho: \Pi_1(\Sigma)\to \SL_2$ to $\rho \circ p_*$. Remember that the Poisson structure depends on a choice of orientation $\mathfrak{o}$ of the boundary arcs. While choosing such a choice $\mathfrak{o}$ for $\mathbf{\Sigma}$ so that the orientation of $a$ and $b$ coincides under the gluing map (so we cannot choose $\mathfrak{o}^+$ for instance) and by choosing  the induced orientation $\mathfrak{o}_{a\#b}$ of the boundary arcs of $\mathbf{\Sigma}_{a\#b}$ (so that $p$ preserves the orientations), the morphism $\theta_{a\#b}$ is Poisson. For the bigon $\mathbb{B}$, while choosing the presentation with a single generator and no relation, and choosing the orientation $\mathfrak{o}^+$, we get a Poisson isomorphism
 $$\mathcal{X}_{\SL_2}^{\mathfrak{o}_+}(\mathbb{B}) \cong \SL_2^D.$$
 By gluing a bigon to a boundary arc $a$ of $\mathbf{\Sigma}$, exactly like for stated skein algebras, we get a Poisson comodule maps $\Delta_a^L: \mathcal{O}[\mathcal{X}_{\SL_2}(\mathbf{\Sigma})] \to \mathcal{O}[\SL_2^D] \otimes \mathcal{O}[\mathcal{X}_{\SL_2}(\mathbf{\Sigma})]$ and $\Delta_a^R: \mathcal{O}[\mathcal{X}_{\SL_2}(\mathbf{\Sigma})] \to \mathcal{O}[\mathcal{X}_{\SL_2}(\mathbf{\Sigma})]\otimes \mathcal{O}[\SL_2^D] $ so that
 
 \begin{theorem}(K. \cite[Theorem $1.3$]{KojuTriangularCharVar})
 The following sequence is exact and each map is Poisson:
$$ 0 \to \mathcal{O}[\mathcal{X}_{\SL_2}(\mathbf{\Sigma}_{a\#b})]^{\mathfrak{o}_{a\#b}} \xrightarrow{\theta_{a\#b}} \mathcal{O}[\mathcal{X}_{\SL_2}(\mathbf{\Sigma})]^{\mathfrak{o}} \xrightarrow{\Delta_a^L - \sigma \circ \Delta_b^R}  \mathcal{O}[\SL_2^D] \otimes \mathcal{O}[\mathcal{X}_{\SL_2}(\mathbf{\Sigma})]^{\mathfrak{o}}.$$

In particular, for a triangulated marked surface $(\mathbf{\Sigma}, \Delta)$ and choosing an orientation $\mathfrak{o}_{\Delta}$ of the edges of $\Delta$, which itself induces an orientation $\mathfrak{o}$ of the edges of $\mathbf{\Sigma}$ and for each face $\mathbb{T}$ induces an orientation $\mathfrak{o}_{\mathbb{T}}$ of the edges of $\mathbb{T}$, we have an exact sequence, where each map is Poisson:
$$ 0 \to  \mathcal{O}[\mathcal{X}_{\SL_2}(\mathbf{\Sigma})]^{\mathfrak{o}} \xrightarrow{\theta^{\Delta}} \otimes_{\mathbb{T}\in F(\Delta)}\mathcal{O}[\mathcal{X}_{\SL_2}(\mathbb{T})]^{\mathfrak{o}_{\mathbb{T}}}  \xrightarrow{\Delta^L - \sigma \circ \Delta^R} \left( \otimes_{e\in \mathring{\mathcal{E}}(\Delta)} \mathcal{O}[\SL_2^D] \right) \otimes \left( \otimes_{\mathbb{T}\in F(\Delta)} \mathcal{O}[\mathcal{X}_{\SL_2}(\mathbb{T})]^{\mathfrak{o}_{\mathbb{T}}}  \right).$$
 \end{theorem}
 
  \textbf{Character varieties for open unmarked surfaces}
  \par The original motivations for the authors of \cite{FockRosly, AlekseevMalkin_PoissonCharVar, GHJW_ModSpacesParBd} to introduce relative character varieties was to be able to study the Poisson geometry of character varieties of unmarked open surfaces $(\Sigma_{g,n}, \emptyset)$ for $n\geq 1$. Like in the closed case, $\mathcal{X}_{\SL_2}(\Sigma_{g,n})$ is singular and its smooth locus $\mathcal{X}^0_{\SL_2}(\Sigma_{g,n})$ is the set of classes of irreducible representations $r:\pi_1(\Sigma_{g,n}, v)\to \SL_2$. 
  Let us admit for a moment that we have a Poisson isomorphism $\mathcal{O}[\mathcal{X}_{\SL_2}(\mathbf{\Sigma})] \cong \mathcal{S}_{+1}(\mathbf{\Sigma})$ (it will be proved in the next subsection).
  Recall the definition of the moment map $\mu : \mathcal{X}_{\SL_2}(\underline{\mathbf{\Sigma}}_{g,n}) \to (\SL_2^{STS})^n$ induced by the embedding of marked surfaces 
  $\mathbf{m}_1^{\cup \mathcal{A}} \to  \underline{\mathbf{\Sigma}}_{g,n}$.
  Recall also from Corollary \ref{coro_remove_BA}, the exact sequence:
  $$ 0\to \mathcal{O}[\mathcal{X}_{\SL_2}(\Sigma_{g,n}, \emptyset)] \xrightarrow{\iota} \mathcal{O}[\mathcal{X}_{\SL_2}(\underline{\mathbf{\Sigma}}_{g,n})] \xrightarrow{\Delta^L_{\mathcal{A}} - (\eta \circ \epsilon)^{\otimes \mathcal{A}} \otimes \id} \left( \mathcal{O}[\SL_2^D] \right)^{n} \otimes  \mathcal{O}[\mathcal{X}_{\SL_2}(\underline{\mathbf{\Sigma}}_{g,n})], $$
 which makes appear $\mathcal{X}_{\SL_2}(\Sigma_{g,n}, \emptyset)$ as a GIT quotient:
 $$ \mathcal{X}_{\SL_2}(\Sigma_{g,n}, \emptyset) = \mathcal{X}_{\SL_2}(\underline{\mathbf{\Sigma}}_{g,n}) \sslash (\SL_2^D)^{n}.$$
 Denote by $\pi:  \mathcal{X}_{\SL_2}(\underline{\mathbf{\Sigma}}_{g,n}) \to  \mathcal{X}_{\SL_2}(\Sigma_{g,n}, \emptyset)$ the quotient map.
 
 \begin{theorem}(Fock-Rosly \cite[Proposition $2$]{FockRosly}, Guruprasad-Huebschmann-Jeffrey-Weinstein \cite[Theorem $9.1$]{GHJW_ModSpacesParBd})
 Let $C_1, \ldots, C_n$ be conjugacy classes in $\SL_2$. Then the set
 $$ \mathcal{X}^0_{\SL_2}(\Sigma_{g,n}) \cap \pi \left(\mu^{-1}(C_1, \ldots, C_n) \right) $$
 is a symplectic leaf.
 \end{theorem}
 
 
 \subsection{Comparison between stated skein algebras and relative character varieties}
 
 We now prove that the Poisson variety $\mathcal{X}(\mathbf{\Sigma})$ is isomorphic to the character variety $\mathcal{X}_{\SL_2}(\mathbf{\Sigma})$.  For unmarked surfaces, the first relation between the two objects was proved by Bullock in \cite{Bullock} using the following simple remark. For $A\in \SL_2$, the Cayley-Hamilton Theorem writes $A^2-\tr(A)A+1=0$ therefore, given $A,B \in \SL_2$, one has
 \begin{equation}\label{eq_bullock} 
 \tr (AB) + \tr (AB^{-1}) =\tr(A)\tr(B).
 \end{equation}
  Now consider a curve $\gamma\subset \Sigma \times[0,1]$, framed in the direction of $1$,  whose regular projection on $\Sigma$ has one crossing at some point $v$. Consider two element $\alpha_v, \beta_v \in \pi_1(\Sigma, v)$ so that $\alpha_v \beta_v$ and $\alpha_p\beta_p$ are the classes of the two resolutions of $\gamma$ at $v$ as depicted in Figure \ref{fig_bullock}. Now for $\rho: \pi_1(\Sigma, v) \to \SL_2$, Equation \ref{eq_bullock} gives
  $$ -\tr (\rho(\alpha_v\alpha_v))= - (-\tr(\rho(\alpha_v\beta_v^{-1}))) - (-\tr(\rho(\alpha_v)))(-\tr(\rho(\beta_v))), $$
  which recovers the skein relation in $\mathcal{S}_{-1}(\Sigma)$. By extending this idea, one gets the
  
   \begin{figure}[!h] 
\centerline{\includegraphics[width=5cm]{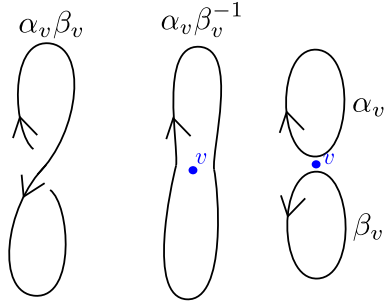} }
\caption{Curves corresponding to a Kauffman resolution of a crossing.} 
\label{fig_bullock} 
\end{figure}

  \begin{theorem}
  \begin{enumerate}
  \item (Bullock \cite{Bullock}) We have an isomorphism of algebras $\Psi: \quotient{\mathcal{S}_{-1}(\Sigma, \emptyset)}{\sqrt{0}} \cong \mathcal{O}[\mathcal{X}_{\SL_2}(\Sigma)]$ sending a curve $\gamma$ to the opposite $-\tau_{\gamma}$ of the associated trace function.
  \item (Przytycki-Sikora \cite{PS00}, Charles-March\'e \cite{ChaMa}): the algebra $\mathcal{S}_{-1}(\Sigma, \emptyset)$ is reduced, so Bullock isomorphism writes  $\Psi: \mathcal{S}_{-1}(\Sigma, \emptyset) \cong \mathcal{O}[\mathcal{X}_{\SL_2}(\Sigma)]$.
  \item (Turaev \cite{Turaev91}) $\Psi$ is Poisson, for the Poisson bracket on  $\mathcal{S}_{-1}(\Sigma, \emptyset)$  induced from $A=-\exp(\hbar/2)$.
  \item (Barett \cite{Barett}) Let $S$ be a spin structure on $\Sigma$ and $\omega_S: \mathrm{H}_1(\Sigma; \mathbb{Z}/2\mathbb{Z}) \to \mathbb{Z}/2\mathbb{Z}$ denote the associated Johnson quadratic form. Then there exists an isomorphism  $\mathcal{S}_A(\Sigma) \to \mathcal{S}_{-A}(\Sigma)$ sending a curve $\gamma$ to $(-1)^{\omega_S([\gamma])} \gamma$. In particular, we have a Poisson isomorphism $\Psi^S: \mathcal{S}_{+1}(\Sigma) \cong \mathcal{O}[\mathcal{X}_{\SL_2}(\Sigma)]$ sending a curve $\gamma$ to $(-1)^{\omega_S([\gamma])+1}\tau_{\gamma}$.
  \end{enumerate}
  \end{theorem}
  
  For unmarked surfaces, both algebras $\mathcal{S}_{+1}(\Sigma)$ and $\mathcal{S}_{-1}(\Sigma)$ are commutative and Bullock's isomorphism $\Psi: \mathcal{S}_{-1}(\Sigma, \emptyset) \cong \mathcal{O}[\mathcal{X}_{\SL_2}(\Sigma)]$ is canonical whereas the isomorphism $\Psi^S: \mathcal{S}_{+1}(\Sigma) \cong \mathcal{O}[\mathcal{X}_{\SL_2}(\Sigma)]$ depends on the choice of a spin structure.  For marked surfaces with $\mathcal{A}\neq \emptyset$, the stated skein algebra $\mathcal{S}_{-1}(\mathbf{\Sigma})$ is no longer commutative (only its graded $0$ part is) so we cannot avoid the fact that our isomorphism $\mathcal{S}_{+1}(\mathbf{\Sigma}) \cong \mathcal{O}[\mathcal{X}_{\SL_2}(\mathbf{\Sigma})]$
  will be non-canonical. An obvious reason is the fact that whereas the Poisson structure on $\mathcal{S}_{+1}(\mathbf{\Sigma})$ is canonical, the Poisson structure on $\mathcal{O}[\mathcal{X}_{\SL_2}(\mathbf{\Sigma})]$ depends on a choice of orientation of the boundary arcs. In \cite{KojuQuesneyClassicalShadows}, the authors defined a notion of \textit{relative spin structure} for marked surfaces, which generalizes spin structures for unmarked surfaces and such that $(1)$ a relative spin structure on $\mathbf{\Sigma}$ induces a canonical relative spin structure on $\mathbf{\Sigma}_{a\#b}$ and $(2)$ an orientation of the edges of a triangle $\mathbb{T}$ defines a relative spin structure on it. Therefore, given a triangulation $\Delta$ of $\mathbf{\Sigma}$ with an orientation $\mathfrak{o}_{\Delta}$ of its edges, we get a relative spin structure $S=S(\Delta, \mathfrak{o}_{\Delta})$ on $\mathbf{\Sigma}$. To a relative spin structure $S$ and a choice $\mathfrak{o}$ of orientation of the boundary arcs, we can associate a map $w_S$ which associates to any oriented embedded curve $(\alpha, \partial \alpha) \subset (\Sigma, \mathcal{A})$ a sign $w_S(\alpha)\in \{ \pm 1\}$. For unmarked surfaces, $w_S$ is related to Johnson quadratic form $\omega_S$ by $w_S(\gamma)=\omega_S([\gamma])+1$.  
  
  \begin{notations}
  For $\mathfrak{o}$ an orientation of the boundary arcs and $a\in \mathcal{A}$, we write $\mathfrak{o}(a)=+1$ of the orientation of $a$ coincides with the orientation $\mathfrak{o}_+(a)$ induced from the orientation of $\Sigma$ and write $\mathfrak{o}(a)=-1$ if they are opposite. For $\alpha$ an oriented embedded arc in $\Sigma$ with $\emptyset \neq \partial \alpha \in \mathcal{A}$, we denote by $X_{\alpha, i, j}\in \mathcal{O}[\mathcal{X}_{\SL_2}(\mathbf{\Sigma})]$ the regular function sending $[\rho]$ to the matrix coefficient $\rho(\alpha)_{i,j}$ and we denote by $U(\alpha)=\begin{pmatrix} X_{\alpha, +, +} & X_{\alpha, +, -} \\ X_{\alpha, -, +} & X_{\alpha, -, -} \end{pmatrix}$ the $2\times 2$ matrix with coefficients in $ \mathcal{O}[\mathcal{X}_{\SL_2}(\mathbf{\Sigma})]$. Recall from Notations \ref{notations_presentations} the definition of the $2\times 2$ matrix $M(\alpha)$ with coefficients in $\mathcal{S}_{+1}(\mathbf{\Sigma})$.
  \end{notations}

     \begin{theorem}\label{theorem_charvar_skein}(K.-Quesney \cite[Theorem $1.3$, Theorem $3.17$]{KojuQuesneyClassicalShadows})
     Let $S$ be a relative spin structure of $\mathbf{\Sigma}$, fix an orientation $\mathfrak{o}$ of its boundary arcs and let $w_S$ be the associated function. Then we have a Poisson isomorphism
    $$ \Psi^S : \mathcal{S}_{+1}(\mathbf{\Sigma})\xrightarrow{\cong} \mathcal{O}[\mathcal{X}_{\SL_2}(\mathbf{\Sigma})]^{\mathfrak{o}}, $$
    characterized by the facts that:
    \begin{enumerate}
    \item  
    For each  embedded arc $\alpha$, one has   
   		$$
   		\Psi^{S} (M(\alpha)) = (-1)^{w_S(\alpha)} (C^{-1})^{1-\mathfrak{o}(\alpha(0))} U(\alpha) (C^{-1})^{\mathfrak{o}(\alpha(1))}.  
   		$$
   		\item For each closed curve   $\gamma$, one has 
   		$$
   		\Psi^{S} (\gamma) = (-1)^{w_S(\gamma)} \tau_{\gamma}. 
   		$$  
		\end{enumerate}
   \end{theorem}
  
  Theorem \ref{theorem_charvar_skein} can be proved using the triangular strategy  by first noting that it can be easily proved in the cases of the bigon and the triangle using explicit computations like we performed in Section \ref{sec_STS}, and then can be proved for any triangulated marked surface $(\mathbf{\Sigma}, \Delta)$ with an orientation $\mathfrak{o}_{\Delta}$ of its edges extending the orientation $\mathfrak{o}$, using the following commutative diagram, where all morphisms are Poisson and both lines are exact:
  $$\begin{tikzcd}
  0 \arrow[r] &
  \mathcal{S}_{+1}(\mathbf{\Sigma}) \arrow[r, "\theta^{\Delta}"] \arrow[d, dotted, "\exists!", "\cong"'] &
  \otimes_{\mathbb{T}\in F(\Delta)} \mathcal{S}_{+1}(\mathbb{T}) \arrow[d, "\cong", "\otimes_{\mathbb{T}} \Psi^{\mathbb{T}, \mathfrak{o}_{\mathbb{T}}}"'] 
  \arrow[r, "\Delta^L -\sigma \circ \Delta^R"]
  &
  \left( \otimes_{e\in \mathring{\mathcal{E}}(\Delta)} \mathcal{O}[\SL_2^D] \right) \otimes \left(   \otimes_{\mathbb{T}\in F(\Delta)} \mathcal{S}_{+1}(\mathbb{T}) \right) 
  \arrow[d, "\cong", "\otimes_e \Psi^{\mathbb{B}} \otimes \otimes_{\mathbb{T}} \Psi^{(\mathbb{T}, \mathfrak{o}_{\mathbb{T}})}"']
  \\
   0 \arrow[r] &
 \mathcal{O}[\mathcal{X}_{\SL_2}(\mathbf{\Sigma})]^{\mathfrak{o}}\arrow[r, "\theta^{\Delta}"] &
  \otimes_{\mathbb{T}\in F(\Delta)}\mathcal{O}[\mathcal{X}_{\SL_2}(\mathbb{T})]^{\mathfrak{o}_{\mathbb{T}}}   \arrow[r, "\Delta^L -\sigma \circ \Delta^R"]&
  \left( \otimes_{e\in \mathring{\mathcal{E}}(\Delta)}\mathcal{O}[\mathcal{X}_{\SL_2}(\mathbb{B})] \right) \otimes \left(   \otimes_{\mathbb{T}\in F(\Delta)}\mathcal{O}[\mathcal{X}_{\SL_2}(\mathbb{T})]^{\mathfrak{o}_{\mathbb{T}}}  \right) 
\end{tikzcd}
$$
 We can alternatively prove Theorem \ref{theorem_charvar_skein} in the case where $\mathcal{A}\neq \emptyset$,  by noticing that the finite presentation of $\mathcal{S}_{A}(\mathbf{\Sigma})$ associated to presentation $\mathbb{P}$ given by Theorem \ref{theorem_presentation} reduces, once taking $A=+1$, to a finite presentation of $\mathcal{O}[\mathcal{X}_{\SL_2}(\mathbf{\Sigma}, \mathbb{P})]$ and then compare the commutators modulo $\hbar^2$ of two non-intersecting curves given by the arcs exchange formulas with the generalized Goldman bracket. We refer to \cite{KojuPresentationSSkein} a proof using this approach.

  \section{Representations of stated skein algebras}
  
  In all this section, $A$ is a root of unity of odd order $N$.
  The goal of this section is to study the finite dimensional (weight) representations $\rho: \mathcal{S}_A(\mathbf{\Sigma}) \to \End(V)$ (i.e. approach Problem \ref{problem_classification}) and to motivate some open problems around the question.
  
  \begin{definition} A finite dimensional representation $\rho: \mathcal{S}_A(\mathbf{\Sigma}) \to \End(V)$ is called
  \begin{enumerate}
  \item a \textit{weight representation} if $V$ is semi-simple as a module over the center of $\mathcal{S}_A(\mathbf{\Sigma})$; 
  \item a \textit{semi-weight representation} if $V$ is semi-simple as module over $\mathcal{S}_{+1}(\mathbf{\Sigma})$ though the Chebyshev-Frobenius morphism;
  \item a \textit{central representation} if for any $x\in \mathcal{S}_{+1}(\mathbf{\Sigma})$, $\rho\circ Ch_A(x)$ is a scalar.
  \end{enumerate}
  \end{definition}
  
  For instance any semi-weight indecomposable representation is central. For a central representation $\rho$, we have $\rho\circ Ch_A(x)=\chi(x) \id$ for all $x\in \mathcal{S}_{+1}(\mathbf{\Sigma})$ and the character $\chi$ induces a closed point of $\mathcal{X}(\mathbf{\Sigma})$ named its \textit{classical shadow}. When an isomorphism $\mathcal{X}(\mathbf{\Sigma}) \cong \mathcal{X}_{\SL_2}(\mathbf{\Sigma})$ is fixed (by choosing a relative spin structure), we can equivalently call classical shadow the point in $\mathcal{X}_{\SL_2}(\mathbf{\Sigma})$.

  Let us fix $[r] \in \mathcal{X}_{\SL_2}(\mathbf{\Sigma})$ and suppose we want to classify all semi-weight indecomposable representations with classical shadow $[r]$. Let $\mathbf{m} \in \mathrm{Specm}(\mathcal{S}_{+1}(\mathbf{\Sigma}))$ be the maximal ideal corresponding to $[r]$ and write 
  $$\mathcal{S}_A(\mathbf{\Sigma})_{[r]}:= \quotient{ \mathcal{S}_A(\mathbf{\Sigma})}{Ch_A(\mathbf{m}) \mathcal{S}_A(\mathbf{\Sigma})}.$$
  The algebra $\mathcal{S}_A(\mathbf{\Sigma})_{[r]}$ is finite dimensional. Therefore Problem \ref{problem_classification} consists in classifying the indecomposable representations of all the finite dimensional algebras $\mathcal{S}_A(\mathbf{\Sigma})_{[r]}$. We face two obvious difficulties: 
  \begin{enumerate}
\item[(i)] there is an infinite number of possible classical shadows $[r]$ so way too much algebras to study and 
\item[(ii)] the problem of classifying all indecomposable modules of a finite dimensional algebra might be out of reach.
\end{enumerate}
  
  Concerning $(i)$, we will review two important theorems that will considerably help to reduce the number of algebras to study, mainly:
  \begin{enumerate}
  \item By the Unicity representations theorem of Frohman-L\^e-Kania-Bartoszynska in \cite{FrohmanKaniaLe_UnicityRep} and a result of Brown-Gordon, there exists a Zariski open dense subset $\mathcal{FAL} \subset \mathcal{X}_{\SL_2}(\mathbf{\Sigma})$, named \textit{fully Azumaya locus} so that for all $[r] \in \mathcal{FAL}$, the algebra $\mathcal{S}_A(\mathbf{\Sigma})_{[r]}$ can be computed easily. 
  \item By Brown-Gordon's theory of Poisson orders, if $[r]$ and $[r]'$ belong to the same equivariant symplectic leaf, then $\mathcal{S}_A(\mathbf{\Sigma})_{[r]}\cong \mathcal{S}_A(\mathbf{\Sigma})_{[r]'}$.
  \end{enumerate}
  
  Concerning $(ii)$, we have the following classification due to Drozd:
  \begin{definition}
  A finite dimensional $\mathbb{C}$ algebra $\mathcal{A}$ is 
  \begin{enumerate}
  \item \textit{of finite representation type} if it has a finite number of isomorphism classes of indecomposable finite dimensional modules; 
  \item \textit{of tame representation type} if it is not of finite representation type and if for every $d\geq 0$, there exists a finite collection of $\mathcal{A}-\mathbb{C}[X]$ bimodules $M_1, \ldots, M_n$ such that any $d$-dimensional indecomposable $\mathcal{A}$ module is isomorphic to a module $M_i\otimes S$ for $S$ a simple $\mathbb{C}[X]$-module.
  \item \textit{of wild representation type} if there exists a functor $F: \mathbb{C}\left< x, y\right>-\Mod \to \mathcal{A}-\Mod$ that preserves indecomposability and isomorphism classes.
  \end{enumerate}
  \end{definition}
  
  \begin{theorem}(Drozd \cite{Drozd}: Trichotomy theorem) A finite dimensional algebra is either of finite, tame or wild representation type.\end{theorem}
  
  The word problem for finitely presented groups is undecidable and can be embedded into the problem of classifying the finite dimensional representations of $\mathbb{C}\left< x, y\right>$, therefore the problem of classifying the representations of a wild algebra is  undecidable. 
  We can reformulate a more reasonable version of Problem \ref{problem_classification} as follows:
  
  \begin{problem}\label{problem_classification2}
  \begin{enumerate}
  \item Classify the equivariant symplectic leaves of $\mathcal{X}_{\SL_2}(\mathbf{\Sigma})$ (this is Problem \ref{problem_classification_leafs}); 
  \item For each leaf, choose a representative $[r]$ and determine the representation type of $\mathcal{S}_A(\mathbf{\Sigma})_{[r]}$; 
  \item If $\mathcal{S}_A(\mathbf{\Sigma})_{[r]}$ it is not wild, classify its indecomposable finite dimensional representations.
  \end{enumerate}
  \end{problem}
  Solving Problem \ref{problem_classification2} would classify all semi-weight non-wild finite representations of $\mathcal{S}_A(\mathbf{\Sigma})$.
  As we shall see in Section \ref{sec_resolution}, this problem has been fully solved only for the bigon by Gordon and Brown in \cite{BrownGordon_OqG}, for the reduced stated skein algebra of $\mathbb{D}_1$ in \cite{KojuQGroupsBraidings} and is easy (and probably well known to the experts) for the closed torus. To attack Problem \ref{problem_classification2} for closed surfaces (for which the first step has been solved in the former section), we now present three families of representations of skein algebras arising from TQFTs, non semi-simple TQFTs and from Quantum Teichm\"uller theory. Obviously all the notions defined in the present section extend word-by-word to reduced stated skein algebras and Problem \ref{problem_classification2} has an analogous formulation.
  
  \subsection{Representations coming from modular TQFTs}
  
  In order to provide a three-dimensional interpretation of the Jones polynomial, Witten introduced in  \cite{Wi2} a Topological Quantum Field Theory (TQFT) based on the $SU(2)$ Chern-Simons functional. A purely algebraic construction, based on the algebraic input of a modular category (plus a choice of square root), were defined by Reshetikhin and Turaev in \cite{RT, Tu}.  The Witten-Reshetikhin-Turaev TQFTs provide representations of Kauffman-bracket skein algebras at roots of unity. Let us describe a simple construction of them following \cite{BHMV2}. 
  Recall that $D(g)$ represents a disc with $g$ inner pairwise non-intersecting open subdiscs removed. Consider the handlebody $H_g=D(g)\times [0,1]$ and the skein module $\mathcal{S}_A(H_g):= \mathcal{S}_A(D(g), \emptyset)$. By identifying a tubular neighborhood $\mathcal{N}$ of $\partial H_g$ with $\Sigma_g \times [0,1]$, we get an oriented embedding $i : \Sigma_g\times [0,1] \to H_g$. We  consider $\mathcal{S}_A(H_g)$ as a $\mathcal{S}_A(\Sigma_g)$-left module by defining for two framed links $L_1 \subset \Sigma_g\times [0,1]$ and $L_2\subset H_g$ the module product $[L_1]\cdot [L_2] := [i(L_1)\cup L_2]$, where $L_2$ has been first isotoped outside $\mathcal{N}$.
  Fix a genus $g$ Heegaard splitting of the sphere, i.e. two homeomorphisms  $S: \partial H_g \cong \overline{\partial H_g}$ and 
$ H_g\bigcup_{S:\partial H_g\rightarrow \overline{\partial H_g}} H_g \cong S^3$. Denote by $\varphi_1, \varphi_2 : H_g \hookrightarrow S^3$ the embeddings in the first and second factors.
\par For two framed links $L_1,L_2 \subset H_g$ , the above gluing defines a link $\varphi_1(L_1)\bigcup \varphi_2(L_2) \subset S^3$. Let $\left< \cdot \right> : \mathcal{S}_A(S^3) \cong \mathbb{C}$ be the isomorphism sending the empty link to $1$. The Hopf pairing is the Hermitian form:
$$ \left( \cdot , \cdot \right)_A^H :\mathcal{S}_A(H_g)\times\mathcal{S}_A(H_g) \rightarrow \mathbb{C}$$
defined by $$\left( L_1 , L_2 \right)_{A}^H := \left< \varphi_1(L_1)\bigcup \varphi_2(L_2) \right>$$
and the spaces $V_A(\Sigma_g)$ are the quotients:
$$ V_A(\Sigma_g) := \quotient{ \mathcal{S}_A(H_g)}{\ker \left(\left( \cdot , \cdot \right)_{A}^H \right)}$$
  The structure of left $\mathcal{S}_A(\Sigma_g)$-module of $\mathcal{S}_A(H_g)$ induces, by passing to the quotient, the so-called \textit{Witten-Reshetikhin-Turaev representation}
  $$ \rho^{WRT}_A : \mathcal{S}_A(\Sigma_g) \to \End(V_A(\Sigma_g)).$$
  In genus $1$, the spaces $V_A(\Sigma_1)$ are easy to compute. 
    \begin{notations} \begin{enumerate}
  \item The Chebyshev polynomial of second kind is the polynomial $S_n(x)$ defined by the recursive formula $S_0(x)=1, S_1(x)=x$ and $S_n(x)=xS_{n-1}(x)-S_{n-2}(x)$. It is related to Chebyshev of first kind by $T_n(x)=S_n(x)-S_{n-2}(x)=XS_{n-1}(x)-2S_{n-2}(x).$
  \item The quantum number $[n]$ is the complex $[n]:= \frac{q^n-q^{-n}}{q-q^{-1}}=S_n[q+q^{-1}]$.
  \end{enumerate}
  \end{notations}
  Since $H_1= D^2\times S^1\cong \mathbb{A}\times [0,1]$ is a thickened annulus, by Theorem \ref{theorem_basis} $\mathcal{S}_A(H_1)\cong \mathbb{C}[X]$ is generated by the class $X$ of the meridian ${0}\times S^1 \subset D^2\times S^1$ with framing towards $1$. So, writing  $v_i:=S_i(X) \in \mathcal{S}_A(H_1)$, the set $\{v_i, i \geq 0\}$ is a basis. A standard computation shows that 
  $$ (v_i, v_j)^H_A=(-1)^{i+j}[(i+1)(j+1)], $$
  from which we deduce that: $(1)$ $[v_{N-1}]= 0$ in $V_A(\Sigma_1)$, $(2)$ $[v_{n+N}]=[v_n]$, $(3)$ $[v_{N-2-n}]=[v_n]$ and $(4)$ the set $\{[v_i] | 0\leq i \leq \frac{N-3}{2} \}$ is a basis of $V_A(\Sigma_1)$. Note that $(1)$ and $(3)$ imply that 
  $$ T_N([v_1])= [v_1]S_{N-1}([v_1]) - 2S_{N-2}([v_1]) = -2.$$
  
  In higher genera, in order to find bases of $V_A(\Sigma_g)$, we need to invoke general properties of TQFTs.
  Whereas most of the literature of the algebraists community studying quantum groups is written using $A$ a root of unity of odd order, most of the literature of the topologists community (since the pioneered work of Witten) only consider roots of unity of even orders. In this survey, we made the choice of considering odd roots of unity since the representation theory is simpler but it has the inconvenient that we need to adapt the existing literature on TQFTs to the (easier) case of roots of odd order. The spaces $V_A(\Sigma_g)$ can be obtained from the Reshetikhin-Turaev construction by using a modular category $\mathcal{C}^{M}$ described as follows. Like most modular categories, $\mathcal{C}^M$ is obtained from a pre-modular category $\widetilde{\mathcal{C}}$ by the so-called purification procedure (the categorical GNS construction) described in \cite{Tu}. For  $\widetilde{\mathcal{C}}$ we can either consider the premodular category of $U_q\mathfrak{sl}_2$ modules at odd roots of unity, or consider the  Cauchy closure $\mathrm{Cauchy}(TL_A)$  of the Temperley-Lieb category. The two obtained modular categories, say $ \mathcal{C}_A^{M1},  \mathcal{C}_A^{M2}$ are equivalent through an  equivalence $F: \mathcal{C}_A^{M1}\cong \mathcal{C}_A^{M2}$ of braided categories. However, when $A$ has odd order, $F$  does not preserve the duality and we need to consider the duality of $\mathcal{C}_A^{M2}$ in order to get representations of the Kauffman-bracket skein algebras. Let us be more precise.

  In brief, $F$ sends the class $S_i$ of the simple $i+1$ dimensional $U_q\mathfrak{sl}_2$-module with highest weight $q^{i}$ to the class $[f_i]$ of the $i^{th}$ Jones-Wenzl idempotent and the quantum traces differ since $\qdim(S_i)=[i+1]$ and $\qdim([f_i])=(-1)^i [i+1]$. Since we want the skein relation 
  $\begin{tikzpicture}[baseline=-0.4ex,scale=0.5,rotate=90] 
\draw [fill=gray!45,gray!45] (-.6,-.6)  rectangle (.6,.6)   ;
\draw[line width=1.2,black] (0,0)  circle (.4)   ;
\end{tikzpicture}
= -([2])
\begin{tikzpicture}[baseline=-0.4ex,scale=0.5,rotate=90] 
\draw [fill=gray!45,gray!45] (-.6,-.6)  rectangle (.6,.6)   ;
\end{tikzpicture}
$ and not $+[2]$,  we need to use the duality of the Temperley-Lieb category
   (see \cite{SnyderTingley_HalfTwist} for details on this subtelty).  What is important for us is that the modular category  $\mathcal{C}_A^{M}$ has a finite set of isomorphism classes of simple objects $\{ S_0=\mathds{1}, S_1,  \ldots, S_{\frac{N-3}{2}}\}$. For $i,j,k \in \mathcal{I}:=\{0, \ldots, \frac{N-3}{2}\}$ the triple $(i,j,k)$ is said  \textit{compatible} if $(1)$ $i+j+k$ is even, $(2)$ they respect the triangular inequalities $i\leq j+k, j\leq i+k, k\leq i+j$ and $(3)$ $i+j+k \leq N-3$. The main properties of $\mathcal{C}_A^{M}$  are:
   
\begin{enumerate}
\item   In $K^0(\mathcal{C}_A^{M})\cong V_A(\Sigma_1)$, we have 
\begin{equation}\label{eq_tcheby}
T_N([S_1])= -2 \mathds{1}.
\end{equation}
\item We have $  \adjustbox{valign=c}{\includegraphics[width=1.9cm]{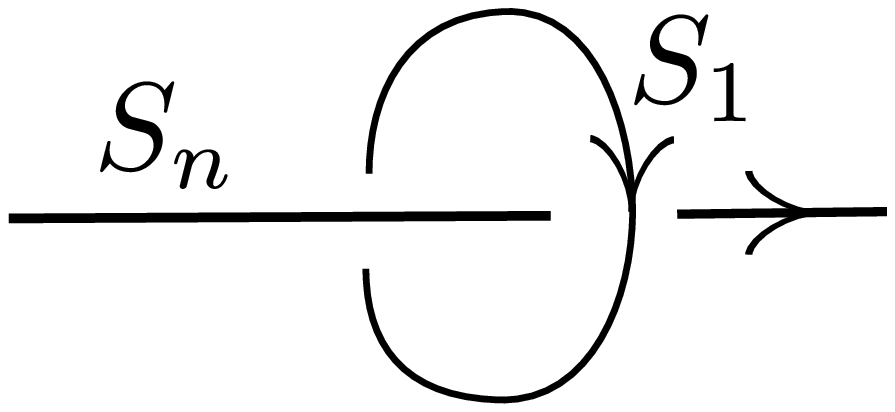}}= \lambda_n  \adjustbox{valign=c}{\includegraphics[width=1.9cm]{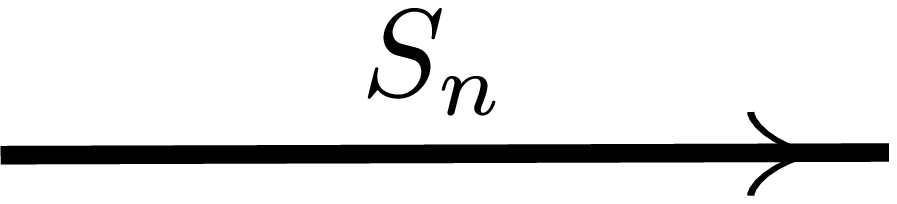}} $, where $\lambda_n=-(q^{n+1}+q^{-(n+1)})$. In particular, the $\lambda_n$ are pairwise distinct.
\item We have $$\Hom_{\mathcal{C}^M}(S_i\otimes S_j \otimes S_k, \mathds{1}) \cong \left\{ \begin{array}{ll} \mathbb{C} & \mbox{, if }(i,j,k)\mbox{ is compatible;} \\ 0 & \mbox{, else.} \end{array} \right. $$
\end{enumerate}
  
   The main property we need is the existence of conformal blocks bases for $V_A(\Sigma_g)$. Let $\mathcal{P}=\{\gamma_e\}_e$ a pants decomposition of $\Sigma_g$, i.e. a maximal set of pairwise non-intersecting non-isotopic closed curves in $\Sigma_g$ and suppose that each $i(\gamma_e)$ bounds a disc $D_e\subset H_g$. Note that there are $3g-3$ such curves if $g\geq 2$ and a single one for $g=1$. Let $\Gamma\subset H_g$ be a banded uni-trivalent graph such that each edge $e$ intersects a single disc $D_e$ exactly once. Let $col(\Gamma)$ be the set of maps $\sigma: \mathcal{E}(\Gamma) \to \mathcal{I}$ such that if $e_1,e_2,e_3$ are three edges adjacent to the same vertex, then $(\sigma(e_1), \sigma(e_2), \sigma(e_3))$ is compatible. For each $\sigma \in col(\Gamma)$, we can define an element $u_{\sigma} \in V_A(\Sigma_g)$ such that
   \begin{enumerate}
   \item for $\sigma=0$, $u_{0}$ is the class of the empty link; 
   \item the set $\{u_{\sigma}\}_{\sigma\in col(\Gamma)}$ forms a basis of $V_A(\Sigma_g)$ and 
   \item for each $\gamma_e \in \mathcal{P}$, one has $\rho_A^{WRT}(\gamma_e) u_{\sigma} = \lambda_{\sigma(e)} u_{\sigma}$.
   \end{enumerate}
   We refer to \cite{Tu} for the construction of this so-called conformal blocks basis. The following lemma is a classic in Lie group theory.
  
  \begin{lemma}\label{lemma_irrep}
  Let $\mathcal{A}$ a $\mathbb{C}$-algebra and $\rho: \mathcal{A} \to \End(V)$ a finite dimensional representation. Let $T\subset \mathcal{A}$ a commutative subalgebra and consider the decomposition
  $$ V= \oplus_{\chi: T\to \mathbb{C}} V^{(\chi)}, \quad V^{(\chi)}:=\{v\in V| \rho(t)v=\chi(t)v, \forall t\in T \}.$$
  Suppose that: 
  \begin{enumerate}
  \item there exists a character $\chi_0$ and $v_0\in V^{(\chi_0)}$ so that $v_0$ is cyclic, i.e. so that $\rho(\mathcal{A})v_0=V$ and
  \item the space $V^{(\chi_0)}$ has dimension  $1$.
\end{enumerate}
Then $\rho$ is irreducible.
\end{lemma}

\begin{proof}
Let $\theta \in \End(V)$ be such that $\theta \rho(a) = \rho(a) \theta$ for all $a\in \mathcal{A}$. Then for every character $\chi$, $\theta (V^{(\chi)}) \subset V^{(\chi)}$, so there exists $\lambda \in \mathbb{C}$ so that $\theta v_0 = \lambda v_0$.  Since $v_0$ is cyclic, for every $v\in V$, there exists $a\in \mathcal{A}$ such that $\rho(a)v_0 = v$ thus $\theta (v)=\theta \rho(a) v_0= \rho(a) \theta v_0= \lambda v$. Therefore $\theta = \lambda \id$ and we conclude using the Sch\"ur lemma.
\end{proof}
  
  The following theorem was stated for $A$ a root of unity of even order, though as we shall show, the proofs extend straightforwardly to the odd case.
  \begin{theorem}
  \begin{enumerate}
  \item (Gelca-Uribe \cite[Theorem $6.7$]{GelcaUribe_SU2}, see also \cite{BonahonWong4}) The representation $\rho^{WRT}$ is irreducible.
  \item (Bonahon-Wong \cite{BonahonWong4}) The classical shadow of $\rho^{WRT}$ is a central representation.
  \end{enumerate}
  \end{theorem}
  
  \begin{proof}
  Consider $T=\{\rho^{WRT}(\gamma_e)\}_e$ and $v_0 \in V_A(\Sigma_g)$ the class of the empty link. The fact that $v_0$ is cyclic is an immediate consequence of the facts that the map $\mathcal{S}_A(\Sigma)\to \mathcal{S}_A(H_g)$ is surjective, that $\mathcal{S}_A(H_g)v_0=\mathcal{S}_A(H_g)$ and that $V_A(\Sigma_g)$ is defined as a quotient of $\mathcal{S}_A(H_g)$. Moreover, since the coefficients $\lambda_i$ are pairwise distinct, $v_0$ is the unique vector of the conformal blocks basis such that $\rho^{WRT}(\gamma_e) v_0= -[2]v_0$ for all $e\in \mathcal{E}(\Gamma)$ so 
 the irreducibility of $\rho_A^{WRT}$ follows from Lemma \ref{lemma_irrep}. Next, using Equation \eqref{eq_tcheby}, we have
 $ \rho_A^{WRT}(T_N((\gamma)))=-2\id$ for all curve $\gamma$, which proves the second assertion.
  \end{proof}
  
  "Which central representation is the classical shadow of $\rho_A^{WRT}$" depends on the arbitrary choice of isomorphism $\mathcal{S}_{+1}(\Sigma_g) \cong \mathcal{O}[\mathcal{X}_{\SL_2}(\Sigma_g)]$, i.e. depends on the choice of a spin structure,   so the question is not relevant.
  
  \begin{corollary}\label{coro_representations_central} For every central representation $[r]\in \mathcal{X}_{\SL_2}^2(\Sigma)$, there exists a simple $\mathcal{S}_A(\Sigma)$-module with classical shadow $[r]$ and whose dimension is strictly smaller than the PI-dimension of $\mathcal{S}_A(\Sigma)$.
  \end{corollary}
  
  \begin{proof}
  Clearly the dimension $\dim (V_A(\Sigma_g)) = | col_N(\Gamma) |$ is strictly smaller than the PI-dimension of  $\mathcal{S}_A(\Sigma)$, which is $N^{3g-3}$ for $g\geq 2$ and $N$ for $g=1$ by Theorem \ref{theorem_center}. Suppose we have fix a spin structure $S$ giving an identification $\mathcal{S}_{+1}(\Sigma_g) \cong \mathcal{O}[\mathcal{X}_{\SL_2}(\Sigma_g)]$ for which $\rho_A^{WRT}$ has classical shadow $x\in \mathrm{H}^1(\Sigma_g; \mathbb{Z}/2\mathbb{Z}) \subset \mathcal{X}_{\SL_2}(\Sigma_g)$ and consider $\chi \in \mathrm{H}^1(\Sigma_g; \mathbb{Z}/2\mathbb{Z}) $. Then the representation $\chi\cdot \rho_A^{WRT} : \mathcal{S}_A(\Sigma_g) \to \End(V_A(\Sigma_g))$ defined by $\chi\cdot \rho_A^{WRT}(\gamma)= (-1)^{\chi([\gamma])} \rho_A^{WRT}(\gamma)$ has classical shadow $x+\mathcal{\chi}\in \mathrm{H}^1(\Sigma_g; \mathbb{Z}/2\mathbb{Z}) $ so every central representation is the classical shadow of such a representation.
  \end{proof}

  \subsection{Representations coming from non semi-simple TQFTs}
  
  Blanchet, Costantino, Geer and Patureau-Mirand defined in \cite{BCGPTQFT} a new family of TQFTs named non semi-simple because their algebraic input is no longer a modular category but rather a so-called $G$-modular relative category (which is non semi-simple in general) as described by De Renzi in \cite{DeRenzi_NSETQFT}. The categories giving rise to representations of the Kauffman-bracket skein algebras are the categories $\mathcal{C}_A^{NS}$ of projective weight representations of the unrolled quantum group $\overline{U}^H_q\mathfrak{sl}_2$  that we consider here at odd roots of unity as studied in \cite{CGP_unrolledQG, DeRenziGeerPatureau_TQFT_QG}. Like in the modular case, we need to change the duality from the usual conventions in order to get the correct skein algebra.
  The construction of these representations is technical and cannot be sketched in one short survey, so instead we focus on its properties. The TQFTs associated to $\mathcal{C}_A^{NS}$ defines spaces $V_A(\Sigma_g, \omega)$ associated to a closed surfaces and a cohomology class $\omega \in \mathrm{H}_1(\Sigma_g; {\mathbb{C}}/{\mathbb{Z}})$ (plus a choice of Lagrangian which is relevant for the present discussion) which induces representations of the skein algebras:
  $$\rho^{BCGP} : \mathcal{S}_A(\Sigma_g) \to \End( V_A(\Sigma, \omega) ).$$
  Note that if we had considered a root of unity of even order, as it is done in \cite{BCGPTQFT}, then the cohomology class would have been in $\mathrm{H}_1(\Sigma_g; {\mathbb{C}}/{2\mathbb{Z}})$ and we would have needed to consider instead the direct sum $\oplus_{\chi \in \mathrm{H}_1(\Sigma_g; \mathbb{Z}/2\mathbb{Z})} V_A(\Sigma_g, \omega + \chi)$ in order to have skein modules.

   Like in the modular case, the fact that the TQFTs can be extended to the circles implies that the spaces $V_A(\Sigma_g, \omega)$ admit bases associated to pants decompositions 
  as described in \cite{BCGP_Bases}. In order to describe them, we need to study $\mathcal{C}_A^{NS}$ a little bit.
Write $q=\exp(\frac{2i\pi k}{N})$ with $k$ prime to $N$ and for $\alpha\in \mathbb{C}$, set $q^{\alpha}:= \exp(\frac{2i\pi k \alpha}{N})$. The unrolled quantum groups has generators $H, K^{\pm 1}, E, F$ and   
 , by definition, in a module $V$ in $\mathcal{C}_A^{NS}$, the element $H$ acts semi-simply on $V$ and if $Hv=\lambda v$ then $Kv= q^{\lambda}v$. Moreover $E^Nv=F^Nv=0$ so every indecomposable representations are highest and lowest weight representation. Recall that a highest weight module is a module $V$ having a cyclic vector $v_0\in V$ such that $F v_0=0$ and $Hv_0=\alpha v_0$ in which case $\alpha$ is called the \textit{highest weight of} $V$. The category $\mathcal{C}_A^{NS}$ is $\mathbb{C}/\mathbb{Z}$-graded $\mathcal{C}_A^{NS}=\bigsqcup_{\overline{\alpha}\in \mathbb{C}/\mathbb{Z}} \mathcal{C}_{\overline{\alpha}}$ where $\mathcal{C}_{\overline{\alpha}}$ is the full subcategory of modules with highest weight $\beta$ such that $k\beta \equiv \alpha \pmod{\mathbb{Z}}$ (the $k$ is introduced here to simplify notations). 
  For $\alpha \in \mathbb{C}$, consider the module $V_{\alpha} \in \mathcal{C}_{\overline{\alpha}}$ with basis $\{v_0, \ldots, v_{N-1}\}$ such that 
  $$ Hv_i = \frac{\alpha}{k} + N-1 -2i \quad, Fv_i=v_{i+1}, Fv_{N-1}=0 \quad, Ev_i=[i]\left( \frac{q^i e^{-2i\pi \alpha/N} - q^{-i} e^{2i\pi \alpha/N}}{q-q^{-1}} \right) v_{i-1}, Ev_0=0.$$
  Note that $V_{\alpha}$ is simple if and only if $\alpha \in \ddot{\mathbb{C}}:= (\mathbb{C}\setminus \frac{1}{2}\mathbb{Z})\cup \frac{N}{2}\mathbb{Z}$. The Casimir element $C=EF+\frac{Kq^{-1} +K^{-1}q}{(q-q^{-1})^2}$ acts on $V_{\alpha}$ by the scalar operator with value $\frac{ e^{2i\pi \alpha/N}+e^{-2i\pi \alpha/N}}{(q-q^{-1})^2}$ from which we deduce that 
  \begin{equation}\label{eq_lambdaE}
    \adjustbox{valign=c}{\includegraphics[width=1.9cm]{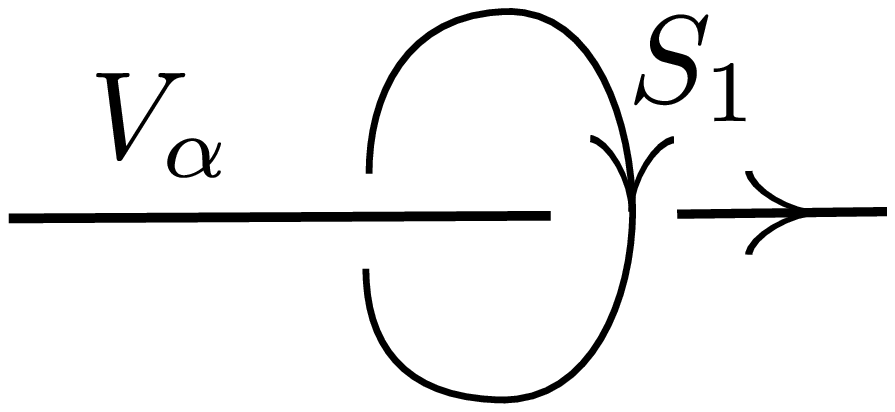}} =\lambda_{\alpha}    \adjustbox{valign=c}{\includegraphics[width=1.9cm]{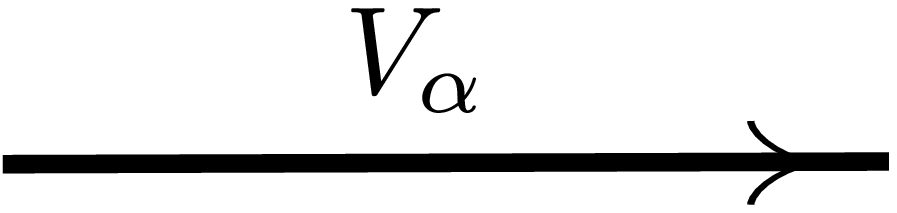}}\mbox{, where }\lambda_{\alpha}:= - (e^{2i\pi \alpha/N}+e^{-2i\pi \alpha/N}) .
  \end{equation}
  In Equation \ref{eq_lambdaE}, the minus sign is present because we changed the duality in $\mathcal{C}_A^{NS}$ as explained before. 
  Note that, as a $U_q\mathfrak{sl}_2$ module, $V_{\alpha}$ only depends on $\overline{\alpha} \in \mathbb{C}/\mathbb{Z}$. 
  For $\overline{\alpha} \notin \{ \overline{0}, \overline{\frac{1}{2}} \}$, the category $\mathcal{C}_{\overline{\alpha}}$ is semi-simple, with $\{ V_{\alpha'}, \alpha' \in \overline{\alpha}\}$ as set of isomorphisms classes of simple objects.

  Fix $\omega \in \mathrm{H}^1(\Sigma_g; \mathbb{C}/\mathbb{Z})$ a class such that $\omega \notin \mathrm{H}^1(\Sigma_g; \frac{1}{2}\mathbb{Z}/\mathbb{Z})$ and $\{\gamma_e\}_e$ a pants decomposition associated to $\Gamma \subset H_g$ a banded trivalent graph as in the previous section. Fix $r: \mathbb{C} \to \mathbb{C}/\mathbb{Z}$ an arbitrary retraction of the quotient map. Write $col_N(\Gamma) = \{ \sigma : \mathcal{E}(\Gamma) \to \{ 0, \ldots, N-1\} \}$. Like in the modular case, to $\sigma \in col_N(\Gamma)$, one can associate a vector $u_{\sigma} \in V(\Sigma_g, \omega)$, by coloring an edge $e$ by the module $V_{r(\omega(\gamma_e))+\sigma(e)}$ in such a way that $(1)$ the set $\{u_{\sigma}, \sigma \in col_N(\Gamma)\}$ is a basis of $V_A(\Sigma_g, \omega)$ and $(2)$ one has $\rho^{BCGP}(\gamma_e) u_{\sigma}= \lambda_{\omega(e)+\sigma(e)} u_{\sigma}$ so 
  $$ T_N(\rho^{BCGP}(\gamma_e))= - (\exp(2i\pi \omega([\gamma])) + \exp(-2i\pi \omega([\gamma]))\id.$$
  Since any curve in $\Sigma_g$ belongs to a pants decomposition, this formula is true for any curve.
  Fix a spin structure $S$ on $\Sigma_g$ with Johnson quadratic form $w_S$ and define  $r_{\omega}:\pi_1(\Sigma_g, v) \to \SL_2$ the diagonal representation defined by 
  $$ r_{\omega}(\gamma):= (-1)^{w_S([\gamma])} \begin{pmatrix} \exp(2i\pi \omega([\gamma]))& 0\\ 0 &\exp(-2i\pi \omega([\gamma])) \end{pmatrix}.$$
  We deduce from the preceding discussion that:
  \begin{enumerate}
  \item The dimension $dim(V_A(\Sigma_g, \omega))= |\mathcal{E}(\Gamma)|^N= \left\{ \begin{array}{ll} N^{3g-3} & \mbox{, if }g\geq 2; \\ N & \mbox{, if }g=1.\end{array} \right.$ is equal to the PI-dimension of $\mathcal{S}_A(\Sigma_g)$ and
  \item The representation $\rho^{BCGP}$ is central with classical shadow $[r_{\omega}] \in \mathcal{X}_{\SL_2}^1(\Sigma_g)$.
  \end{enumerate}
  
  \begin{question}\label{question_BCGP} For $\omega  \in \mathrm{H}^1(\Sigma_g; \mathbb{C}/\mathbb{Z}) \setminus  \mathrm{H}^1(\Sigma_g; \frac{1}{2}\mathbb{Z}/\mathbb{Z})$, is the representation $\rho^{BCGP}$ irreducible ?
  \end{question}
  
  Like in the modular case, it is natural to try to use Lemma \ref{lemma_irrep} to solve Question \ref{question_BCGP} using the conformal block basis $\{u_{\sigma}, \sigma\in col_N(\Gamma)\}$. Since $\lambda_{\sigma(e) +i} \neq \lambda_{\sigma(e)+j}$ for $i\neq j$ in $\{0, \ldots, N-1\}$, one of the hypothesis of Lemma \ref{lemma_irrep} is clearly satisfied. Unlike the modular case, what is absolutely not clear is whether one of the vector $u_{\sigma}$ is cyclic or not. As we shall see in Proposition \ref{prop_equivalences}, $\rho^{BCGP}$ is irreducible for one such class $\omega$ if and only if it is irreducible for all of them so one can use the fact that the representation depends analytically on $\omega$ to solve Question \ref{question_BCGP}. When $\omega \in  \mathrm{H}^1(\Sigma_g; \frac{1}{2}\mathbb{Z}/\mathbb{Z})$, the representations $\rho^{BCGP}$ are much harder to study and relatively unexplored so all natural questions remain open. 
  
  \subsection{Representations coming from quantum tori}
  
  Thanks to the quantum trace, Bonahon and Wong defined in \cite{BonahonWong2} a large family of representations of skein algebras as follows. 
  Let $(\mathbf{\Sigma}, \Delta)$ be a triangulated marked surface and consider $r: \mathcal{Z}_{\omega}(\mathbf{\Sigma}, \Delta)\to \End(V)$ an irreducible representation. Then the composition
  $$ \rho^{BW} : \mathcal{S}^{red}_A(\mathbf{\Sigma}) \xrightarrow{\tr^{\Delta}} \mathcal{Z}_{\omega}(\mathbf{\Sigma}, \Delta)\xrightarrow{r} \End(V), $$
  is called a \textit{quantum Teichm\"uller representation}. Clearly, since the quantum torus $\mathcal{Z}_{\omega}(\mathbf{\Sigma}, \Delta)$ is Azumaya of constant rank, 
  such a representation only depends up to isomorphism on the induced character on the center $\mathbf{Z}_{\mathbf{\Sigma}, \Delta}$ of  $\mathcal{Z}_{\omega}(\mathbf{\Sigma}, \Delta)$. Moreover, by Theorem \ref{theorem_center}, each quantum Teichm\"uller representation has dimension the PI-dimension of $ \mathcal{S}^{red}_A(\mathbf{\Sigma}) $.
    
   Write $\mathcal{Y}(\mathbf{\Sigma},\Delta)=\mathrm{Specm}(\mathbf{Z}_{+1}(\mathbf{\Sigma}, \Delta))$ and $\mathcal{NA}: \mathcal{Y}(\mathbf{\Sigma}, \Delta) \to \mathcal{X}^{red}(\mathbf{\Sigma})$ be the regular map defined by the quantum trace $\tr_{+1}$ at $A=+1$. Note that, since the quantum trace is injective, $\mathcal{NA}$ is dominant and since $\mathcal{X}^{red}(\mathbf{\Sigma})$ is irreducible (because $\mathcal{S}_{+1}^{red}(\mathbf{\Sigma})$ is a domain), then the image of $\mathcal{NA}$ contains a Zariski open dense subset, so a generic point of  $\mathcal{X}^{red}(\mathbf{\Sigma})$ is the classical shadow of a quantum Teichm\"uller representation.
 
  \begin{question}\label{question_QTRep}
  \begin{enumerate}
  \item What is the image of $\mathcal{NA}$ ? In other terms, which elements of $\mathcal{X}^{red}(\mathbf{\Sigma})$ are the classical shadows of  quantum Teichm\"uller representations ?
    \item If two quantum Teichm\"uller representations induce the same character over the center of $\mathcal{S}_A(\mathbf{\Sigma})$, are they isomorphic ?
  \item Given a point $y\in \mathcal{Y}(\mathbf{\Sigma},\Delta)$ and $V(y)$ a corresponding quantum Teichm\"uller representation, is the $\mathcal{S}_A(\mathbf{\Sigma})$ module $V(y)$ simple ? indecomposable ? projective ?
  \end{enumerate}
  \end{question}
  
  Question \ref{question_QTRep} was solved in \cite{KojuQGroupsBraidings} in the case where $\mathbf{\Sigma}=\mathbb{D}_1$. The trick is that $\mathcal{S}_A^{red}(\mathbb{D}_1)$ is isomorphic to the simply laced quantum group $\widetilde{U}_q\mathfrak{gl}_2$, closely related to $U_q\mathfrak{sl}_2$ and for which the representation theory can be studied "by hand". Note that $\mathcal{X}^{red}(\mathbb{D}_1) \cong B^+\times B^-$.
   Then using the fact that $\mathbb{D}_n$ is obtained by gluing $n$ copies of $\mathbb{D}_1$ together, one obtains the
  
  \begin{theorem}(K. \cite{KojuQGroupsBraidings})\label{theorem_D1} Suppose that $\mathbf{\Sigma}=\mathbb{D}_n$ for $n\geq 1$. 
  \begin{enumerate}
  \item For $n\geq 2$, the map $\mathcal{NA}$ is surjective. For $n=1$, an element  $(g_+, g_-) \in B^+ \times B^-$ is not in the image of $\mathcal{NA}$ if and only if either both $g_-$ and $g_+$ are diagonal and not scalar or if none of them is diagonal and both have diagonal elements equal to $\pm 1$.
  \item For any $n\geq 2$, two quantum Teichm\"uller representations of $\mathcal{S}_A^{red}(\mathbb{D}_n)$ which induce the same  character over the center of $\mathcal{S}_A^{red}(\mathbb{D}_n)$ are canonically isomorphic.
  \item For $n=1$, any quantum Teichm\"uller representation of $\mathcal{S}_A^{red}(\mathbb{D}_1)$ is indecomposable and is a projective object in the category of weight modules. Moreover a quantum Teichm\"uller representation of  $\mathcal{S}_A^{red}(\mathbb{D}_1)$ with classical shadow $(g_+,g_-)$ is simple 
  if and only if $g_+g_-^{-1}\neq \pm \mathds{1}_2$ or if $g_+g_-^{-1}=\pm \mathds{1}_2$  and the inner puncture Casimir element $\gamma_p$ is sent to $\pm 2 \id$.
  \end{enumerate}
  \end{theorem}
  
  Thanks to Theorem \ref{theorem_D1}, one can use quantum Teichm\"uller representation to define projective representations of the mapping class group of $\mathbb{D}_n$, i.e. of the braid group as studied in \cite{KojuQGroupsBraidings}. If one could generalize this theorem to other marked surfaces, the same procedure would permit to obtain finite dimensional projective representations of mapping class groups and Torelli groups of more complicated surfaces. This construction might serve as a motivation for solving Question \ref{question_QTRep}.
  
  \vspace{2mm}
  \par The above procedure only provides representations of the stated skein algebras of triangulable surfaces.
  In order to obtain representations of the skein algebras of closed surfaces from quantum Teichm\"uller theory, Bonahon and Wong developed in \cite{BonahonWong3} the following strategy. Let $\mathbf{\Sigma}=(\Sigma_{g,0}, \emptyset)$ a closed surface and $\mathbf{\Sigma}'= (\Sigma_{g,n}, \emptyset)$ the same surface with $n\geq 1$ open discs removed.  For each punctures $p_1, \ldots, p_n$ of $\mathbf{\Sigma}'$ one has a off-puncture ideal $\mathcal{I}_{p_i}$ and writing $\mathcal{I}^{tot}:= \sum_{i=1}^n \mathcal{I}_{p_i}$ the total off-puncture ideal, one gets an exact sequence
  $$ 0 \to \mathcal{I}^{tot} \to \mathcal{S}_A(\mathbf{\Sigma}')\xrightarrow{i_*}  \mathcal{S}_A(\mathbf{\Sigma}) \to 0.$$
  This suggests a natural strategy: choose a triangulation $\Delta$ of $\mathbf{\Sigma}'$ and a quantum Teichm\"uller representation $\rho': \mathcal{S}_A(\mathbf{\Sigma}') \xrightarrow{tr^{\Delta}} \mathcal{Z}_{\omega}(\mathbf{\Sigma}', \Delta) \xrightarrow{r} \End(V)$ and consider the submodule
  $$V^0 := \{ v\in V| \quad x\cdot v = 0,  \forall x \in \mathcal{I}^{tot} \}.$$
  Then, by the exactness of the above sequence,  $\rho'$ induces a representation $\rho: \mathcal{S}_A(\mathbf{\Sigma}) \to \End(V^0)$. The problem is, since $\mathcal{I}^{tot}$ is infinitely generated, it seems extremely difficult to show that $V^0\neq 0$. Bonahon and Wong proved that this is the case, provided that: 
  \begin{enumerate}
\item  the triangulation $\Delta$ is combinatoric, i.e. every edges of $\Delta$ has two distinct endpoints and 
\item the irreducible representation $r:  \mathcal{Z}_{\omega}(\mathbf{\Sigma}', \Delta) \to \End(V)$ sends every Casimir element $H_{p_i}$ to $r(H_{p_i})=-q^{-1} \id_V$.
\end{enumerate}

\begin{definition}
We call quantum Teichm\"uller representation a  representation $\rho^{BW}: \mathcal{S}_A(\mathbf{\Sigma}) \to \End(V^0)$ obtained from an irreducible representation  $r:  \mathcal{Z}_{\omega}(\mathbf{\Sigma}', \Delta) \to \End(V)$ satisfying the above two conditions.
\end{definition}

\begin{theorem}\label{theorem_QTClosed}(Bonahon-Wong \cite{BonahonWong3}) Let $\mathbf{\Sigma}=(\Sigma_{g}, \emptyset)$ be a closed connected marked surface of genus $g\geq 1$.
\begin{enumerate}
\item Every element of $\mathcal{X}_{\SL_2}(\Sigma_g)$ is the classical shadow of a quantum Teichm\"uller representation.
\item If $g=1$, every quantum Teichm\"uller representation has dimension $N$.
\item If $g\geq 2$, every quantum Teichm\"uller representation with non central classical shadow has dimension $N^{3g-3}$ (the PI-dimension of $ \mathcal{S}_A(\mathbf{\Sigma})$).
If the classical shadow is central then dimension is $\leq N^{3g-3}$.
\end{enumerate}
\end{theorem}

These representations have been poorly explored and most natural questions remains open (for $g\geq 2$). The second and third items of Questions \ref{question_QTRep} are also open for a closed surface of genus $g\geq 2$. We can add the following ones:

\begin{question} Suppose $g\geq 2$ and $\mathbf{\Sigma}=(\Sigma_{g}, \emptyset)$.
\begin{enumerate}
\item What is the dimension of quantum Teichm\"uller representation with central classical shadow ?
  \item If $\rho^{BW}$ is a quantum Teichm\"uller representation of $ \mathcal{S}_A(\mathbf{\Sigma})$ with diagonal classical shadow, is it isomorphic to the representation $\rho^{BCGP}$ with same shadow coming from non semi-simple TQFTs ?
\end{enumerate}
\end{question}

  The trick to compute the dimension of such a quantum Teichm\"uller representation is that, whereas the off-puncture ideal $\mathcal{I}_p$ is infinitely generated, its image 
  $$J_p:=  \mathcal{Z}_{\omega}(\mathbf{\Sigma}', \Delta) (\tr^{\Delta} (\mathcal{I}_p) )  \mathcal{Z}_{\omega}(\mathbf{\Sigma}', \Delta), $$
   is extremely simple: it is generated by only two elements.  One of them is the element $H_p +q^{-1}$, where $H_p$ is the Casimir element defined in Definition \ref{def_central_elements}. The second is defined as follows. Let $e\in \mathcal{E}(\Delta)$ be an edge with exactly one endpoint adjacent to $p$. Replace $e$ by two parallel copies $e'$ and $e''$ of the same arc bounding a bigon and let $\mathbf{\Sigma}'(e)$ be the marked surface obtained by removing the interior of this bigon like in Figure \ref{fig_off_diagonal}; so $\mathbf{\Sigma}'$ is obtained from $\mathbf{\Sigma}'(e)$ by gluing $e'$ and $e''$. Since $p$ is a boundary puncture of $\mathbf{\Sigma}'(e)$, we can consider its associated bad arc $\alpha(p)_{-+}$ and denote by $\alpha(p)_{+-}$ the same stated arc with opposite states. The triangulation $\Delta$ of $\mathbf{\Sigma}'$ induces a triangulation $\Delta(e)$ of $\mathbf{\Sigma}'(e)$ and we have a commutative diagram: 
   $$ \begin{tikzcd}
   \mathcal{S}_A(\mathbf{\Sigma}') \arrow[r, hook, "\theta_{e'\#e''}"] \arrow[d, hook, "\tr^{\Delta}"] & 
   \mathcal{S}_A(\mathbf{\Sigma}'(e)) \arrow[d, hook, "\tr^{\Delta(e)}"] \\
   \mathcal{Z}_{\omega}(\mathbf{\Sigma}', \Delta) \arrow[r, hook, "\theta_{e'\#e''}"] & 
   \mathcal{Z}_{\omega}(\mathbf{\Sigma}'(e), \Delta(e))
   \end{tikzcd}
   $$
     
   \begin{figure}[!h] 
\centerline{\includegraphics[width=7cm]{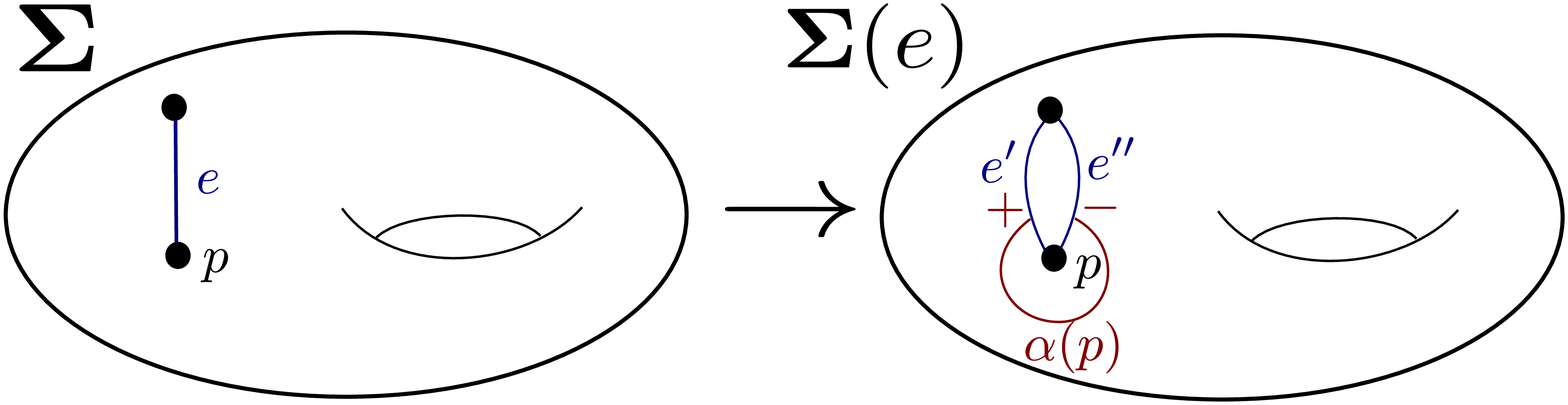} }
\caption{The skein element $\alpha(p)_{+-}$ defining the off-diagonal element $Q_p$.} 
\label{fig_off_diagonal} 
\end{figure} 

The element $\tr^{\Delta(e)} (\alpha(p)_{+-})$ belongs to the image of $\theta_{e'\#e''}$, i.e. there exists an element $Q_p \in \mathcal{Z}_{\omega}(\mathbf{\Sigma}', \Delta)$ such that $\theta_{e'\#e''}(Q_q) = \tr^{\Delta(e)} (\alpha(p)_{+-})$. The element $Q_p$ only depends on the choice of the adjacent edge $e$ up to multiplication by an invertible element of $\mathcal{Z}_{\omega}(\mathbf{\Sigma}', \Delta)$ and is called the \textit{off diagonal element}. Note that in matrix notations we have
$$ \tr^{\Delta(e)} \begin{pmatrix} \alpha(p)_{++} & \alpha(p)_{+-} \\ \alpha(p)_{-+} & \alpha(p)_{--} \end{pmatrix} - \begin{pmatrix} -q^{-1} & 0 \\ 0 & -q \end{pmatrix} = \theta_{e'\#e''} \begin{pmatrix} H_p + q^{-1} & Q_p \\ 0 & H_p^{-1}+q \end{pmatrix}, $$

which explains the name "off diagonal" and permits to prove the

\begin{proposition}(Bonahon-Wong \cite{BonahonWong3})
 If none of the edges of $\Delta$ adjacent to $p$ is a loop edge, then the ideal $J_p$ is generated by $H_p+q^{-1}$ and $Q_p$.
 \end{proposition}
 
 This proposition explains why we need to consider combinatorial triangulations and impose $r(H_p)=-q^{-1}\id_V$. Thanks to this proposition, we see that $V^0= \cap_{p} \ker (\rho'(Q_p))$ and the rest of the proof consists in evaluating the dimension of $\ker (\rho'(Q_p))$.
  
  \subsection{The unicity representations theorem and Azumaya locus}
  
  Let $R$ be a prime complex algebra which is finitely generated over its center $Z$ and let $D$ its PI-dimension. Write $\mathcal{X}:= \mathrm{Specm}(Z)$ and for $x\in \mathcal{X}$ corresponding to a maximal ideal $\mathfrak{m}_x \subset Z$, consider the finite dimensional algebra
  $$R_x:= \quotient{R}{\mathfrak{m}_x R}.$$
  
  \begin{definition} The \textit{Azumaya locus} of $R$ is the subset 
  $$ \mathcal{AL}(R):= \{ x \in \mathcal{X} | R_x \cong Mat_D(\mathbb{C}) \}, $$
  where $Mat_D(\mathbb{C})$ is the algebra of $D\times D$ matrices.
  \end{definition}
  
  Note that any irreducible representation $\rho: R\to \End(V)$ sends central elements to scalar operators so admits a (unique) $x\in \mathcal{X}$ such that $\mathfrak{m}_x\mathcal{A}\subset \ker(\rho)$. 
  
  \begin{remark}\label{remark_AzumayaLocus}
 An irreducible representation  $\rho: \mathcal{A}\to \End(V)$ sends central elements to scalar operators so induces a point  $x\in \mathcal{X}$. If $x\in \mathcal{AL}$, then $V$ is $D$ dimensional. By a theorem of Posner, if $x$ does not belong to the Azumaya locus, then $R_x$ has PI-dimension strictly smaller than $R$, therefore any irreducible representation $\rho: R\to \End(V)$ inducing $x$ has dimension $\dim(V)<D$. So the Azumaya locus admits the following alternative definition:
  $$\mathcal{AL}(R)=\{ x \in \mathcal{X} | x \mbox{ is induced by an irrep of maximal dimension }D\}.$$
  Therefore, if $\rho:R\to \End(V)$ is a $D$-dimensional central representation inducing $x\in \mathcal{X}$, then $\rho$ is irreducible if and only if $x\in \mathcal{AL}$.
  \end{remark}
  
  \begin{definition}
  Let $R$ be a prime, finitely generated $\mathbb{C}$ algebra which has finite rank $D^2$ over its center $Z$. So if $S:=Z\setminus \{0\}$, then by Theorem \ref{theorem_Posner}, $Q= RS^{-1}$ is a central simple algebra with center $F:= ZS^{-1}$ and $F$ admits a field extension $\overline{F}$ such that $R\otimes_Z\overline{F} \cong M_D(\overline{F})$. 
  \begin{enumerate}
  \item The \textit{reduced trace} is the composition
  $$tr : R \hookrightarrow R\otimes_Z \overline{F} \cong M_D(F) \xrightarrow{tr_D} F, $$
  where $tr_D ((a_{ij})_{i,j}):= \frac{1}{D}\sum_i a_{ii}$ (so that $tr(1)=1$).
By \cite[Theorem $10.1$]{Reiner03}, $tr$ takes values in $Z$.  
  \item 
  The \textit{discriminant ideal} of $R$ is the ideal $\mathcal{D} \subset Z$ generated by the elements 
  $$  \det\left( tr(x_ix_j) \right) \in Z \quad, (x_1, \ldots, x_{D^2}) \in R^{D^2}.$$
 
  \end{enumerate}
  \end{definition}

  \begin{theorem}\label{theorem_UnicityRep}(Brown-Milen \cite[Theorem $1.2$]{Brown_AL_discriminant}) If $R$ is $(1)$ prime, $(2)$  finitely generated as an algebra and $(3)$ finitely generated as a module over its center $Z$, then 
  $$ \mathcal{AL}(R)= \mathrm{Specm}(Z) \setminus V(\mathcal{D}).$$
   \end{theorem}
  
In particular, Theorem \ref{theorem_UnicityRep} implies that the Azumaya locus is an open dense set. This fact seems to be well-known to the experts since a long time (see e.g. \cite{BrownGoodearl}) though the author was not able to find to whom attribute this folklore result. The reduced trace for skein algebra of unmarked surfaces was computed in \cite{FrohmanKaniaLe_DimSkein} though the discriminant and the associated Azumaya loci are still unknown (in genus $g\geq 2$).
  
  \begin{corollary}\label{coro_unicity}(\cite{FrohmanKaniaLe_UnicityRep} for unmarked surfaces, \cite[Theorem $1.2$]{KojuAzumayaSkein} for marked surfaces) 
  \begin{enumerate}
  \item 
  The Azumaya loci of $\mathcal{S}_A(\mathbf{\Sigma})$ and $\mathcal{S}^{red}_A(\mathbf{\Sigma})$ are open dense subsets.
  \item There exists an open dense subset $\mathcal{O}\subset \mathcal{X}^{red}(\mathbf{\Sigma})$ so that every irreducible $\mathcal{S}^{red}_A(\mathbf{\Sigma})$ representation with induced character in $\mathcal{O}$  
  is isomorphic to a quantum Teichm\"uller representation and every quantum Teichm\"uller representation with character in $\mathcal{O}$ is irreducible.
  \end{enumerate}
  \end{corollary}
  
   Said differently, a generic irreducible representation is of quantum Teichm\"uller type and a generic quantum Teichm\"uller representation is irreducible. Corollary \ref{coro_unicity} was conjectured by Bonahon and Wong in \cite{BonahonWong2}.

  \begin{proof}
  The first assertion is proved by putting Theorem \ref{theorem_UnicityRep} and Theorem \ref{theorem_center} together. The second follows from the fact, proved in Theorems \ref{theorem_center},  \ref{theorem_centerCF} and \ref{theorem_QTClosed} that the PI-dimension of $\mathcal{S}^{red}_A(\mathbf{\Sigma})$ is equal to the dimension of its quantum Teichm\"uller representations. Since the quantum trace sends injectivly the center $\mathcal{Z}$ of $\mathcal{S}^{red}_A(\mathbf{\Sigma})$  into the center $\mathcal{Z}'$ of $\mathcal{Z}_{\omega}(\mathbf{\Sigma}, \Delta)$, it induces a dominant map $f: \mathrm{Specm}(\mathcal{Z}') \to \mathrm{Specm}(\mathcal{Z})=\mathcal{X}(\mathbf{\Sigma})$ and since $\mathcal{X}(\mathbf{\Sigma})$ is irreducible (by Theorem \ref{theorem_domain}), the image of $f$ contains a Zariski open dense subset $U\subset \mathcal{X}(\mathbf{\Sigma})$ and we can choose $\mathcal{O}:= U \cap \mathcal{AL}$.
\end{proof}
  
  \begin{notations}
  Let $\mathcal{Z}$ and $\mathcal{Z}^{red}$ denote the centers of $\mathcal{S}_A(\mathbf{\Sigma})$ and $\mathcal{S}^{red}_A(\mathbf{\Sigma})$ respectively and set $\widehat{\mathcal{X}}(\mathbf{\Sigma}):= \mathrm{Specm}(\mathcal{Z})$ and $\widehat{\mathcal{X}}^{red}(\mathbf{\Sigma}):= \mathrm{Specm}(\mathcal{Z}^{red})$. We denote by $\pi: \widehat{\mathcal{X}}(\mathbf{\Sigma})\to {\mathcal{X}}(\mathbf{\Sigma})$ and  $\pi': \widehat{\mathcal{X}}^{red}(\mathbf{\Sigma})\to {\mathcal{X}}^{red}(\mathbf{\Sigma})$
    the branched covering induced by the Chebyshev-Frobenius morphisms.
    \end{notations}
    
    \begin{definition} The \textit{fully Azumaya locus}  is the subset $\mathcal{FAL}(\mathbf{\Sigma})\subset \mathcal{X}(\mathbf{\Sigma})$ of elements $x$ such that all elements of $\pi^{-1}(x)$ belong to the Azumaya locus of $\mathcal{S}_A(\mathbf{\Sigma})$. The fully Azumaya locus $\mathcal{FAL}^{red}(\mathbf{\Sigma})\subset \mathcal{X}^{red}(\mathbf{\Sigma})$ is defined similarly. 
    \end{definition}

The fully Azumaya locus was introduced by Brown and Gordon in \cite{BrownGordon_ramificationcenters} in the case of the bigon. Since the projection map $\pi: \widehat{\mathcal{X}}(\mathbf{\Sigma}) \to \mathcal{X}(\mathbf{\Sigma})$ is finite, and since finite morphisms send closed sets to closed sets (\cite[Ex.2.35(b)]{Hart}), Corollary \ref{coro_unicity} implies that the fully Azumaya loci are open dense subsets. The key theorem to work over the fully Azumaya locus is the:

\begin{theorem}(Brown-Gordon \cite[Corollary $2.7$]{BrownGordon_ramificationcenters})
Let $\mathcal{A}$ be an affine prime $\mathbb{C}$-algebra finitely generated over its center $\mathcal{Z}$ and denote by $D$ its PI-dimension. Let $R\subset \mathcal{Z}$ be a subalgebra such that $\mathcal{Z}$ if finitely generated as a $R$-module. Let $M\in \mathcal{AL}(\mathcal{A})$ and $\mathfrak{m}:=M\cap R$. Then
$$ \quotient{\mathcal{A}}{\mathfrak{m} \mathcal{A}} \cong Mat_D \left(\quotient{\mathcal{Z}}{\mathfrak{m}\mathcal{Z}}\right).$$
\end{theorem}

 Recall for $x\in \mathcal{X}(\mathbf{\Sigma})$ the notation $\mathcal{S}_A(\mathbf{\Sigma})_{x}:= \quotient{ \mathcal{S}_A(\mathbf{\Sigma})}{Ch_A(\mathfrak{m}_x) \mathcal{S}_A(\mathbf{\Sigma})}$. Let $\mathcal{Z}$ be the center of $\mathcal{S}_A(\mathbf{\Sigma})$ and write 
$$ Z(x):= \quotient{\mathcal{Z}}{Ch_A(\mathfrak{m}_x)\mathcal{Z}}.$$
For $x\in \mathcal{X}^{red}(\mathbf{\Sigma})$, we define $\mathcal{S}^{red}_A(\mathbf{\Sigma})_{x}$ and $Z^{red}(x)$ similarly.

\begin{corollary}\label{theorem_BG_Ramifications}
If $x\in \mathcal{FAL}(\mathbf{\Sigma})$ and $D:=PI-Dim(\mathcal{S}_A(\mathbf{\Sigma}))$, then $\mathcal{S}_A(\mathbf{\Sigma})_{x}\cong Mat_D(Z(x))$.
Similarly, if $x\in \mathcal{FAL}^{red}(\mathbf{\Sigma})$ and $D':=PI-Dim(\mathcal{S}^{red}_A(\mathbf{\Sigma}))$, then $\mathcal{S}^{red}_A(\mathbf{\Sigma})_{x}\cong Mat_{D'}(Z^{red}(x))$.
\end{corollary}

Note that, since the algebras $Z(x)$ are easy to compute explicitly using Theorem \ref{theorem_center}, by putting Corollary \ref{coro_unicity} and Theorem \ref{theorem_BG_Ramifications} together, we have solved Problem \ref{problem_classification2} generically, i.e. for every classical shadows lying in the fully Azumaya locus. 
We need thus to solve the:

\begin{problem}\label{problem_AL} Compute the (fully) Azumaya loci of $\mathcal{S}_A(\mathbf{\Sigma})$ and $\mathcal{S}_A^{red}(\mathbf{\Sigma})$.
\end{problem}

Let us first state a first easy result towards the resolution of this problem.
Recall that $\mathcal{X}^{red}(\mathbf{\Sigma})\subset \mathcal{X}(\mathbf{\Sigma})$ is a bad arcs leaf.

\begin{lemma}\label{lemma_reduced_leaf}
For a connected marked surface $\mathbf{\Sigma}$ with non-trivial marking, then
the leaf $\mathcal{X}^{red}(\mathbf{\Sigma})\subset \mathcal{X}(\mathbf{\Sigma})$ does not intersect the fully Azumaya locus of $\mathcal{S}_A(\mathbf{\Sigma})$.
\end{lemma}

\begin{proof}
This follows from the facts, proved in Theorem \ref{theorem_center}, that $PI-Dim(\mathcal{S}^{red}_A(\mathbf{\Sigma}))< PI-Dim(\mathcal{S}_A(\mathbf{\Sigma}))$. 
\end{proof}

In order to state Takenov's results in \cite{Takenov_Azumaya}, let us introduce some notations. For the unmarked once-punctured torus $(\Sigma_{1,1}, \emptyset)$, consider the three generators $\alpha_1, \alpha_2, \alpha_3$ of Figure \ref{fig_Takenov} and define the three finite subsets of $\mathcal{X}(\Sigma_{1,1})$:
\begin{align*}
& C_1:=\{x| \chi_x(\alpha_i) \in \{\pm 2\},  \forall i\},  \quad C_2:= \{x| \chi_x(\alpha_i) =0, \forall i\}, \\ 
& C_3:= \left\{x | \exists i \exists \varepsilon_1, \varepsilon_2 = \pm \mbox{ s.t } (\chi_x(\alpha_i), \chi_x(\alpha_{i+1}), \chi_x(\alpha_{i+2}))=(\varepsilon_1 2, \varepsilon_2 \frac{2}{\sqrt{3}} i, \varepsilon_1\varepsilon_2 \frac{2}{\sqrt{3}} ) \right\}.
 \end{align*}
Set
$$\mathcal{O}_{1,1}:= \mathcal{X}(\Sigma_{1,1}) \setminus \left( C_1 \cup C_2 \cup C_3 \right).$$
For the unmarked four-holed sphere $(\Sigma_{0,4}, \emptyset)$ with punctures $p_1,p_2, p_3, p_4$ consider the curve $\alpha$ in Figure \ref{fig_Takenov}. For $x\in \mathcal{X}(\Sigma_{0,4}, \emptyset)$, write $p_i=\chi_x(\gamma_{p_i})$. Let $\mathcal{O}_{0,4} \subset  \mathcal{X}(\Sigma_{0,4}, \emptyset)$ denote the open subset of $x$ such that $\chi_x(\alpha)\neq \pm 2$ and $\chi_x(\alpha)\neq  T_N(r)$ for every $r$ solution of the equation
$$ (r^2 +p_0p_3 r + p_0^2 + p_3^2 -4)(r^2 +p_1p_2r +p_1^2+ p_2^2 -4)=0.$$
\begin{figure}[!h] 
\centerline{\includegraphics[width=7cm]{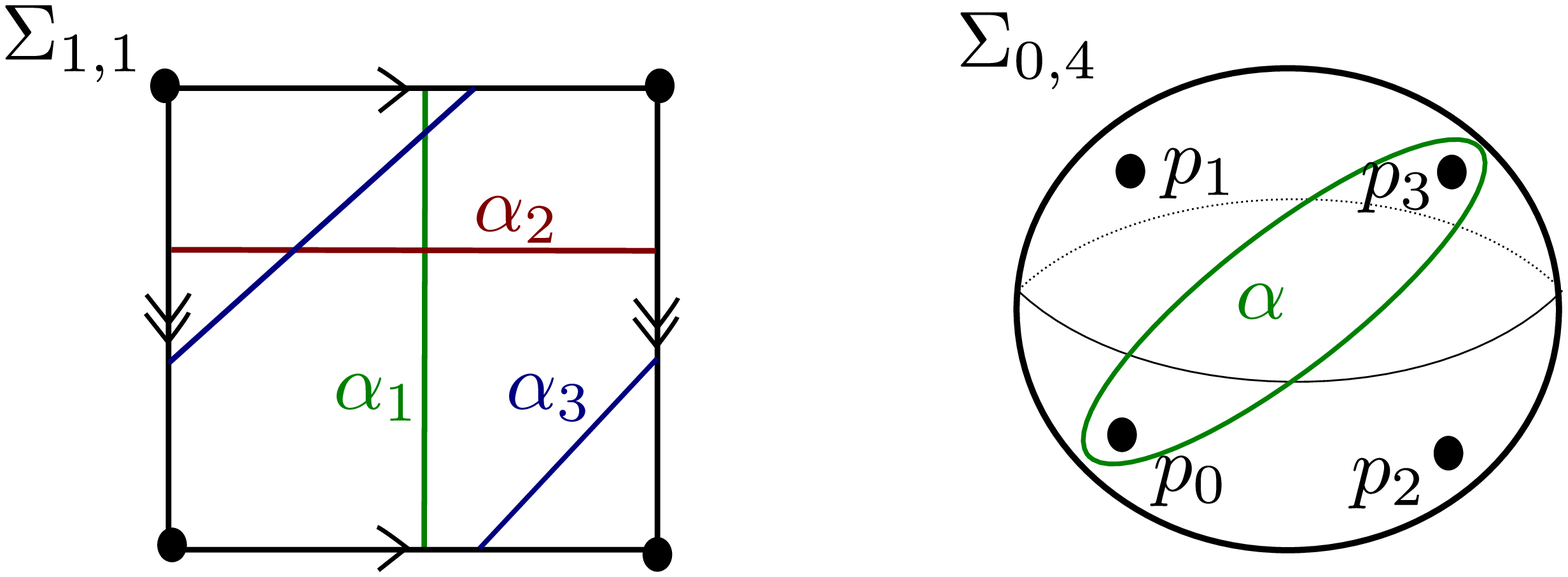} }
\caption{Some curves in $\Sigma_{0,1}$ and $\Sigma_{0,4}$.} 
\label{fig_Takenov} 
\end{figure}

Here is what is known concerning Problem \ref{problem_AL}:

\begin{theorem}
\begin{enumerate}
\item (\textit{Brown-Goodearl }\cite[Theorem C]{BrownGoodearl}:) The Azumaya locus of $\mathcal{O}_q[\SL_2]$ is equal to the smooth locus. Therefore (Brown-Gordon \cite{BrownGordon_OqG}), its fully Azumaya locus is the set of non-diagonal matrices of $\SL_2$.
\item (\textit{Takenov }\cite[Theorem $15$, Theorem $17$]{Takenov_Azumaya}:) The open dense subsets $\mathcal{O}_{1,1}$ and $\mathcal{O}_{0,4}$ are included into the Azumaya loci of $\mathcal{S}_A(\Sigma_{1,1})$ and $\mathcal{S}_A(\Sigma_{0,4})$ respectively.
\item (\textit{K.} \cite{KojuQGroupsBraidings}:) The fully Azumaya locus of $\mathcal{S}^{red}_A(\mathbb{D}_1)$ is equal to the set of elements $(g_+, g_-) \in B_+\times B_-=\mathcal{X}^{red}(\mathbf{D}_1)$ such that $g_+g_-^{-1}\neq \pm \mathds{1}_2$.
\item (\textit{Bonahon-Wong }\cite{BonahonWong4}:) For a closed surface, the central representations of $\mathcal{X}_{\SL_2}(\Sigma_g)$ are not in the Azumaya locus.
\item (\textit{Ganev-Jordan-Safranov } \cite{GanevJordanSafranov_FrobeniusMorphism}:) For a closed surface, the smooth locus of $\mathcal{X}_{\SL_2}(\Sigma_g)$ is included in the Azumaya locus. Moreover, for $\mathbf{\Sigma}_{g,0}^0$, the Azumaya locus is equal to the dense bad arcs leaf $\mathcal{X}^0(\mathbf{\Sigma}_{g,0}^0)$. 
\item (Consequence of \textit{Alekseev-Malkin} \cite{AlekseevMalkin_PoissonLie}:) The open dense subset $D_{00}$ is included in the fully Azumaya locus of $D_q^+(\SL_2)$.
\end{enumerate}
\end{theorem}

The second and third items in this theorem are proved by brute force computations starting with a finite presentation of the skein algebras and hardly generalize. Note that the irreducible representations of $\mathcal{S}_A(\Sigma_{1,1})$ have been classified by Havl\'i\v{c}ek-Po\v{s}ta
in \cite{HavlicekPosta_Uqso3} (after remarking that  $\mathcal{S}_A(\Sigma_{1,1})\cong U'_q\mathfrak{so}_3$), and Takenov's result for $\Sigma_{1,1}$  is derived from this work. The fourth item is an immediate consequence of Corollary \ref{coro_representations_central} which easily generalizes to non-necessary closed unmarked surfaces. The last two items, reviewed in the next subsection, follows from a more general theorem: if an equivariant symplectic leaf is dense, then it is included in the fully Azumaya locus.
Note that none of the proofs makes use of Theorem \ref{theorem_UnicityRep}.

  \subsection{Brown and Gordon's Poisson orders}\label{sec_PO}
  
We now prove that if $x,y$ belong to the same equivariant symplectic leaf, then $\mathcal{S}_A(\mathbf{\Sigma})_x \cong \mathcal{S}_A(\mathbf{\Sigma})_y$ using the theory of Poisson orders. The theory began with the work of DeConcini-Kac on $U_q\mathfrak{g}$ in \cite{DeConciniKacRepQGroups}, the work of DeConcini-Lyubashenko on $\mathcal{O}_q[G]$ in \cite{DeConciniLyubashenko_OqG} and was fully developed by Brown-Gordon in \cite{BrownGordon_PO} that we closely follow.

\begin{definition}
A \textit{Poisson order} is a $4$-uple $(\mathcal{A}, \mathcal{X}, \phi, D)$ where: 
\begin{enumerate}
\item $\mathcal{A}$ is an (associative, unital) affine $\mathbb{C}$-algebra finitely generated over its center $\mathcal{Z}$; 
\item $\mathcal{X}$ is a Poisson affine $\mathbb{C}$ -variety; 
\item $\phi: \mathcal{O}[\mathcal{X}] \hookrightarrow \mathcal{Z}$ a finite injective morphism of algebras; 
\item $D: \mathcal{O}[\mathcal{X}] \to Der (\mathcal{A}) : z\mapsto D_z$ a linear map such that for all $f,g \in  \mathcal{O}[\mathcal{X}]$, we have
$$ D_f(\phi(g))= \phi(\{f,g\})$$.
\end{enumerate}
\end{definition}
 Let us consider a Poisson affine variety $\mathcal{X}$. An ideal $\mathcal{I}\subset \mathcal{O}[\mathcal{X}]$ is a \textit{Poisson ideal} if $\{ \mathcal{I}, \mathcal{O}[\mathcal{X}] \} \subset \mathcal{I}$. Since the sum of two Poisson ideals is a Poisson ideal, every maximal ideal $\mathfrak{m}\subset \mathcal{O}[\mathcal{X}]$ contains a unique maximal Poisson ideal $P(\mathfrak{m}) \subset \mathfrak{m}$. Define an equivalence relation $\sim$ on $\mathcal{X}=\mathrm{Specm}(\mathcal{O}[\mathcal{X}])$ by $\mathfrak{m} \sim \mathfrak{m}'$ if $P(\mathfrak{m})=P(\mathfrak{m}')$. The equivalence class of $\mathfrak{m}$ is named its \textit{symplectic core} and denoted by $C(\mathfrak{m})$.
 
 \begin{definition} The partition $\mathcal{X}=\bigsqcup_{\mathfrak{m}\in \quotient{\mathcal{X}}{\sim}} C(\mathfrak{m})$ is called the \textit{symplectic cores partition} of $\mathcal{X}$.
 \end{definition} 

If $x, y\in \mathcal{X}$ belong to the same smooth strata of $\mathcal{X}$ and are such that $y$ can be obtained from $x$ by deriving along some Hamiltonian flow, Brown and Gordon proved that $C(x)=C(y)$ from which they deduce the

\begin{proposition}(Brown-Gordon \cite[Proposition $3.6$]{BrownGordon_PO}) The symplectic leaves partition is a refinement of the symplectic cores partition.
\end{proposition}

For $(\mathcal{A}, \mathcal{X}, \phi, D)$ a Poisson order and $x\in \mathcal{X}$, write 
$$\mathcal{A}_x:= \quotient{\mathcal{A}}{\phi(\mathfrak{m}_x)\mathcal{A}}.$$
The main theorem of the theory is

\begin{theorem}\label{theorem_PO}(Brown-Gordon \cite[Theorem $4.2$]{BrownGordon_PO}) Let $(\mathcal{A}, \mathcal{X}, \phi, D)$ be a Poisson order. If $x,y\in \mathcal{X}$ belong to the same symplectic core, then $\mathcal{A}_x\cong \mathcal{A}_y$.
\end{theorem}

It can be improved as follows.

\begin{definition} Let $G$ be an affine Lie group. A Poisson order $(\mathcal{A}, \mathcal{X}, \phi, D)$ is said $G$-\textit{equivariant} if $G$ acts  on $\mathcal{A}$ by automorphism such that its action preserves $\phi(\mathcal{O}[\mathcal{X}])\subset \mathcal{A}$ and such that it is $D$ equivariant in the sense that for every $g\in G$, $z\in \mathcal{O}[\mathcal{X}]$ and $a\in \mathcal{A}$, one has
$$ D_{g\cdot z}(a) = g D_z (g^{-1}a).$$
\end{definition}

In this case the action of $g\in G$ induces an isomorphism $\mathcal{A}_x\cong \mathcal{A}_{gx}$. Let us call \textit{equivariant symplectic cores} and \textit{equivariant  symplectic leaves} the $G$-orbits of the symplectic cores and symplectic leaves.

\begin{theorem}\label{theorem_PO_equivariant}(Brown-Gordon \cite[Proposition $4.3$]{BrownGordon_PO}) For $(\mathcal{A}, \mathcal{X}, \phi, D)$ a $G$-equivariant Poisson order, if $x,y\in \mathcal{X}$ belong to the same equivariant symplectic core (or symplectic leaf a fortiori) then $\mathcal{A}_x\cong \mathcal{A}_y$.
\end{theorem}

\par We now endow the triple $(\mathcal{S}_A(\mathbf{\Sigma}), \mathcal{X}(\mathbf{\Sigma}), Ch_A)$ with a structure of $(\mathbb{C}^*)^{\mathcal{A}}$ equivariant Poisson order.

\begin{example}\label{example_PO} 
Let $\mathcal{A}_q$ a free flat affine $\mathbb{C}(q^{\pm 1})$-algebra,   $N\geq 1$ and, writing $t:= N(q^N-1)$,  the $\quotient{\mathbb{C}(q^{\pm 1})}{(q^N-1)}$ algebra $\mathcal{A}_N:= \quotient{\mathcal{A}}{t}$ and $\pi: \mathcal{A}_q \to \mathcal{A}_N$ the quotient map. By fixing a basis $\mathcal{B}$ of $\mathcal{A}_q$ by flatness we can define a linear
 embedding $\hat{\cdot}: \mathcal{A}_{N} \to \mathcal{A}_q$ sending a basis element $b\in \mathcal{B}$ seen as element in $\mathcal{A}_{N}$ to the same element $\hat{b}$ seen as an element in $\mathcal{A}_q$. Note that $\hat{\cdot}$ is a left inverse for $\pi$. Suppose that the  algebra $\mathcal{A}_{+1}=\mathcal{A}_q\otimes_{q=1}\mathbb{C}$ is commutative and suppose there exists a central embedding  $\phi: \mathcal{A}_{+1} \hookrightarrow \mathcal{A}_{N}$ into the center of $\mathcal{A}_{N}$. Write $\mathcal{X}:=\mathrm{Specm}(\mathcal{A}_{+1})$ and define $D: \mathcal{A}_{+1} \to Der(\mathcal{A}_N)$ by the formula
 $$ D_xy:= \pi \left( \frac{ [\widehat{\phi(x)}, \hat{y}]}{N(q^N-1)} \right) .$$
Clearly $D_x$ is a derivation, is independent on the choice of the basis $\mathcal{B}$ and preserves $\phi(\mathcal{A}_{+1})$, so it defines a Poisson bracket $\{\cdot, \cdot \}_N$ on $\mathcal{A}_{+1}$ by 
$$ D_x \phi(y)= \phi (\{x,y\}_N).$$
So, writing $\mathcal{X}=\mathrm{Specm}(\mathcal{A}_{+1})$, then $(\mathcal{A}_N, \mathcal{X}, \phi, D)$ is a Poisson order for this bracket. Note that if $\zeta_N$ is an $N$-th root of unity and $\mathcal{A}_{\zeta_N}=\mathcal{A}_q \otimes_{q=\zeta_N} \mathbb{C}$, we get a Poisson order $(\mathcal{A}_{\zeta_N}, \mathcal{X}, \phi, D)$ as well simply by tensoring by $\mathbb{C}$.

What is not clear is how to compare the above bracket $\{\cdot, \cdot\}_N$ with the one coming from deformation quantization defined in Section \ref{sec_DefoQuantization}, i.e. with the bracket $\{\cdot, \cdot \}_{\infty}$ defined by setting $q=\exp(\hbar)$ and writing $\star$ the product in $\mathcal{A}_q\otimes_{q=\exp(\hbar)} \mathbb{C}[[\hbar]]$, 
by the formula
$$ x\star y - y\star x \equiv \hbar \{x,y\}_{\infty} \pmod{\hbar^2}.$$
\end{example}

\begin{lemma}\label{lemma_bracket_torus}
Let $\mathbb{E}=(E, (\cdot, \cdot))$ be a quadratic pair and consider the Poisson order $(\mathbb{T}_{\zeta_N}(\mathbb{E}), \mathcal{X}(\mathbb{E}), Fr_N, D)$ defined in Example 
\ref{example_PO} using the Frobenius morphism for $\phi$. Then the Poisson bracket $\{\cdot, \cdot\}_N$ coming from $D$ coincides with the Poisson bracket $\{\cdot, \cdot\}_{\infty}$ coming from deformation quantization.
\end{lemma}

\begin{proof}
Let $a,b \in E$. Since $Z_a^N Z_b^N= q^{N^2(a,b)}Z_{a+b}^N$ and $Z_b^N Z_a^N= q^{-N^2(a,b)}Z_{a+b}^N$, we have: 
\begin{align*}
& \pi \left( \frac{ [Fr_N(Z_a), Fr_N(Z_b)]}{N(q^N-1)} \right) =  \pi \left( \frac{ [Z_a^N, Z_b^N]}{N(q^N-1)} \right) \\
&= \pi \left( \frac{q^{N^2(a,b)} - q^{-N^2(a,b)}}{N(q^N-1)} \right) Z_{a+b}^N \\
&= \pi \left( q^{-N(a,b)}\frac{1}{N} \sum_{i=1}^{2N(a,b)} q^{Ni} \right) Fr_N(Z_{a+b})
= Fr_N(2(a,b)Z_{a+b}).
\end{align*}
So $\{Z_a, Z_b\}_N= 2(a,b)Z_{a+b}$. Recall that we obtained the same formula for the bracket $\{\cdot, \cdot\}_{\infty}$ by the computation
$$ Z_a \star Z_b - Z_b \star Z_a = (q^{(a,b)}-q^{-(a,b)}) Z_{a+b} \equiv \hbar 2(a,b)Z_{a+b} \pmod{\hbar^2}.$$
\end{proof}

\begin{proposition}(DeConcini-Lyubashenko \cite{DeConciniLyubashenko_OqG} for the bigon, Ganev-Jordan-Safranov \cite{GanevJordanSafranov_FrobeniusMorphism} for $\mathbf{\Sigma}=(\Sigma_{g,0}, \emptyset)$ and $\mathbf{\Sigma}=\mathbf{\Sigma}_{g,0}^0$)
For any marked surface $\mathbf{\Sigma}$, the map $D$ of Example \ref{example_PO} defines some structures of Poisson orders $(\mathcal{S}_A(\mathbf{\Sigma}), \mathcal{X}(\mathbf{\Sigma}), Ch_A, D)$ and $(\mathcal{S}_A^{red}(\mathbf{\Sigma}, \mathcal{X}^{red}(\mathbf{\Sigma}), Ch_A, D)$ compatible with Goldman's Poisson bracket.
\end{proposition}

\begin{proof} That these $4$-uplet are Poisson orders is clear. What we need to prove is that the Poisson bracket induced by $D$ coincides with the one coming from deformation quantization. When $\mathbf{\Sigma}=\mathbb{B}$ this was done  by DeConcini and Lyubashenko in \cite{DeConciniLyubashenko_OqG} by explicit computations to which we refer. For any other non trivial connected marked surface $\mathbf{\Sigma}$, we can embedded $\mathcal{S}_A^{red}(\mathbf{\Sigma})$ and $\mathcal{S}_A(\mathbf{\Sigma})$ into some quantum torus using Bonahon-Wong's quantum trace or its L\^e-Yu enhancement. In both cases, the Chebyshev-Frobenius morphism is the restriction of the Frobenius morphism so the result follows from Lemma \ref{lemma_bracket_torus}.
\end{proof}

In particular, by putting Corollary \ref{coro_unicity} and Theorem \ref{theorem_PO} together, we see that if a symplectic leaf is dense, then it is included in the fully Azumaya locus, in particular we get the

\begin{theorem}\label{theorem_GJS}
\begin{enumerate}
\item (Ganev-Jordan-Safranov \cite[Theorem $1.1$]{GanevJordanSafranov_FrobeniusMorphism}:) For a closed surface $\mathbf{\Sigma}= (\Sigma_{g,0}, \emptyset)$, the smooth locus $\mathcal{X}_{\SL_2}^0(\Sigma_{g,0})$ is included in the Azumaya locus.
\item (Ganev-Jordan-Safranov \cite[Theorem $1.1$]{GanevJordanSafranov_FrobeniusMorphism}:) For $g\geq 1$, $\mathbf{\Sigma}_{g,0}^0$, the bad arcs leaf $\mathcal{X}^0(\mathbf{\Sigma}_{g,0}^0)=\mu^{-1}(\SL_2^0)$ is equal to the Azumaya locus.
\end{enumerate}
\end{theorem}

\begin{proof}
The leaves $\mathcal{X}_{\SL_2}^0(\Sigma_{g})$ and $\mathcal{X}^0(\mathbf{\Sigma}_{g,0}^0)$ are symplectic by Theorem \ref{theorem_smoothsymplectic}  and Theorem \ref{theorem_GJS} respectively and they are both open dense subsets. Since the Azumaya loci of $\mathcal{S}_A(\Sigma_{g,0})$ and $\mathcal{S}_A(\mathbf{\Sigma}_{g,0}^0)$ are
 also dense by Corollary \ref{coro_unicity}, they intersect non-trivially these leaves and we conclude using Theorem \ref{theorem_PO}. The fact that $\mathcal{X}^1(\mathbf{\Sigma}_{g,0}^0)$ does not intersect the Azumaya locus follows from Lemma \ref{lemma_reduced_leaf}.
 \end{proof}
 
 Note that Corollary \ref{coro_representations_central} implies that 
 
 \begin{corollary} For a closed surface $\Sigma_g$, none of the central representations in $\mathcal{X}_{\SL_2}(\Sigma_g)$ belongs to the Azumaya locus of $\mathcal{S}_A(\Sigma_g)$.
 \end{corollary}

 Let us consider the skein algebra of a closed connected surface $\Sigma_g$ of genus $g\geq 2$. At this stage, the only question that remains open is
 
 \begin{question} Does the diagonal representations of $\mathcal{X}_{\SL_2}^1(\Sigma_g)$ belong to the Azumaya locus of $\mathcal{S}_A(\Sigma_g)$ ?
 \end{question}
 
 L\^e and Yu conjectured in \cite{LeYu_Survey} that the answer is no. Let us collect what we learned from our study:
 
 \begin{proposition}\label{prop_equivalences}
 The following assertion are equivalent:
 \begin{enumerate}
 \item There exists a diagonal representation in  $\mathcal{X}_{\SL_2}^1(\Sigma_g)$ that belongs to the Azumaya locus of $\mathcal{S}_A(\Sigma_g)$; 
 \item All  diagonal representations  belong to the Azumaya locus of $\mathcal{S}_A(\Sigma_g)$;
 \item There exists a class $\omega \in \mathrm{H}^1(\Sigma_g; \mathbb{C}/\mathbb{Z}) \setminus \mathrm{H}^1(\Sigma; \frac{1}{2}\mathbb{Z}/\mathbb{Z})$ for which the representation $\rho_{\omega}^{BCGP}$ coming from non semi simple TQFTs is irreducible; 
 \item For all $\omega \in  \mathrm{H}^1(\Sigma_g; \mathbb{C}/\mathbb{Z}) \setminus \mathrm{H}^1(\Sigma;  \frac{1}{2}\mathbb{Z}/\mathbb{Z})$, then $\rho_{\omega}^{BCGP}$ is irreducible; 
 \item There exists a quantum Teichm\"uller representation $\rho^{BW}$ with classical shadow in $\mathcal{X}_{\SL_2}^1(\Sigma_g)$ which is irreducible; 
 \item All quantum Teichm\"uller representations $\rho^{BW}$ with classical shadow in $\mathcal{X}_{\SL_2}^1(\Sigma_g)$ are  irreducible.
 \end{enumerate}
 Moreover, if these assertions are true, then any two representation $\rho^{BCGP}$ and $\rho^{BW}$ having the same diagonal classical shadow in $\mathcal{X}_{\SL_2}^1(\Sigma_g)$ are isomorphic.
 \end{proposition}
 
 \begin{proof}
 The equivalence between the first two assertions follows from the fact, proved in Lemma \ref{lemma_diagonal_symplectic},  that $\mathcal{X}_{\SL_2}^1(\Sigma_g)$ is a symplectic leaf together with Theorem \ref{theorem_PO}. The other equivalences follow from the fact that both families of representations $\rho^{BCGP}$ and $\rho^{BW}$ have dimension equal to $N^{3g-3}$ which is the PI-dimension of $\mathcal{S}_A(\Sigma_g)$ by Theorem \ref{theorem_center}. 
 \end{proof}
 
 The above Proposition serves as a motivation to solve Question \ref{question_BCGP}. Of course, a much more difficult problem is the 
 
 \begin{problem} Let $g\geq 2$.
 \begin{enumerate}
 \item
 If $[r]\in \mathcal{X}^2_{\SL_2}(\Sigma_g)$ is a central representation, compute the algebra $\mathcal{S}_A(\Sigma_g)_{[r]}$ and find its indecomposable representations (if it's not wild).
 \item For $[r]\in \mathcal{X}^2_{\SL_2}(\Sigma_g)$ a central representation, compare the representations $\rho^{WRT}, \rho^{BCGP}$ and $\rho^{BW}$ with classical shadow $[r]$. In particular, are the representations $\rho^{BCGP}$ and $\rho^{BW}$ isomorphic ? Is $\rho^{WRT}$ a subrepresentation of one of them ?
 \end{enumerate}
 \end{problem}
 
 \vspace{2mm}
 \par 
 
 We now endow the previous Poisson orders with a structure of $(\mathbb{C}^*)^{\mathcal{A}}$-equivariant Poisson order. 
 Let $\varphi : \mathcal{O}_q[\SL_2] \to \mathbb{C}[X^{\pm1}]$ be the surjective morphism defined by 
 $$ \varphi(\alpha_{+-})=\varphi(\alpha_{-+})=0, \quad \varphi(\alpha_{++})=X, \quad \varphi(\alpha_{--})= X^{-1}.$$
 The morphism $\varphi$ is clearly a morphism of Hopf algebras and the induced morphism on $\mathcal{O}[\SL_2] \xrightarrow{Ch_A} \mathcal{O}_q[\SL_2]$ is the diagonal embedding $\mathbb{C}^* \to \SL_2$ sending $z$ to $\begin{pmatrix} z^N & 0 \\ 0 & z^{-N} \end{pmatrix}$. Note that, while identifying $\mathbb{C}[X^{\pm 1}]$ with the reduced algebra $\mathcal{S}_A^{red}(\mathbb{B})$, then $\varphi$ is just the quotient map $\mathcal{S}_A(\mathbb{B}) \to \mathcal{S}_A^{red}(\mathbb{B})$.
 
 Define an algebraic action of $(\mathbb{C}^*)^{\mathcal{A}}$ on $\mathcal{S}_A(\mathbf{\Sigma})$ by the co-action 
 $$\Delta^{diag}: \mathcal{S}_A(\mathbf{\Sigma}) \xrightarrow{\Delta^L} \left(\mathcal{O}_q[\SL_2]^{\otimes \mathcal{A}} \right) \otimes \mathcal{S}_A(\mathbf{\Sigma}) \xrightarrow{ (\varphi^{\otimes \mathcal{A}}) \otimes \id} \mathbb{C}[X^{\pm 1}] \otimes \mathcal{S}_A(\mathbf{\Sigma}).$$

The above action induces by quotient a similar action on $\mathcal{S}_A^{red}(\mathbf{\Sigma})$ and both action preserve the image of the Chebyshev-Frobenius morphism.
The equivariance of $D$ for this action is an immediate consequence of the definition of $D$ so we get the 

\begin{corollary}\label{coro_POSkein} If $x,y \in \mathcal{X}(\mathbf{\Sigma})$ belong to the same equivariant symplectic leaf (or equivariant symplectic core), then $\mathcal{S}_A(\mathbf{\Sigma})_x\cong \mathcal{S}_A(\mathbf{\Sigma})_y$. Similarly, if $x,y \in \mathcal{X}^{red}(\mathbf{\Sigma})$ belong to the same equivariant symplectic leaf (or equivariant symplectic core), then $\mathcal{S}^{red}_A(\mathbf{\Sigma})_x\cong \mathcal{S}^{red}_A(\mathbf{\Sigma})_y$. 
\end{corollary}

In the case of the bigon, the corollary says that the isomorphism class of $\mathcal{O}_q[\SL_2]_g$ only depends on the double Bruhat cell of $g\in \SL_2$. This was the first version of an equivariant Poisson order which appeared in the work of DeConcini-Lyubashenko in \cite{DeConciniLyubashenko_OqG}, so Corollary \ref{coro_POSkein} is a generalization of their work.

\subsection{Resolutions of the classification problem is small cases}\label{sec_resolution}.
\vspace{2mm}
\par 
\textbf{The torus case:} Let $\mathbf{\Sigma}=(\Sigma_1, \emptyset)$ the closed torus. The character variety admits a double branched covering $\pi: (\mathbb{C}^*)^2\to \mathcal{X}_{\SL_2}(\Sigma_1)$, where the covering involution sends $(z_1,z_2)$ to $(z_1^{-1}, z_2^{-1})$ and the four branched points $(\pm 1, \pm 1)$ correspond to the scalar representations.
The symplectic leaves of $\mathcal{X}_{\SL_2}(\Sigma_1)$ are $(1)$ the smooth locus $\mathcal{X}^{1}_{\SL_2}(\Sigma_1)$ of non scalar representations and $(2)$ the four ramification points corresponding to the scalar representations. Since none of the scalar representation belongs to the Azumaya locus  and $\mathcal{X}^{1}_{\SL_2}(\Sigma_1)$ is dense, then the Azumaya locus of $\mathcal{S}_A(\Sigma_1)$ is $\mathcal{X}^{1}_{\SL_2}(\Sigma_1)$. Let us construct its indecomposable representations explicitly. Let $V=\mathbb{C}^N$ with canonical basis $\{e_i, i \in \mathbb{Z}/N\mathbb{Z}\}$ and consider two operators $U,V \in \End(V)$ defined by $Ve_i=e_{i+1}$ and $Ue_i=q^i e_i$, so that $UV=q VU$ and $U^N=V^N=\id_V$. 
 For $(x,y) \in (\mathbb{C}^*)^2$, consider the irreducible(Schr\"odinger) representation $r_{x,y}: \mathcal{W}_q \to \End(V)$ defined by $r_{x,y}(X)=xU$ and $r_{x,y}(V)=yV$. Clearly, the isomorphism class of $r_{x,y}$ only depends on $(x^N, y^N)$ and every simple $\mathcal{W}_q$-module is isomorphic to such a $r_{x,y}$. Recall from Section \ref{sec_ClosedCase} the definition of the Frohman-Gelca isomorphism $\phi^{FG} : \mathcal{S}_A(\Sigma_1) \cong \mathcal{W}_q^+$. By composition
 $$ \rho_{x,y}: \mathcal{S}_A(\Sigma_1) \xrightarrow{\phi^{FG}} \mathcal{W}_q \xrightarrow{r_{x,y}} \End(V), $$
 we get an $N$-dimensional representation of $ \mathcal{S}_A(\Sigma_1) $ which only depends, up to isomorphism, on $x^N$ and $y^N$. Its classical shadow is easily computed from the formulas
  $$T_N(\rho_{x,y}(L))=r_{x,y}(T_N(X+X^{-1}))=r_{x,y}(X^N+X^{-N})= (x^N+x^{-N})\id_V \quad \mbox{and} \quad T_N(\rho_{x,y}(M))=(y^N+y^{-N})\id.$$
  So $\rho_{x,y}$ has classical shadow $\pi(x^N, y^N)$ and every point of $\mathcal{X}_{\SL_2}(\Sigma_1)$ is the shadow of such a representation. Now choose $x=\varepsilon_1$ and $y=\varepsilon_2$, for $\varepsilon_1,\varepsilon_2\in \{ \pm1\}$ so that $\pi(x^N, y^N)$ is a central representation. Recall the definition of the automorphism $\Theta \in \mathrm{Aut}(\mathcal{W}_q)$ sending $X, Y$ to $X^{-1}, Y^{-1}$ and such that $\mathcal{W}_q^+$ is the subalgebra of $\Theta$-invariants elements and define $\theta\in \GL(V)$ by $\theta(e_i)=e_{-i}$. A simple computation shows that 
  $$ r_{\varepsilon_1,\varepsilon_2} (\Theta(Z))= \theta r_{\varepsilon_1,\varepsilon_2}(Z) \theta \quad\mbox{, for all }Z \in \mathcal{W}_q.$$
Therefore, the two subspaces $V^{\pm}:= \{v\in V| \theta (v) = \pm v\} =\Span \{ e_i\pm e_{-i}, i \in \mathbb{Z}/N\mathbb{Z} \}$ are stable under the action of $\mathcal{W}_q^+$ since
$$ \Theta(Z)=Z \mbox{ and }\theta(v)=\pm v \Rightarrow \theta r_{\varepsilon_1,\varepsilon_2}(Z)  \cdot v = r_{\varepsilon_1,\varepsilon_2}(\Theta(Z)) \theta \cdot v=\pm r_{\varepsilon_1,\varepsilon_2}(Z)  \cdot v .$$
Therefore, $\rho_{\varepsilon_1, \varepsilon_2}$ is the direct sum of two  subrepresentations
$$ \rho_{\varepsilon_1, \varepsilon_2}^{\pm} : \mathcal{S}_A(\Sigma_1) \to \End(V^{\pm}).$$
Writing $v_i^{\pm}:= e_i\pm e_{-i}$, we have
 $$ \rho^{\pm}_{\varepsilon_1, \varepsilon_2} (M) v_i^{\pm} = \varepsilon_1 (q^i +q^{-i}) v_i^{\pm} \quad, \rho^{\pm}_{\varepsilon_1, \varepsilon_2} (L) v_i^{\pm} =\varepsilon_2 (v_{i-1}^{\pm} + v_{i+1}^{\pm}), $$
 from which we easily see that  both representation are irreducible and that $\rho^{-}_{+1, +1}$ is isomorphic to the Witten-Reshetikhin-Turaev representation. 
 
 \begin{proposition}
 \begin{enumerate}
 \item Every weight indecomposable representation of $\mathcal{S}_A(\Sigma_1)$  is  isomorphic to either a $\rho_{x,y}$ or a $\rho^{\pm}_{\varepsilon_1, \varepsilon_2}$.
  \item If $\rho^{BW}$ and $\rho^{BCGP}$ are representations of $\mathcal{S}_A(\Sigma_1)$ with classical shadow $\pi(x^N, y^N)$, then $\rho^{BW}\cong \rho^{BCGP}\cong \rho_{x,y}$. Moreover, if $x^N=y^N=1$, then $\rho^{WRT}$ is a sub-representation of them.
 \end{enumerate}
 \end{proposition}
 
 \begin{proof}
 The second assertion is a consequence of the first one. When $(x^{2N}, y^{2N})\neq (1,1)$, the first assertion follows from the fact that the Azumaya locus of $\mathcal{S}_A(\Sigma_1)$ consists in the non scalar representations. To prove the first assertion when $(x,y)=(\varepsilon_1, \varepsilon_2)$, we prove that the representation $\rho_{\varepsilon_1, \varepsilon_2}: \mathcal{S}_A(\Sigma_1)_{\pi(\varepsilon_1, \varepsilon_2)}\to \End(V)$ is faithful.
 Let $\mathcal{I}_{\varepsilon_1, \varepsilon_2}\subset \mathcal{W}_q$ be the ideal generated by $(X^N-\varepsilon_1, Y^N-\varepsilon_2)$ and consider the
   the inclusion map $\iota: \mathcal{W}_q^{+} \hookrightarrow \mathcal{W}_q$. Then $\mathcal{I}^+_{\varepsilon_1, \varepsilon_2} = \iota^{-1}(\mathcal{I}_{\varepsilon_1, \varepsilon_2}) = \mathcal{I}_{\varepsilon_1, \varepsilon_2} \cap \mathcal{W}_q^+$ is clearly generated by $(T_N(X+X^{-1})- 2\varepsilon_1, T_N(Y+Y^{-1})-2\varepsilon_2)$ thus $\iota$ induces an injective morphism:
   $$ \rho_{\varepsilon_1, \varepsilon_2} :  \mathcal{S}_A(\Sigma_1)_{\pi(\varepsilon_1, \varepsilon_2)} \xrightarrow[\phi^{FG}]{\cong} \quotient{\mathcal{W}_q^+}{\mathcal{I}_{\varepsilon_1, \varepsilon_2}^+} \hookrightarrow \quotient{\mathcal{W}_q}{\mathcal{I}_{\varepsilon_1, \varepsilon_2}} \cong \End(V), $$
   and we have proved that $\rho_{\varepsilon_1, \varepsilon_2}$ is faithful. Therefore:
   $$ \mathcal{S}_A(\Sigma_1)_{\pi(\varepsilon_1, \varepsilon_2)} \cong \mathrm{Im}(\rho_{\varepsilon_1, \varepsilon_2})= \End(V^+) \oplus \End(V^-).$$
   This proves the first assertion in the scalar case and concludes the proof.

 \end{proof}

 \textbf{The bigon}(Brown-Gordon \cite{BrownGordon_OqG})
  \par  By a theorem of Brown and Goodearl (\cite[Theorem C]{BrownGoodearl}), the Azumaya locus of $\mathcal{O}_q[\SL_2]$ is equal to the smooth locus of $\widehat{\mathcal{X}}(\mathbb{B})$ from which we deduce that the fully Azumaya locus is the set $\mathcal{FAL}\subset \SL_2^D$ of non diagonal matrices. Recall from Section \ref{sec_STS} that the decomposition into equivariant symplectic leaves corresponds to the double Bruhat cells decomposition $\SL_2^D= X_{00}\bigsqcup X_{01} \bigsqcup X_{01} \bigsqcup X_{11}$, where $X_{11}$ represents the set of diagonal matrices. Recall from Theorem \ref{theorem_center} (which, in that particular case, is due to Enriquez) that the center of $\mathcal{O}_q[\SL_2]$ is generated by the image of the Chebyshev-Frobenius morphism (i.e. by the elements $\alpha_{ij}^N$) together with the elements $\alpha_{+-}^k\alpha_{-+}^{N-k}$ for $1\leq k \leq N-1$.
  
  A simple computation shows that for $g\in \SL_2$ 
  $$ \mathcal{Z}(g) \cong \left\{ 
  \begin{array}{ll} 
 \quotient{\mathbb{C}[X]}{(X^N-1)} \cong \mathbb{C}^{\oplus N} & \mbox{, if }g\in X_{00}; \\
  \quotient{\mathbb{C}[X]}{(X^N)} & \mbox{, if }g\in X_{01}\cup X_{10}; \\
  \quotient{\mathbb{C}[Z_1, \ldots, Z_{N-1}]}{(Z_iZ_j, 1\leq i,j\leq N-1)} & \mbox{, if }g\in X_{11}.
  \end{array}
  \right.$$
  
  Therefore, Corollary \ref{theorem_BG_Ramifications} implies that (this is \cite[Corollary $2.7$ ]{BrownGordon_ramificationcenters}):
  $$\mathcal{O}_q[\SL_2]_g \cong \left\{
  \begin{array}{ll}
  Mat_N(\mathbb{C})^{\oplus N} & \mbox{, if }g\in X_{00}; \\
  Mat_N\left(\quotient{\mathbb{C}[X]}{(X^N)} \right) & \mbox{, if }g\in X_{10}\cup X_{01}.
  \end{array} \right.$$
 So for each $g\in \SL_2\setminus X_{11}$, there exists a unique (up to isomorphism) indecomposable weight module $V(g)$ with classical shadow $g$ and this module has dimension $N$. Moreover, when $g\in X_{10}\cup X_{01}$, $V(g)$ is the subrepresentation of a tower of semi-weight indecomposable modules described as follows. Let $V=\mathbb{C}^N$ with basis $\{e_1, \ldots, e_N\}$ and define a representation $r : \quotient{\mathbb{C}[X]}{(X^N)} \to \End(V)$ by $r(X)e_i=e_{i+1}$, for all $1\leq i \leq N-1$ and $r(X)e_N=0$. For $1\leq i \leq N$, let $V_i:= \Span\{ e_i, e_{i+1} \ldots, e_N\} \subset V$. We thus have a tower $V_N\subset V_{N-1} \subset \ldots \subset V_1=V$ and each $V_i$ is stable by $\quotient{\mathbb{C}[X]}{(X^N)}$. The algebra $Mat_N(\quotient{\mathbb{C}[X]}{(X^N)})$ naturally acts on $\quotient{\mathbb{C}[X]}{(X^N)}^{\oplus N}$ by matrix multiplication so composing with $r$  and setting  $W_i:= V_i^{\oplus N}$, we get some representations $r_i : Mat_N(\quotient{\mathbb{C}[X]}{(X^N)}) \to \End(W_i)$ from which we define a semi-weight indecomposable representation: 
 $$\rho_g^i : \mathcal{O}_q[\SL_2] \to \mathcal{O}_q[\SL_2]_g \cong Mat_N(\quotient{\mathbb{C}[X]}{(X^N)}) \xrightarrow{r_i} \End(W_i).$$
 So we get a flag $V(g)=W_N\subset \ldots \subset W_1$ of representations over $V(g)$. In particular, $V(g)$ is not projective in this case, though it is still a projective object in the category of weight representations.

 When $g\in X_{11}$, Gordon and Brown proved in \cite{BrownGordon_OqG} that $\mathcal{O}_q[\SL_2]_g$ is wild, so we will avoid this case.

\vspace{2mm}
\par
\textbf{The algebra $\mathcal{S}^{red}_A(\mathbb{D}_1)$}
 \par
 Using the fact $\mathcal{S}^{red}_A(\mathbb{D}_1)$ is isomorphic to $\widetilde{U}_q\mathfrak{gl}_2$, the author were able to classify by hand in \cite{KojuQGroupsBraidings} every finite dimensional weight indecomposable representations of $\mathcal{S}^{red}_A(\mathbb{D}_1)$. Retrospectively, we can deduce from this classification that  the fully Azumaya locus is equal to the  subset of elements  $(g_+, g_-) \in B^+ \times B^- =\mathcal{X}^{red}(\mathbb{D}_1)$ such that $g_+ g_-^{-1}\neq \pm \mathds{1}_2$. By Theorem \ref{theorem_center}, 
  the center $\mathcal{Z}$ of $\mathcal{S}^{red}_A(\mathbb{D}_1)$ is generated by the image of the Chebyshev-Frobenius morphism together with the Casimir elements $\alpha_{\partial}^{\pm 1}$ and $\gamma_p$ corresponding to the two boundary components of the annulus $\mathbb{D}_1$. So for  $(g_+, g_-) \in B^+\times B^-$ writing $t:= \tr( g_+ g_-^{-1})\in \mathbb{C}$ and $h_{\partial}\in \mathbb{C}^*$ the lower-right matrix coefficient of $g_+g_-$, an element of the fiber $\pi^{-1}((g_+,g_-))$ of 
the branched covering $\pi: \widehat{\mathcal{X}}^{red}(\mathbb{D}_1) \to {\mathcal{X}}^{red}(\mathbb{D}_1) $ is an element $((g_+,g_-), x, y)$ such that $T_N(x)=t$ and $y^N=h_{\partial}$ so $\pi$ is branched at  these $(g_+, g_-)$ such that $t=\pm 2$. 
  A simple computation shows that 
  $$ \mathcal{Z}((g_+, g_-)) \cong \quotient{\mathbb{C}[X,Y]}{(T_N(X)-t, Y^N -1)} \cong 
  \left\{
  \begin{array}{ll}
  \mathbb{C}^{\oplus N^2} & \mbox{, if }t\neq \pm 2; \\
   \mathbb{C}^{\oplus N} \oplus (\quotient{\mathbb{C}[X]}{(X-1)^2})^{\oplus \frac{N(N-1)}{2}} & \mbox{, if }t=\pm 2.
   \end{array}
   \right.
   $$
   So Theorem \ref{theorem_BG_Ramifications} implies that if $g_+g_-^{-1}\neq \pm \mathds{1}_2$ then
   $$\mathcal{S}_A^{red}(\mathbb{D}_1)_{(g_+,g_-)} \cong 
   \left\{
   \begin{array}{ll}
   Mat_N(\mathbb{C})^{\oplus N^2} & \mbox{, if }t\neq \pm 2; \\
  Mat_N(\mathbb{C})^{\oplus N} \oplus Mat_N(\quotient{\mathbb{C}[X]}{(X-1)^2})^{\oplus  \frac{N(N-1)}{2}} & \mbox{, if }t=\pm 2.
  \end{array}
  \right.
  $$
  So,  for each $\hat{g}=(g_+, g_-, x, y)$ such that $g_+g_-^{-1}\neq \pm \mathds{1}_2$, we have a unique (up to isomorphism) indecomposable weight module $V(\hat{g})$ which, moreover, has dimension $N$ and is simple. And when $\tr(g_+g_-) = \pm 2$ and $x\neq \mp 2$, then $V(\hat{g})$ is a submodule of a larger indecomposable semi-weight module.
  Using the classification in \cite[Appendix B]{KojuQGroupsBraidings}, we can say moreover than when $g_+g_-^{-1}= \pm \mathds{1}_2$  there exists a $N$ dimensional  indecomposable weight module $V(\hat{g})$ with classical shadow $\hat{g}$ as well and that the latter is simple if and only if $x=\mp 2$. If $x\neq \mp 2$, then it contains a simple proper submodule $S(\hat{g}) \subset V(\hat{g})$ and any weight indecomposable module is isomorphic to either a $V(\hat{g})$ or a $S(\hat{g})$.
    
Eventually, under the identification 
   $\mathcal{S}^{red}_A(\mathbb{D}_1)\cong \widetilde{U}_q\mathfrak{gl}_2$ and using the fact that each $\widetilde{U}_q\mathfrak{gl}_2$-module sending the central element $H_{\partial}$ to $1$ (having $y=1$) induces a $U_q\mathfrak{sl}_2$ module, then the simple non-projective modules $S(\hat{x})$ induce the $U_q\mathfrak{sl}_2$-modules defining the colored Jones polynomials, their projective cover $S(\hat{x}) \subset V(\hat{x})$ induce the modules of the ADO invariants, the $V(\hat{x})$ for which $ \tr(g_+g_-^{-1})=  2$ and $\chi_{\hat{x}}(\gamma_p) = - 2$ induces the module of the Kashaev invariant and the others are the diagonal, semi-cyclic and cyclic modules.

  \section{Generalizations to other gauge groups}
  
  It is impossible in one short survey to cover all results about skein algebras and their representations, so we made the restrictive choice to only consider the Kauffman-bracket stated skein algebras, to consider only roots of unity of odd order and to consider only thickened surfaces and not $3$-manifolds in general. Let us say a few worlds about skein algebras for different gauge groups.

  Since the work of Walker \cite{Walker_SkeinCat}, it is known that we can associate to any ribbon category $\mathcal{C}$  skein algebras $\mathcal{S}_{\mathcal{C}}(\Sigma)$ for unmarked surfaces. The ones considered in this paper are the ones associated to Temperley-Lieb categories. Natural choices for $\mathcal{C}$ are, for $G$ a reductive affine Lie group,  the subcategory $\mathrm{Rep}_{\zeta}(G)$ of the category of finite dimensional weight representations of Lusztig integral form $U_q\mathfrak{g}^L\otimes_{q=\zeta_N} \mathbb{C}$ at some root of unity $\zeta$, generated by modules with weights in the lattice defined by $G$, from which we get skein algebras $\mathcal{S}_{\zeta}^{G}(\Sigma)$. The Kauffman bracket skein algebra is then isomorphic to the skein algebra at $G=\SL_2$. It is expected (it is work in progress \cite{CostantinoKojuLe_TannakianSSkein}) that we can extend stated skein algebras to any Tannakian ribbon category and so define stated skein algebras $\mathcal{S}_q^G(\mathbf{\Sigma})$ for any marked surfaces. All properties, like the identification between the algebra of the bigon and $\mathcal{O}_q[G]$, the flatness, the gluing/cutting/excision and fusion formulas, the Chebyshev morphism, the quantum trace and the relation with relative character varieties, are expected to hold for any $G$.
  Here is what is known and claimed:
  \begin{enumerate}
  \item The algebra $\mathcal{S}_q^{\mathbb{C}^*}(\mathbf{\Sigma})$ is trivial, and studied in \cite{KojuQuesneyQNonAb}: it is the quantum torus associated to the quadratic pair $(\mathrm{H}_1(\Sigma, \mathcal{A}; \mathbb{Z}), (\cdot, \cdot))$, where $(\cdot, \cdot)$ is the relative intersection form defined in Definition \ref{def_relative_intersection}. All properties extend to this case.
  \item Higgins has defined in \cite{Higgins_SSkeinSL3} the algebra $\mathcal{S}_q^{\SL_3}(\mathbf{\Sigma})$ using Kuperberg's category of webs. He identified $\mathcal{S}_q^{\SL_3}(\mathbb{B})$ with $\mathcal{O}_q[\SL_3]$, proved the gluing and fusion properties and was able to define explicit bases proving the flatness of the algebras. In particular, all ingredients to use the triangular strategy are present in his work.
  \item L\^e and Sikora have announced \cite{LeSikora_ToAppear} that they will define the algebras  $\mathcal{S}_q^{\SL_N}(\mathbf{\Sigma})$, prove the gluing formula and define a quantum trace using the triangular strategy. A previous step towards the definition of a $\SL_N$ quantum trace was accomplished by Douglas in \cite{Douglas_QTraceSLN}.
  \item Stated skein algebras are closely related to locally constant categories valued factorizaton algebras on surfaces. The theory was studied by Ben Zvi-Brochier-Jordan in \cite{BenzviBrochierJordan_FactAlg1, BenzviBrochierJordan_FactAlg2}. A locally constant factorization algebra is completely determined, up to contractible choices, by its induced representation of the framed little disc operads (by a theorem of Lurie). In this case, such a choice corresponds to the choice of a ribbon category $\mathcal{C}$ and the corresponding factorization algebra can be defined abstractly, as in \cite{AyalaFrancis_FactAlgPrimer}, using a left Kan extension. A more explicit construction was studied by Cooke in \cite{Cooke_FactorisationHomSkein} using skein categories. These are locally constant pre-factorization algebras based on $\mathcal{C}$ and a surface $\Sigma$ for which the excision property (the equivalent for factorization algebras of the descent condition for a pre-cosheaf for being a cosheaf) was proved in \cite{Cooke_FactorisationHomSkein}. 
  The precise derivation of stated skein algebras from factorization algebras was sketched in \cite[Remark 2.21]{GunninghamJordanSafranov_FinitenessConjecture} and will be fully develop by Ha\"ioun in the next forthcoming paper \cite{Haioun_ToAppear}.
  In this language, the skein algebras $\mathcal{S}_q^G(\Sigma, \emptyset)$ for closed surfaces and the stated skein algebras $\mathcal{S}_q^G(\mathbf{\Sigma}_{g,0}^0)$ have been studied in \cite{BenzviBrochierJordan_FactAlg2, GanevJordanSafranov_FrobeniusMorphism}. In particular: (1) the fact that they define quantization deformations of  relative character varieties (the generalization of Theorem \ref{theorem_charvar_skein}) was proved in \cite{BenzviBrochierJordan_FactAlg1, BenzviBrochierJordan_FactAlg2}, $(2)$ the Chebyshev-Frobenius morphism was defined in \cite{GanevJordanSafranov_FrobeniusMorphism} where it appears very naturally directly from Lusztig's Frobenius functor $Fr_A : \mathrm{Rep}(G) \to \mathrm{Rep}_A(G)$, $(3)$ the centers of $\mathcal{S}_q^G(\Sigma, \emptyset), \mathcal{S}_q^G(\mathbf{\Sigma}_{g,0}^0)$ have been proved in \cite{GanevJordanSafranov_FrobeniusMorphism} to be equal to the images of the Chebyshev-Frobenius morphisms and these algebras have been proved to be finitely generated over their center and $(4)$ Theorem \ref{theorem_GJS} was initially stated (in an appropriate more general form) for arbitrary $G$.
  \item  Moreover, all results of Alekseev-Malkin and Hodges-Levasseur of Section \ref{sec_STS} and the results of Brown-Gordon for the bigon were initially stated for arbitrary $G$.
  \end{enumerate}

\bibliographystyle{amsalpha}
\bibliography{biblio}

\end{document}